\documentclass{memo-l}

\usepackage{latexsym,amssymb,amsfonts,amsmath,amsthm,comment,url}
\usepackage[utf8]{inputenc}
\usepackage{xspace}
\usepackage[original]{imakeidx}\makeindex
\usepackage[refpage,notocbasic]{nomencl}\makenomenclature

\usepackage[figuresright]{rotating}
\usepackage{tikz-cd}
\usetikzlibrary{decorations.pathmorphing}
\usetikzlibrary{arrows,arrows.meta}
\tikzcdset{arrow style=tikz, diagrams={>=stealth}}
\usepackage{mathtools}
\usepackage[T1]{fontenc}
\usepackage[inline]{enumitem}
\usepackage{doi}
\usepackage{hyperref}
\hypersetup{colorlinks=true,allcolors=blue}
\usepackage[capitalise]{cleveref}

\usepackage[draft]{optional}
\usepackage{xcolor}
\newcommand{\plan}[1]{}
\newcommand{\BA}[1]{}
\newcommand{\PN}[1]{}
\newcommand{\MS}[1]{}
\newcommand{\DT}[1]{}
\opt{draft}{
\renewcommand{\plan}[1]{\textcolor{blue}{Plan: #1}\PackageWarning{TODO}{TODO: #1}}
\renewcommand{\BA}[1]{\textcolor{orange}{BA: #1}\PackageWarning{TODO}{TODO: #1}}
\renewcommand{\PN}[1]{\textcolor{purple}{PN: #1}\PackageWarning{TODO}{TODO: #1}}
\renewcommand{\MS}[1]{\textcolor{magenta}{MS: #1}\PackageWarning{TODO}{TODO: #1}}
\renewcommand{\DT}[1]{\textcolor{red}{DT: #1}\PackageWarning{TODO}{TODO: #1}}
}

\newcommand{\exo}[1]{#1^{\mathrm{e}}}
\newcommand{\Sig}{\mathsf{Sig}}
\newcommand{\Nat}{\mathbb{N}}
\newcommand{\exosum}{\mathbin{\exo{+}}}
\newcommand{\FinSet}{\mathsf{FinSet}}

\newcommand{\U}{\mathcal{U}}
\newcommand{\T}{\mathcal{T}}
\newcommand{\Tcat}{\mathcal{T_{\text{cat}}}}
\newcommand{\Ustrict}{\exo{\U}}
\newcommand{\Usharp}{\U^{\textrm{sharp}}}
\newcommand{\pathinduction}{$[=]$-induction\xspace}

\newcommand{\converts}{\equiv}
\let\steq\converts
\newcommand{\define}{:\converts}
\newcommand{\bottomlevel}{\bot}
\newcommand{\bottom}[1]{{#1}_\bottomlevel}
\newcommand{\Set}{\ensuremath{\mathsf{Set}}}
\newcommand{\strictiso}{\cong}

\newcommand{\derivcat}[2]{\ensuremath{{#1}'_{#2}}}
\newcommand{\derivdia}[1]{\ensuremath{{#1}'}}
\newcommand{\doublederivdia}[1]{\ensuremath{{#1}''}}

\renewcommand{\L}{\mathcal{L}}
\newcommand{\M}{\mathcal{M}}
\newcommand{\N}{\mathcal{N}}
\newcommand{\F}{\mathcal{F}}

\newcommand{\Struc}[1]{\mathsf{Str}({#1})}
\newcommand{\isIso}{\mathsf{LvlEquiv}}
\newcommand{\isAdjEq}{\mathsf{isAdjEq}}
\newcommand{\isBiInv}{\mathsf{isBiInv}}
\newcommand{\leqv}[1][]{\approxeq_{#1}}
\newcommand{\isEquiv}{\mathsf{isEquiv}}
\newcommand{\refl}{\mathsf{refl}}
 \newcommand{\uStruc}[1]{\mathsf{uStr}({#1})}

\newcommand{\VSS}{\mathsf{SSEquiv}}

\newcommand{\streqv}{\mathsf{SEquiv}}

\makeatletter
\def\twoheadrightarrowfill@{\arrowfill@{\relbar\joinrel\relbar}\relbar\twoheadrightarrow}
\newcommand\xtwoheadrightarrow[2][]{\ext@arrow 0055{\twoheadrightarrowfill@}{#1}{#2}}
\newcommand{\strucequiv}{\ensuremath{\xtwoheadrightarrow{\smash{\raisebox{-1.2mm}{$\scriptstyle\sim$}}}}}
\makeatother

\newcommand{\fanout}[2]{\ensuremath{\mathsf{Fanout}_{#2}(#1)}}
\newcommand{\fanoutfun}[3]{\ensuremath{{#1}_{#3}(#2)}}
\newcommand{\IC}{\mathsf{D\Sig}}

\newcommand{\Lcat}{\L_{\text{cat}}}
\newcommand{\LcatE}{\L_{\text{cat+E}}}

\newcommand{\eqdef}{\define}

\newcommand{\copair}[2]{\langle #1,#2\rangle}

\newcommand{\transport}{\ensuremath{\mathsf{transport}}}

\newcommand{\ob}[1]{\ensuremath{{#1}_0}}

\newcommand{\concat}{\ensuremath{\cdot}}
\newcommand{\idtoiso}{\ensuremath{\mathsf{idtoiso}}}
\newcommand{\idtolvle}{\ensuremath{\mathsf{idtolvle}}}
\newcommand{\idtoindisc}{\ensuremath{\mathsf{idtoindisc}}}
\newcommand{\isiso}{\ensuremath{\mathsf{isiso}}}

\newcommand{\idtovss}{\ensuremath{\mathsf{idtosse}}}
\newcommand{\lvletovss}{\ensuremath{\mathsf{lvletosse}}}
\newcommand{\idtoeqv}{\ensuremath{\mathsf{idtoeqv}}}
\newcommand{\ua}{\ensuremath{\mathsf{ua}}}

\newcommand{\istype}[1]{\mathsf{is}\mbox{-}{#1}\mbox{-}\mathsf{type}}

\newcommand{\nminusone}{\ensuremath{(n-1)}}
\newcommand{\nminustwo}{\ensuremath{(n-2)}}

\newcommand{\defemph}[1]{\textbf{#1}}
\newcommand{\Prop}{\mathsf{Prop}}
\newcommand{\PropU}{\mathsf{Prop}_\U}
\newcommand{\PropUp}{\mathsf{Prop}_{\U'}}
\newcommand{\SetU}{\mathsf{Set}_\U}
\newcommand\C{\mathcal{C}}
\newcommand{\D}{\mathcal{D}}

\newcommand{\foldsiso}{\asymp}
\newcommand{\fiso}{\foldsiso}
\newcommand{\trans}[2]{{#1}_*(#2)}
\newcommand{\transfun}[1]{{#1}_*}
\newcommand{\zerotype}{\mathbf{0}}
\newcommand{\onetype}{\mathbf{1}}
\newcommand{\ttt}{\star}

\newcommand{\hetE}{\widetilde{E}} 

\newcommand{\pro}{\mathsf{pro}\text{-}}

\newcommand{\match}{\mathbb{M}}
\newcommand{\func}[2]{[#1,#2]}
\newcommand{\rfunc}[2]{[#1,#2]_{\mathsf{Rfib}}}

\DeclareFontFamily{U}{min}{}
\DeclareFontShape{U}{min}{m}{n}{<-> udmj30}{}
\newcommand\yon{\!\text{\usefont{U}{min}{m}{n}\symbol{'210}}\!}

\newcommand{\paperorbook}{book\xspace}

\makeatletter

\def\prd#1{\@ifnextchar\bgroup{\prd@parens{#1}}{\@ifnextchar\sm{\prd@parens{#1}\@eatsm}{\prd@noparens{#1}}}}
\def\prd@parens#1{\@ifnextchar\bgroup%
  {\mathchoice{\@dprd{#1}}{\@tprd{#1}}{\@tprd{#1}}{\@tprd{#1}}\prd@parens}%
  {\mathchoice{\@dprd{#1}}{\@tprd{#1}}{\@tprd{#1}}{\@tprd{#1}}}}
\def\@eatsm\sm{\sm@parens}
\def\prd@noparens#1{\mathchoice{\@dprd@noparens{#1}}{\@tprd{#1}}{\@tprd{#1}}{\@tprd{#1}}}
\def\lprd#1{\@ifnextchar\bgroup{\@lprd{#1}\lprd}{\@@lprd{#1}}}
\def\@lprd#1{\mathchoice{{\textstyle\prod}}{\prod}{\prod}{\prod}({\textstyle #1})\;}
\def\@@lprd#1{\mathchoice{{\textstyle\prod}}{\prod}{\prod}{\prod}({\textstyle #1}),\ }
\def\tprd#1{\@tprd{#1}\@ifnextchar\bgroup{\tprd}{}}
\def\@tprd#1{\mathchoice{{\textstyle\prod_{(#1)}}}{\prod_{(#1)}}{\prod_{(#1)}}{\prod_{(#1)}}}
\def\dprd#1{\@dprd{#1}\@ifnextchar\bgroup{\dprd}{}}
\def\@dprd#1{\prod_{(#1)}\,}
\def\@dprd@noparens#1{\prod_{#1}\,}

\def\sm#1{\@ifnextchar\bgroup{\sm@parens{#1}}{\@ifnextchar\prd{\sm@parens{#1}\@eatprd}{\sm@noparens{#1}}}}
\def\sm@parens#1{\@ifnextchar\bgroup%
  {\mathchoice{\@dsm{#1}}{\@tsm{#1}}{\@tsm{#1}}{\@tsm{#1}}\sm@parens}%
  {\mathchoice{\@dsm{#1}}{\@tsm{#1}}{\@tsm{#1}}{\@tsm{#1}}}}
\def\@eatprd\prd{\prd@parens}
\def\sm@noparens#1{\mathchoice{\@dsm@noparens{#1}}{\@tsm{#1}}{\@tsm{#1}}{\@tsm{#1}}}
\def\lsm#1{\@ifnextchar\bgroup{\@lsm{#1}\lsm}{\@@lsm{#1}}}
\def\@lsm#1{\mathchoice{{\textstyle\sum}}{\sum}{\sum}{\sum}({\textstyle #1})\;}
\def\@@lsm#1{\mathchoice{{\textstyle\sum}}{\sum}{\sum}{\sum}({\textstyle #1}),\ }
\def\tsm#1{\@tsm{#1}\@ifnextchar\bgroup{\tsm}{}}
\def\@tsm#1{\mathchoice{{\textstyle\sum_{(#1)}}}{\sum_{(#1)}}{\sum_{(#1)}}{\sum_{(#1)}}}
\def\dsm#1{\@dsm{#1}\@ifnextchar\bgroup{\dsm}{}}
\def\@dsm#1{\sum_{(#1)}\,}
\def\@dsm@noparens#1{\sum_{#1}\,}

\theoremstyle{plain}
\newtheorem{theorem}{Theorem}
\newtheorem{proposition}[theorem]{Proposition}
\newtheorem{lemma}[theorem]{Lemma}
\newtheorem{corollary}[theorem]{Corollary}
\theoremstyle{definition}
\newtheorem{definition}[theorem]{Definition}
\newtheorem{notation}[theorem]{Notation}

\newtheorem{remark}[theorem]{Remark}
\newtheorem{example}[theorem]{Example}

\numberwithin{theorem}{chapter}
\numberwithin{equation}{chapter}
\numberwithin{section}{chapter}

\newcommand{\tobeprovedas}[1]{To be proved as \cref{#1}}

\begin{document}

\frontmatter


\allowdisplaybreaks



\title{The Univalence Principle}

\author{Benedikt Ahrens}
\address{Delft University of Technology, The Netherlands, and University of Birmingham, United Kingdom}
\email{B.P.Ahrens@tudelft.nl}
\thanks{}

\author{Paige Randall North}
\address{Department of Mathematics and Department of Electrical and Systems Engineering, University of Pennsylvania, Philadelphia, Pennsylvania, USA}
\email{pnorth@upenn.edu}
\thanks{}

\author{Michael Shulman}
\address{Department of Mathematics, University of San Diego, San Diego, California, USA}
\curraddr{}
\email{shulman@sandiego.edu}
\thanks{}

\author{Dimitris Tsementzis}
\address{Princeton University and Rutgers University, New Jersey, USA}
\curraddr{}
\email{dimitrios.tsementzis@gmail.com}
\thanks{}

\date{}




\begin{abstract}
  The Univalence Principle is the statement that equivalent mathematical structures are indistinguishable.
  We prove a general version of this principle that applies to all set-based, categorical, and higher-categorical structures defined in a non-algebraic and space-based style, as well as models of higher-order theories such as topological spaces.
  In particular, we formulate a general definition of indiscernibility for objects of any such structure, and a corresponding univalence condition that generalizes Rezk's completeness condition for Segal spaces and ensures that all equivalences of structures are levelwise equivalences.

  Our work builds on Makkai's First-Order Logic with Dependent Sorts, but is expressed in Voevodsky's Univalent Foundations (UF), extending previous work on the Structure Identity Principle and univalent categories in UF.
  This enables indistinguishability to be expressed simply as identification, and yields a formal theory that is interpretable in classical homotopy theory, but also in other higher topos models.
  It follows that Univalent Foundations is a fully equivalence-invariant foundation for higher-categorical mathematics, as intended by Voevodsky.
\end{abstract}

\maketitle

\cleardoublepage
\thispagestyle{empty}
\vspace*{13.5pc}
\begin{center}
 \textit{In Memory of Vladimir Voevodsky}
\end{center}
\cleardoublepage

\setcounter{tocdepth}{0}  

\tableofcontents

\mainmatter

\chapter{Introduction}
\label{sec:intro}

\section{From structuralism to FOLDS}
\label{sec:struc-folds}

\noindent
What is the univalence principle?
Succinctly, it is the statement that
\begin{quote}
  \it Equivalent mathematical structures are indistinguishable.
\end{quote}
The meaning of ``equivalence'' varies with the mathematical structure in question.
For classical set-based structures such as groups, rings, fields, topological spaces, and so on, the relevant notion of equivalence is \emph{isomorphism}.
Every beginning abstract algebra student, for instance, learns that isomorphic groups are indistinguishable from the perspective of group theory: they have all the same ``group-theoretic properties''.
A general statement of this sort can be made using category theory: two isomorphic objects of any category are indistinguishable from the perspective of that category.
Philosophically, this categorical point of view has been advocated as an approach to \emph{mathematical structuralism}~\cite{Benac,Awo96,Awo04}.

However, one of the novel aspects of category theory is that the relevant notion of equivalence for categories \emph{themselves} is \emph{not} isomorphism, but rather \emph{equivalence of categories}.
The student of category theory likewise learns that equivalent categories are indistinguishable from the perspective of category theory, having all the same ``category-theoretic properties''.
One can generalize this to a statement about objects of 2-categories, and so on.

Indeed, nowadays an increasing number of mathematicians work with objects such as $\infty$-categories, for which the appropriate notion of equivalence is even weaker.
Moreover, the need to transfer properties across such equivalences is even greater, since $\infty$-categories have many very different-looking presentations that are used for different purposes.
A popular aim in $\infty$-category theory is to work ``model-independently''; that is, to only use properties and constructions on $\infty$-categories that are independent of the presentation chosen.
This essentially relies on the univalence principle: equivalent presentations of $\infty$-categories should be indistinguishable to $\infty$-category-theoretic operations and properties.

In trying to give precise mathematical expression to these ideas, however, various problems arise.
For instance, can we really make sense of a statement like ``equivalent categories have all the same category-theoretic properties'' as a theorem to prove, rather than merely a definition of what we mean by ``category-theoretic property''?
Certainly there are \emph{some} properties that equivalent categories, or even isomorphic groups, can fail to share; for instance, $G=\{0\}$ and $H = \{1\}$ are isomorphic groups (with their unique group structures) but $0\in G$ while $0\notin H$.

In the case of groups, one answer is that a property is group-theoretic if it is expressible in the formal first-order language of group theory.
This works for other set-based structures, but fails for categories: there is an ordinary first-order theory of categories, but it can distinguish categories up to isomorphism, not just up to equivalence.
This problem was solved by Blanc \cite{Blanc} and Freyd \cite{Freyd}, who devised a syntax for category-theoretic properties, and showed that such properties are invariant under equivalence of categories.
Informally, a property is category-theoretic if it is expressible in a form of first-order logic that never refers to equality of objects.
This makes isomorphic \emph{objects} of a category indistinguishable, and therefore makes equivalent \emph{categories} indistinguishable as well.

Blanc and Freyd's syntax relies on \emph{dependently typed logic} to entirely eliminate equality of objects from the theory of categories.
(The point is that two arrows can only be composed when the domain of one \emph{equals} the target of the other; but by stating composition instead as a \emph{family} of maps $C(y,z) \times C(x,y) \to C(x,z)$ we can avoid referring to equality of objects.)
Makkai~\cite{MFOLDS} generalized this to higher categories, introducing a language for higher-categorical properties called First Order Logic with Dependent Sorts (FOLDS), and proving that FOLDS-properties are invariant under FOLDS-equivalence.

Importantly, Makkai's FOLDS is not specific to $n$-categories (for any value of $n$, including $\infty$), but involves general notions of signature and equivalence for all kinds of higher-categorical structures.
Indeed, while categories and $n$-categories are undoubtedly important, they are really just the most prominent examples of a zoo of categorical and higher-categorical structures that play a growing role in mathematics and its applications.
This zoo includes various kinds of monoidal categories, multicategories, polycategories, fibred categories, enriched categories, enhanced categories, $\dagger$-categories, and many more that are continually being discovered.
But in principle, any such structure, known or yet unknown, can be encoded into Makkai's framework, leading to a notion of equivalence between structures and a notion of categorical property that is invariant under equivalence.

\section{Space-based definitions of higher structure}
\label{sec:space-based}

Makkai's framework is powerful, but it has limitations.
One is that it enforces a particular style of definition for higher categorical structures, which may be called \emph{non-algebraic} and \emph{set-based}.
An algebraic categorical structure is one whose operations (composition, identities, and so on) are specified by functions, as in the standard definitions of category and bicategory.
By contrast, a non-algebraic structure is one whose operations are determined by witnesses that stand in a relation to their inputs and outputs.
For example, the best-known definition of $\infty$-category, the quasi-categories defined by Joyal~\cite{joyal:qcat-kan} and used extensively by Lurie~\cite{Lurie,lurie:ha}, is non-algebraic in that composites of 1-simplices are witnessed by 2-simplices---although the \emph{identities} in a quasi-category are given algebraically by the simplicial degeneracy maps, so a quasi-category is still too algebraic to fit in Makkai's framework.

A quasi-category is also set-based, in the sense that a quasicategory is a collection of \emph{sets} equipped with structure.
The alternative to this is a \emph{space-based} structure, exemplified by Rezk's complete Segal spaces~\cite{rezk:css} (CSS) as a model for $\infty$-categories: a CSS is a collection of \emph{spaces} (in the sense of $\infty$-groupoids, represented by Kan complexes or CW-complexes) equipped with structure.

Of course, in the standard set-theoretic foundations for mathematics, a space is also defined in terms of sets, so a space-based definition can be expanded out to a set-based one.
However, space-based definitions have many advantages, foremost among which is a simpler definition of equivalence: an equivalence of CSS is simply a functor that is a \emph{levelwise} equivalence on each underlying space.
This often leads to better-behaved Quillen model categories (e.g.,~\cite{rezk:theta}), and makes it easier to ``internalize'' by replacing spaces with objects of any sufficiently structured $\infty$-category (e.g.,~\cite{lurie:ha}).
But most importantly for us, it means that the univalence principle for space-based structures can be reduced directly to the univalence principle for spaces: when the constituent spaces of two CSS are indistinguishable, so are the CSS themselves.

However, the space-based approach requires some care to formulate correctly: we can't just copy a set-based definition and make all the sets into spaces, since then there would be superfluous data.
We need to ensure that the ``internal'' homotopical structure of each constituent space coincides with the ``naturally defined'' homotopical structure on its set of points induced by the higher morphisms present in the structure being defined.
In~\cite{rezk:css,rezk:theta} this condition is called \emph{completeness}, and in other contexts it can be viewed as a \emph{stack} condition.
We will call it \emph{univalence}, because it is a ``local'' version of the univalence principle for the \emph{objects} of our space-based structure: if two isomorphic objects of a category are truly indistinguishable, then they should be related by a path in the \emph{space} of objects.

Part of what we achieve in this \paperorbook is to formulate a version of Makkai's theory for space-based structures.
This requires, firstly, giving a general definition of a univalence property for such structures, which requires a general notion of when two \emph{objects} of such a structure are equivalent.
The latter is a nontrivial task because of the generality of the ``structures'' in question: they may have many kinds of ``morphism'' with different shapes and behavior, and it is not always obvious which of these should figure into a notion of equivalence, and how.

Our solution is to take the local univalence principle as a definition: we define two objects to be \emph{indiscernible} if they cannot be distinguished by \emph{any} of the higher ``morphisms'' of the structure in question (see \cref{sec:indis-intro} for more discussion).
In familiar cases this reduces, by a Yoneda-like argument, to familiar notions of isomorphism and equivalence; but it also gives a correct answer in more unusual situations.
Then we will define a structure to be univalent if the path-space between any two objects is equivalent to the space of indiscernibilities between them, and prove that these univalent structures have the good behavior of CSS: the equivalences between them are the levelwise equivalences of underlying spaces.
Finally, we use this to deduce a univalence principle for univalent structures from the univalence principle for spaces: two equivalent univalent structures are indistinguishable by any property expressible in a certain dependently typed language.

\section{Univalent foundations}
\label{sec:uf-intro}

Another limitation of Makkai's framework is that his univalence principle pertains only to \emph{properties}; i.e., statements about a single structure that could be either true or false.
By contrast, mathematics is concerned not just with properties of a single structure, but with \emph{constructions} on objects and relations \emph{between} objects.

With this in mind, and inspired by Makkai (see \cite[p.~1279]{voevodsky_2015}), Voevodsky conceived Univalent Foundations (UF) with a more ambitious goal: a foundational language for mathematics, \emph{all} of whose constructions are invariant under equivalences of structures.
Since in UF proofs are particular constructions, this implies a similar invariance of properties.

For a foundational language for mathematics to satisfy such a ``global'' univalence principle, we must exclude any properties such as $0\in G$ mentioned above.
Like the structure-specific languages of Blanc, Freyd, and Makkai, UF achieves this using dependent types.
In fact UF is a form of homotopy type theory (HoTT), and so we often speak of ``HoTT/UF''.
We give a brief review of HoTT/UF in \cref{sec:review-uf}.

More than this is true, however.
The univalence principle is a generalization of the \emph{indiscernibility of identicals}---the statement that \emph{equal} objects have the same properties---in which equality is replaced by a suitable notion of equivalence.
However, in a foundational language for mathematics, the converse \emph{identity of indiscernibles} generally also holds automatically, because of the presence of haecceities.
Philosophically, a \emph{haecceity}\index{haecceity} is the property of an object being itself; mathematically we use it to refer to the property of ``being equal to $x$'', for some fixed $x$.
Now if $x$ and $y$ have \emph{all} the same properties, then this applies also to this property of ``being equal to $x$''; hence, since $x$ is equal to $x$, also $y$ is equal to $x$.

This means that in any foundational language satisfying the global univalence principle, \emph{equivalent structures must be equal}.
This may seem impossible, but HoTT/UF achieves it by expanding the notion of ``equality'', enabling it to carry information and coincide with equivalence.
The basic objects of HoTT/UF are \emph{types}, which behave not like discrete sets but like spaces in homotopy theory, and the foundational notion of ``equality'' behaves like \emph{paths} in such a space.
(We generally refer to this notion of equality as ``identification'', to avoid the conceptual baggage that comes along with words like ``equality'' and ``path''.)

An essential feature of HoTT/UF is Voevodsky's \emph{univalence axiom}\index{axiom!univalence}
\begin{equation}
 \mathsf{univalence} : \forall{(X, Y: \U)},\;  (X = Y)  \simeq  (X \simeq Y),
 \label{eq:ua}
\end{equation}
which says that for any two types (e.g., sets) $X$ and $Y$, the \emph{identification type} $X=Y$ is equivalent to the \emph{equivalence type} $X\simeq Y$.
By the indiscernibility of identicals, this then implies that equivalent types are indistinguishable by \emph{all} properties, and also all constructions.
In other words, in HoTT/UF, the univalence principle for single types (the most trivial sort of ``mathematical structure'') holds essentially by definition.

It was observed by Coquand and proven in~\cite{COQUAND20131105} and in~\cite[Section~9.8]{HTT}\footnote{The formalization of \cite{COQUAND20131105} compares the two independent results.} that this single postulate implies that the same kind of strong univalence principle also holds for a wide range of set-based mathematical structures such as groups and rings.
Namely, the type of identifications between two structures is equivalent to the type of \emph{isomorphisms} between them, and therefore isomorphic structures are indistinguishable by all properties and constructions.
This latter result has become known as the \emph{Structure Identity Principle}, a term coined by Aczel~\cite{Aczel_SIP}.
As pointed out in~\cite{Awo13} and~\cite{tsem-ufassf}, philosophically this can be viewed as the ultimate formulation of mathematical structuralism.
Similarly, in \cite[Theorem~6.17]{AKS13} the univalence principle was proven for categories: identifications of univalent categories (those satisfying the ``local'' univalence principle for objects, as in space-based structures such as CSS) are equivalent to equivalences of categories.

In this \paperorbook, we generalize these results to other higher-categorical structures.
Using a general notion of structure inspired by Makkai's, and the notions of indiscernibility of objects and local univalence described earlier, we show in HoTT/UF that identifications of univalent structures are equivalent to equivalences of structures.

Thus, equivalent univalent structures are indistinguishable in HoTT/UF: not only logical properties, but also all mathematical constructions, are invariant under such equivalence.
This result, formulated precisely as the centerpiece of this work in \cref{thm:hsip2}, goes a long way towards showing that HoTT/UF really does realize Voevodsky's goal of a fully equivalence-invariant foundation.

The connection between our \cref{thm:hsip2} and the (set-theoretic) space-based approaches discussed in \cref{sec:space-based} is that the former yields the latter by way of a model construction.
Specifically, in the simplicial model of HoTT/UF constructed by Voevodsky~\cite{KL12} in the model category of simplicial sets, the types of HoTT/UF are interpreted by Kan complexes, i.e., spaces of homotopy theory.
Thus, any structure defined inside HoTT/UF immediately yields a space-based structure in set-theoretic mathematics.
Hence, proving the internal univalence principle for structures in HoTT/UF immediately entails an analogous theorem for space-based structures in simplicial sets.

Even if we were not interested in HoTT/UF for its own sake, the type-theoretic approach also has other advantages over working more explicitly with space-based structures in homotopy theory.
For instance, the native presence of homotopy theory makes it easy to incorporate higher homotopy types in the definition of our structures, in particular allowing group actions and higher group actions to appear even in our non-algebraic context.
Type-theoretic \emph{universes} provide an extremely convenient language for working with classifying spaces, allowing us to view their points as literally being the objects they classify.
And type-theoretic arguments can much more easily be verified for correctness using a computer proof assistant.

The inductive nature of type-theoretic arguments also suggests useful new abstractions.
For instance, by isolating those properties of Makkai's signatures that are essential for our arguments, we are led ineluctably to a more general notion of signature.
These ``functorial signatures'', defined in \cref{sec:abstract-signatures}, turn out to be general enough to encompass \emph{higher}-order logic, including structures such as topological and uniform spaces and suplattices within our theory.
Moreover, unlike ordinary higher-order logic, functorial signatures carry enough data to determine a non-invertible notion of morphism of structures, which specializes to the correct notions in examples such as continuous or sup-preserving maps.

A final, very significant, advantage of HoTT/UF is that simplicial sets are not its only model.
Although the type-theoretic language sounds and feels as if we were talking about ordinary spaces, working with concrete points and paths between them, there is nevertheless a machine that ``compiles'' this language to yield definitions and theorems that make sense in any $\infty$-topos.\footnote{For a long time, some of the pieces of this machine were missing from the literature; indeed, resolving one of these was the last project Voevodsky was working on.
However, with the appearance of~\cite{shulman:univinj,bblm:initiality} (though still unpublished), all the pieces of relevance to the current \paperorbook seem to be resolved.}

Thus, the advantage of space-based structures mentioned in \cref{sec:space-based} that they can be more easily internalized in other $\infty$-categories is achieved \emph{automatically} if such structures are defined internally in HoTT/UF.
For instance, in this way the univalent categories of~\cite{AKS13} can be interpreted as \emph{stacks} of 1-categories over any site, and similarly for other higher-categorical structures.

\section{Indiscernibility}
\label{sec:indis-intro}

We now say a few more words about the notion of \emph{indiscernibility} for objects of a structure that underlies all of our work.
The name comes from the fact that it is a relativization of the \emph{identity of indiscernibles} to a particular structure.
We mentioned above that because of haecceities, identity of indiscernibles (two objects with all the same properties are identical) is automatic in a foundational theory.
But if we restrict the ``properties'' in question to those expressible in terms of a particular structure, we obtain a nontrivial notion that turns out to specialize to a correct definition of isomorphism/equivalence for objects of any categorical structure.

We can already see this in operation for set-level structures.
Consider the example of a preordered set: a set $P$ together with a binary relation $\leq$ that is reflexive and transitive.
The univalence principle of \cite[Section~9.8]{HTT} for preordered sets says that the type of isomorphisms between two preordered sets (i.e., isomorphisms of underlying sets respecting the relation) is equivalent to the type of identifications.

However, the univalence principle for preordered sets we arrive at in this \paperorbook is somewhat different, based on the above notion of indiscernibility.
In the case of a preordered set $P$, two elements $x,y$ of $P$ are \emph{indiscernible} if they behave in exactly the same way: that is, $x \leq z$ iff $y \leq z$ for all $z$, and $z \leq x$ iff $z \leq y$ for all $z$ in $P$.
But this is equivalent to saying that $x \leq y$ and $y \leq x$.

Since our general notion of equivalence involves indiscernibilities, we then find that two preordered sets are \emph{equivalent} if there are functions between the underlying sets $f: P \leftrightarrows Q : g$, respecting the relations, such that $ gf(x) $ is not necessarily equal to, but \emph{indiscernible} from, $x$ and likewise $fg (y) $ is indiscernible from $y$.
For the univalence principle to hold for this notion of equivalence, it must be that such equivalences coincide with isomorphisms, which means that indiscernibilities must coincide with equalities.
But this says precisely that $P$ is antisymmetric, i.e., a partial order.

Of course, this is nothing but a specialization of the notions of isomorphism of objects and equivalence of categories, when preorders are regarded as categories with at most one morphism between any two objects.
For a more novel example, consider topological spaces: sets $X$ together with a subset $\mathcal O$ of their powerset $\mathcal P (X)$ satisfying suitable axioms.
It turns out that two points $x,y $ of a topological space $X$ are indiscernible when $x \in U$ iff $y \in U$ for every open set $U$ in $\mathcal O$.
Then a \emph{univalent} topological space is exactly a $T_0$-space, and we find that for two univalent topological spaces, their identification type is equivalent to the type of equivalences up to indiscernibility.

Our notion of indiscernibility is inspired by Makkai's ``internal identity'' for objects of a FOLDS-structure; see, e.g., \cite{makkai-blpc}.
Other notions related to our indiscernibility have appeared elsewhere in the literature.
For instance, Levy \cite{DBLP:conf/popl/Levy17} studies isomorphism of types in simply-typed lambda calculi with effects, defining a notion of ``contextual isomorphism'' that is similar in spirit to our indiscernibilities; see \cref{sec:unnat} for a few more details.

We will argue that indiscernibility gives a correct notion of equivalence between objects of \emph{all} categorical and higher-categorical structures, when properly formulated; and that the resulting \emph{univalent structures} (where indiscernibility coincides with equality) are usually the correct notion of such structures to work with in HoTT/UF, and likewise when interpreted into homotopy theory yield the correct space-based definitions of such structures.
To that end, in \cref{sec:examples,sec:egs-higherorder} we will survey a large number of categorical structures and their notions of indiscernibility and univalence.
More broadly, we view this \paperorbook as laying out the foundations of a general approach to categorical and higher-categorical structures in HoTT/UF.

\section{Limitations of our theory and comparison to other work}
\label{sec:limitations}

In this \lcnamecref{sec:limitations}, we explain what we do \emph{not} achieve in this work, and how our work relates to other work on the univalence principle.

In \cref{sec:conclusion} we will discuss a number of open problems for future work, but to avoid disappointing the reader, we want to mention at the outset a few important ways in which our current theory is incomplete.
The first is that at present we consider only structures of \emph{finite} categorical dimension, e.g., $n$-categories for finite $n$ but not $\infty$-categories, $(\infty,n)$-categories, or $(\infty,\infty)$-categories.
We hope that our definitions and results should have infinite-dimensional analogues, but since $\infty$ introduces unique complications (e.g., there is more than one candidate notion of $(\infty,\infty)$-category, depending on whether the equivalences are defined ``inductively'' or ``coinductively'') we have chosen to begin with the simpler finite-dimensional case.

The second is that our framework is still, like Makkai's, entirely non-algebraic.
The ambient structure of HoTT/UF certainly allows, and even encourages, us to define operations and algebraic structures.
However, at present our notion of indiscernibility is only defined for purely ``relational'' or non-algebraic structures.
It is always possible to encode operations non-algebraically, as we show in many examples, but it would be preferable not to have to go through this encoding step manually.

Furthermore, the notion of equivalence of structures that appears in our univalence principle is what may be called a \emph{strong} equivalence, which in the case of 1-categories specializes to a pair of functors in both directions together with natural isomorphisms relating their composites to identities.
However, the univalence principle for 1-categories of~\cite{AKS13} also applies to \emph{weak} equivalences---single functors that are fully faithful and essentially surjective---which implies, in particular, that any weak equivalence between univalent 1-categories is a strong equivalence.
(Indeed, this is one of the reasons that univalent categories are the ``good'' notion of category when working in HoTT/UF.)
We can define a notion of weak equivalence of arbitrary structures, but we have been unable to extend our univalence principle to such equivalences (though we have no counterexample either).

Previous work on the Structure Identity Principle \cite{COQUAND20131105, HTT} has already been mentioned.
One the one hand, our results are much more general than those of~\cite{COQUAND20131105} and~\cite[Section~9.8]{HTT} in that they also include higher-categorical and higher-order examples.
On the other hand, they are also less general in that they apply only to structures built from types in a way specified by a kind of signature, rather than to additional structure of any sort added to objects of any category.
See also \cref{rem:diff-sip-up}.

The relationship between FOLDS and HoTT/UF has also been explored, from a different perspective, in~\cite{tsem-foliso}, wherein an extension of the syntax for FOLDS is developed that interprets ``natively'' into the HoTT/UF notion of equality.
A broader investigation (not focused on HoTT/UF) of the relation between FOLDS (and dependently-typed first-order logic more generally) and the semantics of dependent type theories was pursued in~\cite{palmgren2019categories}.

\section{Structure of this work}
\label{sec:synopsis}

This work is structured in three \lcnamecrefs{sec:general-theory}, preceded by an introduction to univalent foundations and two-level type theory in \cref{sec:review-uf}.

\subsection{Coarse structure}

In this work, we introduce two notions of \emph{theory} and \emph{model of a theory}.
Given a theory $\T$ and two models $M$ and $N$ of $\T$, there are three notions of ``sameness'' for them:
\begin{enumerate}
 \item Identification $M = N$;
 \item Levelwise equivalence $M \leqv N$;
 \item Equivalence $M \simeq N$.
\end{enumerate}
We first show
that Voevodsky's univalence axiom entails that identifications coincide with levelwise equivalence for any models $M$ and $N$.
We then prove
that for \emph{univalent} structures $M$ and $N$, levelwise equivalence also coincides with equivalence.
The composition of these results yields our \textbf{Univalence Principle}:
for univalent models,
\[ (M = N) \simeq (M \simeq N) .\]

In \cref{sec:general-theory}, we introduce \emph{diagram theories} and their \emph{models}, and state our main definitions and results for such diagram theories.

In \cref{sec:examples}, we study many examples of diagram theories and compare the indiscernibilities and equivalences in these examples to the usual notions of sameness and equivalence.

In \cref{sec:ho}, we introduce \emph{functorial theories}---generalizing diagram theories---and their \emph{models}. We give all definitions in detail and state and prove our results for such functorial theories.
We also give a translation from diagram theories to functorial theories.

\subsection{Fine structure}

\Cref{sec:general-theory} is dedicated to \emph{diagram theories} and their \emph{univalent models}.
Throughout this \lcnamecref{sec:general-theory}, most proofs, and even some definitions, are deferred to \cref{sec:ho}, where we study a more general notion of theory (functorial theories, \cref{def:abstract_signature}) and their models.

Before defining diagram theories in general, we start out, in \cref{sec:folds-cats}, by considering an example diagram theory in detail: the theory of categories.
We compare there our notion of model of the theory of categories with a more traditional definition of categories, our notion of indiscernibility of objects in a model with categorical isomorphism, and equivalence of models with categorical equivalence.
This chapter also serves to connect our work to existing ideas in the HoTT/UF literature, for those readers familiar with the latter.

In \cref{sec:folds-signatures-two}, we introduce \emph{diagram theories} and their \emph{models}.
Readers with backgrounds and interests in mathematical/categorical logic may want to pay special attention to this chapter, since it introduces the ``syntax'' and ``semantics'' of the system in which we will be stating our results.
This system is different from traditional logic in certain ways, e.g., our signatures are best thought of as categories, our notion of ``proposition'' is defined ``semantically'' rather than syntactically, etc.

In \cref{sec:indisc-folds}, we sketch our definitions of \emph{indiscernibility} and \emph{univalence} of models of a diagram theory.

In \cref{sec:hsip-folds}, we state our Univalence Principle for the special case of diagram theories.
This section is intended to give a fairly concrete understanding of our definitions and results in the setting most familiar to working geometers, topologists, and homotopy theorists, who may be less interested in the detailed inductive structures required for the proofs in \cref{sec:ho}.

In \cref{sec:examples}, we present diagram theories for many mathematical structures.
We usually spell out the signatures explicitly, but describe the axioms only informally.
However, most axioms could be formally stated in the language of FOLDS described in \cref{sec:foldssig-by-eg}, and thus obtained via the translation sketched in \cref{rem:axioms-from-folds}.

Specifically, we start in \cref{sec:set-egs} by considering theories whose underlying signatures are of height less or equal to $2$; univalent models are then sets equipped with some structure.

In \cref{sec:1cat-egs} we consider theories for categories with extra structure, such as certain limits, functors, natural transformations, multicategories, categorical structures for the interpretation of type theories, and many others.
These structures can all be defined and studied with essentially the same technology used for ordinary categories in~\cite{HTT}, but our framework provides a unifying perspective.

In \cref{sec:hcat-egs} we study theories for higher-categorical structures, such as bicategories and double categories.
This is the place where we first start to see the real advantages of our framework for doing higher category theory in HoTT/UF.
While the particular case of univalent bicategories have been previously studied by~\cite{DBLP:conf/rta/AhrensFMW19}, our theory provides a general machine for defining these and other higher-categorical structures.

The structures considered until this point are \emph{weak} in the higher-categorical sense.
In \cref{sec:strict-egs} we show how to encode \emph{strict} categorical structures;
importantly, they are obtained by adding, to the weak theories, additional structure and properties.

In \cref{sec:graph-egs} we study theories with signatures of height 3 that are not categorical, i.e., of theories of objects and arrows but without composition or identities, such as directed multigraphs and Petri nets.
These are interesting mainly as counterexamples, and to explore the boundaries of the generality of our theory.

In \cref{sec:restr-indis-egs} we present theories of ``enhanced'' (higher) categories, that is, categories with additional structure that is not categorical---such as a $\dagger$-structure.
We show that, in such structures, our notion of indiscernibility coincides with a well-known notion of ``good'' isomorphism.

In \cref{sec:unnat} we study theories involving, in particular, object-only functors and unnatural transformations.
Such structures frequently arise in the study of semantics of programming languages.
Again, our notion of indiscernibility coincides, for these examples, with well-known notions of ``good'' isomorphisms for these examples.

Throughout \cref{sec:examples}, we often omit the adjective ``diagram''; by ``signatures'' and ``theories'', we always mean diagram signatures and diagram theories, as opposed to the functorial signatures and diagrams of \cref{sec:ho}.

In \cref{sec:ho}, we start out, in \cref{sec:abstr-sign-transl}, with an in-depth study of diagram signatures---specifically, the notion of ``derivation'' of such signatures that was introduced in \cref{sec:general-theory}.
The results of this study suggest a more general, (co)inductive definition of signatures: our \emph{functorial signatures} of \cref{def:abstract_signature}.
These signatures, and their corresponding structures, are easier to reason about in the abstract, and we prove most of the statements of \cref{sec:general-theory} only for structures of functorial signatures.
At the same time, the study immediately yields a translation of diagram signatures into functorial signatures, made explicit in \cref{thm:translation}.
We conclude this \lcnamecref{sec:abstr-sign-transl} with the definition of \emph{functorial theories} and their \emph{models}, in complete analogy to diagram theories and their models.

\Cref{sec:structures} is dedicated to the study of levelwise equivalence of functorial structures. We prove here that identifications coincide with levelwise equivalences.

In \cref{sec:FOLDS-iso-uni}, we define notions of indiscernibility and univalence for structures of functorial signatures, and for models of functorial theories.
We then prove two results about the homotopy levels of structures and models.

\Cref{sec:hsip} is dedicated to the proof of the main result of our work.
Models of functorial theories again admit three notions of sameness; the main result of this work,  \cref{thm:hsip2}, shows that for univalent models, all three coincide.

We conclude this work, in \cref{sec:egs-higherorder}, with some examples of functorial theories that are not, to our understanding, expressable as diagram theories.

\section{Version history}

An extended abstract for this \paperorbook was published in the conference proceedings of LICS 2020 \cite{hsip_lics}.
There, diagram signatures and functorial signatures were called FOLDS-signatures and abstract signatures, respectively.
We have also changed the title and all other instances of the phrase `Higher Structure Identity Principle' to `Univalence Principle' in this \paperorbook.

Compared to the extended abstract, this \paperorbook contains some new results:
\begin{itemize}
\item a comparison of $\L$-structures with Reedy-fibrant diagrams, see \cref{thm:reedy-struc}, and
\item a variant of the univalence principle for \emph{essentially} split-surjective equivalences, see \cref{thm:hsip2}.
\end{itemize}
We also benefit from the additional space provided here to make our exposition more pedagogical; specifically, we give
\begin{itemize}
\item an overview of our results in the special case of diagram signatures (see \cref{sec:general-theory}) and a translation from diagram signatures to functorial signatures in \cref{sec:abstr-sign-transl}, and
\item many more examples of diagram signatures (in \cref{sec:examples}) and functorial signatures (in \cref{sec:egs-higherorder}).
\end{itemize}

\newcommand{\grantsponsor}[3]{#2}
\newcommand{\grantnum}[2]{#1}

\section{Acknowledgments}

  Nicolai Kraus provided helpful advice on 2LTT.
  Paul Blain Levy pointed out a possible connection to his work on contextual isomorphisms, and provided helpful comments on an earlier version.
  Elif Uskuplu pointed out typos in a previous version of this work.
  We are very grateful to all of them for their input.
  We furthermore thank the anonymous referees of the LICS version, and of the present AMS Memoir, for their helpful suggestions.

  Ahrens and North acknowledge the support of the \grantsponsor{}{Centre for Advanced Study (CAS)}{} in Oslo, Norway, which funded and hosted the research project \emph{\grantnum{}{Homotopy Type Theory and Univalent Foundations}} during the 2018/19 academic year.

This work was partially funded by \grantsponsor{}{EPSRC}{https://epsrc.ukri.org/} under agreement \grantnum{EP/T000252/1}{EP/T000252/1}.

This material is based
on research sponsored by \grantsponsor{id}{The United States Air Force Research
  Laboratory}{https://www.wpafb.af.mil/AFRL/} under agreement number \grantnum{FA9550-15-1-0053}{FA9550-15-1-0053}, \grantnum{FA9550-16-1-0212}{FA9550-16-1-0212}, \grantnum{FA9550-17-1-0363}{FA9550-17-1-0363},
\grantnum{FA9550-21-1-0009}{FA9550-21-1-0009}, and \grantnum{FA9550-21-1-0334}{FA9550-21-1-0334}. The
U.S.\ Government is authorized to reproduce and distribute reprints
for Governmental purposes notwithstanding any copyright notation
thereon. The views and conclusions contained herein are those of
the authors and should not be interpreted as necessarily
representing the official policies or endorsements, either
expressed or implied, of the United States Air Force Research
Laboratory, the U.S. Government, or Carnegie Mellon University.

  This material is based upon work supported by the
  \grantsponsor{GS100000001}{National Science
    Foundation}{http://dx.doi.org/10.13039/100000001} under Grant
  No.~\grantnum{DMS-1554092}{DMS-1554092}.
  Any opinions, findings, and
  conclusions or recommendations expressed in this material are those
  of the authors and do not necessarily reflect the views of the
  National Science Foundation.

\chapter[Introduction to HoTT/UF]{Introduction to two-level homotopy type theory and univalent foundations}
\label{sec:review-uf}

In this \lcnamecref{sec:review-uf} we give a brief introduction to the formal language of Homotopy Type Theory and Univalent Foundations (HoTT/UF), including Two-Level Type Theory (2LTT).
Semantically, HoTT/UF can be viewed as a convenient syntax for working with a Quillen-model-category-like structure, namely a category equipped with a class of morphisms called ``fibrations'' (the cofibrations and weak equivalences of a model category are not represented explicitly in the syntax, although they play a role in constructing concrete interpretations).\index{fibration!in a model category}\index{cofibration!in a model category}
For intuition, the reader is free to think of this as the category of topological spaces.
The formal system used in~\cite{HTT} assumes that all objects are fibrant (e.g., the category of fibrant objects of a model category); we will instead use a form of two-level type theory~\cite{2LTT} that doesn't make this assumption.

The formal language of HoTT/UF is based on Martin-L\"{o}f Type Theory (or MLTT for short), which has, as primitive objects, \defemph{types}\index{type} and \defemph{elements}\index{element} (a.k.a.\ \defemph{terms}).\index{term}
Each element is an element of a specific type.
We write $t : T$ to say that $t$ is an element of the type $T$, e.g., we write $2 : \Nat$ to say that $2$ is an element of the type of natural numbers.
We write $s,t:T$ to abbreviate ``$s:T$ and $t:T$''.

Generally speaking, types and their elements are used in the same way that sets and their elements are used in a ZFC-based foundation for mathematics.
At any given point in a mathematical construction or proof, we are working in a \defemph{context}\index{context} containing a number of \defemph{variables}\index{variable}, each declared to belong to a particular type (which we also write $x:T$).
The types and elements we then construct can depend on the values of these variables, for which particular elements of the appropriate types can later be substituted.

Semantically, types that don't depend on any variables represent objects of the category, and terms that don't depend on any variables represent global elements (morphisms out of the terminal object).
More generally, a term of type $T$ depending on variables $x:A$ and $y:B$, say, represents a morphism $A\times B \to T$.
Likewise, a type depending on $x:A$ and $y:B$ represents an object $P$ equipped with a morphism $P \to A\times B$, i.e., an object of the slice category over $A\times B$.
A term belonging to such a type dependent on $x:A$ and $y:B$ represents a \emph{section} of $P \to A\times B$, i.e., a morphism $A\times B \to P$ such that the composite $A\times B \to P \to A\times B$ is the identity.

Two-level type theory~\cite{2LTT} enhances this picture by distinguishing between \emph{fibrant types}\index{type!fibrant} and \emph{non-fibrant types}.
Semantically, a non-fibrant type in some context represents an arbitrary morphism into that context, while a fibrant type represents a morphism into the context that is a ``fibration'' of some sort.
Importantly, when regarding HoTT/UF as a foundational language for mathematics, \emph{only the fibrant types} are ``true'' mathematical objects; the non-fibrant types should be regarded as a sort of ``internalized metatheory'' making it easier to reason generically about the fibrant ones.
In line with this philosophy, we adopt a suggestion of Ulrik Buchholtz and refer to non-fibrant types as \defemph{exotypes}\index{exotype}, reserving the unadorned word \emph{type} to refer to the fibrant ones (though we sometimes retain the adjective ``fibrant'' for emphasis).
All types are exotypes, but not all exotypes are types.

The syntax of 2LTT is described in \cite[\S2.1]{2LTT}; in the rest of this \lcnamecref{sec:review-uf} we summarize it.
A reader already familiar with 2LTT can skip most of this \lcnamecref{sec:review-uf}, except for \cref{sec:sharp} where we introduce an apparently-new notion of \emph{sharp exotype}.
The upshot of \crefrange{sec:basic-types}{sec:cofibrancy} is that we work in 2LTT with axioms M2 (Russell universes), T1--T3 (strictness of conversion), and A5 (exo-equality reflection) from~\cite[Section~2.4]{2LTT}.\index{axiom!of 2LTT}
We will briefly mention the semantics of these axioms in \cref{sec:semantics}, but our main purpose in assuming them is expositional clarity; e.g., it allows us to develop our theory with \emph{two} notions of equality rather than three.
None of our results depend essentially on the axioms, and in particular should be just as valid if exo-equality only satisfies the UIP axiom.
Notationally, when an operation has both a fibrant version and an exo-version, we decorate the \emph{exo-version} with a superscript ``e'' (for ``exo-''), rather than decorating the fibrant version as in~\cite{2LTT}.

\section{Basic types and type constructors}
\label{sec:basic-types}

Firstly there are a few basic types: an empty type $\zerotype$, a singleton type $\onetype$ with unique element $\ttt:\onetype$, and a type of natural numbers $\Nat$, with suitable term constructors.\nomenclature[0]{$\zerotype$}{empty type}\nomenclature[1]{$\onetype$}{singleton type}\nomenclature[1]{$\ttt$}{unique term of $\onetype$}\nomenclature[N]{$\Nat$}{type of natural numbers}
Moreover, given exotypes $A$ and $B$, we can form the function exotype $A \to B$ of functions from $A$ to $B$ and the product exotype $A \times B$ of pairs of elements from $A$ and $B$, each of which is a type (i.e., fibrant) if $A$ and $B$ are.\nomenclature[\]{$A \to B$}{(exo)type of functions between (exo)types $A$ and $B$}\nomenclature[\]{$A \times B$}{product (exo)type of (exo)types $A$ and $B$}

The situation for disjoint sums is somewhat different, corresponding to the fact that fibrations are not in general closed under such sums: two exotypes $A$ and $B$ have a disjoint sum exotype $A \exosum B$, but it is not generally a type even if $A$ and $B$ are; instead for two types $A$ and $B$ we have a distinct disjoint sum type $A + B$.\nomenclature[\]{$A \exosum B$}{disjoint sum exotype of $A$ and $B$}\nomenclature[\]{$A + B$}{disjoint sum type of types $A$ and $B$}
Similarly, there is an empty exotype $\exo{\zerotype}$ and an exotype of exo-natural numbers $\exo{\Nat}$ that may not coincide with $\zerotype$ and $\Nat$ (since the initial object and natural numbers object may not be fibrant).\nomenclature[0]{$\exo{\zerotype}$}{empty exotype}\nomenclature[n]{$\exo{\Nat}$}{exotype of exo-natural numbers}
What distinguishes these type/exotype pairs is that the fibrant versions can only be mapped out of into other fibrant types; e.g., there is a unique map $\zerotype \to A$ for any type $A$, but if $A$ is only an exotype we can only get a unique map $\exo{\zerotype} \to A$.
Similarly, we can define functions into any exotype by recursion on $\exo{\Nat}$, but to define a function by recursion on $\Nat$ we must know that the target is fibrant.

There is also a ``type of types'', called a universe\index{universe}, and denoted by $\U$. Its elements are types, e.g., $\Nat : \U$.\nomenclature[u]{$\U$}{universe}
This allows us to specify a type \emph{family}, parametrized by elements of a type, say, $A$, as a function, say, $B : A \to \U$, from the parametrizing type $A$ into the universe.
We also have an exo-universe $\Ustrict$, an exotype whose elements are exotypes.\nomenclature[u]{$\Ustrict$}{exo-universe}

Semantically, $\U$ is a classifying space of fibrations, so that $B:A\to \U$ can also be regarded as a fibration with codomain $A$, whose fibers are the types $B(x)$.
Similarly, $\Ustrict$ is a classifying space of arbitrary maps (not necessarily fibrations).
To avoid paradoxes \`{a} la Russell, there is a hierarchy of such universes, but we sweep this detail under the rug with the conventional typical ambiguity~\cite[Section~1.3]{HTT}, using $\U$ and $\Ustrict$ to denote unspecified universes.

The type constructions $\to$ and $\times$ generalize to (exo)type families.
Specifically, given $A : \Ustrict$ and $B : A \to \Ustrict$,
\begin{itemize}
 \item We can form the exotype $\prd{a : A}B(a)$ of dependent functions. An element $f : \prd{a : A}B(a)$ is a function that returns, on an input $a : A$, an element of the type $B(a)$.
  \nomenclature[p]{$\prd{a : A}B(a)$}{(exo)type of dependent functions from $A$ to $B$}
   If $b[x]$ is a term containing a variable $x : A$, then we can build the function $\lambda x. b[x] : \prd{a : A}B(a)$.\footnote{This is the type-theoretic way to write the function $x \mapsto b[x]$.}\nomenclature[l]{$\lambda x. b[x]$}{possibly-dependent function sending input $x$ to output $b[x]$}
   Function application is written $f(a)$ as usual, or sometimes just $f\, a$.
 \item We can form the exotype $\sm{a : A}B(a)$ of dependent pairs. An element $s : \sm{a: A}B(a)$ consists of two components, where the first component $\pi_1(s) : A$ and the second component $\pi_2(s) : B(\pi_1(s))$. Given $a : A$ and $b : B(a)$, we write $(a,b) : \sm{a : A}B(a)$.\nomenclature[s]{$\sm{a : A}B(a)$}{(exo)type of dependent pairs consisting of elements from $A$ and $B$}\nomenclature[p]{$\pi_1$}{first component of a dependent pair}\nomenclature[p]{$\pi_2$}{second component of a dependent pair}\nomenclature[\]{$(a,b)$}{dependent pair of $a:A$ and $b:B(a)$}
   When $B$ is regarded as a fibration or map into $A$, then $\sm{a : A}B(a)$ is the domain of this map.
\end{itemize}
Both $\prd{a : A}B(a)$ and $\sm{a : A}B(a)$ are fibrant types if $A$ and $B$ are, i.e., if $A:\U$ and $B:A\to \U$.

A function of two variables can be expressed as usual as $f:A\times B\to C$, but it is more usual to write it in ``curried'' form $f:A\to (B\to C)$, the type of which we abbreviate as $A\to B\to C$.
Similarly, a function of two dependent variables can be expressed either as $\big(\sm{x:A} B(x)\big) \to C$ or as $\prd{x:A} \big(B(x) \to C\big)$, and we write the latter curried form as $\prd{x:A} B(x) \to C$.
Of course, $C$ could also depend on $x:A$ and $y:B(x)$.

\section{Notions of identity}
\label{sec:identity}

Given two elements $a,b: A$ of the same (exo)type, we can ask whether they are equal.
In fact, there are two different ways to ask this question.

The first, called \emph{strict equality} or \defemph{exo-equality}\index{equality!exo-}, and written $a\steq b$, corresponds semantically to ``point-set level'' equality, i.e., actual equality of objects or morphisms in a model category.\nomenclature[\]{$a \steq b$}{exo-equality between $a$ and $b$}
This equality is a congruence for all constructions in type theory.
In particular, it is \defemph{convertible},\index{convertible} meaning that if we have $a\steq b$ and an (exo)type family $P:A\to \Ustrict$, then any element $u:P(a)$ \emph{is itself also} an element of $P(b)$, i.e., $u:P(b)$.

Aside from the rules making it a congruence, exo-equality has more ``interesting'' generators such as $\pi_1(a,b) \converts a$ and $\pi_2(a,b) \converts b$, and $(\lambda x.b[x])(a) \converts b[a]$.
(Here, $b[a]$ denotes the term obtained by substituting $a$ for $x$ in $b[x]$.)\nomenclature[\]{$b[a]$}{substitution of $a$ for $x$ in $b[x]$}
Dually, for any $s:\sm{a:A}B(a)$ we have $s \converts (\pi_1(s),\pi_2(s))$, and for any $f:\prd{a:A}B(a)$ we have $f\converts \lambda x.f(x)$; and more simply, for any $u:\onetype$ we have $u\converts \ttt$.
In particular, the universal properties of $\Sigma$-types, $\Pi$-types, function types, and product types hold up to exo-equality.
In addition, any \emph{definition} gives an exo-equality: when defining a new symbol $t$ to equal an expression $b$, we write $t\eqdef b$, and ever after we have an exo-equality $t\steq b$.\nomenclature[\]{$a \eqdef b$}{definition of $a$ as $b$}

Exo-equality is internalized, in the sense that for any exotype $A$ and elements $a,b:A$ there is an exotype $a\steq b$ such that to give an element of $a\steq b$ means to show that $a$ and $b$ are exo-equal, and there can be only one such element, which we call \defemph{an exo-equality}.
Semantically, the exotype $x\steq y$ depending on the variables $x,y:A$ is the diagonal $A \to A\times A$.
Importantly, even if $A$ is fibrant, this diagonal is not usually a fibration; hence even if $A$ is a type, the exo-equality is only an exotype.

The other notion of equality is called the \emph{(Martin-Löf) identity type} or the \emph{identification type}.
Semantically, this represents a \emph{fibrant replacement} of the diagonal, a.k.a.\ a \emph{path type}.
Internally, this means for any (fibrant!)\ type $A$ and elements $a,b:A$ there is a fibrant type written $a =_A b$, or simply $a = b$.\nomenclature[\]{$a = b$}{type of identifications of $a$ and $b$}
In particular, for any $a : A$, we have $\refl_a : a = a$.\nomenclature[refl]{$\refl_a$}{reflexivity identification of $a$}
But unlike the exo-equality type, there can be more than one element of $a=b$: if $A$ represents a space-like object, then distinct elements of $a=b$ represent distinct paths or homotopies from $a$ to $b$.
We refer to elements of $a=b$ as \defemph{identifications}\index{type!identity} of $a$ with $b$.

The identification type is not convertible in the sense outlined above; instead it is \defemph{transportable}.\index{transportable}
That is, given $p : a = b$ and a type family $P:A\to \U$, any $u:P(a)$ induces a different, but corresponding, element of $P(b)$, written $\transport^P(p,u)$ or $\trans{p}{u}$.\nomenclature[t]{$\transport^P(p,u)$}{transport of $u:P(a)$ along $p:a=b$}\nomenclature[\]{$\trans{p}{u}$}{transport of $u:P(a)$ along $p:a=b$}
More generally, given a type family
\[P : \tprd{a, b : A} (a = b) \to \U,\]
to construct a function $f$ of type $\prd{a, b : A}{p : a = b} P(a,b, p)$ it suffices to specify, for any $a : A$, an element of $P(a, a, \refl_a)$. We refer to this principle as ``\pathinduction''.

Note that by convertibility for exo-equality, for any fibrant type $A$ and $a, b : A$, we have a map $(a \steq b) \to (a = b)$.
We sometimes use this implicitly to ``coerce'' an exo-equality to an identification.
Indeed, recalling the pairs of type and exotype formers such as $\exo{\Nat}$ and $\Nat$, and $\exo{\zerotype}$ and $\zerotype$, we could also write $a\converts b$ as $a \mathbin{\exo{=}} b$.
In general, the exotype versions of these pairs of operations satisfy their universal property up to exo-equality, while the fibrant versions satisfy their universal property up to identifications: e.g., for any exotype $A$ the map $\exo{\zerotype}\to A$ is unique up to exo-equality, while for any fibrant type $A$ the map $\zerotype\to A$ is unique up to identification.

We say that a function $f:A\to B$ between exotypes is an \defemph{isomorphism}\index{isomorphism!exo-} (or \defemph{exo-isomorphism} for emphasis) if there is $g:B\to A$ such that $g\circ f \steq 1_A$ and $f\circ g \steq 1_B$.
For instance, if an exotype $A$ is isomorphic to a fibrant type $B$, we assume that $A$ is itself fibrant; this is axiom (T3) of~\cite{2LTT}.\index{axiom!of 2LTT}
The exotype of isomorphisms from $A$ to $B$ is defined as
\[A \cong B \eqdef \sm{f:A \to B}{g:B \to A} (f \circ g \steq 1_B) \times (g \circ f \steq 1_A).\]
Similarly, we say that a function $f:A\to B$ between (fibrant!)\ types is an \defemph{equivalence} if there is $g:B\to A$ such that $g\circ f = 1_A$ and $f\circ g = 1_B$.
In particular, given two (fibrant) types $A$ and $B$, if a function $f : A \to B$ is an exo-isomorphism, then it is also an equivalence (since exo-equalities give rise to identifications).
The type $A \simeq B$ of equivalences from $A$ to $B$ requires some care to define; see \cref{sec:hlevel}.

\begin{remark}
It is very important that the ``mathematical'' notion of equality is the identification type, \emph{not} the exo-equality.
  In other words, when making a piece of mathematics formal in 2LTT, equality should be expressed using identifications, and mathematical structures such as groups or number systems should be built from (fibrant) types, not from exotypes.
  For instance, the correct definition of ``category'' internal to 2LTT is the one we will give in \cref{def:precategory}, not the notion of exo-category that we will give in \cref{def:exo-category}.

Exo-equality should be regarded as a sort of ``metatheoretic'' or ``syntactic'' equality, used for convenience but somewhat accidental in its behavior.
For instance, we can of course prove by induction that for $m,n:\Nat$ we have $m+n=n+m$, but the corresponding exo-equality $m+n\steq n+m$ cannot be proven, since it is not fibrant and we cannot use induction on $\Nat$ to construct an element of an exotype.\footnote{Of course, we \emph{can} prove that $m+n\steq n+m$ for all $m,n:\exo{\Nat}$, by induction on $\exo{\Nat}$.}
Indeed, if addition on $\Nat$ is defined by recursion on its second argument, then we cannot even prove $0+n \steq n$ for all $n:\Nat$, although we do have $n+0 \steq n$ by definition (and the situation is reversed if we define addition the other way); this sort of thing is what we mean by ``accidental''.
\end{remark}

As we will see in later sections, our signatures will involve exo-equality, and thus we will have only an exotype \emph{of} signatures.
This is reasonable because the study of general signatures and theories is properly a meta-mathematical activity.
But for a fixed signature, the types of structures, of maps between structures, of indiscernibilities within a structure, and so on, will all be fibrant, which is as we would hope because these all belong to mathematics proper.
In other words, we only need 2LTT (as opposed to the type theory of~\cite{HTT} without proper exotypes) because we want to treat all signatures, of all (finite) dimensions, uniformly.

\section{Stratification of types by their ``homotopy level''}
\label{sec:hlevel}

The type of identifications $a =_A b$ is itself a type, so has its own type of identifications; given $p, q : a =_A b$, we can form the type $p =_{a = b} q$, and so on.
The resulting tower of identification types extracts the higher homotopical information in the type $A$, and in general it need never trivialize.

By contrast, the similar tower of exo-equality exotypes trivializes after one step: if $p,q : a\steq b$ then necessarily $p\steq q$.
In other words, we assume the axiom called Uniqueness of Identity Proofs (UIP) for exo-equality.\index{axiom!uniqueness of identity proofs}
This corresponds to the fact that a model category is itself a 1-category, but represents an $\infty$-category through its notions of homotopy.

Voevodsky devised a stratification of types according to the ``complexity'' of their identity types as follows.\index{homotopy level}
\begin{itemize}
  \item Say that $A$ is a $(-2)$-type (or \defemph{contractible})\index{type!contractible} when there is an element of $\sm{a : A}\prd{x : A} x = a$. Intuitively, this means that $A$ has a unique element.
  \item Inductively, say that $A$ is an \defemph{$(n+1)$-type} if all of its identity types $a = a'$ are $n$-types, i.e., there is an element of $\prd{a:A}{a':A} \istype{n}(a=a')$.\nomenclature[istype]{$\istype{n}(A)$}{the proposition that $A$ is a homotopy $n$-type}
\end{itemize}
The unit type $\onetype$ is contractible --- and indeed, a type $A$ is contractible if and only if it is equivalent to $\onetype$.
Furthermore, given a type $A$ and $a : A$, the type $\sm{y : A} y = a$ is contractible;
this follows from \pathinduction (cf.~\cref{sec:identity}).

The next two levels above $-2$ are especially important:
\begin{itemize}
\item A $(-1)$-type is also called a \defemph{proposition}.\index{proposition}
  Equivalently, $A$ is a proposition if one can construct a function $\prd{a, a' : A} a = a'$.
  Intuitively, a proposition contains \emph{at most one} element (up to identification).
  Types that are propositions are used to represent logic inside of HoTT/UF: see \cref{sec:logic}.
\item A $0$-type is also called a \defemph{set}.\index{set}
  The types that appear in ordinary mathematics, such as the natural numbers, the real numbers, and so on, are all sets.
  Types that are not sets typically appear when working category-theoretically with large collections of structures.
\end{itemize}

The above stratification applies only to fibrant types, and refers to the fibrant identity type.
We can similarly stratify all the exotypes with reference to exo-equality, obtaining notions of \defemph{exo-contractible}, \defemph{exo-proposition}, and \defemph{exo-set}.
Because we assume UIP for exo-equality, as noted above, \emph{all exotypes are exo-sets} (including fibrant types that may have higher homotopy level in the fibrant sense); thus the exo-hierarchy has only three levels.

Now suppose given a type $A:\U$ and a family $B : A \to \U$ such that $B$ is pointwise a proposition.
Then for two elements $u,v : \sm{x:A}B(x)$, the identity type $u=v$ in $\sm{x:A}B(x)$ is equivalent to the identity type $\pi_1(u) = \pi_1(v)$ in $A$.
For this reason we refer to $\sm{x:A}B(x)$ as a \defemph{subtype} of $A$ (see also \cref{sec:logic}).
Similarly, given an exotype $A:\Ustrict$ and a family $B : A \to \Ustrict$ such that $B$ is pointwise an exo-proposition, we refer to $\sm{x:A}B(x)$ as a \defemph{sub-exotype} of $A$.

As an example, it can be proven that a function $f : A \to B$ between types is an equivalence if and only if its fibers are contractible, i.e., if there is an element of
\[\isEquiv(f) \eqdef \prd{b : B}\istype{(-2)}\Big(\tsm{a : A}f(a) = b\Big).\]\nomenclature[isequiv]{$\isEquiv(f)$}{the proposition that $f$ is an equivalence of types}%
Unlike the more na\"ive $\sm{g:B\to A} (g\circ f = 1_A)\times (f\circ g = 1_B)$, the above type turns out to be a proposition.
Thus, $\sm{f : A \to B} \isEquiv(f)$ is a subtype of $A\to B$.
We take this as the definition of the type of equivalences:
\[(A \simeq B) \eqdef \sm{f : A \to B} \isEquiv(f). \]\nomenclature[\]{$A \simeq B$}{type of equivalences between types $A$ and $B$}%
This ensures that for equivalences $f$ and $g$, the type $f=g$ is independent (up to equivalence) of whether we regard $f$ and $g$ as elements of $A\to B$ or $A\simeq B$.
(There are also many other ways to achieve this; see~\cite[Chapter 4]{HTT}.)

The type $\onetype$ is contractible, $\zerotype$ is a proposition, and $\Nat$ is a set.
In addition, the identification types of ``composite'' types can be characterized from their constituent pieces, e.g.
\begin{itemize}
\item For $a, a' : A$ and $b, b' : B$, we have $\bigl((a, b) = (a', b')\bigr) \simeq (a = a') \times (b = b')$.
\item For $a, a' : A$, $b : B(a)$, and $b' : B(a')$, we have $\bigl((a, b) = (a', b')\bigr) \simeq \sm{p : a = a'} \trans{p}{b} = b'$.
\item For $f, g : \prd{a : A}B(a)$, we have $(f = g) \simeq \textstyle\prd{a : A} f(a) = g(a)$.\footnote{Technically this is an additional axiom, called \defemph{function extensionality}.\index{axiom!function extensionality}
  Voevodsky showed that it follows from his univalence axiom~\eqref{eq:univalence-axiom}.}
\end{itemize}
More specifically, in each case we have a particular equivalence that sends the reflexivity element on the left to an element built from reflexivities on the right; by the specification of $=$ above, this specifies the map from left to right uniquely.
Corresponding facts (expressed with $\cong$ instead of $\simeq$) are also true for exotype constructors and exo-equality.

Similarly, given $A, B : \U$, we have an equivalence
\begin{equation}
 (A = B) \simeq (A \simeq B)
 \label{eq:univalence-axiom}
\end{equation}
mapping $\refl_A$ to the identity equivalence on $A$.
That this map is an equivalence is Voevodsky's \defemph{univalence axiom}\index{axiom!univalence}.
It entails in particular that $\U$ is not a set, since the type $\Nat : \U$ has non-trivial automorphisms and thus, by the univalence axiom, non-trivial self-identifications.

There is no analogue of the univalence axiom for $\Ustrict$: it is not assumed to be fibrant, so we cannot even form a type ``$A =_{\Ustrict} B$'', while there is little intelligible we can say about $A \steq B$ for exotypes $A,B:\Ustrict$.
This is another sense in which exo-equality is ``accidental''.

\section{Logic in UF}
\label{sec:logic}

While in ZF(C), logic is wrapped around set theory, in HoTT/UF (as in type-theoretic foundations for mathematics more generally) it is the other way round: types contain logic.
Specifically, in HoTT/UF the logical propositions are those types that are called propositions (i.e., $(-1)$-types; cf.\ \cref{sec:hlevel}), and a proof of such a proposition $A$ is given by an element of $A$. By definition, such a proof is unique (up to identifications) if it exists.

In this representation, the basic connectives of logic correspond to already-extant operations on types.
For instance, if $A$ and $B$ are propositions, then so are $A\to B$ and $A\times B$, and they represent the logical propositions ``if $A$ then $B$'' and ``$A$ and $B$'' respectively.
This integrated approach to logic is sometimes referred to as the ``Curry--Howard correspondence'' or the ``propositions-as-types principle''.

A predicate on a type $A$ is, accordingly, a function $P : A \to \U$ that is pointwise a proposition.
Given such a predicate, we can form the types $\prd{a: A} P(a)$ and $\sm{a : A} P(a)$. The former is a proposition and corresponds to $\forall(x : A).Px$.
However, the latter is not in general a proposition; different elements in $A$ satisfying $P$ give rise to different elements in $\sm{a : A}P(a)$.
Indeed, recall that in \cref{sec:hlevel} we called $\sm{a : A} P(a)$ a \emph{subtype} of $A$; it behaves like the set $\{ a:A \mid P(a) \}$ of elements satisfying $P$.

To represent $\exists$, therefore, we employ another type construction of HoTT/UF that maps any type $A$ universally to a proposition $\Vert A\Vert$, called the \defemph{propositional truncation} of $A$.\nomenclature[\]{$\Vert A \Vert$}{propositional truncation of a type $A$}
Intuitively, $\Vert A\Vert$ is the proposition that ``$A$ has an element''; but if this is the case, then $\Vert A\Vert$ has only one element no matter how many elements $A$ has.
Then $\exists (x:A).Px$ corresponds to $\Vert \sm{a : A}P(a)\Vert$.
Similarly, for propositions $A$ and $B$ the disjoint sum $A+B$ is not generally a proposition;
its propositional truncation $\Vert A+B\Vert$ is what represents the logical proposition ``$A$ or $B$''.
This representation of logic using propositional truncation distinguishes HoTT/UF from the more traditional propositions-as-types principle used in Martin-L\"{o}f Type Theory.

We write $\PropU \eqdef \sm{A:\U} \istype{(-1)}(A)$ for the type of propositions, which is a subtype of $\U$.\nomenclature[prop]{$\PropU$}{type of propositions}
This plays the role of the ``set of truth values'' or ``subobject classifier''.
In particular, a predicate can equivalently be defined as a map $P:A\to \PropU$.
We will also sometimes refer to such predicates as \defemph{subsets}; i.e., we identify subsets with their characteristic functions.

The \defemph{Law of Excluded Middle} is the assertion that $\PropU$ is equivalent to $\onetype + \onetype$, so that the above logic is classical.
This is an additional axiom which it is consistent to assume (and many readers may prefer to do so); but omitting it permits more general models (see \cref{sec:semantics}), and we will have no need for it.\index{excluded middle}

As an example of the role played by propositional truncation, we say that a function $f:A\to B$ is \defemph{surjective} if $\prd{b:B} \Vert \sm{a:A} fa = b \Vert$.
This notion behaves like the usual notion of surjectivity, where preimages exist but are not specified.
By contrast, we say $f$ is \defemph{split-surjective} if $\prd{b:B} \sm{a:A} fa = b$; in this case it is equipped with a specified section $s:B\to A$ such that $f \circ s = 1_B$.
The statement that every surjective function between sets is split-surjective is a form of the \defemph{Axiom of Choice}; like the Law of Excluded Middle this is consistent to assume, but we will not use it.

The ``dual'' of a surjective map is an \defemph{embedding}\index{embedding}, which is a function $f:A\to B$ that induces equivalences on identification types, $(x=_A y) \simeq (fx =_B fy)$ for all $x,y:A$.
This is a homotopical refinement of injectivity, and is equivalent to saying that each of its fibers $\sm{a:A} fa = b$ is a proposition.
A function that is both surjective and an embedding is automatically an equivalence.

Finally, we note that in general, for different universes $\U$ and $\U'$, the types $\PropU$ and $\PropUp$ are not equivalent.
All we can say is that if $\U$ is contained in $\U'$, there is an embedding from $\PropU$ to $\PropUp$.
The axiom of \defemph{propositional resizing}\index{axiom!propositional resizing} from~\cite[\S3.5]{HTT} asserts that this embedding is always an equivalence.
This axiom follows from excluded middle, and moreover holds in all higher topos models.
We will use the propositional resizing axiom when discussing examples from higher-order logic in \cref{sec:egs-higherorder}, but nowhere else.

\section{Finiteness}
\label{sec:finiteness}

For any natural number $n:\Nat$, there is a corresponding \defemph{finite type}\index{type!finite}, written $\Nat_{<n}$.\nomenclature[n]{$\Nat_{<n}$}{finite type}
If $n$ is a concrete external numeral such as 2 or 3, then $\Nat_{<n}$ is equivalent to the sum of $n$ copies of $\onetype$:
\[\Nat_{<n} \simeq \overbrace{\onetype + \cdots + \onetype}^{n}.\]
Similarly, for any exo-natural number $n:\exo\Nat$, there is an \defemph{exofinite exotype}\index{exotype!exofinite} $\exo{\Nat}_{<n}$, which for a numeral $n$ is isomorphic to the exo-sum of $n$ copies of $\onetype$:
\[\exo{\Nat}_{<n} \cong \overbrace{\onetype \exosum \cdots \exosum \onetype}^{n}.\]\nomenclature[n]{$\exo{\Nat}_{<n}$}{exofinite exotype}%
There is a unique map $\exo{\Nat}\to \Nat$ preserving zero and successor, which we treat as an implicit coercion, allowing us to regard any exo-natural number as a natural number.
For a numeral $n$ we will sometimes write $\Nat_{<n}$ and $\exo{\Nat}_{<n}$ as $\mathbf{n}$ and $\exo{\mathbf{n}}$ respectively; thus for instance $\mathbf{3} \simeq \onetype + \onetype+\onetype$ while $\exo{\mathbf{3}} \cong \onetype \exosum\onetype \exosum\onetype$.

In both cases, $\Sigma$-types over concrete (exo)finite (exo)types reduce to (exo)sums.
For instance, if $A:\mathbf{3} \to \U$ and $B:\exo{\mathbf{3}}\to \Ustrict$, we have
\[ (\tsm{j:\mathbf{3}} A_j) \simeq A_0 + A_1 + A_2
  \qquad\text{and} \qquad
  (\tsm{j:\exo{\mathbf{3}}} B_j) \cong B_0 \exosum B_1 \exosum B_2.
\]

\section{Semantics of HoTT/UF and 2LTT}
\label{sec:semantics}

Building on work such as~\cite{HoffStreich,AW2007}, Voevodsky devised a model of MLTT in simplicial sets \cite{KL12} that satisfies the univalence axiom~\eqref{eq:univalence-axiom}.
Each type $A$ is interpreted as a Kan complex (a concrete model for an $\infty$-groupoid), with elements $a:A$ representing 0-cells, identifications $p:a=a'$ representing 1-cells, and so on.
The universe $\U$ is interpreted by the base of a universal Kan fibration: a fibration of which every fibration (with fibers belonging to some universe) is a pullback.
Thus, a type family $B:A\to \U$ induces by pullback a fibration $B\twoheadrightarrow A$.
The type constructors such as $\times,\to,+,\Sigma,\Pi$ are interpreted by standard categorical constructions (see the references).
Of particular note is the identity type $(\cdot=\cdot) : A\to A \to \U$, which is interpreted by a \emph{path object} $A^{\Delta^1}$.\nomenclature{$A^{\Delta^1}$}{path object of a simplicial set $A$}
This model was extended to two-level type theory in~\cite[\S2.5]{2LTT}, including axioms T1--T3 and A5.

Subsequent work (e.g.,~\cite{lw:localuniv,ls:hits,shulman:univinj}) has shown that HoTT/UF can be interpreted in a much wider class of Quillen model categories, which are general enough to present all Grothendieck--Lurie $\infty$-toposes~\cite{Lurie}.
The basic ideas of the interpretation are the same; the main difference is that unlike simplicial sets, these models fail to satisfy nonconstructive principles such as the Law of Excluded Middle or the Axiom of Choice (providing a motivation even for a purely classical mathematician to avoid nonconstructive principles).
Thus, any mathematics written in HoTT/UF without using these principles can be automatically interpreted ``internally'' to any higher topos.
See, e.g.,~\cite{shulman_logic_space} for further discussion.
These more general models have not been formally extended to two-level type theory, but the same techniques should apply --- although axiom T1 should not be expected to hold, since in such cases the notion of ``fibration'' used in defining the universe is often structure rather than a mere property.\index{axiom!of 2LTT}

Other interpretations are also possible, e.g., in cubical sets \cite{bezem_et_al:LIPIcs:2014:4628} (see also \cite[\S2.5.2]{2LTT}).
These have the advantage of good computational and constructive behavior.

Finally, 2LTT can be interpreted in presheaves over any model of MLTT.
This is used in \cite[Proposition 2.18]{2LTT} to show that 2LTT with axioms T1, T2, and A5 (but not T3) is conservative over the type theory of \cite{HTT}.

\section{Fibrations and cofibrations}
\label{sec:cofibrancy}

Recall that we assume axiom T3: any exotype that is exo-isomorphic to a fibrant type is itself fibrant.
Following~\cite[Definition 3.7]{2LTT}, we say that a function $f:A\to B$ between exotypes is a \defemph{fibration}\index{fibration}\index{fibration!in 2LTT} if each of its exo-fibers is fibrant, i.e., if
\[ \prd{b:B} \sm{F:\U} \Big(F \steq \big(\tsm{a:A} f(a) \steq b\big)\Big). \]
This is equivalent to the existence of a type family $F:B\to \U$ and an exo-isomorphism $A \cong \sm{b:B} F(b)$ over $B$; in other words, $\U$ is a ``classifier of fibrations''.
Semantically, this coincides with the model-categorical notion of fibration.

As per~\cite[Corollary~3.19(i)]{2LTT}, we call an exotype $B$ \defemph{cofibrant}\footnote{\cite{2LTT} also defines a notion of when a function is a \emph{cofibration},\index{cofibration!in 2LTT} but we will have no need for that.}\index{exotype!cofibrant} if for any family of fibrant types $Y : B \to \U$, the exotype $\prd{b:B} Y(b)$ is fibrant, and moreover if each $Y(b)$ is contractible then so is $\prd{b:B} Y(b)$.
Although this notion of cofibrancy is similar to the model-categorical one\footnote{Specifically, to the pullback-corner axiom relating cofibrations and fibrations in a monoidal or enriched model category.}, it does not coincide with it.
In particular, in the simplicial-set-based models mentioned in \cref{sec:semantics}, all objects are cofibrant; but not all types can be shown to be cofibrant in 2LTT.
What can be shown is the following (see~\cite[Lemma~3.24]{2LTT}):
\begin{itemize}
\item All fibrant types are cofibrant.
\item $\exo{\zerotype}$ is cofibrant, and if $A$ and $B$ are cofibrant so are $A \exosum B$ and $A\times B$.
  In particular, all exofinite exotypes are cofibrant.
\item If $A$ is cofibrant and $B:A\to \Ustrict$ is such that each $B(a)$ is cofibrant, then $\sm{a:A} B(a)$ is cofibrant.
\end{itemize}
It does not seem to be possible to prove that $\exo{\Nat}$ is cofibrant, but this is a reasonable axiom to add since it holds in higher topos models.
We will not need it in this \paperorbook, but we expect it may prove useful when extending our results to signatures of infinite height.
(In~\cite{2LTT} this axiom is called A3, while A2 is a stronger but still semantically reasonable version of it; both hold in the models of~\cite{KL12,shulman:univinj}.)\index{axiom!of 2LTT}

\section{Sharpness}
\label{sec:sharp}

In a model category, every object has a fibrant replacement: a weak equivalence to a fibrant object.
This cannot be internalized in 2LTT because it is not stable under pullback; see~\cite[\S 2.7]{2LTT}.
But there are \emph{some} exotypes that do admit a fibrant replacement, and this yields a useful notion intermediate between cofibrancy and fibrancy.
(Unlike the rest of this \lcnamecref{sec:review-uf}, the material in this \lcnamecref{sec:sharp} is new.)

\begin{definition}
  An exotype $B$ is \defemph{sharp}\index{exotype!sharp} if it is cofibrant and it has a ``fibrant replacement'', meaning that there is a fibrant type $RB$ and a map $r:B\to RB$ such that for any family of fibrant types $Y:RB \to \U$, the precomposition map
  \begin{equation}
    (-\circ r) : \Big(\prd{c:RB} Y(c)\Big) \to \Big( \prd{b:B} Y(r b) \Big)\label{eq:sharp-precomp}
  \end{equation}
  is an equivalence of types.
  (Note that the codomain of~\eqref{eq:sharp-precomp} is fibrant because $B$ is assumed cofibrant.)
\end{definition}

In model category theory, a morphism is called sharp~\cite{rezk:sharp-maps}\footnote{Other terms used for sharp maps include ``right proper maps'', ``weak fibrations'', ``h-fibrations'', ``W-fibrations'', and ``fibrillations''.} if pullback along it preserves weak equivalences; but in a right proper model category, this is equivalent to its having a pullback-stable fibrant replacement.
This motivates our use of the word.

\begin{lemma}\label{thm:sharp-variations}
  Let $B$ be a cofibrant exotype, $RB$ a fibrant type, and $r:B\to RB$ a map.  The following are equivalent:
  \begin{enumerate}
  \item The map~\eqref{eq:sharp-precomp} is an equivalence for any $Y$, so that $B$ is sharp.\label{item:sv1}
  \item The map~\eqref{eq:sharp-precomp} has a section for any $Y$.\label{item:sv2}
  \item The map~\eqref{eq:sharp-precomp} is an equivalence whenever $Y \eqdef \lambda x. Z$ is constant at some fibrant type $Z:\U$, hence $(RB \to Z) \simeq (B\to Z)$.\label{item:sv3}
  \end{enumerate}
\end{lemma}
\begin{proof}
  This is analogous to other universal properties in type theory such as~\cite[Theorems 5.5.5 and 7.7.7]{HTT}.
  Of course~(\ref{item:sv1}) implies~(\ref{item:sv2}) and~(\ref{item:sv3}).

  Assuming~(\ref{item:sv2}), to show~(\ref{item:sv1}) it suffices to show that for any $f,g:\prd{c:RB} Y(c)$, if $f\circ r = g\circ r$, then $f=g$.
  But this follows from~(\ref{item:sv2}) applied to the (fibrant) type family $Y'(c) \eqdef (f(c) = g(c))$.

  Finally, assuming~(\ref{item:sv3}) and given any $Y:RB\to \U$, we define $Z\eqdef \sm{c:RB} Y(c)$, and let $\sigma:(B\to Z) \to (RB\to Z)$ be an inverse of $(-\circ r)$ for this $Z$.
  To show~(\ref{item:sv2}), let $f:\prd{b:B} Y(rb)$ and define $f' : B \to Z$ by $f'(b) \eqdef (rb,fb)$.
  Then we have $\sigma f' : RB \to Z$ and an identification $q : \sigma f' \circ r = f'$.
  Hence, in particular, for any $b:B$ we have $\pi_1(q_b) : \pi_1(\sigma f'(r(b))) = \pi_1(f'(b)) \converts rb$, which is to say that $\pi_1 \circ \sigma f' \circ r : RB \to B$ is equal to $r$.
  Thus, by~(\ref{item:sv3}) for the codomain $RB$, we have $p : \pi_1 \circ \sigma f' = 1_{RB}$ that precomposes with $r$ to yield $\pi_1 q$.

  In particular, for any $c:RB$ we have some $p_c : \pi_1(\sigma f'(c)) = c$ such that $p_{rb} = \pi_1(q_b)$.
  Define $g:\prd{c:RB} Y(c)$ by $g(c) \eqdef \trans{(p_c)}{\pi_2(\sigma f'(c))}$.
  Then we have $g(rb) = \trans{(q_b)}{\pi_2(\sigma f'(rb))}$; but $\pi_2(q_b)$ identifies this with $\pi_2(f'(b))$, which is $fb$.
  So $g\circ r = f$, and~(\ref{item:sv2}) holds.
\end{proof}

\begin{lemma}\label{lem:sharp}\
  \begin{enumerate}
  \item All fibrant types are sharp.
  \item $\exo{\zerotype}$ is sharp, and if $A$ and $B$ are sharp so are $A \exosum B$ and $A\times B$.
  \item If $A$ is sharp and $B:A\to \Ustrict$ is such that each $B(a)$ is sharp, then $\sm{a:A} B(a)$ is sharp.
  \item Each $\exo{\Nat}_{<n}$ is sharp.
  \item If $\exo{\Nat}$ is cofibrant, then it is sharp.
  \end{enumerate}
\end{lemma}
\begin{proof}
  If $B$ is fibrant, we can take $RB \define B$.
  The fibrant replacement of $\exo{\zerotype}$ is $\zerotype$; then both domain and codomain of~\eqref{eq:sharp-precomp} are contractible.
  The fibrant replacement of $A \exosum B$ is $RA + RB$, and --- writing $\mathsf{inl}$ and $\mathsf{inr}$ for the canonical functions of type $RA \to RA + RB$ and $RB \to RA + RB$, respectively --- we have a commutative square
  \[
    \begin{tikzcd}
      \prd{c:RA+RB} Y(c) \ar[d] \ar[r,"\simeq"]  &
      \big(\prd{a:RA} Y(\mathsf{inl}(a))\big) \times \big(\prd{b:RB} Y(\mathsf{inr}(b))\big) \ar[d,"\simeq"]\\
      \prd{b:A\exosum B} Y(r b) \ar[r,"\cong"'] &
      \big(\prd{a:A} Y(\mathsf{inl}(r a))\big) \times \big(\prd{b:B} Y(\mathsf{inr}(r b))\big)
    \end{tikzcd}
  \]
  showing that the left-hand map is an equivalence.
  Similarly, the fibrant replacement of $A\times B$ is $RA \times RB$.

  For $\sm{x:A} B(x)$, since each $B(a)$ is sharp we have a pointwise fibrant replacement $RB : A \to \U$, and since $A$ is sharp this can be extended to $\widetilde{RB} : RA \to \U$ with equivalences $e : \prd{a:A} RB(a) \simeq \widetilde{RB}(ra)$.
  Now we can take the fibrant replacement of $\sm{x:A} B(x)$ to be $\sm{x:RA} \widetilde{RB}(x)$, and for any $Y:(\sm{x:RA} \widetilde{RB}(x)) \to \U$ we can decompose~\eqref{eq:sharp-precomp} into a chain of equivalences
  \begin{align*}
    \Big(\tprd{z:\sm{x:RA} \widetilde{RB}(x)} Y(z)\Big)
    &\simeq \tprd{x:RA}{y:\widetilde{RB}(x)} Y(x,y) \\
    &\simeq \tprd{a:A}{y:\widetilde{RB}(ra)} Y(ra,y) \\
    &\simeq \tprd{a:A}{y:RB(a)} Y(ra,e_ay) \\
    &\simeq \tprd{a:A}{b:B(a)} Y(ra,e_a(rb)) \\
    &\simeq \tprd{z:\sm{a:A}B(a)} Y(rz).
  \end{align*}

  We prove that the fibrant replacement of $\exo{\Nat}_{<n}$ is $\Nat_{<n}$ by induction on $n$, using the facts that $\exo{\Nat}_{<0} \cong \exo{\zerotype}$ and $\exo{\Nat}_{<n+1} \cong \exo{\Nat}_{<n} \exosum \onetype$, and similarly for $\Nat_{<n}$.

  Finally, suppose $\exo{\Nat}$ is cofibrant.
  We take its fibrant replacement to be $\Nat$, with $r:\exo{\Nat} \to \Nat$ defined by recursion such that $r(\exo{0})\converts 0$ and $r(m + \exo{1}) \converts r(m)+1$.
  We will prove that \cref{thm:sharp-variations}(\ref{item:sv2}) holds.
  By recursion on $n:\Nat$, we define a function assigning to every $Y:\Nat\to\U$ and $f:\prd{m:\exo{\Nat}} Y(rm)$ an element $s(n,Y,f) : Y(n)$.
  (This is a valid induction on $\Nat$ since $\prd{Y:\Nat\to\U}{f:\prd{m:\exo{\Nat}} Y(rm)} Y(n)$ is fibrant, by the assumed cofibrancy of $\exo{\Nat}$.)
  First, when $n\eqdef 0$, we take $s(0,Y,f) \eqdef f(\exo{0})$, which is valid since $r(\exo{0})\converts 0$.
  Next, inductively assuming $s(n,Y,f)$ defined for all $Y$ and $f$, we define
  \[s(n+1,Y,f) \eqdef s(n,(\lambda x. Y(x+1)),(\lambda x.f(x+\exo{1}))),\]
  which is valid since $r(m+\exo{1}) \converts r(m)+1$.
  This completes the definition of $s$.
  It remains to show that $s(rm,Y,f) = fm$ for any $m:\exo{\Nat}$; but this is immediate by induction on $m$.
\end{proof}

In fact, we do not know of any exotypes that can be proven cofibrant but cannot be proven sharp.\footnote{Since the identification type is semantically a fibrant replacement of the diagonal, it is natural to guess that $a=b$ might also be a fibrant replacement of $a\steq b$ internally.  However, in general $a\steq b$ need not be cofibrant, and moreover $a=b$ only has a universal property as a \emph{family} with at least one endpoint varying.  So to include this case we would need to generalize our notion of ``sharp exotype'' to consider ``sharp families'' or ``sharp functions''.}

For our purposes, the main advantage of sharp types over cofibrant ones is that they have identity types, in the following sense.
If $B$ is sharp with fibrant replacement $r:B\to RB$, and $x,y:B$, we write
\begin{align*}
  (x=_B y) &\define (rx =_{RB} ry)\\
  \refl_x &\define \refl_{rx}.
\end{align*}

\begin{lemma}\label{lem:sharp-id}
  If $B$ is sharp and $b:B$, then for any $Y: \prd{y:B} (b=y) \to \U$ and $d:Y(b,\refl_b)$, there is a
  \[ j : \prd{y:B}{p:b=y} Y(y,b) \]
  and an identification (not an exo-equality) $j(b,\refl_b) = d$.
\end{lemma}
\begin{proof}
  Since $B$ is sharp and $(b=y) \to \U$ is fibrant, there is a
  \[\widetilde{Y} : \prd{z:RB} (rb = z) \to \U\]
  and an equivalence $e_{y,p}: Y(y,p) \simeq \widetilde{Y}(ry,p)$ for all $y:B$ and $p:b=y$.
  Thus, by \pathinduction for the identification type of $RB$, we have an $f:\prd{z:RB}{p:rb = z} \widetilde{Y}(z,p)$ such that $f(rb,\refl_{rb}) \steq e_{b,\refl_{rb}}(d)$.
  Let $j(y,p) \define e_{y,p}^{-1}(f(ry,p))$; then $j(b,\refl_b) \steq e_{b,\refl_b}^{-1}(e_{b,\refl_b}(d)) = d$ as desired.
\end{proof}

We define a function between sharp exotypes to be an \defemph{equivalence}\index{equivalence!of sharp exotypes} if it induces an equivalence between their fibrant replacements.
Note that in general, such an equivalence may not have any inverse map backwards at the level of the sharp exotypes.
In particular, every sharp exotype is equivalent to its fibrant replacement.
Moreover, the sharpness-preserving type constructors from \cref{lem:sharp} such as $\times$, $\exosum$, and $\Sigma$ all preserve equivalences in this sense, as do $\Pi$-types with sharp domain and fibrant codomain.

\section{Further reading}

More information on UF may be found, e.g., in \cite{grayson-intro-to-uf}.
An introduction to HoTT/UF with an emphasis on synthetic mathematics is given in \cite{shulman_logic_space}.
The book \cite{HTT} constitutes a comprehensive reference on HoTT.
And, of course, \cite{2LTT} is the main reference for 2LTT.

\part{Theory of diagram structures}\label{sec:general-theory}

In this \lcnamecref{sec:general-theory} we summarize our theory and results for a particular class of signatures and their (univalent) structures.
We call these \emph{diagram signatures}; they are essentially the same as Makkai's ``vocabularies'' \cite{MFOLDS}, but formulated in the setting of 2LTT. Almost accidentally, formulating vocabularies in 2LTT turns out to make them more general, see, e.g., \cref{eg:species,eg:unbiased-sym-monoidal}.
Eventually, in \cref{sec:ho}, we will formulate and prove our results in the even more general context of ``functorial signatures''.
But we begin with diagram signatures because they are an easier context in which to explain the ideas, and general enough to include the majority of examples.

We start, in \cref{sec:folds-cats}, by studying the particular signature $\LcatE$ of categories.
We review the notion of (univalent) categories in HoTT, and analyze how to recover this notion in terms of $\LcatE$-structures, in a way that can be transferred to other signatures.
The case of categories also provides one of several running examples for the whole of \cref{sec:general-theory}.

In \cref{sec:folds-signatures-two} we define general diagram signatures and their structures.
This requires some care to ensure that the exotype of structures is \emph{fibrant}, that is, lives in the HoTT fragment of 2LTT.
We also define notions of ``axiom'' and ``theory'', but to avoid introducing a new formal syntax we take these to be semantic in nature, appealing to the notion of property provided by the ambient logic (HoTT/UF).

In \cref{sec:indisc-univ-folds} we state the definitions of indiscernibility and univalence for diagram structures, generalizing the notions of isomorphism and univalence in categories that we reviewed in \cref{sec:folds-cats}.
We furthermore state results, in \cref{thm:hlevel-folds,thm:hlevel1-folds}, which give an upper bound for the homotopy level of the types within a univalent $\L$-structure, and for the type of univalent $\L$-structures, respectively, in terms of the height of $\L$.
This generalizes the fact that in a univalent category, the hom-types are sets (0-types) and the type of objects is a 1-type.

Finally, in \cref{sec:hsip-folds}, we formulate our univalence principle for diagram signatures.
We define a notion of equivalence for structures, consisting of ``fiberwise'' split-surjective maps.
Our univalence principle states that, for any two univalent $\L$-structures $M$ and $N$, the type of equivalences $M \simeq N$ is equivalent, via a canonical map, to the type of identifications $M = N$.
This principle, along with the results stated in \cref{sec:indisc-univ-folds}, will be proven in \cref{sec:ho} in the more general context of functorial signatures.

\chapter{Categories: an extended example}
\label{sec:folds-cats}

In \cref{sec:folds-signatures-two} we will study diagram theories and their models in general, but first we present, in this \lcnamecref{sec:folds-cats}, our prototypical example: the theory of categories.
This discussion will help to motivate the general definitions.

In \cref{sec:foldssig-by-eg} we introduce diagram signatures and axioms informally, with reference to the diagram signature $\LcatE$ for categories.
(Formal definitions will follow in \cref{sec:foldssig}.)

In \cref{sec:categories-hott} we review the ``reference definition'' of categories in HoTT/UF, first given in~\cite{AKS13} and~\cite[Chapter 9]{HTT}.
This definition is ``algebraic'' in that identities and composition are given by operations rather than by relations.

In \cref{sec:struct-our-sign} we identify suitable axioms of $\LcatE$-structures to carve out the (pre)categories amongst the $\LcatE$-structures;
this yields the theory $\Tcat$ for (pre)categories.
We construct an equivalence of types between the type of categories and the type of $\LcatE$-structures satisfying these properties.

The axioms of this theory can be partitioned into ``categorical'' and ``homotopical'' axioms.
In \cref{sec:univ-cond-categ} we show that the homotopical axioms are equivalent to a \textbf{univalence condition} on the models of that theory.
This univalence condition entails a univalence principle for these models.
Its definition does not rely on the categorical axioms, and can be generalized to structures of other diagram signatures.

\section{Diagram theories by example}
\label{sec:foldssig-by-eg}

Intuitively, a diagram signature specifies the sorts, and the dependencies between sorts, in a kind of mathematical structure.
It consists of a particular kind of category, whose objects specify the sorts, and whose morphisms $A \to B$ specify a dependency of sort $A$ on sort $B$.

For instance, a graph consists of objects and, for any two objects, arrows between them.
Since the arrows are parametrized by pairs of objects, the structure of graphs is specified by the following diagram signature:
\[
  \begin{tikzcd}
    A \ar[d, shift right] \ar[d, shift left]
    \\
    O
  \end{tikzcd}
\]
To first approximation, a structure for a diagram signature should be a diagram of that shape in the category of types.
Thus, a structure for the diagram signature of graphs would consist of two types $MA$ and $MO$ and two functions $MA \rightrightarrows MO$.
However, the intent is that $MA$ should be, not one type, but a \emph{family} of types indexed by pairs of objects, $(MA(x,y))_{x,y: MO}$.
In 2LTT we can express this by saying that the induced function $MA \to MO\times MO$ is a fibration.
We will make this precise for an arbitrary diagram signature in \cref{sec:foldssig}; for now we observe that such a reinterpretation of diagrams is only possible for a very restricted class of categories, namely the \emph{inverse} ones.
A similar restriction on categories for this purpose was used by Makkai~\cite{MFOLDS}, who called them \emph{simple}.

When we say that a diagram signature is ``for'' a given class of structures, we generally mean that its structures \emph{include} that class and that we are primarily interested in them.
For instance, the signature for pointed sets\index{set!pointed} looks as follows:
\[
  \begin{tikzcd}
    P \ar[d]
    \\
    X
  \end{tikzcd}
\]
A structure $M$ for this signature is given by a type $MX$ together with a type family $(MP(x))_{x: MX}$, where $p: MP(x)$ signifies that $x: MX$ is the chosen point.
For this to truly represent a pointed set we require that such a $p$ exists for exactly one $x:MX$ (and that $MX$ is a set and each $MP(x)$ is a proposition).

To cut down the class of structures to those of interest, Makkai~\cite{MFOLDS} introduced a formal logic over diagram signatures, called First-Order Logic with Dependent Sorts (FOLDS).
Its formulas are built from $\top$ (``true'') and $\bot$ (``false'') as atomic predicates, and with universal and existential quantification over sorts, as well as logical connectives $\wedge$ and $\vee$, for recursively constructing more complicated formulas.
Any formula in this logic can be interpretated as a predicate on structures in HoTT/UF via the usual interpretation of logic; note that the interpretation of $\exists$ and $\vee$ involves propositional truncation.
A \emph{theory} is a pair of a signature and a collection of formulas called \emph{axioms} over it, and a \emph{model} of a theory is a structure for the signature that satisfies all of the axioms.

For instance, the ``existence'' requirement for a structure as above to be a pointed set\index{set!pointed} can be expressed by the axiom
\[ \exists (x : X). \exists (p : P(x)). \top,\]
which we will also abbreviate as
\[ \exists (x:X). P(x). \]
A given structure $M$ satisfies this axiom if there exist $x : MX$ and $p : MP(x)$; in other words, the axiom is interpreted in $M$ straightforwardly as $\exists (x:MX). MP(x)$.

\begin{notation}\label{notn:no-models-cats}
  Because of the ``tautological'' nature of this interpretation of syntactic axioms, when only one structure $M$ is being discussed, we will often abuse notation by dropping the ``$M$'' in front of its interpretation of the sorts.
  Thus, for example, when discussing a structure for the signature for pointed sets, we will talk simply about $X:\U$ and $P : X\to \U$ (rather than $MX$ and $MP$).
  However, when more than one structure is under consideration (such as when we discuss morphisms of structures in \cref{sec:eq-of-cat}), we will always retain the structure names on all the sorts.
\end{notation}

Similarly, the ``uniqueness'' requirement for a structure to be a pointed set can be expressed by the axiom
\[ \forall (x,y:X). \forall \bigl(p : P(x) \land P(y)\bigr). (x=y). \]
This uses the abbreviation introduced above, and we in turn abbreviate it as
\[ \forall (x,y:X). \bigl(P(x) \land P(y) \to (x=y)\bigr). \]
However, note that this also requires an ``equality'' proposition.
As in ordinary first-order logic, equality can be considered a basic part of the language, but such a ``logic with equality'' can always be represented inside ``logic without equality'' by making equality into an atomic relation symbol.
In our situation, this means modifying the above signature to
\begin{equation}
  \begin{tikzcd}
    P \ar[d] & E \ar[dl,shift left] \ar[dl,shift right]
    \\
    X
  \end{tikzcd}
  \label{eq:sig-pointed-set}
\end{equation}
with a new sort $E$ interpreted by a type family $(ME(x,y))_{x,y:MX}$.
Now uniqueness can be expressed as
\[ \forall (x,y:X). (P(x) \land P(y) \to E(x,y)). \]
When adding equality in this way, we always include axioms making it an equivalence relation, and moreover a \defemph{congruence}\index{congruence} for all the other predicates.  
By a congruence for a family of predicates, we mean an equivalence relation that is compatible with the predicates, i.e., such that if $E(x,y)$ then $x$ and $y$ are interchangeable in any place of any of the predicates.
In the case of pointed sets this means, again using the two abbreviations introduced above,
\begin{gather}
  \forall (x:X). E(x,x) \label{eq:pset-refl}\\
  \forall (x,y:X). (E(x,y) \to E(y,x)) \label{eq:pset-sym}\\
  \forall (x,y,z:X). (E(x,y) \land E(y,z) \to E(x,z)) \label{eq:pset-trans}\\
  \forall (x,y:X). (E(x,y) \land P(x) \to P(y)).\label{eq:pset-PE}
\end{gather}

Of course, in an arbitrary structure for the extended signature, the relation $E$ may not actually be interpreted by equality; if it is, one calls the structure \emph{standard}.\footnote{With reference to ``identity of indiscernibles'' from \cref{sec:uf-intro}, standard equality amounts to adding haecceities into the structure explicitly.}
We will see in \cref{sec:indisc-univ-folds} (particularly \cref{eg:ht2-univalence}) that standardness, along with ``homotopy level'' requirements such as that $X$ is a set and each $P(x)$ is a proposition, will follow automatically from our notion of univalence for structures.\footnote{Makkai, working in a set-theoretic framework, equipped each signature with a collection of top-level sorts regarded as ``relations'', and required that in any model the corresponding sets were propositions.}

More generally, we can represent any \emph{predicate} by a sort, with axioms ensuring that its corresponding types are propositions, and we can represent any \emph{function} by the predicate of its graph.\footnote{It is not unreasonable to directly include functions in addition to dependent sorts in a signature, obtaining something like Cartmell's \cite{Cart86} Generalized Algebraic Theories.  Indeed, a diagram signature can be regarded as an especially simple sort of GAT; see also \cite[pp.~1--6]{MFOLDS}.  It is an interesting question whether the results of this \paperorbook can be extended to more general GATs; for now we restrict ourselves to the simple case.
The relationship of our functorial signatures (\cref{sec:ho}) to GATs is unclear to us.}
For instance, the diagram signature $\Lcat$ of categories\index{signature!of categories} is shown in \cref{fig:signatures}, along with the related diagram signature $\LcatE$ of categories with equality (see below).
Here the sort $I$ is a predicate for ``being an identity arrow'', while the sort $T$ is a predicate for ``$h$ is the composite of $f$ and $g$'' ($T$ stands for ``triangle'').
On the left we have written numbers indicating the \emph{rank} of each sort.
The \emph{height} of a signature is one more than the maximal rank of any sort; thus $\Lcat$ and $\LcatE$ have height 3, while the signatures for graphs and pointed sets have height 2.

Note that there are some relations on the composite arrows in $\Lcat$, as shown (e.g., the two composites $I\to A\rightrightarrows O$ are equal).
When an $\Lcat$-diagram is reinterpreted using families of types, the types corresponding to any given sort $A$ depend on as many copies of each other sort $B$ as there are morphisms $A\to B$ in the signature.
For instance, since there is one arrow $I\to A$ and one arrow $I\to O$, the type family corresponding to $I$ is $(MI_x(f))_{x:MO, f:MA(x,x)}$.
Similarly, since there are three arrows $T\to A$, whose composites with $d,c$ are equal in pairs yielding three arrows $T\to O$, the type family of $T$ is
\[(MT_{x,y,z}(f,g,h))_{x,y,z:MO, f:MA(x,y),g:MA(y,z),h:MA(x,z)}.\]

\begin{figure}
\[
  \begin{tikzcd}[ampersand replacement = \&]
2 \&
    T \ar[rd, "t_2" description, shift left = .5, bend left=25]  \ar[rd, "t_0" description, shift right = .5,bend right = 25] \ar[rd, "t_1" description]
    \&
    I \ar[d, "i" near start]
    \&
    T \ar[rd, "t_2" description, shift left = .5, bend left=25]  \ar[rd, "t_0" description, shift right = .5,bend right = 25] \ar[rd, "t_1" description]
    \& 
    I \ar[d, "i" near start]
    \&
    E \ar[ld, "e_1"' near start, shift right] \ar[ld, "e_2" near start, shift left]
    \\
1 \&
    \&
    A \ar[d, "d"', shift right] \ar[d, "c", shift left]
    \&
    \&
    A \ar[d, "d"', shift right] \ar[d, "c", shift left]
    \\
0 \&
    \&
    O 
    \&
    \&   
    O
 \end{tikzcd}
\]

\begin{equation*}
 ci = di,
 ct_0 = dt_1,
 dt_0 = dt_2,
 ct_1 = ct_2,
 de_1 = de_2,
 ce_1 = ce_2
\end{equation*}

\caption{The diagram signatures $\Lcat$, for categories (left), and $\LcatE$, for categories with equality predicate on arrows (right).
The morphisms are subject to the indicated equalities.
}

\label{fig:signatures}
\end{figure}

Of course, we need to impose axioms to restrict to the structures that represent categories.
For instance, we require that any two composable arrows have a composite:
\begin{equation}
  \forall (x,y,z:O). \forall (f:A(x,y)). \forall (g:A(y,z)). \exists (h:A(x,z)). T_{x,y,z}(f,g,h).
  \label{eq:has-comp}
\end{equation}
Other axioms involve equality, e.g., the uniqueness of composites
\begin{multline}
  \forall (x,y,z:O). \forall (f:A(x,y)). \forall (g:A(y,z)). \forall (h,h':A(x,z)).\\ T_{x,y,z}(f,g,h)\land T_{x,y,z}(f,g,h') \to (h=h').
  \label{eq:comp-unique}
\end{multline}
As in the case of pointed sets, we can represent this inside ``logic without equality'' by adding an equality sort (and suitable congruence axioms for it).
Since there are no axioms involving equality of \emph{objects} (this is one of the virtues of a dependently typed formulation of categories), it suffices to add a single sort $E$ representing equality of arrows, as shown in $\LcatE$ on the right of \cref{fig:signatures}.
As usual, we add axioms making $E$ an equivalence relation, and a congruence for the other top-level sorts $T$ and $I$ (where all free variables should be considered to be universally quantified):
\begin{gather}
  E_{x,y} (f,f) \label{eq:E-refl}\\
  E_{x,y}(f,g) \to E_{x,y}(g,f) \label{eq:E-sym}\\
  E_{x,y}(f,g) \land E_{x,y}(g,h) \to E_{x,y}(f,h) \label{eq:E-trans}\\
  E_{x,x}(f,g) \land I_x(f) \to I_x(g) \label{eq:E-I}\\
  E_{x,y} (f,f') \land E_{y,z} (g,g') \land E_{x,z} (h,h') \land T_{x,y,z}(f,g,h) \to T_{x,y,z}(f',g',h').\label{eq:E-cong}
\end{gather}

We will not make any formal use of the logic of FOLDS in this \paperorbook; we mention it mainly to help motivate the inclusion of equality sorts in a signature.
(We will see later that there are also other good reasons for doing this; see \cref{rem:equality-in-axioms,rmk:equality}.)
It is straightforward to add an equality on any sort, like $A$, that is one rank below the top; as we will see, this restriction also makes sense semantically since in HoTT/UF it is only these sorts that can be expected to consist of sets.
In \cref{sec:strict-egs} we will discuss a way to add equalities to other sorts as well.

\begin{remark}\label{rem:equality-in-axioms}
  In fact, in the particular case of categories, the reference to the equality sort $E$ in axioms might be avoidable, by replacing uses of $E_{x,y}(f,g)$ by something like $\exists(i:A(x,x)).I_x(i) \land T_{x,x,y}(i,f,g)$.
  But even for general signatures, where such a trick might not be available, we do not really need to add a ``local'' equality predicate for the purpose of stating axioms.
    This is because the foundation we are working in (homotopy type theory, or more precisely 2LTT), has a ``global'' equality predicate, namely the type of identifications.
  Since the only notion of ``axiom'' we consider in this work (given in \cref{sec:axioms}) is very general and can refer to the type of identifications,
  adding equality sorts to our signatures (e.g., to the signature of pointed sets or of categories) is not necessary to be able to express the axioms of a pointed set or of a category.

  In future work there may be reasons to consider more restricted notions of axiom, for which purpose explicit equality sorts might be necessary.
  However, the main reason for adding equality sorts to signatures is that they can change the (univalent) models significantly---compare, for instance, the signatures and their models of \cref{eg:prop,eg:set}.
  We say more on this topic in \cref{rmk:equality}.
\end{remark}

With the general concept of diagram signature and structure in hand (though with formal definitions deferred to \cref{sec:folds-signatures-two}), in the rest of this \lcnamecref{sec:folds-cats} we investigate the $\LcatE$-structures in more detail.
In particular, we want to know what further requirements must be imposed on such structures to make them ``behave like categories'', both syntactically (internal to HoTT/UF) and semantically (in the higher-topos models thereof).

\section{Categories in HoTT}
\label{sec:categories-hott}

In this \lcnamecref{sec:categories-hott} we review the definition of \defemph{category} given in~\cite{AKS13} and also in ~\cite[Chapter 9]{HTT}.
Both start by defining a \defemph{precategory}\index{category!pre-} $\C$ as follows.

\begin{definition}
  \label{def:precategory}
  A \defemph{precategory} $\C$ consists of the following:
  \begin{enumerate}
  \item A type $\ob{\C}$ of objects.
    \label{item:cat-ob}
  \item For each $a,b:\ob{\C}$, a type $\C(a,b)$ of {morphisms}.
  \item For each $a:\ob{\C}$, a morphism $1_a:\C(a,a)$.
  \item For each $a,b,c:\ob{\C}$, a function\label{item:cat-comp}
    \[  (\circ) : \C(b,c) \to \C(a,b) \to \C(a,c). \]
  \item For each $a,b:\ob{\C}$ and $f:\C(a,b)$, we have $f = {1_b\circ f}$ and $f = {f\circ 1_a}$.
    \label{item:cat-unit}
  \item For each $a,b,c,d:\ob{\C}$ and $f:\C(a,b)$, $g:\C(b,c)$, $h:\C(c,d)$, we have ${h\circ (g\circ f)} ={(h\circ g)\circ f}$.
    \label{item:cat-assoc}
  \item For any $a, b : \ob{\C}$, the type $\C(a,b)$ is a set.
    \label{item:cat-homset}
  \end{enumerate}
\end{definition}

The definitions of functors, natural transformations, and other categorical notions are straightforward; see~\cite{AKS13} and~\cite[Chapter 9]{HTT}.

Note that $\ob{\C}$ may not be a set, and for ``large'' precategories it almost never is.
For instance, $\ob{\Set}$ is the type of sets, which by univalence is a proper 1-type.
However, allowing arbitrary types of objects is problematic too.
For instance, while the statement ``a fully faithful and essentially surjective functor is an equivalence'' in ZF is equivalent to the axiom of choice, for precategories in HoTT/UF it is generally false, even if the axiom of choice is assumed.

Precategories are also ``wrong'' semantically.
For instance, when interpreted in Voevodsky's simplicial set model~\cite{KL12}, precategories do \emph{not} correspond to traditional categories defined in set-theoretic foundations.
Roughly speaking, the interpretation in this model of a precategory consists of a category $C$ (of the usual set-based sort), a Kan complex (i.e., a homotopy type or $\infty$-groupoid) $K$, and an essentially surjective functor $\Pi_1(K) \to C$, where $\Pi_1(K)$ denotes the fundamental groupoid.
In~\cite{af:flagged} this is called a \emph{flagged category}.

In terms of classical homotopy theory, the meaning of a precategory can be explained by defining its \emph{nerve} to consist of the types
\begin{align*}
  \C_0 &\phantom{\define} \text{(the type of objects)} \\
  \C_1 &\define \sm{a,b:\ob{\C}} \C(a,b) \\
  \C_2 &\define \sm{a,b,c:\ob{\C}}{f:\C(a,b)}{g:\C(b,c)}{h:\C(a,c)} h= g\circ f
\end{align*}
and so on, with $\C_n$ the type of ``$n$-simplices'' in $\C$ consisting of $n+1$ objects as the vertices, $\binom{n+1}{2}$ arrows as the 1-simplices, and identifications and higher identifications filling in the higher-dimensional faces.
This yields a semisimplicial simplicial set that satisfies the Segal condition and has degeneracies up to homotopy, and such that the map $\C_1 \to \C_0\times \C_0$ has fibers that are equivalent to discrete sets.
Thus, after a suitable rectification of the degeneracies, we can say that precategories correspond to \emph{Segal spaces} in the sense of~\cite{rezk:css} such that $\C_1 \to \C_0\times \C_0$ has homotopy-discrete fibers, which by the results of~\cite{af:flagged} are equivalent to the above description of flagged categories.

There are two ways to restrict Segal spaces to obtain a correct definition of $(\infty,1)$-categories (each of which therefore includes an equivalent copy of the collection of 1-categories).
The first, called a \emph{Segal category}~\cite{dks:hocomm}, requires the simplicial set of 0-simplices to be discrete (thus the above $K$ is discrete and the functor $K = \Pi_1(K) \to C$ is bijective on objects).
This corresponds to the following requirement on a precategory in HoTT/UF.

\begin{definition}
  A \defemph{strict category}\index{category!strict} is a precategory $\C$ such that the type $\ob{\C}$ is a set.
\end{definition}

By contrast, a \emph{complete Segal space}~\cite{rezk:css} requires the path-spaces in the simplicial set of 0-simplices to be equivalent, in a canonical way, to the spaces of equivalences obtained from the category structure.
That is, we require that $K$ is a homotopy 1-type, hence determined by its fundamental groupoid, and that the functor $\Pi_1(K) \to C$ is an equivalence onto the maximal subgroupoid of $C$.
Inside HoTT/UF, this corresponds to a ``local univalence'' condition on a precategory, ensuring that equality in the type of objects coincides with the relevant notion of ``sameness'' for objects of a category in category theory, namely isomorphism.
We now recall this condition precisely, starting with the usual definition of isomorphism:
\begin{definition}
  Let $\C$ be a precategory.
  For any morphism $f : \C(a,b)$ in $\C$, we define the type
  \[ \isiso(f) \eqdef \sm{g : \C(b,a)} (f \circ g = 1_b) \times (g \circ f = 1_a) \]
  
  The type of isomorphisms from $a$ to $b$ is then the type of pairs of a morphism $f : \C(a,b)$ together with a witness of the fact that $f$ is an isomorphism:
  \[ (a \cong b) \enspace \eqdef \enspace \sm{f : \C(a,b)} \isiso(f). \]

  We define the family of functions
  \[ \idtoiso : \prd{a, b : \ob{\C}} (a = b) \to (a \cong b) \]
  by \pathinduction, sending $\refl_a$ to the identity isomorphism on $a$.
\end{definition}

As expected, one can show that, for a morphism $f : \C(a,b)$, the type $\isiso(f)$ is a proposition; in particular, the inverse of $f$ is unique if it exists.
This entails that the type $a \cong b$ of isomorphisms is a set.

\begin{definition}\index{category!univalent}
  \label{def:univ-cat}
  A \defemph{univalent category}\footnote{Due to the many advantages mentioned below, in HoTT/UF univalent categories are often called simply ``categories'', although in some references that unadorned word refers instead to precategories.} is a precategory $\C$ such that for any $a, b : \ob{\C}$, the function
  \[ \idtoiso_{a,b} : (a = b) \to (a \cong b) \]
  is an equivalence.
\end{definition}

In particular, in a univalent category, ``isomorphic objects are equal''.
Note also that in a univalent category $\C$, the type $\ob{\C}$ of objects is a 1-type, since its identity types are all sets (0-types).

The homotopy theories of Segal categories and of complete Segal spaces are equivalent~\cite{bergner:infty-one}.
Similarly, when strict categories and univalent categories are interpreted in the simplicial set model, they both yield a notion that is equivalent to ordinary 1-categories as defined in set theory.
However, although the notions of Segal category and strict category are arguably more obvious, there are numerous advantages to complete Segal spaces and univalent categories, such as:
\begin{enumerate}
\item For strict categories, the statement ``a fully faithful and essentially surjective functor is an equivalence'' is again equivalent to the axiom of choice.
  While this is an improvement over the situation for precategories, for univalent categories this statement is simply \emph{true}.
\item Internal to HoTT/UF, the vast majority of ``naturally occurring'' large categories, such as sets, groups, rings, fields, topological spaces, etc., are univalent, while practically none of them are strict.
\item Every precategory is weakly equivalent to a univalent category (its ``Rezk completion'' or ``univalent completion''), while it is impossible to prove in HoTT/UF that every precategory, or even every univalent category, is weakly equivalent to a strict one.
  Indeed, there are models of HoTT/UF in which not every 1-type admits a surjection from a set (see~\cite{ntypes-cover}).
  Regarding 1-types as univalent groupoids (i.e., univalent categories with all morphisms invertible), we find that in such a model not every precategory is weakly equivalent to a strict category.
\item The type of univalent groupoids is provably equivalent to the type of 1-types, but this is not the case for the type of strict groupoids \emph{even if} we assume that every univalent category is weakly equivalent to some strict category.
\item When HoTT/UF is interpreted in the $(\infty,1)$-topos of stacks of $\infty$-group\-oids on some site using the models of~\cite{shulman:univinj}, univalent categories correspond to stacks of 1-categories on the same site, while strict categories correspond to internal categories in the 1-topos of sheaves of sets on that site.
  The former are, generally speaking, much more important.
\item A map of complete Segal spaces is a category-theoretic equivalence just when it is a levelwise equivalence of bisimplicial sets.
  In HoTT/UF this has the following even more pleasing manifestation:
\end{enumerate}

\begin{theorem}[{\cite[Theorem~6.17]{AKS13}}]\label{thm:univalence-for-categories}\index{univalence principle!for categories}
  For univalent categories $\C$ and $\D$, let $\C \simeq \D$ be the type of categorical equivalences between $\C$ and $\D$; then
  \[ (\C= \D) = (\C\simeq \D). \]
\end{theorem}

Our goal, therefore, is to formulate a general notion of ``univalence'' for other categorical structures, for which we can prove an analogue of \cref{thm:univalence-for-categories}.

\begin{remark}
  Since univalent categories are much better behaved than precategories, and most naturally-ocurring categories in HoTT/UF are univalent, one generally prefers working with univalent categories.
  However, non-univalent precategories do arise at times.
  For example, a na\"{\i}ve definition of the Kleisli category $\C_T$ of a monad $T : \C \to \C$ --- taking as objects those of $\C$, and as morphisms $\C_T(a,b) \eqdef \C(a,Tb)$ --- generally only yields a precategory, even if $\C$ is univalent.
  There are other definitions of the Kleisli category that yield a univalent version (e.g.\ as the category of free $T$-algebras); but for some applications it is the non-univalent one that is needed.

  One place where this arises is in the categorical semantics of programming languages, where various abstract structures essentially codify the non-univalence of categories of this sort.
  We will discuss some such structures in \cref{sec:unnat}.
  For instance, in \cref{eg:thunk-force} we discuss the theory of ``thunk-force categories'', which axiomatize the structure of a Kleisli (pre)category; they are used to interpret computations with side effects in functional programming languages.
  Our general notion of univalence for $\L$-structures (given in \cref{def:univalence-cond-on-diag-structures}), for any signature $\L$, yields, in particular, a notion of \emph{univalent} thunk-force category.
  Any thunk-force category, univalent or not, has an underlying precategory; but univalence of the thunk-force category in the sense of \cref{def:univalence-cond-on-diag-structures} does \emph{not} entail univalence of the underlying precategory.
  This is due to the presence of ``non-categorical'' structure in a thunk-force category; details are given in \cref{eg:thunk-force}.
\end{remark}

\section{Structures of our signature for categories}
\label{sec:struct-our-sign}

In this \lcnamecref{sec:struct-our-sign}, we give an equivalent definition of the precategories of the previous section, carving them out of the structures for the signature $\LcatE$.

Anticipating the explicit calculation from \cref{def:structures-lcat-lcate,ex:struc-lcat}, a structure $M$ for the diagram signature $\Lcat$ from \cref{fig:signatures} consists of
\begin{align*}
  MO&:\U\\
  MA&:MO\times MO \to \U\\
  MI &: \tprd{x:MO} MA(x,x) \to \U\\
  MT &: \tprd{x,y,z:MO} MA(x,y) \to MA(y,z) \to MA(x,z) \to \U\\
\intertext{while a structure for $\LcatE$ additionally contains a family}
  ME &: \tprd{x,y:MO} MA(x,y) \to MA(x,y) \to \U
\end{align*}

This forms the underlying data of a (pre)category: a type of objects, types of morphisms, and properties of ``being an identity'' and ``being a composite''.

To carve out from the type of $\LcatE$-structures the precategories of \cref{def:precategory}, we consider the following ``categorical'' $\LcatE$-axioms.
\begin{enumerate}
\item $E$ is a congruence with respect to $T$, $I$, and $E$ itself (an equivalence relation).
  \label{item:e-congruence}
\item Composition of any two composable arrows should exist:  \label{item:comp-exists}
  \[\forall(x,y,z : O).\forall(f:A(x,y)).\forall(g:A(y,z)).\exists (h:A(x,z)).T_{x,y,z}(f,g,h).\]
\item Composites are unique:  \label{item:comp-unique}
  \begin{multline*}
    \forall (x,y,z:O). \forall (f:A(x,y)). \forall (g:A(y,z)). \forall (h,h':A(x,z)).
    \\
    T_{x,y,z}(f,g,h)\land T_{x,y,z}(f,g,h') \to E_{x,z}(h,h').
    \end{multline*}
\item Identities are unique:  \label{item:id-unique}
  \[\forall (x : O).\forall (f, g : A(x,x)).I_x(f) \to I_x(g) \to E_{x,x}(f,g).\]
\item Composition is right-unital: \label{item:id-r-unital}
  \[ \forall(x,y : O).\forall(f : A(x,y).\forall(g : A(y,y)).I_y(g) \to T_{x,y,y}(f,g,f). \]
\item Composition is left-unital:   \label{item:id-l-unital}
  \[ \forall(x,y : O).\forall(f : A(x,x).\forall(g : A(x,y)).I_x(f) \to T_{x,x,y}(f,g,g).\]
\item Composition is associative: \label{item:comp-assoc}
  \begin{align*}
    &\forall (w,x,y,z).\forall(f : A(w,x)).\forall(g:A(x,y).\forall (h : A(y,z)).
    \\
    &\forall(k : A(w,y)).\forall(l : A(x,z)).\forall(m,n : A(w,z).
    \\
    &T_{w,x,y}(f,g,k) \to T_{x,y,z}(g,h,l) \to T_{w,y,z}(k,h,m) \to T_{w,x,z}(f,l,n) \to E_{w,z}(m,n).
  \end{align*}

\end{enumerate}
In addition to the categorical axioms, we ask the structures to satisfy the following ``homotopical'' $\LcatE$-axioms:

\begin{enumerate}[resume]
\item $T$, $I$, $E$ consist pointwise of propositions; \label{item:pw-prop}
\item $A$ consists pointwise of sets; \label{item:pw-set}
\item The structure is \defemph{standard} with equality $E$:\label{item:standard}
  \[ \forall(x,y : O).\forall(f,g : A(x,y)).E_{x,y}(f,g) \leftrightarrow f = g.\]
\end{enumerate}

Note that Axiom (\ref{item:standard}) above asserts (semantically) that the ``internal'' or ``formal'' equality predicate $E$ coincides (propositionally) with the ``actual'' identity type of $A$. 
Without this axiom, there would be more models of $\Tcat$-precategories in the sense of \cref{def:folds-precat} (immediately below) than precategories in the sense of \cref{def:precategory}, since the former could model $E$ as any congruence (not necessarily the one given by the identity type).

\begin{definition}
  \label{def:folds-precat}
  We define $\Tcat$ to be the theory with signature $\LcatE$, and with axioms given by \cref{item:e-congruence,item:comp-exists,item:comp-unique,item:id-unique,item:id-r-unital,item:id-l-unital,item:comp-assoc,item:pw-prop,item:pw-set,item:standard} above.
  A \defemph{$\Tcat$-precategory} is a $\Tcat$-model, i.e., an $\LcatE$-structure satisfying the $\Tcat$-axioms.
\end{definition}

We can now state and prove the following result:

\begin{lemma}
  \label{lem:equiv-precat-folds}
  The type of precategories of \cref{def:precategory} is equivalent to the type of $\Tcat$-precategories.
\end{lemma}
\begin{proof}
  The underlying data of $O$ and $A$ are the same.
  In one direction, let $I_{x}(f) \eqdef (f = 1_x)$ and $T_{x,y,z}(f,g,h) \eqdef (h = g\circ f)$.
  In the other, let $1_x$ be the unique $f:A(x,x)$ with $I_x(f)$, and $g\circ f$ the unique $h$ with $T_{x,y,z}(f,g,h)$.
  (Here we use the principle of unique choice, which holds in univalent foundations; see~\cite[\S3.9]{HTT}.)
\end{proof}

\begin{remark}[E-categories]\index{category!E-}
  An \emph{E-category} (see, e.g., \cite{palmgren:LIPIcs:2018:10055}) consists of a type of objects, and, for each pair $(a,b)$ of objects, a \emph{setoid}\index{setoid} $\hom(a,b)$ of morphisms from $a$ to $b$, together with operations of identity and composition satisfying the usual categorical axioms up to the setoid relation.
  Models of the theory given by the signature $\LcatE$ and the categorical axioms of \cref{item:e-congruence,item:comp-exists,item:comp-unique,item:id-unique,item:id-r-unital,item:id-l-unital,item:comp-assoc} are closely related to E-categories.
  One difference is that E-categories are ``algebraic'', in the sense that identities and composition are given as functions instead of (functional) relations.
\end{remark}

\section{Univalence at \texorpdfstring{$O$}{O}}
\label{sec:univ-cond-categ}
\label{sec:univalence-at-o}

We now consider how to express the univalence condition of \cref{def:univ-cat}, and, specifically, the notion of isomorphism, in $\Tcat$-precategories.
Of course, a trivial solution would be to rely on the equivalence of \cref{lem:equiv-precat-folds} to reuse the traditional definition of isomorphism.
Our goal, however, is to define a notion of isomorphism that
\begin{enumerate}
\item only depends on the diagram structure underlying a $\Tcat$-precategory, not on the axioms imposed on it; and
\item is applicable to other signatures.
\end{enumerate}
We will refer to the resulting notion of isomorphism as an \emph{indiscernibility}.

Recall that by the Yoneda lemma, an isomorphism $\phi : a\cong b$ in a category $\C$ is equivalently a natural family of isomorphisms of sets $\phi_{x\bullet} : \C(x,a) \cong \C(x,b)$, where naturality in $x$ means that $\phi_{y\bullet}(g) \circ f = \phi_{x\bullet}(g\circ f)$.
In the language of $\LcatE$-structures, the operation $\circ$ is replaced by the relation $T$, with a new variable $h$ for the composite $g\circ f$.
Thus, for $a, b : O$ in an $\LcatE$-model, we might attempt to define indiscernibilities $a \fiso b$ to consist of the following.
\begin{itemize}
\item For each $x:O$, an isomorphism $\phi_{x\bullet}:A(x,a) \cong A(x,b)$; and \label{item:foldsiso1}
\item For each $x,y:O$, $f:A(x,y)$, $g:A(y,a)$, and $h:A(x,a)$, we have\label{item:foldsiso2}
  $T_{x,y,a}(f,g,h) \leftrightarrow T_{x,y,b}(f,\phi_{y\bullet}(g),\phi_{x\bullet}(h))$.
\end{itemize}
This looks promising, but it privileges one of the variables of $A$ over the other, and the relation $T$ over $I$ (and $E$). To arrive at a mechanical definition of indiscernibility purely from the underlying diagram signature, we need to avoid such arbitrary choices.

It is more natural, therefore, to give equivalences between hom-sets with $a$ and $b$ substituted into \emph{all} possible ``collections of holes'':
\begin{align}
&\text{For any $x:O$, an isomorphism $\phi_{x\bullet}:A(x,a) \cong A(x,b)$;} \label{item:foldsiso1a} \\
&\text{For any $z:O$, an isomorphism $\phi_{\bullet z}:A(a,z) \cong A(b,z)$;} \label{item:foldsiso1b} \\
&\text{An isomorphism $\phi_{\bullet\bullet}:A(a,a) \cong A(b,b)$.}\label{item:foldsiso1c}
\end{align}
and similar logical equivalences between all possible ``relations with holes'':
\begin{align}
  T_{x,y,a}(f,g,h) &\leftrightarrow T_{x,y,b}(f,\phi_{y\bullet}(g),\phi_{x\bullet}(h)) \label{eq:Txya}\\
  T_{x,a,z}(f,g,h) &\leftrightarrow T_{x,b,z}(\phi_{x\bullet}(f),\phi_{\bullet z}(g),h) \label{eq:Txaz}\\
  T_{a,z,w}(f,g,h) &\leftrightarrow T_{b,z,w}(\phi_{\bullet z}(f),g,\phi_{\bullet w}(h)) \label{eq:Tazw}\\
  T_{x,a,a}(f,g,h) &\leftrightarrow T_{x,b,b}(\phi_{x\bullet}(f),\phi_{\bullet\bullet}(g),\phi_{x\bullet}(h)) \label{eq:Txaa}\\
  T_{a,x,a}(f,g,h) &\leftrightarrow T_{b,x,b}(\phi_{\bullet x}(f),\phi_{x\bullet}(g),\phi_{\bullet\bullet}(h)) \label{eq:Taxa}\\
  T_{a,a,x}(f,g,h) &\leftrightarrow T_{b,b,x}(\phi_{\bullet\bullet}(f),\phi_{\bullet x}(g),\phi_{\bullet x}(h)) \label{eq:Taax}\\
  T_{a,a,a}(f,g,h) &\leftrightarrow T_{b,b,b}(\phi_{\bullet\bullet}(f),\phi_{\bullet\bullet}(g),\phi_{\bullet\bullet}(h)) \label{eq:Taaa}\\
  I_{a,a}(f) &\leftrightarrow I_{b,b}(\phi_{\bullet\bullet}(f)) \label{eq:Iaa}  \\
   E_{x,a}(f,g) &\leftrightarrow E_{x,b}(\phi_{x\bullet }(f),\phi_{x\bullet}(g)) \label{eq:Exa}\\
   E_{a,x}(f,g) &\leftrightarrow E_{b,x}(\phi_{\bullet x}(f),\phi_{\bullet x}(g)) \label{eq:Eax}\\
   E_{a,a}(f,g) &\leftrightarrow E_{b,b}(\phi_{\bullet \bullet}(f),\phi_{\bullet \bullet}(g)) \label{eq:Eaa}
\end{align}
for all $x,y,z,w:O$ and $f,g,h$ of appropriate types.
Fortunately, the additional data here are redundant. Since $\phi_{x\bullet}$, $\phi_{\bullet z}$, and $\phi_{\bullet\bullet}$ preserve identities and $E$ is equivalent to identity by hypothesis, we obtain~\eqref{eq:Exa} to~\eqref{eq:Eaa}.
Just as~\eqref{eq:Txya} means the $\phi_{x\bullet}$ form a natural isomorphism,~\eqref{eq:Tazw} means the $\phi_{\bullet z}$ form a natural isomorphism, and~\eqref{eq:Txaz} means these natural isomorphisms arise from the same $\phi : a\cong b$.
Given this, any one of~\cref{eq:Txaa,eq:Taxa,eq:Taax} ensures that $\phi_{\bullet\bullet}$ is conjugation by $\phi$, and then the other two follow automatically, as do~\cref{eq:Taaa,eq:Iaa}.
This suggests the following definition.

\begin{definition}\label{defn:folds-iso-obj}\index{indiscernibility!of objects in a category}
  For $a,b$ objects of an $\LcatE$-structure, an \defemph{indiscernibility} from $a$ to $b$ consists of data as in \cref{item:foldsiso1a,item:foldsiso1b,item:foldsiso1c} satisfying \cref{eq:Txya,eq:Txaz,eq:Tazw,eq:Txaa,eq:Taxa,eq:Taax,eq:Taaa,eq:Iaa,eq:Exa,eq:Eax,eq:Eaa}.
  We write $a \foldsiso b$ for the type of such indiscernibilities.
\end{definition}

\begin{theorem}\label{thm:iso-foldsiso}
  In any $\Tcat$-precategory, the type $a\foldsiso b$ of indiscernibilities from $a$ to $b$ is equivalent to the type of isomorphisms $a\cong b$.\qed
\end{theorem}

\begin{definition}
  A \defemph{univalent $\Tcat$-category} is a $\Tcat$-precategory such that for all $a,b:MO$, the canonical map $(a=b)\to (a\foldsiso b)$ is an equivalence.
\end{definition}

We can conclude:

\begin{theorem}
  A $\Tcat$-precategory is univalent iff its corresponding precategory is a univalent category. \qed
\end{theorem}

The point of our definition of indiscernibility is that it can be derived algorithmically from the diagram signature for categories, by an algorithm which applies equally well to
\begin{enumerate*}[label=(\arabic*)]
\item any sort in $\LcatE$ and
\item any diagram signature.
\end{enumerate*}
We will give this mechanism explicitly in \cref{sec:indisc-folds,sec:FOLDS-iso-uni}.
Then, for any $a,b:MK$ in some structure $M$, there will be a canonical map $(a=_K b)\to (a\foldsiso b)$, and we call $M$ \textbf{univalent at $K$} if these are equivalences.

This leads to an obvious question: what does univalence mean at the other sorts of $\LcatE$?

\section{Univalence at top-level sorts}
\label{sec:univalence-at-t}

Let $M$ be an $\LcatE$-structure, and let $t,t':MT_{x,y,z}(f,g,h)$.
Following the same recipe as for $O$ in the previous \lcnamecref{sec:univ-cond-categ}, the type $t\foldsiso t'$ of indiscernibilities should consist of consistent equivalences between all types dependent on $t$ and $t'$.
But there are no such types in the signature, so $t\foldsiso t'$ is contractible.
The same reasoning applies to $I$ and $E$.
Thus, the univalence condition for these sorts will assert simply that all of their path-types are contractible, i.e., that they are propositions.
Note that this is just \cref{item:pw-prop} of \cref{def:folds-precat}.

\begin{definition}
  An $\LcatE$-structure $M$ is \defemph{univalent at $T$, $I$, and $E$} if and only if the canonical maps
  \begin{align*}
    (t= t') &\to (t\foldsiso t')
    \\
    (i=i') &\to (i\foldsiso i')
    \\
    (e=e') &\to (e\foldsiso e')
  \end{align*}
   are equivalences
  for all inhabitants of the types $MT$, $MI$, and $ME$, respectively. \qed
\end{definition}

\begin{theorem}
  An $\LcatE$-structure $M$ is univalent at $T$, $I$, and $E$ if and only if \cref{item:pw-prop} of \cref{def:folds-precat} is satisfied.\qed
\end{theorem}

\section{Univalence at \texorpdfstring{$A$}{A}}
\label{sec:univalence-at-a}

Finally, we define indiscernibilities for arrows of an $\LcatE$-structure.
We assume that the structure is univalent at $T$, $I$, and $E$, that is, that $T$, $I$, and $E$ are propositions pointwise. Furthermore, we assume that $E$ is a congruence for $T$ and $I$.
Let $a, b : O$ and $f, g : A(a,b)$.

An indiscernibility between $f,g:A(a,b)$ in $M$ should consist of logical equivalences between instances of $T$, $I$, and $E$ with $f$ replaced by $g$ in ``all possible ways'', clearly beginning with
\begin{align}
  T_{x,a,b}(u,f,v) &\leftrightarrow T_{x,a,b}(u,g,v) \label{eq:fia1}\\
  T_{a,x,b}(u,v,f) &\leftrightarrow T_{a,x,b}(u,v,g) \label{eq:fia2}\\
  T_{a,b,x}(f,u,v) &\leftrightarrow T_{a,b,x}(g,u,v) \label{eq:fia3}
\end{align}
for all $x:O$ and $u,v$ of appropriate types.
But 
how do we put $f$ in two or three of the places in $T$ in the most general way?
In \cref{sec:indisc-folds,sec:FOLDS-iso-uni} we will see that the answer is to assume an equality between objects and transport $f$ along it.

\begin{definition}\label{defn:foldsiso-arrows}
  For $f,g:A(a,b)$ in an $\LcatE$-structure that is univalent at $T$, $I$, and $E$, an \defemph{indiscernibility} from $f$ to $g$ consists of the logical equivalences shown in~\cref{eq:fia1,eq:fia2,eq:fia3,eq:fia4,eq:fia5,eq:fia6,eq:fia7,eq:fia8,eq:fia9,eq:fia10,eq:fian}, for all $p:a=a$, $q:b=a$, and $r:b=b$.
\end{definition}

\begin{align}
  T_{a,a,b}(\trans{q}{f},f,u) &\leftrightarrow T_{a,a,b}(\trans{q}{g},g,u) \label{eq:fia4}\\
  T_{a,b,b}(\trans{p}{f},u,f) &\leftrightarrow T_{a,b,b}(\trans{p}{g},u,g) \label{eq:fia5}\\
  T_{a,a,b}(u,\trans{r}{f},f) &\leftrightarrow T_{a,a,b}(u,\trans{r}{g},g) \label{eq:fia6}\\
  T_{a,a,b} (\trans{(p,q)}f,\trans r f,f) &\leftrightarrow T_{a,a,b}(\trans{(p,q)}g,\trans r g,g) \label{eq:fia7}\\
  I_{a} (\trans q f) &\leftrightarrow I_a(\trans q g) \label{eq:fia8}\\
  E_{a,b}(f,u) &\leftrightarrow E_{a,b}(g,u) \label{eq:fia9}\\
  E_{a,b}(u,f) &\leftrightarrow E_{a,b}(u,g) \label{eq:fia10}\\
  E_{a,b} (\trans{(p,r)}f,f) &\leftrightarrow E_{a,b} (\trans{(p,r)}g,g)\label{eq:fian}
\end{align}

Since $T$, $I$, and $E$ are propositions, so is the type $f\foldsiso g$ of indiscernibilities.
And $f\foldsiso f$, so by \pathinduction we have $(f=g) \to (f\foldsiso g)$.

\begin{theorem}\label{thm:2-univ-is-1-univ}
  Let $M$ be univalent at $T$, $I$, and $E$, and let $E$ be a congruence for $T$ and $I$. Then the following are equivalent:
  \begin{enumerate}
  \item the map $(f=g) \to (f\foldsiso g)$ is an equivalence for all $f,g$; \label{item:A-univ}
  \item $M$ is standard and $MA$ is a set pointwise. \label{item:standard-set}
  \end{enumerate}
\end{theorem}
\begin{proof}
  Since $f\foldsiso g$ is a proposition, the former condition implies that each $A(a,b)$ is a set.
  Thus, for \ref{item:A-univ} $\Rightarrow$ \ref{item:standard-set}, it suffices to show $E_{a,b}(f,g)\Rightarrow (f\foldsiso g)$, which holds since $E$ is a congruence for $T$ and $I$.

  For \ref{item:standard-set} $\Rightarrow$ \ref{item:A-univ}, we must show $(f\foldsiso g) \Rightarrow (f=g)$ in a standard structure.
  (The maps back and forth are then automatically equivalences, since by assumption both sides are propositions.)
  But since $E_{a,b}(f,f)$ always, $f\foldsiso g$ implies $E_{a,b}(f,g)$, hence $f=g$ by standardness.
\end{proof}

Thus, by extending the ``univalence'' condition of a category from the sort $O$ to the sort $A$, we encompass automatically the assumption that the hom-types in a precategory are sets and that the structure is standard.
These are the remaining homotopical axioms (\cref{item:pw-set,item:standard}) from \cref{def:folds-precat}.

\begin{remark}[On equality predicates in a signature]
  More generally, suppose $\L$ is a signature containing a sort $S$ at rank one below top-level and an equality predicate $E_S \rightrightarrows S$.
  If $M$ is an $\L$-structure satisfying axioms saying that $E_S$ is a congruence for all the sorts that depend on $S$,
  then univalence of $M$ means that the type (family) $MS$ is pointwise a set with identifications given by $ME_S$.

  In \cref{eg:pre-po-sets} we describe in detail the effect of adding an equality sort to the signature of a type with a binary relation on it.  
  In \cref{sec:strict-egs} we study how to add equality predicates to sorts at lower rank, e.g., to the sort of objects of a category.
\end{remark}

In sum, all the ad-hoc-looking homotopical conditions on a $\Tcat$-precategory are equivalent to univalence conditions on the various sorts of $\LcatE$.
A category is hence equivalently an $\LcatE$-structure that
\begin{enumerate*}[label=(\arabic*)]
\item satisfies the categorical axioms; and
\item is univalent at all of its sorts.
\end{enumerate*}
Our goal in the rest of \cref{sec:general-theory}, therefore, is to define notions of indiscernibility and univa\-lence for any signature $\L$, generalizing the theory of univalent categories to arbitrary higher-categorical structures.

\section{Equivalence of categories}
\label{sec:eq-of-cat}

In \cref{sec:univ-cond-categ}, we described a notion of \emph{isomorphism} or \emph{indiscernibility} that
\begin{enumerate}
  \item only depends on the diagram structure underlying a $\Tcat$-precategory, not on the axioms imposed on it; and \label{item:only-depend-struc}
  \item is applicable to other signatures. \label{item:applic-other-strux}
  \end{enumerate}
In this \lcnamecref{sec:eq-of-cat}, we develop a notion of \emph{equivalence} that is not only equivalent to the usual notion of equivalence between univalent categories, but that also satisfies \cref{item:only-depend-struc,item:applic-other-strux} above.

Consider a fully faithful and essentially surjective functor \cite[Def.~6.7]{AKS13} between univalent $\Tcat$-categories $M$ and $N$. Putting this in terms of $\Tcat$-structures, a functor $e: M \to N$ consists of the following functions.\index{equivalence!of categories}
(As noted in \cref{notn:no-models-cats}, we retain the structure names on the interpretations of sorts when multiple structures are present.)
\begin{align*}
  e_O & :  MO \to NO \\
  e_A & : \tprd{x, y: MO} MA(x, y) \to NA(ex, ey) \\
  e_I & : \tprd{x : MO, i : MA(x,x)} MI_{x}(i) \to NI_{ex} (ei) \\
  e_T & :
        \begin{multlined}[t][0.85\linewidth]
          \tprd{x,y,z : MO, f : MA(x,y), g : MA(y,z), h : MA(x,z)}\hfill \\
          MT_{x,y,z}(f,g,h) \to NT_{ex,ey,ez} (ef,eg,eh)
        \end{multlined}\\
  e_E & : \tprd{x,y : MO, f : MA(x,y), g : MA(x,y)} ME_{x,y}(f,g) \to NE_{ex,ey} (ef,eg)
\end{align*}
The functor $e$ is essentially surjective just when $e_O$ is pointwise surjective (this is because $N$ is univalent), and $e$ is fully faithful just when $e_A$ is a pointwise equivalence.

This second condition, that each $e_A(x,y)$ is an equivalence, is equivalent to the condition that $e_A(x,y)$ is surjective and an embedding \cite[Cor.~4.6.4]{HTT}. 
If $e_A(x,y)$ is an embedding, then (by definition) the function $(f = g) \to (ef = eg)$ for $f, g: e_A(x,y)$ is an equivalence. 
Since our categories are univalent, $f=g$ is equivalent to the type of indiscernibilities $f \foldsiso g$, which is equivalent to $ME_{x,y}(f,g)$ by \cref{thm:2-univ-is-1-univ}; and similarly, $ef=eg$ is equivalent to $NE_{ex,ey}(ef,eg)$. Thus, $e_A(x,y)$ is an embedding if and only if for all $f, g: e_A(x,y)$ the function $e_E: ME_{x,y}(f,g) \to NE_{ex,ey}(ef,eg)$ is an equivalence --- or, equivalently, a surjection, since $ME_{x,y}(f,g)$ and $NE_{ex,ey}(ef,eg)$ are propositions. 
Thus, the condition that $e_A$ is a pointwise equivalence is itself equivalent to the condition that $e_A$ and $e_E$ are pointwise surjections.

Note also that the morphisms $e_I$ and $e_T$ are also pointwise surjections whenever $e$ is fully faithful. 
To see that $e_I$ is a surjection at each $x : MO$ and $f : MA(x,x)$, for each $j : NI_{e(x)} (ef)$, we need to find terms $i : MI_x(f)$ and $p : e_I i = j$. Note that the axioms for a category imply that any two identity morphisms are the same (that is, if $\alpha, \beta: A(a,a)$ and $I_a(\alpha)$, $I_a(\beta)$ are inhabited, then $\alpha = \beta$ is). Consider then the identity $1_x$. It is sent to an identity on $ex$, so we find that $e(1_x) = ef$. But since $e_A(x,x)$ is an embedding, we find that $1_x = f$. Thus there is some term $i : MI_x(f)$ and $p : e_I i = j$ is satisfied since its ambient type, $NI_{e(x)} (ef)$, is a proposition. A similar argument shows that $e_T$ is a pointwise surjection.

Thus, we can say that an \emph{equivalence} between two univalent $\Tcat$-structures $M$ and $N$ consists of functions $e_O, e_A, e_I, e_T, e_E$ as above which are all pointwise surjections. We have the following.

\begin{theorem}
  An equivalence between two univalent $\Tcat$-categories, in the above sense, is exactly an equivalence of univalent categories in the usual sense.
\end{theorem}

This notion of equivalence (a levelwise, pointwise surjection) motivates our general notion of equivalence between structures. It should be noted that we actually consider \emph{split} surjections (\cref{def:vss}) in the general case; although non-split surjections suffice for categories, we do not know if they suffice for all theories.
Moreover, we introduce many different shades of equivalence to accomplish our goals (e.g., levelwise equivalence (\cref{def:lvleqv-folds}, \cref{def:iso-of-structures}), relative equivalence (\cref{def:eqv})), but when the structures in question are univalent, these are all equivalent, as we will show.

\chapter{Diagram signatures in Two-Level Type Theory}
\label{sec:folds-signatures-two}

To state and prove general theorems about higher-categorical structures, we need a general definition of what is meant by a ``higher-categorical structure''.
There are many approaches to this; we will take the ``geometric'' or ``non-algebraic'' one, in which a structure is specified by a \emph{diagram} of sets or spaces with properties.
Specifically, we will use \emph{Reedy fibrant diagrams} of spaces (i.e., types) on certain \emph{inverse \mbox{(exo-)cat}\-egories}, which are ``maximally non-algebraic'': the functorial actions can all be encapsulated by type dependency.

This latter point was already realized by Makkai~\cite{MFOLDS}, who used diagrams of sets on inverse categories to give a similar general context for higher-categorical structures, along with a language called First-Order Logic with Dependent Sorts (FOLDS).
In contrast to HoTT/UF, FOLDS is not a foundational system for mathematics, but a kind of first-order logic designed for higher categorical structures.
We will not use the logical syntax of FOLDS, but we adopt and generalize its notions of signature and structure.
We will refer to the particular inverse categories we use as as \emph{diagram signatures} (Makkai called them ``vocabularies'').

\section{Exo-categories}
\label{sec:s-categories}

Before we can give our first definition of signature in \cref{sec:foldssig}, we review, in this section, the definition of \emph{exo-categories} in 2LTT (see also \cite[Definition~3.1]{2LTT}).

\begin{definition}\index{category!exo-}
  \label{def:exo-category}
An \defemph{exo-category} $\C$ is given by the following data:
\begin{enumerate}
\item An exotype $\ob{\C}$ of \emph{objects} (also often denoted $\C$);
\item For each $x,y \colon \ob{\C}$ an exotype $\C(x,y)$ of \emph{arrows};
\item For each $x \colon \ob{\C}$ an arrow $1_x \colon \C(x,x)$; and
\item A \emph{composition} map $(\circ) \colon \C(y,z) \rightarrow \C(x,y) \rightarrow \C(x,z)$ that is associative and for which $1$ is a left and right unit, both up to exo-equality.
\end{enumerate}
\end{definition}

\begin{remark}[Precategories vs.\ exo-categories]\index{category!pre-}
  For emphasis, we list here the differences between precategories (\cref{def:precategory}) and exo-categories (\cref{def:exo-category}):
  \begin{enumerate}
  \item In precategories, the exotypes of objects and morphisms are required to be fibrant.
  \item In precategories, the axioms are formulated with respect to identifications, while in exo-categories they are formulated with respect to exo-equalities.
  \item In precategories, \cref{item:cat-homset} of \cref{def:precategory} ensures that equality of arrows is a property; in particular, that the associativity and unitality witnesses are unique.
    Such a condition would not make sense for general exo-categories; but even if the hom-exotypes of an exo-category are fibrant, they may not be sets.
  \end{enumerate}

\end{remark}

\begin{example}
Any exouniverse $\Ustrict$ gives rise to an exo-category, also denoted $\Ustrict$, with objects $A : \Ustrict$ and morphisms $\Ustrict(A,B) \eqdef A \to B$.
The corresponding fibrant universe $\U$ is a full sub-exo-category of $\Ustrict$.
\end{example}

\begin{definition}\index{functor!exo-}\index{natural transformation!exo-}
An \defemph{exo-functor} $F : \C \to \D$ consists of a function $\ob{F} : \ob{\C} \to \ob{\D}$ and functions $F_{x,y} : \C(x,y) \to \D(\ob{F}x,\ob{F}y)$ preserving identity and composition up to exo-equality. We denote both $\ob{F}$ and $F_{x,y}$ by just $F$.
An \defemph{exo-natural transformation} $\alpha : F \Rightarrow G : \C \to \D$ consists of a family of morphisms $(\alpha_x : \D(Fx, Gx))_{x : \ob{\C}}$ satisfying the naturality axiom by an exo-equality.
\end{definition}

\section{Diagram signatures}
\label{sec:foldssig}

After the introduction of diagram signatures by example in \cref{sec:folds-cats}, we now move on to a formal definition of diagram signatures in 2LTT.
Our diagram signatures are indexed by their height; diagram signatures of height $n$ and their morphisms form an exo-category.
Moreover, each diagram signature is itself an ``inverse exo-category'', with exotypes of objects and morphisms: this turns out to give a very useful midway point between the entirely internal (with types of objects and morphisms) and the entirely external (an inverse category in the metatheory, with sets of objects and morphisms).

\begin{definition}[{\cite[\S4.2]{2LTT}}]\index{category!exo-!inverse}
  An \defemph{inverse exo-category} is an exo-category $\L$ together with a functor $\mathsf{rk}:\L\to (\exo{\Nat})^{\mathrm{op}}$ (where $\exo{\Nat}$ is regarded as an exo-category with $\exo{\Nat}(m, n) \eqdef (m \leq n)$) that reflects identities.\nomenclature[rk]{$\mathsf{rk}$}{rank function of a diagram signature}
  Thus each object is assigned a natural number, called its \defemph{rank}, such that every nonidentity morphism strictly decreases rank.
  An inverse exo-category has \defemph{height} $p$ if all of its objects have rank $<p$.
  (In particular, only the empty exo-category has height 0.)
\end{definition}

In particular, therefore, an inverse exo-category has an exotype of objects $\ob{\L}$ equipped with a function $\mathsf{rk}:\ob{\L}\to\exo{\Nat}$.
It is often convenient to regard this instead as a \emph{family} of exotypes \emph{indexed} by $\exo{\Nat}$.
That is, if we write $\L(n)$ for the exotype of objects of rank $n$:
\[\L(n) \eqdef \sm{L:\ob{\L}} (\mathsf{rk}(L) \converts n),\]
then we have $\ob{\L}\cong \sm{n:\exo{\Nat}}\L(n)$.
More precisely, the slice exo-category $\Ustrict / \exo{\Nat}$ is equivalent to the functor exo-category $\exo{\Nat}\to\Ustrict$ (where $\exo{\Nat}$ is here regarded as a discrete exo-category).
Thus, if we define an \defemph{indexed inverse exo-category} to be a type family $\L:\exo{\Nat}\to\Ustrict$ together with the structure of an inverse exo-category on its image $\ob{\L}$ in $\Ustrict / \exo{\Nat}$, we obtain an equivalent notion of inverse exo-category.
(Once we define morphisms of diagram signatures in \cref{sec:funct-deriv}, we can say that the exo-categories of inverse exo-categories and of indexed inverse exo-categories are equivalent.)
We will generally pass back and forth between these two viewpoints silently, trusting the context to disambiguate.

Note that neither $\L(n)$ nor the hom-types $\hom_\L(K,L)$ need be fibrant.
However, in a diagram signature we will require sharpness of the types $\L(n)$ and cofibrancy of the following \emph{fanout} exotype, which gathers all the dependencies of a sort.

\begin{definition}
  Given an inverse exo-category $\L$, the \defemph{fan\-out exotype of $K:\L(n)$ at $m<n$} is 
  \[ \fanout{K}{m} \define \sm{L:\L(m)} \hom_\L(K,L).\]\nomenclature[fanout]{$\fanout{K}{m}$}{fanout of an inverse exo-category at sort $K$ and rank $m$}
\end{definition}

\begin{definition}\label{def:material-sig}\index{signature!diagram}
A \defemph{diagram signature of height $p$} is an inverse exo-category $\L$ of height $p$ for which 
\begin{enumerate}
 \item each $\L(n)$ is sharp; and
 \item each exotype $\fanout{K}{m}$ is cofibrant.
\end{enumerate}
 The exotype of diagram signatures of height $p$ is denoted by $\IC(p)$.\nomenclature[dsig]{$\IC(p)$}{exotype of diagram signatures of height $p$}
\end{definition}

There are several reasons for these restrictions.
One is that, as we will see in \cref{sec:deriv-folds}, they ensure that the type of structures for a diagram signature is fibrant.
They are also necessary for the definition of indiscernibility (\cref{def:indisc-in-diag-structure}).

\begin{remark}\label{rmk:exofinite-sigs}
  Many, if not most, naturally-occurring diagram signatures are finite.\footnote{Or, in the case of signatures of infinite height, they are ``locally finite'' in that the sets $\L(n)$ and $\hom_\L(K,L)$ are finite.}
  For instance, the diagram signatures $\Lcat$ and $\LcatE$ shown in \cref{fig:signatures} have four and five objects respectively, and their homsets are also finite.
  When interpreting such pictures as exo-categories in 2LTT, we interpret these finite sets as \emph{exofinite exotypes} $\exo{\Nat}_{<n}$, which as shown in \cref{sec:sharp} are sharp.

  Indeed, if we instead used finite types $\Nat_{<n}$ (at least in a na\"ive way), we would not in general obtain an exo-category, since then the associativity law could only be proven to hold up to identification, rather than up to strict equality.
  (The unit laws are not a problem, since inverse (exo-)categories have no nontrivial endomorphisms, so composition with identity morphisms can just be defined to be the identity operation.
  Note also that nontrivial instances of associativity only arise for inverse exo-categories of height $\ge 4$.)
  
  However, it will be crucial for our inductive approach explained in \cref{sec:deriv-folds} that our signatures are not \emph{required} to be finite, since the ``derivative'' operation does not preserve finiteness.
  Semantically, this extra generality is closely related to the \emph{internal inverse categories} of~\cite{shulman_univalence_ei}.

  In addition to the non-finite examples produced by derivation, in \cref{sec:examples} we will also encounter some naturally-occurring examples of infinite signatures (though still of finite height).
  For those which are \emph{countably} infinite, we can remain close to the spirit of exofiniteness by using $\exo{\Nat}$ for countably infinite families of sorts and morphisms, as long as we assume that $\exo{\Nat}$ is cofibrant (and hence sharp, by \cref{lem:sharp}).
  In addition, arbitrary types of sorts and morphisms (such as $\Nat$) are unproblematic in signatures of height $\le 3$, since then there are no nontrivial associativity relations to prove.
  All the example signatures we consider in this \paperorbook will be covered by one of these two cases.
  Moreover, we expect that uncountably infinite signatures of height $\ge 4$ can probably also be represented as strict exo-categories using a technique like that of~\cite{shulman_univalence_ei}, which encodes composition and strict associativity using type dependency.
\end{remark}

\section{Reedy fibrant diagrams}
\label{sec:struct-foldssig}

As suggested in \cref{sec:foldssig-by-eg}, a structure for a diagram signature $\L$ should be an exo-functor $\L\to \Ustrict$ such that the image of each sort ``is'' a family of fibrant types dependent on all the relevant types of lower rank.
This can be formalized with the notion of \emph{Reedy fibration}\index{fibration!Reedy} imported from homotopy theory.

Let $\L$ be an inverse exo-category, $K:\L(n)$, and $M$ an exo-functor $\L\to\Ustrict$.
We will also refer to such an $M$ as an \defemph{exo-diagram}\index{diagram!exo-} on $\L$, or just a \defemph{diagram}\index{diagram} for short.
The \emph{matching object} of $M$ at $K$ is designed to capture a family of elements of $M$ at lower-rank sorts that together provide all the dependencies that an object of $MK$ might have.
Two ways of formalizing this can be found in~\cite[Definition 4.4 and Lemma 4.5]{2LTT}; we give a third in terms of our fanout types.

\begin{definition}[{\cite[\S4.3]{2LTT}}]\label{def:matching}
  The \defemph{matching object} of an exo-functor $M:\L\to\Ustrict$ at $K:\L(n)$, denoted $\match_K M$, is the sub-exotype (cf.~\cref{sec:identity}) of
  \[ \prd{m<n}{(L,f):\fanout{K}{m}} ML \]
  consisting of those $d$ such that for any ${m_2<m_1<n}$, given ${(L_1,f_1):\fanout{K}{m_1}}$ and $(L_2,f_2):\fanout{K}{m_2}$ with ${g:\hom_\L(L_1,L_2)}$ such that $f_2 \steq g \circ f_1$, we have
  \[Mg(d(m_1,L_1,f_1)) \steq d(m_2,L_2,f_2).\]\nomenclature[matching]{$\match_K M$}{matching object of diagram $M$ at sort $K$}%
  We say $M$ is \defemph{Reedy fibrant} if for all $K$, the induced map
  \[ MK \to \match_K M \]
  that sends $x:MK$ to $\lambda m. \lambda (L,f). Mf(x)$, is a fibration.
  Let $\rfunc{\L}{\Ustrict}$ denote the exo-category of Reedy fibrant exo-diagrams.
\end{definition}

It is straightforward to verify that this definition is equivalent to the ones found in \cite{2LTT}, which in turn are rephrasings of the standard homotopy-theoretic definition.

\begin{example}[Reedy fibrancy for structures of height 1]
  If $K$ has rank 0, then there are no $m<0$, hence $\match_K M \cong \onetype$.
  Thus, Reedy fibrancy at rank-0 sorts simply means that $MK$ is fibrant.
\end{example}

\begin{example}[Reedy fibrancy for graph structures]\label{eg:rfib-gph}
  For an exo-diagram $M$ on the signature $A \rightrightarrows O$ for graphs, we have $\match_A M \cong MO \times MO$.
  Thus, $M$ is Reedy fibrant if $MO$ is fibrant and the map $MA \to MO\times MO$ is a fibration, which is to say that $M$ is determined by a type $MO:\U$ and a type family $MA : MO \to MO \to \U$.
\end{example}

\begin{example}[Reedy fibrancy for category structures]\label{def:structures-lcat-lcate}
  For an exo-diagram $M$ on the signature $\Lcat$ for categories, we have $\match_A M \cong MO \times MO$ as for graphs.
  The matching object $\match_I M$ is the pullback of $MA$ along the diagonal $MO \to MO\times MO$, or equivalently $\sm{x:MO} MA(x,x)$.
  Similarly, $\match_T M$ is the triple fiber product of three pullbacks of $MA$ to $MO\times MO\times MO$, or equivalently $\sm{x,y,z:MO} MA(x,y)\times MA(y,z) \times MA(x,z)$.
  Thus, a Reedy fibrant diagram on $\Lcat$ is determined by a type $MO:\U$, a type family $MA : MO \to MO \to \U$, and two further type families
  \begin{align*}
    MI &: \Big(\sm{x:MO} MA(x,x) \Big) \to \U\\
    MT &: \Big(\sm{x,y,z:MO} MA(x,y)\times MA(y,z) \times MA(x,z)\Big) \to \U
  \end{align*}
  or equivalently
  \begin{align*}
    MI &: \prd{x:MO} MA(x,x) \to \U\\
    MT &: \prd{x,y,z:MO} MA(x,y) \to MA(y,z) \to MA(x,z)\to \U.
  \end{align*}
  A Reedy fibrant diagram on $\LcatE$ adds to this a further type family
  \[ ME : \prd{x,y:MO} MA(x,y) \to MA(x,y) \to \U.\]
\end{example}

The meaning of ``determined by'' in these examples is somewhat subtle.
For instance, in \cref{eg:rfib-gph} it does \emph{not} mean that the exotype of Reedy fibrant diagrams is \emph{isomorphic} to the fibrant type $\sm{MO:\U} MO\to MO \to \U$.
Nor does it mean that they are \emph{equivalent as types}; indeed that doesn't even make sense, since the former may not be fibrant.
What is true is that the \emph{exo-category} $\rfunc{\L}{\Ustrict}$ is \emph{equivalent as an exo-category} to one whose (exo)type of objects is $\sm{MO:\U} MO\to MO \to \U$.

This situation is generic: the exo-category of Reedy fibrant diagrams on any diagram signature is equivalent, as an exo-category, to an exo-category with a fibrant type of objects.\footnote{For this result it would suffice to assume that each $\L(n)$ is cofibrant, as is each fanout exotype.  Our stronger assumption of sharpness of $\L(n)$ will not be needed until \cref{sec:indisc-folds,sec:FOLDS-iso-uni}; see, e.g., \cref{def:partial-struc}.}
This is essentially proven in~\cite[\S4.5]{2LTT}; we will give a different proof in \cref{sec:deriv-folds}.
It is the elements of this fibrant type that we will refer to as \textbf{$\L$-structures}.\index{structure!diagram}

\section{Derivatives of signatures and diagram structures}
\label{sec:deriv-folds}

In \cref{sec:struct-foldssig} we gave a version of the usual definition of Reedy fibrant diagrams for a diagram signature.
This definition is well-suited to arguments that are ``inductive at the top'': that is, where in the inductive step we assume that something has been done at all sorts of rank $<n$ and proceed to extend it to rank $n$.

However, our arguments will be ``inductive at the bottom'': we assume that something has been done at all sorts of rank $>0$ and proceed to extend it to rank $0$.
For this purpose we need a different characterization of Reedy fibrant structures.
The crucial observation is that if we fix the value of a structure on the rank-0 sorts, then the rest of that structure can be represented as a diagram on the following \emph{derived} signature.

\begin{definition}\label{app:defn:derivcat}\index{derivative!of a diagram signature}
Let $\L$ be an inverse exo-category of height $p >0$, and let $M : \L(0) \to \Ustrict$. The \defemph{derivative of $\L$ with respect to $M$} is the inverse exo-category $\derivcat{\L}{M}$ of height $p-1$ with objects and morphisms defined as follows:\nomenclature[L]{$\derivcat{\L}{M}$}{derivative of diagram signature $\L$ with respect to $M$}
\begin{align*}
\derivcat{\L}{M}(n) &\define \sm{K : \L(n+1)} \prd{F: \fanout{K}{0}}  M(\pi_1 F)\\
\hom_{\derivcat{\L}{M}}((K_1, \alpha_1), (K_2, \alpha_2)) &\define \sm{f: \hom(K_1,K_2)} \prd{F: \fanout{K_2}{0}}\alpha_1 (F \circ f) \steq \alpha_2(F)
\end{align*}
where $\pi_1 : \fanout{K}{0} \to \L(0)$ is the projection and $F \circ f$ denotes the 
function $\fanout{K_2}{0} \to \fanout{K_1}{0}$ given by precomposition.
\end{definition}

\begin{example}[Derivation of a structure of height 1]\label{ex:deriv-empty}
  If $p \converts 1$ then $\L_{>0}$ is empty.
  Thus, no matter what $M : \L(0) \to \Ustrict$ we choose, $\derivcat{\L}{M}$ is the empty signature.
\end{example}

\begin{example}[Derivation of a structure of height 2]\label{ex:deriv-height-one}
  If $\L$ has height $2$, then it consists of two exotypes $\L(0)$ and $\L(1)$ and a family of hom-exotypes $\hom_\L: \L(1) \to \L(0)\to \Ustrict$.
  Then for any $M : \L(0) \to \Ustrict$, the derivative $\derivcat{\L}{M}$ has height $1$, consisting of just a single exotype of sorts of rank $0$.
  Each such sort is, by definition, a sort $K:\L(1)$ in $\L$ of rank $1$ together with a function $\fanout{K}{0} \to ML$.

  As a particular example, for the diagram signature $A \rightrightarrows O$ of graphs, we have $\L(0) \strictiso \L(1) \strictiso \onetype$ and $\hom_\L(A,O) \converts \exo{\mathbf{2}}$.
  Thus $\fanout{A}{0} \strictiso \onetype \times \exo{\mathbf{2}}$, which is isomorphic to $\exo{\mathbf{2}}$.
  Hence for $M : \onetype \to \Ustrict$, which is determined up to exo-equality by a single type $MO:\U$, the derivative $\derivcat{\L}{M}$ has rank-0 sorts indexed by $\onetype \times ((\onetype \times \exo{\mathbf{2}}) \to MO)$, which is isomorphic to $MO\times MO$.
  In the future we will generally elide isomorphisms of this sort.
\end{example}

\begin{example}[Derivation of a category structure]\label{ex:deriv-lcat}
  We have $\Lcat(0) \strictiso \onetype$ (the single sort $O$), so a type family $M:\Lcat(0) \to \U$ is determined by a single type $MO$.
  The derivative $\derivcat{(\Lcat)}{MO}$ then has rank-0 sorts $A(x,y)$ indexed by (a type isomorphic to) $MO\times MO$, one family of rank-1 sorts $I(x)$ indexed by (a type isomorphic to) $MO$, and a second family of rank-1 sorts $T(x,y,z)$ indexed by (a type isomorphic to) $MO\times MO\times MO$.
  To be precise, this means the exotype of rank-1 sorts is isomorphic to $MO \exosum (MO\times MO\times MO)$; note that this is neither exofinite nor fibrant, but it is sharp.
  There is an arrow from $I(x)$ to $A(x,x)$, and arrows from $T(x,y,z)$ to $A(x,y)$, $A(y,z)$, and $A(x,z)$.

  The derivative $\derivcat{(\LcatE)}{M}$ is similar, but with a third family of rank-1 sorts $E(x,y)$ indexed by $MO\times MO$, with two arrows from $E(x,y)$ to $A(x,y)$.

  If we take the ``second derivative'' of $\Lcat$ at some $MA:MO\times MO \to \U$, we obtain a height-1 signature $\derivcat{(\derivcat{(\Lcat)}{MO})}{MA}$ whose exotype of rank-0 sorts is (isomorphic to)
  \[
    \Big(\sm{x:MO} MA(x,x)\Big) \exosum \Big(\sm{x,y,z:MO} MA(x,y) \times MA(y,z) \times MA(x,z) \Big).
  \]
  The second derivative $\derivcat{(\derivcat{(\LcatE)}{MO})}{MA}$ is similar, with an extra exo-summand $\sm{x,y:MO} MA(x,y)\times MA(x,y)$.
\end{example}

Intuitively, in $\derivcat{\L}{M}$ we take the ``indexing'' of all sorts by $O$ and move it ``outside'' the signature, incorporating it into the types of sorts.
Note that this would be impossible if our inverse categories were metatheoretic in the ordinary sense, e.g., syntactic and externally finite.
2LTT is just right.

\cref{app:defn:derivcat} applies to any inverse exo-category, but it preserves diagram signatures:

\begin{proposition}\label{prop:deriv-good}
  Let $\L$ be a diagram signature of height $p>0$ and $M : \L(0) \to \U$.
  Then the inverse exo-category $\derivcat{\L}{M}$ is a diagram signature.
\end{proposition}

\begin{proof}
Since each $\fanout{K}{0}$ is cofibrant and each $M(\pi_1 F)$ is fibrant, we have that $\prd{F: \fanout{K}{0}} M(L)$ is fibrant. Since $\L(n+1)$ is sharp, so is 
\[\sm{K : \L(n+1)} \prd{F: \fanout{K}{0}} M(\pi_1 F).\]

  Now consider $n: \exo{\Nat}_{< p}$, $m: \exo{\Nat}_{< n}$, and $(K,\alpha):\derivcat{\L}{M}(n)$, 
  We have
  \begin{align*}
    \fanout{K,\alpha}{m}
    &\converts \sm{(L,\beta): \derivcat{\L}{M}(m)} \hom_{\derivcat{\L}{M}} (K, L)
    \\ 
   & \strictiso \sum_{\substack{L: \L(m+1) \\ \beta: \prd{\fanout{L}{0}} M(L) \\  f: \hom( K,L)}}
   \prd{(N,g): \fanout{L}{0} } \alpha (N, gf) \steq \beta(N, g) 
   \\
   & \strictiso \sum_{\substack{G: \fanout{K}{m+1} \\ \beta: \prd{\fanout{L}{0}}M(L)}}  
 \prd{(N,g): \fanout{L}{0} } \alpha (N, gf) \steq \beta(N, g)
   \\
nn      & \strictiso \sm{G: \fanout{K}{m+1}} 
  \onetype   \\
  & \strictiso \fanout{K}{m+1}
  \end{align*}
  Here, we expand ${\derivcat{\L}{M}(m)}$ and $ \hom_{\derivcat{\L}{M}} (K, L)$ to get the first isomorphism. We rearrange pairs and use the definition of $\fanout{K}{m+1}$ to get the second isomorphism.
  To get the third, observe that 
  \[    \sm{\beta: \prd{\fanout{L}{0}} M(L)}  
   \prd{(N,g): \fanout{L}{0} } \alpha (N, gf) \steq \beta(N, g) 
 \]
 is isomorphic to $\onetype$.
 
Since $\fanout{K}{m+1}$ is cofibrant, so is $ \fanout{K,\alpha}{m}$.
\end{proof}

Now we can state our ``bottom-up'' characterization of Reedy fibrant diagrams, although we postpone the proof until \cref{sec:abstr-sign-transl}; see \cref{thm:deriv-reedy,thm:reedy-struc}.

\begin{definition}\label{def:folds-struc}
  Let $\L$ be a diagram signature; we define the type $\Struc{\L}$ of \defemph{$\L$-structures} inductively on its height.\nomenclature[Str]{$\Struc{\L}$}{type of structures of diagram signature $\L$}
  If $\L: \IC(0)$, we define $\Struc{\L} \define \onetype$.
  If $\L:\IC(n+1)$, we define
  \[\Struc{\L} \define  \sm{\bottom{M} : \L(0) \to \U } \Struc{\derivcat{\L}{\bottom{M}}}.\]
  We write the two components of $M:\Struc{\L}$ as $(\bottom{M},\derivdia{M})$.
\end{definition}

\begin{remark}
Technically, this is a definition by recursion of a function
\[\Struc{-} : \prd{n:\exo{\Nat}} (\IC(n) \to \U).\]
The closure properties of $\U$, and the cofibrancy of $\L(0)$, ensure that this function is well-defined.
Thus, in particular, \emph{each $\Struc{\L}$ is a fibrant type}.
In the future we will make more definitions of this sort.
\end{remark}

\begin{remark}
  Recall that the rank functor is part of the data of a diagram signature.
  It is not obvious from the definitions that the $\L$-structures of a diagram signature $\L$ are independent of the rank functor of $\L$.
  This independence will be shown in \cref{cor:struc-norank}; the comparison goes via the Reedy fibrant diagrams of \cref{sec:struct-foldssig}.
\end{remark}

\begin{notation}
  \label{notation:no-bottom}
  Given a signature $\L$, an $\L$-structure $M$, and $K : \L(0)$, we often write $MK$ instead of $\bottom{M}K$. Similarly, for $L : \L(1)$, we write $ML$ instead of $\bottom{(\derivdia{M})}L$ and so on.
\end{notation}

\begin{theorem}[\tobeprovedas{thm:reedy-struc}]
  For any diagram signature $\L$, the exo-category $\rfunc{\L}{\Ustrict}$ has fibrant hom-types, and is equivalent to an exo-category whose (exo)type of objects is $\Struc{\L}$.
\end{theorem}

\begin{example}[Structures for the signature of graphs]\label{ex:struc-gph}
  For the signature $A\rightrightarrows O$ of graphs, the type $\Struc{\L}$ is, to be completely precise,
  \begin{align*}
    \sm{MO : \onetype \to \U} \Struc{\derivcat{\L}{MO}}
    &\converts
    \sm{MO : \onetype \to \U}{MA : \onetype \times ((\onetype \times \exo{\mathbf{2}}) \to MO) \to \U}
    \Struc{\derivcat{(\derivcat{\L}{MO})}{MA}}\\
    &\converts
    \sm{MO : \onetype \to \U}{MA : \onetype \times ((\onetype \times \exo{\mathbf{2}}) \to MO) \to \U} \onetype
  \end{align*}
  (see \cref{ex:deriv-height-one}).
  However, this is isomorphic to
  \[\sm{MO : \U} (MO\times MO \to \U)
  \]
  and we will generally elide isomorphisms of this sort.
\end{example}

\begin{example}[Structures for the signature of categories (with equality)]\label{ex:struc-lcat}
  For the signature $\Lcat$ for categories (see \cref{ex:deriv-lcat}), the type $\Struc{\Lcat}$ is (isomorphic to)
  \begin{multline*}
    \sm{MO:\U}{MA:MO\times MO\to \U}
    \Big(\big(\tsm{x:MO} MA(x,x)\big) \to \U\Big) \times\\
    \Big(\big(\tsm{x,y,z:MO} MA(x,y)\times MA(y,z) \times MA(x,z)\big) \to \U\Big).
  \end{multline*}
  Similarly, $\Struc{\LcatE}$ is (isomorphic to)
  \begin{multline*}
    \sm{MO:\U}{MA:MO\times MO\to \U}
    \Big(\big(\tsm{x:MO} MA(x,x)\big) \to \U\Big) \times\\
    \Big(\big(\tsm{x,y,z:MO} MA(x,y)\times MA(y,z) \times MA(x,z)\big) \to \U\Big) \times\\
    \Big(\big(\tsm{x,y:MO} MA(x,y)\times MA(x,y)\big) \to \U\Big).
  \end{multline*}
  These are exactly as we claimed in \cref{def:structures-lcat-lcate}.
\end{example}

Essentially by definition, $\rfunc{\L}{\Ustrict}$ is a ``weak classifier'' for Reedy fibrant diagrams, in that every Reedy fibrant diagram on $\L$ is a pullback of a generic one over $\rfunc{\L}{\Ustrict}$.
However, this pullback is not in general unique.
By contrast, $\Struc{\L}$ is a \emph{strong} classifier of Reedy fibrant diagrams, in the same way that the fibrant universe $\U$ is a strong classifier of types.
This is a consequence of \cref{prop:uni1-folds} below, which we will prove in \cref{sec:structures} as \cref{prop:uni1} more generally for functorial signatures.

\begin{definition}\label{def:lvleqv-folds}\index{equivalence!of Reedy fibrant diagrams!levelwise}
  A morphism $f:M\to N$ of Reedy fibrant diagrams is a \defemph{levelwise equivalence} if each commutative square
  \[
    \begin{tikzcd}
      MK \ar[d] \ar[r] & NK \ar[d] \\
      \match_K M \ar[r] & \match_K N
    \end{tikzcd}
  \]
  is a homotopy pullback, i.e., each induced map of (fibrant) fibers is an equivalence.
  Let $M \leqv N$ denote the (fibrant) type of levelwise equivalences.\nomenclature[\bequiv1]{$M \leqv N$}{type of levelwise equivalences between diagram structures $M$ and $N$}
\end{definition}

\begin{proposition}[\tobeprovedas{prop:uni1}]\label{prop:uni1-folds}
  For any diagram signature $\L$ and $M,N:\Struc{\L}$, the canonical map $\idtolvle : (M = N) \to (M \leqv N)$ is an equivalence.\nomenclature[idtolvle]{$\idtolvle$}{function from identifications to levelwise equivalences of diagram structures}
\end{proposition}

The proof of \cref{prop:uni1}, hence also that of \cref{prop:uni1-folds}, relies on the univalence axiom;\index{axiom!univalence}
conversely, the univalence axiom can be recovered as an instance of \cref{prop:uni1-folds}, for the signature consisting of just one sort.

\begin{example}[Levelwise equivalence of pointed sets]\label{ex:sip-pointed-set}
 Consider the theory of pointed sets of \cref{eq:sig-pointed-set}, with the following underlying signature.
 \[
   \begin{tikzcd}
    P \ar[d] & E \ar[dl,shift left] \ar[dl,shift right]
    \\
    X 
  \end{tikzcd}
 \]
 A levelwise equivalence $f : M \to N$ of models of that theory is precisely an isomorphism of pointed sets, i.e., an isomorphism of sets preserving the chosen point.
\end{example}

\begin{example}[Levelwise equivalence of set-structures]\label{ex:sip-set-dia}
  Consider two exo-dia\-grams $M$ and $N$ of a diagram signature $\L$ such that at every sort of $\L$, $M$ and $N$ are sets.
  Then a levelwise equivalence between $M$ and $N$ is a natural transformation $M \to N$ that is an isomorphism at every sort of $\L$.
  Thus, in such cases levelwise equivalence produces the appropriate notion of sameness for set-level structures: that which would often be called an isomorphism of structured sets.
  We will see a number of examples of this sort in \cref{sec:set-egs}.
\end{example}

\begin{remark}\label{rem:diff-sip-up}
In particular, for set-level structures that can be encoded using diagram signatures, \cref{prop:uni1-folds} reproduces the results of \cite{COQUAND20131105} and \cite[Section 9.8]{HTT}.\footnote{The notion of signatures for set-level structures considered in \cite{COQUAND20131105} and \cite[Section 9.8]{HTT} is \emph{prima facie} more general, though we do not know of any examples of their work that cannot be expressed in ours.}
However, our work is more general because we deal also with higher-categorical structures, in which case levelwise equivalence is not the ``correct'' notion of sameness, as shown by the following example.
\end{remark}

\begin{example}[Levelwise equivalence of (pre)categories]\label{ex:sip-cat}
  Levelwise equivalences between $\Tcat$-precategories correspond precisely to \emph{isomorphisms of precategories} from \cite[Def.~6.9]{AKS13} and~\cite[Def. 9.4.8]{HTT}.
  These are functors $f:M\to N$ and $g:N\to M$ that induce equivalences on hom-types and also equivalences on types of objects (relative to homotopical \emph{identifications} of objects, not isomorphisms in the category structure).
\end{example}

In general, isomorphisms of precategories are too strong of a notion.
Instead, we would expect to consider \emph{equivalences of (pre)categories}, where the composites $g \circ f$ and $f\circ g$ on objects are only \emph{isomorphic} to the identity rather than \emph{identified} with it.
Thus, we want to replace the notion of levelwise equivalence in \cref{prop:uni1-folds,prop:uni1} with a kind of equivalence ``up to'' a notion of sameness that is \emph{derived from the structure}.
  
This latter notion of sameness for elements of a structure is what we will call \emph{indiscernibility}.
We will define it in \cref{sec:indisc-univ-folds,sec:hsip-folds}, following the ideas we described in \cref{sec:folds-cats} for the case of $\LcatE$.
Specifically, indiscernibilities are defined in \cref{def:indisc-in-diag-structure}, and equivalences of structures are defined in \cref{def:eqv-folds}.
We will then improve \cref{prop:uni1-folds} to \cref{thm:hsip-folds}, which deals with this improved kind of equivalence.

First, however, we conclude this \lcnamecref{sec:folds-signatures-two} with a definition of \emph{axioms} and \emph{theories} over a diagram signature.

\section{Axioms and theories}
\label{sec:axioms}

Unlike Makkai's notion of axiom defined using FOLDS, our axioms are \emph{not} syntactically defined through an inductive set of sentences.
Instead, we use the notion of \emph{proposition} of our ambient HoTT/UF to obtain a semantic notion of axiom.

\begin{definition}\label{def:diag-axiom}\index{axiom!for a diagram signature}
 Let $\L$ be a diagram signature. An \defemph{$\L$-axiom} is a function 
 $\Struc{\L} \to \PropU$.
\end{definition}

\begin{example}[Axioms for pointed sets]
  Consider the diagram signature $P \to X$ for pointed sets from \cref{sec:foldssig-by-eg}.
  Recall that the axiom $\exists (x : X).P(x)$ is a shorthand for $\exists (x:X).\exists (p : P(x)).\top$.
  The latter formula straightforwardly translates to the axiom
  \begin{align*}
    \Struc{\L} & \to \PropU
    \\
    M & \mapsto \exists (x:MX).\exists (p : MP(x)).\onetype.
  \end{align*}
  Here, we use the notational convention of \cref{notation:no-bottom}.
\end{example}

\begin{remark}[Axioms from FOLDS]
  \label{rem:axioms-from-folds}
  More generally, any FOLDS-axiom gives rise, in a mechanical way, to an axiom in the sense of \cref{def:diag-axiom}:
  for this, we map
  \begin{itemize}
  \item   $\top$ and $\bot$ to $\onetype$ and $\zerotype$, respectively (both of which are propositions), and
    
  \item  $\forall$, $\exists$, $\wedge$, and $\vee$ to their logical counterparts in HoTT/UF
  (where the translation of $\exists$ and $\vee$ uses propositional truncation).
  \end{itemize}

\end{remark}

\begin{example}[Axioms for categories]
 The axioms given in \cref{eq:has-comp,eq:comp-unique,eq:E-refl,eq:E-sym,eq:E-trans,eq:E-I,eq:E-cong} in \cref{sec:foldssig-by-eg} straightforwardly give rise to axioms for the signature $\LcatE$ via the translation sketched in \cref{rem:axioms-from-folds}.
\end{example}

\begin{definition}\label{def:diagram-theory}\index{theory!diagram}\index{model!of a diagram theory}
  A \defemph{diagram theory} is a pair $(\L,T)$ of a diagram signature $\L$ and a family $T$ of $\L$-axioms indexed by a cofibrant exotype.  A \defemph{model of a theory $(\L,T)$} then consists of an $\L$-structure $M$ together with a proof of $t(M)$ for each axiom $t$ of $T$.
  A \defemph{morphism of models} is a morphism of the underlying structures.
\end{definition}
For instance, a list of five $\L$-axioms can be specified by a family indexed by the exofinite exotype $\exo{\mathbf{5}}$.
The cofibrancy condition on the indexing exotype ensures that the exotype of models of a theory is fibrant. The exotypes of morphisms, of isomorphisms, and of equivalences of models are fibrant as well.

In \cref{sec:examples}, we will discuss a wide range of particular theories and their univalent models.

\chapter{Indiscernibility and univalence for diagram structures}
\label{sec:indisc-univ-folds}
\label{sec:indisc-folds}

In this \lcnamecref{sec:indisc-univ-folds} and the next we state our definitions and results about indiscernibility and univalence for diagram signatures.
We postpone many proofs until \cref{sec:ho}, where we will give them in the context of a more general notion of signature that we define in \cref{sec:abstr-sign-transl}.
However, since most of our examples are diagram signatures, we can discuss them first in \cref{sec:examples}.

We start in this \lcnamecref{sec:indisc-univ-folds} with most of the definitions of indiscernibility of objects within an $\L$-structure.%
\footnote{Our notion of indiscernibility is inspired by Makkai's notion of ``internal identity'', which has so far only been discussed in talks, but not appeared in print. See, for instance, \cite{makkai-blpc}.}
We then define a structure to be univalent when indiscernibility coincides with identification of objects.

Let $M$ be an $\L$-structure, $K : \L(0)$, and $a, b : MK$.%
\footnote{Here and below, we make use of the convention of \cref{notation:no-bottom}, writing $MK$ for $\bottom{M}K$.}
(To deal with sorts of rank $>0$, we simply derive $\L$ and $M$ enough times to bring the sort down to rank 0.)
To define indiscernibilities from $a$ to $b$, we
consider a new $\L$-structure obtained by adding to $M$ one element at sort $K$: a ``joker'' element. 
We can substitute this new element by $a$ or by $b$; 
below, we call the obtained structures $\partial_a M$ and $\partial_b M$, respectively. 
An indiscernibility from $a$ to $b$ will be defined below to be a levelwise equivalence of structures from $\partial_a M$ to $\partial_b M$ that is the identity on all the sorts not depending on the joker element.
Intuitively, this means that $a$ and $b$ are indiscernible when one cannot discern one from the other using the rest of the structure $M$.

To make this more precise, recall that for any $N: \L(0) \to \U$, we have a \emph{derivative} signature $\derivcat{\L}{N}$ (\cref{app:defn:derivcat}), and that $M$ is determined by $\bottom{M} : \L(0) \to \U$ together with an $\derivcat{\L}{\bottom{M}}$-structure $\derivdia{M}$ (\cref{{def:folds-struc}}).
Let
\[ [K] \eqdef  \lambda L. \left(  K=L \right) : \L(0) \to \U, \]
which makes sense since $\L(0)$ is sharp.
We define $\hat a: \prd{L: \L(0)} ([K](L) \to ML)$ by applying \cref{lem:sharp-id} to $a : MK$.
Let $\bottom{M}+[K]$ denote the pointwise disjoint union in $\L(0)\to \U$, and $\copair{1_M}{\hat{a}}:  \prd{L: \L(0)}  (\bottom{M}+[K])(L)\to ML $ the pointwise copairing.
There is an induced morphism of diagram signatures $\derivcat{\L}{\copair{1_{\bottom{M}}}{\hat{a}}} : \derivcat{\L}{\bottom{M}+[K]} \to \derivcat{\L}{\bottom{M}}$, along which we can pull back $\derivdia{M}$, and define
\[\partial_a M  \eqdef  (\derivcat{\L}{\copair{1_{\bottom{M}}}{\hat{a}}})^*\derivdia{M} : \Struc{\derivcat{\L}{\bottom{M} + [K]}}.\]\nomenclature{$\partial_a M$}{pullback of $\derivdia{M}$ along extension by $a$}%
This is not a complete definition since we have not defined morphisms of diagram signatures and the functoriality of derivatives and structures.
We will give these definitions in \cref{sec:abstr-sign-transl} (see specifically \cref{def:derivation_functorial_action,prop:deriv-opfib,rem:two-functorialities}), but this partial definition will suffice to state our theorems and allow the reader to understand the examples.

There is also an induced morphism $\iota_{\bottom{M}} : \derivcat{\L}{\bottom{M}} \to \derivcat{\L}{\bottom{M} + [K]}$, and the pullback of $\partial_a M$ along $\iota_{\bottom{M}}$ is $\derivdia{M}$.

\begin{definition}
  \label{def:indisc-in-diag-structure}
  For $\L:\IC(n+1)$, $K:\L(0)$, $M:\Struc{\L}$, and $a,b:MK$, we define an \defemph{indiscernibility from $a$ to $b$} to be a levelwise equivalence $\partial_a M \leqv \partial_b M$ that restricts along $\iota_{\bottom{M}}$ to the identity of $\derivdia{M}$.
  We write $a \foldsiso b$ for the (fibrant) type of indiscernibilities.\nomenclature[\bbindisc1]{$a \foldsiso b$}{type of indiscernibilities between $a$ and $b$ in a diagram structure}
\end{definition}

See \cref{def:iso-within-a-structure} for the full definition.
There is a canonical identity indiscernibility $a\foldsiso a$, which induces a map $(a=b) \to (a\foldsiso b)$.

\begin{definition}
  \label{def:univalence-cond-on-diag-structures}
  For $K:\L(0)$, a structure $M:\Struc{\L}$ is \defemph{univalent at $K$} if the map $(a=b) \to (a\foldsiso b)$ is an equivalence for all $a,b:MK$.
  We say $M$ is \defemph{univalent} if it and all its derivatives are univalent at all rank-0 sorts of their signatures.
\end{definition}

\begin{definition}
  Given a theory $\T \steq (\L,T)$, a $\T$-model is \defemph{univalent} if its underlying $\L$-structure is univalent.
\end{definition}

\begin{example}[Univalence for structures of height 1]\label{eg:ht1-univalence}
  Suppose $\L$ has height 1, hence is just a type ${\L}(0)$.
  Consider an $\L$-structure $M: \L(0) \to \U$ and $a,b:M(K)$.
  Then $\partial_a M$ and $\partial_b M$ are structures for the trivial signature of height 0, hence uniquely identified; thus $(a \fiso b)=\onetype$.
  So any structure of a signature $\L$ of height 1 is univalent just when it consists entirely of propositions.
\end{example}

\begin{example}[Univalence for types with a unary predicate]\label{eg:ht2-univalence}
  Suppose $\L$ is the signature of pointed sets of Diagram~\eqref{eq:sig-pointed-set}.
  In this case, we have $\L(0) \eqdef \onetype$ (whose single element we denote by $X$), and $\bottom{M} \eqdef MX : \U$, while $\derivdia{M}$ consists of the sorts $MP(x)$ and $ME(x,y)$.
  By  \cref{eg:ht1-univalence}, $\derivdia{M}$ is univalent exactly when all these types are propositions.

  We have $\derivcat{\L}{\bottom{M}}(0) \cong MX \exosum MX \times MX$, and thus
  \begin{align*}
  \derivcat{\L}{\bottom{M} + [X]}(0) &\cong (MX + \onetype) \exosum ((MX+ \onetype) \times (MX + \onetype))\\
  &\simeq (MX + \onetype) \exosum ((MX \times MX) + MX + MX + \onetype).
  \end{align*}
  The latter is an ``equivalence of sharp exotypes'' in the sense of \cref{sec:sharp}, which induces an equivalence of fibrant types upon mapping into $\U$.
  Thus, for $a : MX$, the structure $(\partial_a M)(0) : \derivcat{\L}{\bottom{M} + [X]}(0) \to \U$ is determined up to equivalence by the functions
  \begin{align*}
    (\lambda x. MP(x)) &: MX \to \U  \\
    (\lambda z. MP(a)) &: \onetype \to \U  \\
    (\lambda u. ME(\pi_1 u,\pi_2 u)) &: MX\times MX \to \U  \\
    (\lambda x.ME(x,a)) &: MX \to \U  \\
    (\lambda y.ME(a,y)) &: MX \to \U  \\
    (\lambda z.ME(a,a)) &: \onetype \to \U.
  \end{align*}
  An indiscernibility $a \fiso b$, a.k.a.\ a levelwise equivalence $\partial_a M \leqv \partial_b M$, thus consists (up to equivalence) of equivalences of types
  \begin{align}
    MP(x) & \simeq MP(x) \label{eq:pset-px}
    \\
    MP(a) & \simeq MP(b)\label{eq:pset-pab}
    \\
    ME(x,y) & \simeq ME(x,y) \label{eq:pset-exy}
    \\
    ME(x,a) & \simeq ME(x,b) \label{eq:pset-exab}
    \\
    ME(a,y) & \simeq ME(b,y) \label{eq:eset-eaby}
    \\
    ME(a,a) & \simeq ME(b,b) \label{eq:eset-aabb}
  \end{align}
  for all $x,y:MX$.
  The condition on restriction along $\iota$ says that the equivalences of \cref{eq:pset-px,eq:pset-exy} are the identity.

  In a univalent structure $M$, the types $MP(x)$ and $ME(x,y)$, and hence the type $a \fiso b$, are propositions.
  If we assume the axioms of \cref{eq:pset-refl,eq:pset-sym,eq:pset-trans,eq:pset-PE} stating that $E$ is a congruence, then we can show that the type $a \fiso b$ is equivalent to $ME(a,b)$, and hence the univalence condition for $X$ says that $M$ is standard.

  If we omit the sort $E$ from the signature, then an indiscernibility $a \fiso b$ is exactly an equivalence $MP(a) \simeq MP(b)$.
\end{example}

To illustrate the impact of an equality predicate on univalence, we consider, in the next example, partially ordered types.

\begin{example}[Univalence for sets with a binary relation]\label{eg:partially-ordered-types}
  Now suppose $\L$ is the following diagram signature:
    \[
    \begin{tikzcd}
     {[\le]} \ar[d, shift right] \ar[d, shift left]
     \\
     X
    \end{tikzcd}
  \]
  and let $M$ be an $\L$-structure.
  As in \cref{eg:ht2-univalence}, we have $\L(0) \eqdef \onetype$ (whose single element we denote by $X$), and $\bottom{M} \eqdef MX : \U$; here, $\derivdia{M}$ consists of the sorts $M[\leq](x,y)$, which we abbreviate as $x \leq y$.
   By  \cref{eg:ht1-univalence}, $\derivdia{M}$ is univalent exactly when all these types are propositions.

   We have that $\derivcat{\L}{\bottom{M}}(0) \cong MX \times MX$, and
  \begin{align*}
  \derivcat{\L}{\bottom{M} + [X]}(0) &\cong (MX+ \onetype) \times (MX + \onetype)\\
  &\simeq (MX \times MX) + MX + MX + \onetype.
  \end{align*}
  Thus, for $a : MX$, the structure $(\partial_a M)(0) : \derivcat{\L}{\bottom{M} + [X]}(0) \to \U$ is determined up to equivalence by the functions
  \begin{align*}
  (\lambda u.(\pi_1 u\leq \pi_2 u)) &: MX \times MX \to \U  \\
  (\lambda x.(x\leq a)) &: MX \to \U  \\
  (\lambda y.(a \leq y)) &:  MX \to U \\
  (\lambda z.(a \leq a)) &:  \onetype \to \U.
\end{align*}
An indiscernibility $a \fiso b$, a.k.a.\ a levelwise equivalence $\partial_a M \leqv \partial_b M$, thus consists of equivalences of types
  \begin{align}
    (x \leq y) & \simeq (x \leq y) \label{eq:potype-xy}
    \\
    (x \leq a) & \simeq (x \leq b) \label{eq:potype-xab}
    \\
    (a \leq y) & \simeq (b \leq y) \label{eq:potype-eaby}
    \\
    (a \leq a) & \simeq (b \leq b) \label{eq:potype-aabb}
  \end{align}
  for all $x,y:MX$.
  The condition on restriction along $\iota$ says that the equivalences of \cref{eq:potype-xy} are the identity.

  In a univalent structure $M$, the types $(x \leq y)$, and hence the type $a \fiso b$, are propositions.
  If we assume the axioms of reflexivity $x \leq x$ and antisymmetry $x \leq y \to y \leq x$, then we can show that the type $a \fiso b$ reduces to $(a \leq b) \wedge (b \leq a)$, as mentioned in \cref{sec:indis-intro}.
  The univalence condition then reads as
  \[ (a = b) \simeq (a \fiso b) \simeq \big((a \leq b) \wedge (b \leq a)\big); \]
  that is, it asserts antisymmetry.
\end{example}
Variants of this theory including an equality predicate are given and compared in \cref{eg:pre-po-sets}.

\begin{example}[Univalence for category structures]
  Recall from \cref{ex:deriv-lcat,ex:struc-lcat} that for $\L\converts\LcatE$, we have
  \begin{alignat*}{2}
    {\L}(0) &\eqdef \onetype &\qquad
    {(\derivcat{\L}{MO})}(0) &\eqdef MO \times MO \\
    \bottom{M} &\eqdef MO:\U &\qquad
    \bottom{(\derivdia{M})} &\eqdef MA : MO\times MO \to \U,
  \end{alignat*}
  while
   $\doublederivdia{M}$ consists of the sorts $MT_{x,y,z}(f,g,h)$, $MI_x(f)$, and $ME_{x,y}(f,g)$.
  By \cref{eg:ht1-univalence}, $\doublederivdia{M}$ is univalent just when all these types are propositions.
  Now for any $a,b:MO$, we have
  \[ (MA+[A(a,b)])(x,y) \simeq MA(x,y) + ((a=x)\times (b=y)). \]
  Thus, the height-1 signature $\derivcat{(\derivcat{\L}{MO})}{MA+[A(a,b)]}$ is equivalent to
  \begin{multline*}
    \big(\tsm{x,y,z:MO} (MA(x,y) + ((a=x) \times (b=y))) \\ \times (MA(y,z) + ((a=y)\times (b=z))) \times (MA(x,z) + ((a=x)\times (b=z)))\big) \\
    \exosum \big(\tsm{x:MO} (MA(x,x) + ((a=x) \times (b=x)))\big)\\
    \exosum \big(\tsm{x,y:MO} (MA(x,y) + ((a=x) \times (b=y))) \\ \times (MA(x,y) + ((a=x) \times (b=y))) \big).
  \end{multline*}
  By distributing $\sum$ and $\times$ over $+$, replacing $\exosum$ by $+$ up to equivalence,
  and using the fact that for fixed $a : A$, the type $\sm{y : A} a = y$ is contractible (cf.~\cref{sec:hlevel}) and hence equivalent to $\onetype$, the above sharp exotype is equivalent to
  \begin{align}
    \MoveEqLeft \big(\tsm{x,y,z:MO} MA(x,y)\times MA(y,z) \times MA(x,z)\big) \label{eq:sia0} \\
    &+ \big(\tsm{z:MO} MA(b,z)\times MA(a,z)\big)\label{eq:sia1}\\
    &+ \big(\tsm{x:MO} MA(x,a)\times MA(x,b)\big)\label{eq:sia2}\\
    &+ \big(\tsm{y:MO} MA(a,y)\times MA(y,b) \big)\label{eq:sia3}\\
    &+ \big((a=b)\times MA(a,b)\big)\label{eq:sia4}\\
    &+ \big(MA(a,a)\times (b=b)\big)\label{eq:sia5}\\
    &+ \big((a=a)\times MA(b,b)\big)\label{eq:sia6}\\
    &+ \big((a=b)\times (a=a)\times (b=b)\big)\label{eq:sia7}\\
    &+ \big(\tsm{x:MO} MA(x,x)\big) \label{eq:sia00}\\
    &+ \big((a=b)\big)\label{eq:sia8}\\
    &+ \big(\tsm{x,y:MO} MA(x,y) \times MA(x,y)\big) \label{eq:sia000}\\
    &+ \big(MA(a,b)\big)\label{eq:sia9}\\
    &+ \big(MA(a,b)\big)\label{eq:sia10}\\
    &+ \big((a=a) \times (b=b)\big).\label{eq:sian}
  \end{align}
  Thus for $f,g:MA(a,b)$, an identification $\partial_f M = \partial_g M$ consists of equivalences between instances of the predicates $MT,MI,ME$ indexed over the types~\eqref{eq:sia0}--\eqref{eq:sian}.
  The condition on restriction along $\iota$ says that the equivalences corresponding to~\eqref{eq:sia0}, \eqref{eq:sia00}, and~\eqref{eq:sia000} are the identity, while those corresponding to~\eqref{eq:sia1}--\eqref{eq:sia3}, \eqref{eq:sia4}--\eqref{eq:sia7}, \eqref{eq:sia8}, and~\eqref{eq:sia9}--\eqref{eq:sian} yield respectively the equivalences~\eqref{eq:fia1}--\eqref{eq:fia3}, \eqref{eq:fia4}--\eqref{eq:fia7}, \eqref{eq:fia8}, and~\eqref{eq:fia9}--\eqref{eq:fian} from \cref{sec:univalence-at-a}.
  Hence, indiscernibilities $f\fiso g$ in the sense of \cref{def:indisc-in-diag-structure,def:iso-within-a-structure} coincide with the indiscernibilities from \cref{defn:foldsiso-arrows}.

  Now moving back down to the bottom rank, an $(\derivcat{\L}{MO})$-structure consists of $MA : MO\times MO \to \U$ together with appropriately typed families $MT$, $MI$, and $ME$.
  Since \( (MO + [O]) = MO + \onetype\), for $a:MO$ the $0^{\mathrm{th}}$ rank of $\partial_a M$ is
  \[ (\partial_a M)A : (MO+\onetype) \times (MO+\onetype) \to \U \]
  or equivalently 
  \[ (\partial_a M)A : (MO\times MO) + MO + MO + \onetype \to \U \]
  consisting of the types $(MA(x,y))_{x,y:MO}$, $(MA(a,y))_{y:MO}$, $(MA(x,a))_{x:MO}$, and $MA(a,a)$.
  The $1^{\mathrm{st}}$ rank consists of $MT$, $MI$, and $ME$ pulled back appropriately to these families.
  Thus, a levelwise equivalence $\partial_a M \leqv \partial_b M$ consists of equivalences
  \begin{align}
    MA(x,y) &\simeq MA(x,y) \label{eq:sio0}\\
    MA(x,a) &\simeq MA(x,b) \label{eq:sio1}\\
    MA(a,y) &\simeq MA(b,y) \label{eq:sio2}\\
    MA(a,a) &\simeq MA(b,b) \label{eq:sio3}
  \end{align}
  for all $x,y:MO$ that respect the predicates $MT$, $MI$, $ME$.
  The condition on restriction along $\iota$ says that the equivalences~\eqref{eq:sio0} are the identity, while the remaining~\eqref{eq:sio1}--\eqref{eq:sio3} correspond respectively to the equivalences $\phi_{x\bullet}$, $\phi_{\bullet y}$, and $\phi_{\bullet\bullet}$ (\cref{item:foldsiso1a,item:foldsiso1b,item:foldsiso1c}) from \cref{sec:univalence-at-o}.
  Finally, respect for $MT$, $MI$, $ME$ specializes to \cref{eq:Txya,eq:Txaz,eq:Tazw,eq:Txaa,eq:Taxa,eq:Taax,eq:Taaa,eq:Iaa,eq:Exa,eq:Eax,eq:Eaa}.
  Thus, indiscernibilities $a\fiso b$ in the sense of \cref{def:indisc-in-diag-structure,def:iso-within-a-structure} coincide with the indiscernibilities from \cref{defn:folds-iso-obj}.
\end{example}

If $\L$ is a signature of height $n$, we can give an upper bound, in terms of $n$, for types occurring in a univalent $\L$-structure, and for the type of univalent $\L$-structures:

\begin{proposition}[\tobeprovedas{thm:hlevel}]\label{thm:hlevel-folds}
  If $n>0$, $\L$ has height $n$, $M:\Struc{\L}$ is univalent, and $K:\L(0)$, then $MK$ is an $(n-2)$-type.
\end{proposition}

\begin{proposition}[\tobeprovedas{thm:hlevel1}]\label{thm:hlevel1-folds}
  If $\L$ has height $n$, then the type of univalent $\L$-structures is an $(n-1)$-type.
\end{proposition}

\begin{example}[Homotopy levels for univalent categories]
  For the diagram signature $\LcatE$ of height $3$, \cref{thm:hlevel-folds} states that the type of objects of a univalent $\LcatE$-structure is a $1$-type.
  In particular, the type of objects of a univalent $\Tcat$-category is a 1-type.

  Similarly, \cref{thm:hlevel1-folds} states that the type of univalent $\LcatE$-structures is a $2$-type.
  Thus, since this type contains the type of univalent $\Tcat$-categories as a subtype (cf.~\cref{sec:identity}), the latter is also a $2$-type.
\end{example}

Finally, and perhaps surprisingly, we note that in general, morphisms of structures need \emph{not} preserve indiscernibility.
The following toy example makes the point; we will see in \cref{eg:premonoidal} that this can also fail in ``real-world'' categorical structures.

\begin{example}[A structure morphism that doesn't preserve indiscernibilities]\label{eg:indis-notpres}
  Let $\L$ be the height-2 signature such that an $\L$-structure consists of a type $MA$ and a binary relation $MR : MA \to MA \to \U$.
  Univalence at $R$ means that each $MR(x,y)$ is a proposition; while $a\fiso b$, for $a,b:MA$, means that $MR(x,a)\leftrightarrow MR(x,b)$ for all $x$, $MR(a,x)\leftrightarrow MR(b,x)$ for all $x$, and $MR(a,a)\leftrightarrow MR(b,b)$.

  Let $M$ be the $\L$-structure with $MA = \{a,b\}$ and $MR(x,y)$ always false, and $N$ the $\L$-structure with $NA=\{a,b,c\}$ with $NR$ always false except that $NR(a,c)$ is true.
  Let $\bottom{f}:MA\to NA$ be the inclusion, so that $\derivdia{M} = (\derivcat{\L}{\bottom{f}})^* \derivdia{N}$.
  Then $a\fiso b$ in $M$, but $\bottom{f}a \not\fiso \bottom{f}b$ in $N$.
\end{example}

Functors between categories, and morphisms between most of the other categorical examples to be discussed in \cref{sec:examples}, do generally preserve indiscernibilities.
But this is only because the indiscernibilities in such cases admit an equivalent ``diagrammatic'' characterization by a suitable ``Yoneda lemma'' (as described for categories in \cref{sec:folds-cats}).
Note that the existence of such a Yoneda lemma depends on the theory (i.e., the axioms) as well as the signature.
We do not know a general condition on a theory ensuring that morphisms between its structures preserve indiscernibilities.

\chapter{The univalence principle for diagram structures}
\label{sec:hsip-folds}

Our goal is to prove a univalence principle for a notion of equivalence of univalent structures that is \textit{a priori} weaker than levelwise equivalence.
In the case of (pre)categories $M,N$, there are two natural candidates for such a notion.
(Recall the notions of surjective and split-surjective function from \cref{sec:logic}.)
\begin{itemize}
\item A \emph{weak equivalence} is a functor $f:M\to N$ that is fully faithful (each function $MA(x,y) \to NA(fx,fy)$ is an isomorphism of sets) and essentially surjective ($\prd{y:NO} \Vert \sm{x:MO} (fx \cong y) \Vert$).
\item A \emph{(strong) equivalence} is a functor $f:M\to N$ for which there is a functor $g:N\to M$ and natural isomorphisms $f g \cong 1_N$ and $g f \cong 1_M$.
  By~\cite[Lemma 6.6]{AKS13}, this is equivalent to being fully faithful and \emph{split} essentially surjective ($\prd{y:NO} \sm{x:MO} (fx \cong y)$).
\end{itemize}
In addition, there are two important related auxiliary notions:
\begin{itemize}
\item A \emph{surjective weak equivalence} is a functor $f:M\to N$ that is fully faithful (each function $MA(x,y) \to NA(fx,fy)$ is an isomorphism of sets) and surjective on objects ($\prd{y:NO} \Vert \sm{x:MO} (fx=y) \Vert$).
\item A \emph{split-surjective equivalence} is a functor $f:M\to N$ that is fully faithful (each function $MA(x,y) \to NA(fx,fy)$ is an isomorphism of sets) and split-surjective on objects ($\prd{y:NO} \sm{x:MO} (fx=y)$).
\end{itemize}
Note that the latter two do not require knowing what an isomorphism between objects is.
Furthermore, fully-faithfulness can be split into fullness (each function $MA(x,y) \to NA(fx,fy)$ is surjective\footnote{Or split-surjective; in the presence of faithfulness the two are equivalent.}) and faithfulness (each function $MA(x,y) \to NA(fx,fy)$ is injective), while faithfulness is equivalent to \emph{surjectivity on equalities}: each function $ME_{x,y}(p,q) \to NE_{fx,fy}(fp,fq)$ is surjective (which implies a similar property for $T$ and $I)$.
This suggests the following generalizations that apply to all diagram structures.

\begin{definition}\label{def:vss-folds}
  A morphism $f:M\to N$ of Reedy fibrant diagrams over a diagram signature $\L$ is a \defemph{surjective weak equivalence} (resp.\ a \defemph{split-surjective equivalence}) if for all sorts $K$, the maps on fibers induced by the commutative square
  \begin{equation}
    \begin{tikzcd}
      MK \ar[d] \ar[r,"f_K"] & NK \ar[d] \\
      \match_K M \ar[r,"\match_K f"'] & \match_K N
    \end{tikzcd}\label{eq:reedy-sq}
  \end{equation}
  are surjective (resp.\ split-surjective).
\end{definition}

Note that being a surjective or split-surjective weak equivalence is not obviously a symmetric notion:
it does not seem straightforward to ``invert'' such an equivalence.
But with the help of \cref{thm:hsip-vss-folds}, below, we can obtain an inverse to a split-surjective equivalence between \emph{univalent} structures.

Makkai defined surjective weak equivalences under the name \emph{very surjective morphisms}; other names for them include \emph{Reedy surjections} and \emph{trivial fibrations}.
Unfortunately, we are currently unable to prove our desired general result with surjective weak equivalences, so for the present we restrict to the split-surjective equivalences.
We write $M \strucequiv N$ for the type of split-surjective equivalences from $M$ to $N$.\nomenclature[\cequiv1]{$M\strucequiv N$}{type of split-surjective equivalences between diagram structures $M$ and $N$}
Our first main result is:

\begin{theorem}[\tobeprovedas{thm:hsip}]\label{thm:hsip-vss-folds}
  For any diagram signature $\L$ and $M,N: \Struc{\L}$ such that $M$ is univalent, the canonical map
  \[ \idtovss: (M = N) \to (M \strucequiv N)\]\nomenclature[idtosse]{$\idtovss$}{function from identifications to split-surjective equivalences of diagram structures}%
  is an equivalence.
\end{theorem}

Makkai was unable to define a general notion of non-surjective equivalence directly, instead considering \emph{spans} of surjective equivalences.
However, with our notion of indiscernibility we can avoid this detour.\footnote{Also, although spans of surjective equivalences give the correct \emph{relation} of equivalence, they do not give the correct \emph{homotopy type of} equivalences, unless the apices of the spans are constrained to be univalent so that \cref{thm:hsip-vss-folds} applies.}
If $z:\match_K M$, we denote the map on fibers induced by the square~\eqref{eq:reedy-sq} by $f_{K,z} : MK_z \to NK_{\match_K f(z)}$.

\begin{definition}\label{def:eqv-folds}
  A morphism $f:M\to N$ of Reedy fibrant diagrams over a diagram signature $\L$ is an \defemph{equivalence} if for all sorts $K$ and all $z:\match_K M$ we have $\prd{y:NK_{\match_K f(z)}}\sm{x:MK_z} f_{K,z}(x) \foldsiso y$.
  Similarly, it is a \defemph{weak equivalence} if for all $K$ and $z$ we have $\prd{y:NK_{\match_K f(z)}}\Vert \sm{x:MK_z} f_{K,z}(x) \foldsiso y\Vert$.
\end{definition}

We write $M\simeq N$ for the type of equivalences.\nomenclature[\equiv1]{$M\simeq N$}{type of equivalences between diagram structures $M$ and $N$}
Our second main result is:

\begin{theorem}[\tobeprovedas{thm:hsip2}]\label{thm:hsip-folds}
  For any diagram signature $\L$ and $M,N: \Struc{\L}$ that are both univalent, the canonical map
  \[\idtoeqv: (M = N) \to (M \simeq N)\]\nomenclature[idtoeqv]{$\idtoeqv$}{function from identifications to equivalences of diagram structures}%
  is an equivalence.
\end{theorem}

\begin{example}[Equivalence of univalent category structures]\label{eg:vss-cat}
  An equivalence between univalent $\Tcat$-categories is the same as a fully faithful and split essentially surjective functor, which by \cite[Lemma~6.6]{AKS13} is the same as an equivalence of categories.
  Thus, \cref{thm:hsip-folds} specializes to \cite[Theorem~6.17]{AKS13}. See \cref{sec:eq-of-cat} for a few more details.
\end{example}

One of Makkai's goals was to define, for a given (diagram) signature $\L$, a language for properties that are invariant under $\L$-equivalence.
He calls such invariance the ``Principle of Isomorphism'' \cite{MakSFAM}:
\begin{quote}
    The basic character of the Principle of Isomorphism is that of a constraint on the
    language of Abstract Mathematics; a welcome one, since it provides for the separation of sense from nonsense. 
\end{quote}
Working in 2LTT, we do not need to devise a language for invariant properties ourselves; instead, we can rely on the homotopical fragment of 2LTT to sufficiently constrain our language.
Recall our notion of ``axiom'' from \cref{def:diag-axiom}.

\begin{corollary}[of \cref{thm:hsip-folds}]
  \label{cor:axiom-invariance}
 Any $\L$-axiom $t$ is invariant under equivalence of univalent $\L$-structures:
 given univalent $\L$-structures $M$, $N$ and an equivalence $M \simeq N$, then $t(M) \leftrightarrow t(N)$.
\end{corollary}

\begin{remark}[Axioms invariant under weak equivalences]\label{eg:ax-weq}
  We anticipate that one can construct a ``univalent completion'' operation that associates, to any structure $M$ of a signature $\L$, its free univalent completion $\widehat{M}$, together with a weak equivalence $M \to \widehat{M}$.
  In light of this completion, it would make sense to restrict our notion of $\L$-axiom to those maps $\Struc{\L} \to \PropU$ that are invariant under weak equivalence.
  We have not checked that all the axioms presented in the examples of \cref{sec:examples} are indeed invariant under weak equivalence.
  Most of our axioms can be expressed in Makkai's language FOLDS \cite{MFOLDS}, which was designed to be invariant under equivalence; we expect it to serve this function in our context as well, though we have not verified it for our notion of ``weak equivalence''.
\end{remark}

\begin{remark}\label{rmk:equality}
  We can now finally give a more comprehensive explanation of the inclusion of equality sorts (and their associated axioms) in our theories.
  In \cref{sec:foldssig-by-eg} we introduced these sorts in order to state axioms involving equality in the style of Makkai's FOLDS.
  However, our notion of ``axiom'' in \cref{def:diag-axiom} is so general that it allows us to formulate such axioms without equality sorts; we can simply refer directly to the identification types of the other sorts.
  Why then do we include equality sorts in our signatures?

  One answer is that, as noted in \cref{eg:ax-weq}, we hope that there is a more restrictive notion of axiom that would be invariant under weak equivalence, and we expect that equality sorts would be needed to express axioms involving equality in such a way.
  But in addition, the inclusion or exclusion of equality sorts in a signature has a direct effect on the resulting notions of indiscernibility and univalence for its structures.
  We have seen that when a sort $B$ at one below top rank has an equality sort $E_B$ above it, then univalence at $B$ says simply that $B$ is a family of sets with standard equality $E_B$ (and in \cref{sec:strict-egs} we will discuss a way to extend this to sorts at lower ranks as well).
  But depending on the signature, omitting $E_B$ could lead to different notions of indiscernibility and univalence at $B$. For instance, in \cref{eg:partially-ordered-types}, we give a signature (without an equality sort) together with axioms whose univalent structures are partially ordered sets. In \cref{item:preset} of \cref{eg:pre-po-sets}, we add to that signature with axioms an equality sort together with axioms asserting that it is reflexive and a congruence for the other sorts: then the univalent structures are \emph{pre-ordered} sets. That is, by adding an equality sort, we are able to capture a wider class of structures.
  Thus, including or excluding equality sorts is one way to ``fine-tune'' the resulting notion of univalent structure.
\end{remark}

In \cref{sec:ho} we will prove all of the results stated above, in fact obtaining them as special cases of analogous results for a higher-order notion of signature.
However, before delving into that, in \cref{sec:examples} we will survey a large number of examples that fit into the first-order framework of diagram signatures.

\part{Examples of diagram structures}\label{sec:examples}

In this \lcnamecref{sec:examples}, we present diagram theories (as defined in \cref{def:diagram-theory}) for many mathematical structures.
We usually spell out the signatures explicitly, but describe the axioms only informally.
However, most axioms could be formally stated in the language of FOLDS described in \Cref{sec:foldssig-by-eg}, and thus obtained via the translation sketched in \Cref{rem:axioms-from-folds}.

For each theory, we describe its models, its indiscernibilities, and its univalent models.
In most cases, we also describe the morphisms and equivalences (in the sense of \cref{def:eqv-folds}) between univalent models (although sometimes full univalence is not needed to characterize the morphisms).
We aim to show that in most if not all cases, the indiscernibilities coincide with the ``expected'' notion of isomorphism or internal equivalence, while the equivalences between models correspond to the ``expected'' notions of equivalence used in practice.

We start in \Cref{sec:set-egs} by considering theories over signatures of height $\le 2$.
The univalent models of such theories are sets, or families of sets, equipped with some structure.
Of particular interest are \cref{eg:prop,eg:set}, which illustrate the effect of adding an equality sort as described in \cref{rmk:equality}.
Similarly, \cref{eg:pre-po-sets} continues \cref{eg:partially-ordered-types}, studying the difference between preordered sets and partially ordered sets, the latter presented both with and without an equality sort.

In addition, in \cref{eg:fol} we show that our diagram theories include the theories of traditional first-order logic.
Since our logic is purely relational, this requires encoding functions in terms of their graphs.

In \cref{sec:1cat-egs} we consider theories for categories with extra structure built from functors and natural transformations.
This includes categories with certain specified limits or colimits, as well as categories with monoidal structures, and so on.
We also consider other 1-categorical structures such as multicategories, categorical structures for the interpretation of type theories, and many others.

Again, since our structure is purely relational, functors must be encoded like functions, in terms of their graphs.
The relevant kind of ``graph of a functor'' is an \emph{anafunctor} (a notion also due to Makkai~\cite{makkai:avoiding-choice}), which we discuss in \cref{eg:anafunctors}.
Univalence helps to ensure that this representation is accurate: univalent saturated anafunctors correspond precisely to functors (without any need for an axiom of choice, in contrast to the situation in set-based category theory).

We also discuss in \cref{sec:1cat-egs} some interesting examples of signatures where the type of \emph{sorts} (of a given rank) is not a set but a higher type.
Such signatures are particularly useful to specify ``unbiased'' operations involving some symmetry, that is, operations defined for any arity where inputs can be swapped.
Examples include unbiased symmetric monoidal categories (\cref{eg:unbiased-sym-monoidal}) and unbiased (or ``fat'') symmetric multicategories (\cref{eg:fat-sym-multicats}).

In \cref{sec:hcat-egs} we study theories for higher-categorical structures, such as bicategories and double categories.
As always, the functorial operations on such structures must be encoded as anafunctors.
A prototypical example is the representation of bicategories as \emph{anabicategories} in \cref{eg:bicats}.
Although we do not discuss it in detail, similar methods can be used to represent weak $n$-categories for any finite $n$ (but not $n=\infty$, as our theory does not yet handle infinite-height structures).

In \cref{sec:strict-egs} we show how to encode \emph{strict} categorical structures in our framework.
This includes strict structures in the usual higher-categorical sense, such as strict 2-categories, as compared to the \emph{weak} notions such as bicategories that we studied in \Cref{sec:hcat-egs}.
But it also includes ``strict 1-categories'' in the sense of \cref{sec:categories-hott}, having a set of objects rather than a more general type of objects.
The two are closely related; e.g., a strict 2-category in the first sense must have hom-categories that are strict in the second sense.
Importantly, all sorts of strict structures are obtained by \emph{adding}, to the weak theories, additional structure and properties in the form of additional equality sorts.

The prototypical example, of strict 1-categories, is discussed in \cref{eg:strict-1-cat}.
When defining strict 2-categories (\cref{eg:strict-2-cat}), we do not equip the sort of \emph{objects} with a strict equality; thus while a strict 1-category is strict at the level of objects, a strict 2-category is strict only at the level of 1-morphisms.
This is a closer match for the way strict 2-categories are traditionally used.
In principle it would be perfectly possible to also consider ``ultra-strict'' 2-categories with equality of objects, but we expect such things to occur even more rarely in practice.

In \cref{sec:graph-egs} we study theories with signatures of height 3 that are not categorical: e.g., theories of objects and arrows that lack composition or identities, such as directed multigraphs and Petri nets.
The resulting notions of indiscernibility and univalence are a little strange, suggesting that our theory is better-suited to categorical structures; but with the imposition of strictness conditions as in \cref{sec:strict-egs} we can eliminate the strange behavior.

In \cref{sec:restr-indis-egs} we study ``enhanced'' (higher) categories, that is, categories with additional structure that is not described purely in terms of functors and natural transformations.
In most cases, we can show that the indiscernibilities in such structures coincide with a well-known notion of ``good'' isomorphism.
Perhaps the best-known example is $\dagger$-categories (\cref{ex:dagger}), which have a ``reversal'' operation on 1-morphisms; in this case the indiscernibilities are the ``unitary'' isomorphisms.

Finally, in \cref{sec:unnat}, we study theories of categories equipped with even less-categorical operations, such as non-functorial operations on objects, or unnatural transformations.
Such structures often arise in the study of semantics of programming languages.
Again, our notion of indiscernibility tends to coincide in these examples with well-known notions of ``good'' isomorphism.
For example, in a \emph{thunk-force category} (\cref{eg:thunk-force}; a.k.a.\ an \emph{abstract Kleisli category}), the structure includes a certain unnatural transformation, and the indiscernibilities of objects are the isomorphisms on which this transformation is natural.
Similarly, in a \emph{premonoidal category} (\cref{eg:premonoidal}) the indiscernibilities are the \emph{central isomorphisms}.

\subsection*{Terminology and Notation}\label{sec:no-models}

Throughout this \lcnamecref{sec:examples}, we often omit the adjective ``diagram''.
By ``signatures'' and ``theories'', we always mean diagram signatures and diagram theories, as opposed to the functorial signatures and diagrams of \Cref{sec:ho}.

  Moreover, as previously suggested in \cref{notn:no-models-cats}, when only one structure $M$ is being discussed, we will often abuse notation by dropping the ``$M$'' in front of its interpretation of the sorts, writing, for instance, $x : O$ instead of $x : MO$.
  However, when more than one structure is under consideration (such as when discussing morphisms of structures), we will always retain the structure names on all the sorts.

\chapter{Structured sets}
\label{sec:set-egs}

As previously noted, the goal of \cref{sec:examples} is to explore a large number of examples of mathematical structures, to get a feel for how widely applicable our notions of indiscernibility and univalence are.
In this chapter we begin with some fairly trivial examples; later we will build up to more complicated ones.

\begin{example}[Propositions]\label{eg:prop}\index{proposition}
  The theory of propositions has the following underlying signature:
  \[
    \begin{tikzcd}
      P
    \end{tikzcd}
  \]
  and no axioms.
  A univalent model of this theory is exactly a proposition.
  A morphism of such models is an implication; an equivalence is a bi-implication.
\end{example}

\begin{example}[Sets]\label{eg:set}\index{set}
  \label{eg:sets}
  The theory of sets has the following underlying signature:
  \[
    \begin{tikzcd}
      E \ar[d, shift right] \ar[d, shift left]
      \\
      X
    \end{tikzcd}
  \]
  We assume axioms turning $E$ into an equivalence relation.
  In a univalent model, $E(a,b)$ is a proposition for any $a,b:X$.
  For any $a,b:X$, the type $a \fiso b$ of indiscernibilities is then also a proposition, and furthermore $a \fiso b \leftrightarrow E(a,b)$.
  Thus, univalence at $X$ signifies that $X$ is a set with equality given by $E$.

  A morphism of univalent models is just a function of sets; an equivalence of such models is a bijection.
\end{example}

\begin{remark}
  In \cref{eg:sets}, we can view the signature of sets to be obtained from that of propositions by adding an equality predicate on top, thus ``bumping up'' the homotopy level. It is natural to ask whether one can similarly obtain a theory for 1-types, or $n$-types more generally.
  A naïve attempt to define a theory of 1-types might start out with the following signature:
  \[
    \begin{tikzcd}
      E_2 \ar[d, shift right] \ar[d, shift left]
      \\      
      E_1 \ar[d, shift right] \ar[d, shift left]
      \\
      X
    \end{tikzcd}
  \]
  However, it is not clear to us if there are suitable axioms on a structure for this signature ensuring that $X$ is a 1-type \emph{with identifications given by $E_1$}.
  One solution is to add sorts and axioms to the signature to obtain the theory of groupoids (which is just the theory of categories with an extra invertibility axiom); a univalent model then is exactly a 1-type.
  We are not aware of a simpler theory of 1-types.
  Similarly, the simplest theory of $n$-types that we know of is obtained by adding invertibility axioms to a theory of $n$-categories (see \cref{sec:restr-indis-egs}).
\end{remark}

\begin{example}[Preordered and partially ordered sets, continuing \cref{eg:partially-ordered-types}]\index{set}\index{preorder}\index{poset}
\label{eg:pre-po-sets}
  In this \lcnamecref{eg:pre-po-sets}, we consider three very similar theories:
  \begin{enumerate}
   \item \label{item:poset} \textit{Partially ordered sets with equality sort:}
Consider a theory of partially ordered sets with underlying signature:
  \[
    \begin{tikzcd}
     {[\le]} \ar[d, shift right] \ar[d, shift left]
     &
     E \ar[dl, shift right] \ar[dl, shift left]
     \\
     X.
    \end{tikzcd}
  \]
  We write the relation $[\le](x,y)$ infix as $x\le y$.
  We assume that $E$ is a congruence for $[\le]$, and we furthermore assume axioms for reflexivity, transitivity and antisymmetry: $(x \le y) \to (y \le x) \to E(x,y)$. 
  Given $a, b : X$ in a model, we have $a \fiso b \leftrightarrow E(a,b)$. If the model is univalent, then both these types are propositions (by univalence at $[\le]$ and $E$) and coincide with $a = b$ (by univalence at $X$).
  A univalent model of this theory thus consists of a set $X$ with equality given by $E$, equipped with a partial order.
  
  \item \label{item:preset} \textit{Preordered sets:}
    Now consider the theory of preordered sets, with underlying signature as above, where
    $E$ is assumed to be a congruence for $[\le]$, and we furthermore assume reflexivity, and transitivity, but not antisymmetry.
    Given $a, b : X$ in a model, we have that $a \fiso b \leftrightarrow E(a,b)$. If the model is univalent, then both these types are propositions and coincide with $a = b$.
    A univalent model of this theory thus consists of a set $X$ with equality given by $E$, equipped with a preorder.
  
  \item \label{item:potype} \textit{Partially ordered sets without equality sort:}
    For comparison, recall the theory studied in \cref{eg:partially-ordered-types}, which has underlying signature:
   \[
     \begin{tikzcd}
       {[\le]} \ar[d, shift right] \ar[d, shift left]
       \\
       X.
     \end{tikzcd}
   \]
   We assert axioms of reflexivity and transitivity, but not antisymmetry (note that there is no relation $E$ with respect to which this axiom could be stated).
   Given a model, univalence at $[\le]$ ensures that $[\le]$ is pointwise a proposition.
   Given $a,b : X$, the type $a \fiso b$ then is a proposition and equivalent to $(a \le b) \wedge (b \le a)$ (using reflexivity and transitivity).
   Univalence at $X$ therefore entails antisymmetry, stated with respect to identifications.
   Of course, it also entails that $X$ is a set.
   A univalent model of this theory thus consists of a set $X$ equipped with a partial order.
  \end{enumerate}
In conclusion, the univalent models of the theory of \cref{item:potype} are antisymmetric without this being explicitly postulated as an axiom; they are thus the same as the univalent models of the theory of partially ordered sets of \cref{item:poset}.
To obtain a theory of preordered sets where elements can be distinguished beyond the distinction induced by the order, it suffices to add, to the theory of \cref{item:potype}, a dedicated equality relation, as in \cref{item:preset}.

A morphism of univalent models of any of these theories is a function that is monotone, i.e., preserves the inequality.
It is an equivalence if it is a bijection that also reflects the inequality, i.e., an isomorphism of pre- or partially-ordered sets.
\end{example}

\begin{example}[First-order logic]\label{eg:fol}\index{logic!first-order}
  Consider an arbitrary many-sorted first-order theory $T$ with only relation symbols.\index{theory!first-order}
  We can make this a diagram signature with one rank-0 sort for each sort of $T$ and one rank-1 sort for each relation symbol of $T$, plus equality sorts (assumed to be congruences):
  
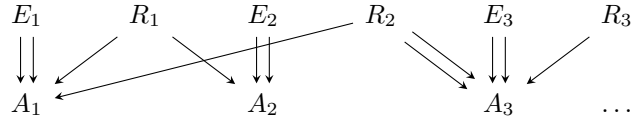
\[
    \begin{tikzcd}
      E_1 \ar[d, shift right] \ar[d, shift left]
      &
      R_1 \ar[dl] \ar[dr]
      &
      E_2 \ar[d, shift right] \ar[d, shift left]
      &
      R_2 \ar[dlll] \ar[dr, shift right] \ar[dr, shift left]
      &
      E_3 \ar[d, shift right] \ar[d, shift left]
      &
      R_3 \ar[dl]
      \\
      A_1
      &
      &
      A_2
      &
      &
      A_3 & \dots
    \end{tikzcd}
  \]
  In the most common cases, $T$ has finitely many sorts and relations, so this is an exofinite signature.
  However, since it has only height 2, there are no nontrivial compositions, so it would be unproblematic to allow arbitrary sharp exotypes (including fibrant types) of sorts and relations (see \cref{rmk:exofinite-sigs}).

  As always, since $E_i$ and $R_i$ have nothing dependent on them, univalence at those sorts simply makes them proposition-valued.
  And since $E_i$ is a congruence, by a similar argument as in \cref{thm:2-univ-is-1-univ}, univalence at $A_i$ makes it a set whose equality is $E_i$.
  Thus, we recover first-order logic with equality.\index{logic!first-order with equality}
  Our logic has only relations and no functions, but as noted before we can always encode a function as a relation using its graph.
  Any instance of this example, with sorts $(A_i)_{i : I}$, is also an instance of the SIP \cite[Section~9.9]{HTT} over $\Set^I$, including, for instance, posets (in which case we recover \cref{eg:pre-po-sets}\ref{item:poset}), monoids, groups, and fields.\index{Structure Identity Principle}\index{poset}\index{monoid}\index{group}\index{field}
  In particular, any essentially algebraic theory is a first-order theory, hence can be represented via a signature of this form.\index{theory!essentially algebraic}

  A morphism of univalent structures is a function that preserves the truth of the relation symbols, i.e., a homomorphism of first-order structures.
  Note that when a function is encoded by their graphs, such preservation by a morphism is equivalent to its commuting with the functional actions.
  Such a morphism is an equivalence if it is a bijection that also reflects the truth of relations, i.e., an isomorphism of first-order structures.
\end{example}

In this way, finitary first-order theories of structured sets more or less coincide with theories formulated on exofinite diagram signatures of height 2.
Similarly, when interpreted semantically they yield the usual models of such first-order theories in the category of sets, or in more general 1-toposes.\index{set!category of}\index{topos}

However, our notion of diagram signature is more general than this, because the exotypes $\L(n)$ of sorts of rank $n$ are not required to be sets.\index{signature!diagram}
This allows the direct incorporation of group actions (or higher group actions) in structures, such as the following.\index{group!action}

\begin{example}[Combinatorial Species]\label{eg:species}\index{species!combinatorial}
  A \emph{combinatorial species}~\cite{joyal:species} is defined as a presheaf of sets on the groupoid of finite sets.\index{set!finite}
  We can represent these as models of a theory with a height-2 signature in which $\L(0)$ and $\L(1)$ are both $\FinSet$, the 1-type of finite sets defined by
  \begin{align*}
  \FinSet
  &\eqdef \tsm{X:\U} \exists(n:\Nat). (X \simeq [n])\\
  &\converts \tsm{X:\U} \big\Vert \tsm{n:\Nat} (X \simeq [n]) \big\Vert.
  \end{align*}
  We denote the elements of $\L(0)$ and $\L(1)$ corresponding to $X:\FinSet$ by $A_X$ and $E_X$ respectively.
  The signature includes two morphisms $E_X \rightrightarrows A_X$ for every $X$, and we assert axioms making each of these an equivalence relation on $A_X$.
  (Because the height is 2, there are no nontrivial compositions, so there is no problem with making this a strict exo-category even though it contains fibrant types --- see \cref{rmk:exofinite-sigs}.)
  The fact that $\L(0)$ and $\L(1)$ are not sets makes it hard to draw this signature non-misleadingly, but we can give it a try, denoting the identifications in these types by loops:
  \[
    \begin{tikzcd}
      E_{[0]} \ar[d,shift left] \ar[d,shift right] &
      E_{[1]} \ar[d,shift left] \ar[d,shift right] &
      E_{[2]} \ar[d,shift left] \ar[d,shift right] & &
      E_{[3]} \ar[d,shift left] \ar[d,shift right] & &
      \dots \\
      A_{[0]} &
      A_{[1]} &
      A_{[2]} \ar[loop,out=-30,in=30,looseness=4,"S_2"'] & &
      A_{[3]} \ar[loop,out=-30,in=30,looseness=4,"S_3"'] \ar[loop,out=-25,in=25,looseness=4] & &
      \dots
    \end{tikzcd}
  \]
  Here $[n]$ denotes the standard $n$-element finite set $\{0,1,2,\dots,n-1\}$, while $S_n$ denotes its automorphism group, the symmetric group on $n$ elements.\index{group!symmetric}\index{group!of automorphisms}
  These loop ``arrows'' here are not morphisms in the inverse exo-category that constitutes the signature for combinatorial species. Rather, they represent identifications --- for example, the permutation $(12) : S_3 \eqdef (A_{[3]} = A_{[3]}) $. These act (by transport) on the types $ A_{[n]}$ in any structure,%
\footnote{In terms of the corresponding functorial signature $\L$ (see \cref{sec:abstr-sign-transl}), whose underlying $\bottom{\L}$ is given by the type $\FinSet$, a structure includes, in particular, a map $\bottom{M} \steq MA : \FinSet \to \U$. Any isomorphism $[n] \cong [n]$ of finite sets corresponds to an identification by the univalence axiom, and thus is mapped, by $\bottom{M}$, to an identification, and hence an equivalence of types, $MA_{[n]} \simeq MA_{[n]}$.}
  and there is a sense in which both kinds of arrows can be regarded as morphisms in the same ``category'' (see, e.g.,~\cite{shulman_univalence_ei}).
  We have not notated loops at the types $E_{[n]}$, but they are there too.\index{category!exo-}\index{signature!functorial}\index{axiom!univalence}

  As in \cref{eg:fol}, univalence at $E_{[n]}$ makes it consist of propositions, while univalence at $A_{[n]}$ makes it a set with equality $E_{[n]}$.
  Thus, a univalent structure for this theory is simply a function $A : \FinSet \to \Set$, which is the natural formalization of a combinatorial species in HoTT/UF (see, e.g.,~\cite{yorgey:thesis}).
  And in the simplicial set model, $\Set$ and $\FinSet$ are respectively interpreted by homotopy 1-types that are equivalent to the nerves of the usual groupoids of sets and finite sets; thus a function $\FinSet \to \Set$ is equivalent to a functor between these groupoids, which is the classical notion of combinatorial species.\index{simplicial set model of univalent foundations}

  A morphism $f : M \to N$ of models of this theory consists of functions $f_{[n]} : MA_{[n]} \to NA_{[n]}$ that commute with the actions on $MA_{[n]}$ and $NA_{[n]}$  induced by $S_n$.%
  \footnote{In terms of the corresponding functorial theory (see \cref{sec:abstr-sign-transl}),
    a morphism of models $f : M \to N$ consists, in particular, of a morphism of type families $\bottom{M} \to \bottom{N}$, which automatically commutes with the equivalences induced by the identifications $S_n$.}
  A morphism of models is an equivalence precisely when all the $f_{[n]}$ are bijections.

  In summary, the action on arrows of combinatorial species, and the naturality condition on morphisms of species, are encoded in the homotopical structure of the type $\FinSet$ that serves as an indexing type in our theory of combinatorial species.
\end{example}

\chapter{Structured 1-categories}
\label{sec:1cat-egs}

Generally speaking, structures on 1-categories involve height-3 signatures.
Of course, the example of 1-categories themselves was discussed at length in \cref{sec:folds-cats}.

\begin{example}[Categories with binary products]\label{eg:bin-product}\index{category!with binary products}
 Binary products could be asserted to exist in a category purely in terms of axioms,
 on top of the signature $\Lcat$ of categories of \cref{sec:folds-cats}.
 However, the resulting morphisms of models would not say anything about products.
 
 Instead, we can integrate some data from the products into the signature as follows:
  
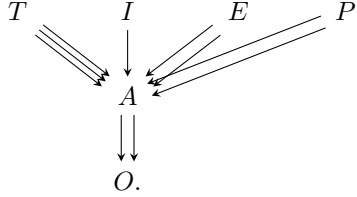
\[
     \begin{tikzcd}[ampersand replacement=\&]
      T \ar[dr] \ar[dr, shift left] \ar[dr, shift right] \& I \ar[d] \& E \ar[dl, shift right] \ar[dl, shift left]
      \& P \ar[dll, shift left] \ar[dll, shift right]
      \\
      \& A \ar[d, shift left] \ar[d, shift right]
      \\
      \& O.
    \end{tikzcd}
\]
 Here we assert suitable equalities of arrows such that $P$ depends on three objects $x,y,z:O$ and two arrows $f:A(y,x)$ and $g:A(y,z)$.
 We assert with axioms that there is a $w:P_{x,y,z}(f,g)$ if and only if $x \xleftarrow{f} y \xrightarrow{g} z$ is a product diagram.

 Univalence at $P$ then means that $P$ is pointwise a proposition.
 We assert with axioms that the equality $E$ is a congruence for $P$ as well as for the other top-sorts, ensuring that univalence at $A$ still means that $A$ is pointwise a set with equality $E$.
 An indiscernibility $a \fiso b$ is an isomorphism $\phi : a \cong b$ that is compatible with $P$ in the sense that, e.g., $P_{x,a,z}(f,g) \leftrightarrow P_{x,b,z}(f \circ \phi^{-1}, g \circ \phi^{-1})$.
 But the compatibility with $P$ is automatic, and hence an indiscernibility $a \fiso b$ is simply an isomorphism $a \cong b$.
 
 A morphism of structures is a functor between the underlying categories that preserves product diagrams.
 An equivalence of univalent structures is an equivalence of categories (which of course preserves product diagrams).

 As variants of this theory we consider, in \cref{eg:cart-monoidal-cats}, the theory of cartesian monoidal categories, and, in \cref{eg:cat-w-binprod-functor}, the theory of categories with functorially specified binary products.
\end{example}

\begin{example}[Categories with pullbacks]\label{eg:pullbacks}\index{category!with pullbacks}
  Pullbacks can be added to the theory of categories analogously to binary products (\cref{eg:bin-product}).
  This gives rise to the following signature:
  
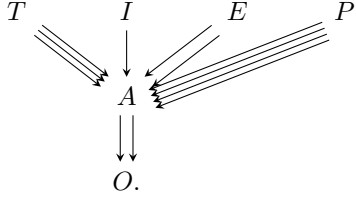
\[
     \begin{tikzcd}[ampersand replacement=\&]
      T \ar[dr] \ar[dr, shift left] \ar[dr, shift right] \& I \ar[d] \& E \ar[dl, shift right] \ar[dl, shift left]
      \& P \ar[dll, shift left]\ar[dll, shift left=2] \ar[dll, shift left = 3]\ar[dll]
      \\
      \& A \ar[d, shift left] \ar[d, shift right]
      \\
      \& O.
    \end{tikzcd}
\]
  Here we assert suitable equalities of arrows such that $P$ depends on objects and morphisms forming a diagram
 \[
  \begin{tikzcd}[ampersand replacement=\&]
     z \ar[d, "h"] \ar[r, "k"] \&  x \ar[d, "f"]
     \\
     y \ar[r, "g"] \& w.
  \end{tikzcd}
 \]
 We assert with axioms that there is a $w : P_{w,x,y,z}(f,g,h,k)$ if and only if this diagram commutes and is a pullback square.

 The discussion of univalent models of this theory is then analogous to that of \cref{eg:bin-product};
 in particular, an indiscernibility $a \fiso b$ of objects $a,b:O$ is simply an isomorphism $a \cong b$.

 We can give a similar treatment to categories with other limits and/or co\-limits.
\end{example}

\begin{example}[Presheaves]\label{eg:presheaves}\index{presheaf}
 The theory of a category with a presheaf on it has the following underlying signature:
 
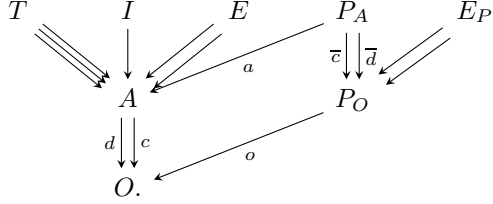
\[
   \begin{tikzcd}
     T \ar[dr] \ar[dr, shift left] \ar[dr, shift right]
     &
     I \ar[d]
     &
     E \ar[dl, shift left] \ar[dl, shift right]
     &
     P_A \ar[d, shift left, "\overline{d}"] \ar[d, shift right, "\overline{c}"'] \ar[dll, "a"]
     &
     E_P \ar[dl, shift left] \ar[dl, shift right]
     \\
     &
     A \ar[d, shift left, "c"] \ar[d, shift right, "d"']
     &
     &
     P_O \ar[dll, "o"]
     \\
     &
     O.
   \end{tikzcd}
 \]
Here we assert the equations $da \steq o\overline{d}$ and $ca \steq o\overline{c}$, among others.
Given a structure for this signature, the proposition $P_A(f, a, b)$ signifies that the function $P(f)$ maps $b$ to $a$.
For this to be a well-defined function, the axioms we impose to carve out the presheaves among the structures must include
\[
 \begin{gathered}
  P_A(f, a, b) \to P_A(f, a', b) \to E_P(a, a')
  \\
 \forall (f:A(x,y)). \forall (b:P_O(y)). \exists (a:P_O(x)). P_A(f,a,b).
 \end{gathered}
 \]
 as well as functoriality axioms.
 We also assume that $E$ and $E_P$ are congruences for all the top-rank sorts, in addition to the usual category axioms for $(O,A,T,I,E)$.
 
Given a model, univalence at $P_A$ means that $P_A$ is a proposition pointwise, indicating that $P(f)(b) = a$.
Univalence at $P_O$ means that $P_O(x)$ is a set with equality $E_P$, as expected for a presheaf.
Similarly, univalence at $A$ means that $A(x,y)$ is a set with equality $E$, as in a category: the additional dependency $P_A$ doesn't disrupt this since $E$ is a congruence for it as well.

Finally, an indiscernibility between $x,y:O$ consists of an indiscernibility $x\fiso y$ (hence just an isomorphism $g:x\cong y$) in the underlying category, together with a coherent bijection on values of the presheaf $P_O(x) \cong P_O(y)$.
Since this ``coherence'' includes in particular respect for $P_A$, it follows by an argument similar to that of \cref{sec:folds-cats} that it must be simply the functorial action of $g$.
Thus, a univalent model of this theory is precisely a univalent category together with a presheaf on it.\index{category!univalent}

A map $f : \hom(M,N)$ of structures consists of a functor $f$ between the underlying categories (given by the components $O$, $A$, $T$, $I$, $E$), together with a natural transformation between the presheaves specified by $M$ and $N$ (given by the components $P_O$, $P_A$, and $E_P$). Here, the component of $f$ on $P_A$ encodes naturality.
If $f$ is an equivalence of structures, then its underlying functor is an equivalence; moreover essential surjectivity on the component $P_A$ implies injectivity of the underlying natural transformation, whereas on $P_O$ it implies that the natural transformation is pointwise surjective---thus it is a natural isomorphism.
\end{example}

\begin{example}[(Ana)functors; {\cite[Section~6]{MFOLDS}}]\label{eg:anafunctors}\index{functor}\index{anafunctor}
Just as we can represent functions between sets by relations, we can describe two categories and a functor between them by adding ``relations'' between their objects and morphisms:
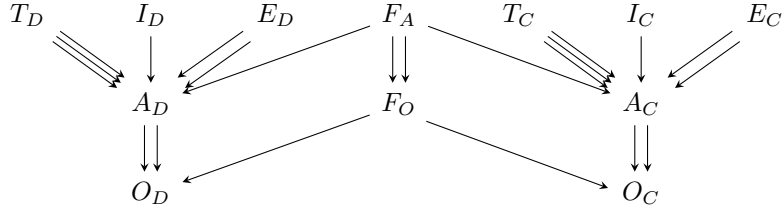
\[
  \begin{tikzcd}
    T_D \ar[dr] \ar[dr, shift left] \ar[dr, shift right]
    &
    I_D \ar[d]
    &
    E_D \ar[dl, shift left] \ar[dl, shift right]
    &
    F_A \ar[d, shift left] \ar[d, shift right] \ar[dll] \ar[drr]
    &
    T_C \ar[dr] \ar[dr, shift left] \ar[dr, shift right]
    &
    I_C \ar[d]
    &
    E_C \ar[dl, shift left] \ar[dl, shift right]
    \\
    &
    A_D \ar[d, shift left] \ar[d, shift right]
    &
    &
    F_O \ar[dll] \ar[drr, shift right]
    &
    &
    A_C \ar[d, shift left] \ar[d, shift right]
    \\
    &
    O_D
    &
    &
    &
    &
    O_C
  \end{tikzcd}
\]
with the obvious equations on arrows. Here, $O_D$ is the sort of objects of the domain category (with $D$ in $O_D$ standing for ``domain''), $O_C$ the sort of objects in the codomain, and $F_O(x,y)$ the sort of ``witnesses that $F(x) = y$''.
For instance, if $F$ is a cartesian product functor, then an element of $F_O((x_1,x_2),y)$ would be a product diagram $x_1 \leftarrow y \to x_2$.
In general, if $F$ is an arbitrary functor, a standard way to make it an anafunctor is by defining $F_O(x,y)$ to be the set of isomorphisms $F x \cong y$.
Note that $F_O$ does not generally consist of mere propositions: an object $y:O_C$ can ``be the image'' of $x:O_D$ in more than one way.
(For instance, a single object can be a cartesian product of two objects $x_1,x_2$ in more than one way, and there can be more than one isomorphism $F x \cong y$.)
We impose an axiom stating that $F$ does have a value on each possible input object, i.e., $\forall (x : O_D). \exists (y : O_C). F_O(x,y)$.

Given two such witnesses $w_1:F_O(x_1,y_1)$ and $w_2:F_O(x_2,y_2)$ and morphisms $g:A_D(x_1,x_2)$ and $h:A_C(y_1,y_2)$, an element of $F_A(w_1,w_2,g,h)$ represents the assertion that $F(g) = h$ according to the witnesses $w_1$ and $w_2$.
We impose an axiom saying that given $w_1,w_2,g$ there exists a unique $h$ such that $F_A(w_1,w_2,g,h)$; we often denote this unique $h$ by $F_{w_1,w_2}(g)$.
We also assert that $E_D$ and $E_C$ are congruences for the relation $F_A$, and that composition and identities are preserved.
In the case of composition, this means $F_{w_1,w_3}(g\circ f) = F_{w_2,w_3}(g) \circ F_{w_1,w_2}(f)$; while in the case of identities it means $F_{w,w}(1_x) = 1_y$ for any $w:F_O(x,y)$.

The result is what Makkai~\cite{makkai:avoiding-choice} calls an \emph{anafunctor}, whose ``values'' can be specified only up to isomorphism.
If dependencies are forgotten, it can be thought of as a span of functors between (pre)categories in which the first leg is a surjective equivalence ($F_O$ is the type of objects of the middle category and $F_A$ its type of arrows).
Note that for any $w_1:F_O(x,y_1)$ and $w_2:F_O(x,y_2)$ we have $F_{w_1,w_2}(1_x) : A_C(y_1,y_2)$, which is an isomorphism with inverse $F_{w_2,w_1}(1_x)$; thus ``any two values of $F$ are canonically isomorphic''.
In addition, we impose the ``existential saturation'' axiom that for any $w:F_O(x,y)$ and isomorphism $g:y\cong z$ in the codomain, there exists a $w':F_O(x,z)$ such that $F_{w,w'}(1_x) = g$: that is, any object $z$ isomorphic to a value $y$ of $F(x)$ is also a possible value of $F(x)$, in such a way that the induced isomorphism $y\cong z$ is the specified one.

As in \cref{eg:presheaves}, univalence at $A_D$ and $A_C$ just means they are sets whose equalities are the congruences $E_D$ and $E_C$ (this requires the assumption that these are congruences for $F_A$ as well).
Univalence at $F_O$ is more subtle, since $F_O$ doesn't come with a specified congruence.
For $w_1,w_2:F_O(x,y)$, the type of indiscernibilities $w_1\fiso w_2$ is the proposition asserting that $w_1$ and $w_2$ induce the same action on all arrows whose domain or codomain (or both) is $x$.
That is, for any $g:A_D(x,x')$ and $w':F_O(x',y')$ we have $F_{w_1,w'}(g) = F_{w_2,w'}(g)$, and likewise when $g:A_D(x',x)$ or $g:A_D(x,x)$.
In particular, taking $w'=w_2$ and $g = 1_x$ this implies that $F_{w_1,w_2}(1_x) = F_{w_2,w_2}(1_x) = 1_y$.
But conversely, if we have $w_1,w_2:FO(x,y)$ such that $F_{w_1,w_2}(1_x) = 1_y$, then for any $g:A_D(x,x')$ and $w':F_O(x',y')$ we have
\[ F_{w_1,w'}(g) = F_{w_1,w'}(g\circ 1_x) = F_{w_2,w'}(g) \circ F_{w_1,w_2}(1_x) = F_{w_2,w'}(g) \circ 1_y = F_{w_2,w'}(g)
\]
by functoriality, and similarly in the other cases.

Thus, $w_1\fiso w_2$ is equivalent to $F_{w_1,w_2}(1_x) = 1_y$; hence in a univalent model this latter equality implies $w_1=w_2$.
This implies that for any $w:F_O(x,y)$ and $g:y\cong z$ the $w'$ in the existential saturation axiom is unique; for if $F_{w,w'}(1_x) = F_{w,w''}(1_x) = g$, then by functoriality $F_{w',w''}(1_x) = g \circ g^{-1} = 1_z$ and hence $w'=w''$.
Thus, a univalent model for our theory is a \emph{saturated} anafunctor in Makkai's sense, whose values are determined \emph{exactly} up to unique isomorphism.\index{anafunctor!saturated}
(The standard way of making a functor into an anafunctor, with $F_O(x,y) \eqdef (F x \cong y)$, is always saturated.)

Saturation, in turn, ensures that univalence at $O_D$ and $O_C$ reduces to ordinary univalence of the underlying domain and codomain categories.
For instance, \emph{a priori} an indiscernibility $y_1\fiso y_2$ in $O_C$ is an isomorphism $\phi : y_1\cong y_2$ in the codomain category (arising from the specified equivalences on the sorts $A_C$, $T_C$, $I_C$, and $E_C$) equipped with a transport function for $F_O$ that respects $F_A$.
Explicitly, this means we have bijections $\phi_{x\bullet}: F_O(x,y_1) \cong F_O(x,y_2)$ that commute with $F_A$, in the sense that
\begin{equation*}
  F_{w,w'}(g) = F_{\phi_{x\bullet}(w),w'}(g) \circ \phi
\end{equation*}
for all $w'$, and similarly in the other variable and in both variables together.
But by the uniqueness aspect of saturation, $F_{\phi_{x\bullet}(w),w'}(g) = F_{w,w'}(g)\circ \phi^{-1}$ uniquely determines $\phi_{x\bullet}(w)$ if it exists.
Specifically, setting $w'\eqdef w$ and $g \eqdef 1_x$ in the equation above yields the condition
  $F_{\phi_{x\bullet}(w), w}(1_x) = \phi$; by the uniqueness part of saturation, applied to $w$ and $\phi$, this condition determines $\phi_{x\bullet}(w)$.

On the other hand, for any $w$ a $\phi_{x\bullet}(w)$ with this property does exist, by existential saturation combined with functoriality.
Specifically, given $w : F_O(x, y_1)$, we set $\phi_{x\bullet}(w)$ to be the unique $\trans \phi w$ such that $F_{w,\trans \phi w}(1_x) = \phi$ using saturation.
  For $g : D_A(x, x')$ and $w' : F_O(x', y_2)$ we then have
  $F_{w, w'}(g) = F_{\trans \phi w, w'}(g) \circ F_{w, \trans \phi w}(1_x) = F_{\trans \phi w, w'}(g) \circ \phi$, as required.
Thus, an indiscernibility in $O_C$ is nothing but an ordinary isomorphism in the codomain category, so univalence at $O_C$ reduces to univalence of the latter category.

A similar argument applies at $O_D$.
The only difference is that we haven't asserted existential saturation directly on the domain, so we have to prove that for any isomorphism $\phi : x_1\cong x_2$ and $w:F_O(x_1,y)$ there is a $w':F_O(x_2,y)$ such that $F_{w,w'}(\phi)= 1_y$.
However, we have \emph{some} $w'':F_O(x_2,y')$ and an isomorphism $F_{w,w''}(\phi) : y \cong y'$, so by existential saturation we can transport $w''$ back to a $w':F_O(x_2,y)$ with the desired property.

Once we have univalence at both $F_O$ and $O_C$, we can prove that for each $x:O_D$ the type $\sm{y:O_C} F_O(x,y)$ is a proposition.
For if we have $y,y':O_C$ with $w:F_O(x,y)$ and $w':F_O(x,y')$, then there is an isomorphism $h:y\cong y'$ such that $F_{w,w'}(1_x)= h$.
By univalence at $O_C$, this $h$ comes from an identification $p:y=y'$, and the transported witness $\trans p w : F_O(x,y')$ acts on arrows by conjugating with $h$.
But since $F_{w,w'}(1_x)= h$, by functoriality this implies that $\trans p w$ acts the same as $w'$ on all arrows; hence by univalence at $F_O$ they are equal, and so $(y,w) = (y',w')$ in $\sm{y:O_C} F_O(x,y)$.

Thus $\sm{y:O_C} F_O(x,y)$ is a proposition; but since it is inhabited (this is one of the axioms of an anafunctor), it is contractible.
In particular, there is a \emph{function} $F:O_D \to O_C$ with $W: \prd{x:O_D} F_O(x,F(x))$.
We can then make this into the object-function of an ordinary functor $F: C \to D$, from which the original saturated anafunctor can be recovered in the standard way with $F_O(x,y) \eqdef (F x \cong y)$.
Thus, although Makkai originally introduced anafunctors in~\cite{makkai:avoiding-choice} to avoid using the axiom of choice, in univalent foundations any anafunctor is represented by an actual functor, since the only relevant choices are unique ones.\index{axiom of choice}
(This was also observed, in somewhat different terminology, in~\cite{AKS13}.)
However, the notion of anafunctor is still useful because it is the only way to represent a functor via a single structure in our framework.

Finally, a morphism of models $f:M\to N$ consists of functors between the domain and codomain categories, say $f_D : M_D \to N_D$ and $f_C :M_C \to N_C$, together with functions $MF_O(x,y) \to NF_O(f_D(x),f_C(y))$ that respect $F_A$.
Applying this to the $MW(x) : MF_O(x,MF(x))$ constructed above, we obtain $NF_O(f_D(x),f_C(MF(x)))$.
Since we also have \[NW(f_D(x)) : NF_O(f_D(x),NF(f_D(x))),\] we obtain an isomorphism $f_C(MF(x)) \cong NF(f_D(x))$.
Respect for $F_A$ implies that these isomorphisms are natural in $x$, and they determine the entire morphism of models uniquely (this uses saturation again).
That is, a morphism between saturated anafunctors is a square of functors that commutes up to isomorphism.
\end{example}

\begin{example}[Profunctors]\index{profunctor}
A profunctor from $C$ to $D$ is a functor $F: C^{\text{op}} \times D \to \Set$.
We can represent two categories and a profunctor between them using the following signature:
 
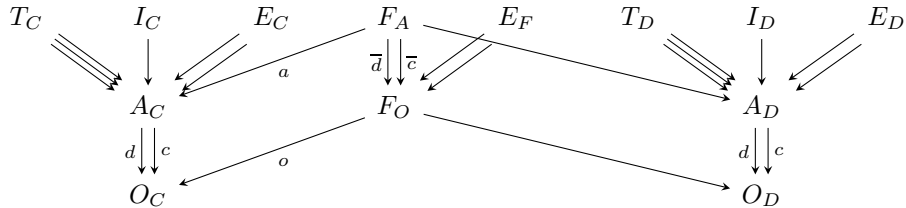
\[
  \begin{tikzcd}
    T_C \ar[dr] \ar[dr, shift right] \ar[dr, shift left]
    &
    I_C \ar[d]
    &
    E_C \ar[dl, shift right] \ar[dl, shift left]
    &
    F_A \ar[d, shift left, "\overline{c}"] \ar[d, shift right, "\overline{d}"'] \ar[dll, "a"] \ar[drrr]
    &
    E_F \ar[dl, shift right] \ar[dl, shift left]
    &
    T_D \ar[dr] \ar[dr, shift right] \ar[dr, shift left]
    &
    I_D \ar[d]
    &
    E_D \ar[dl, shift right] \ar[dl, shift left]
    \\
    &
    A_C \ar[d, shift left, "c"] \ar[d, shift right, "d"']
    &
    &
    F_O \ar[dll, "o"] \ar[drrr]
    &
    &
    &
    A_D \ar[d, shift left, "c"] \ar[d, shift right, "d"']
    \\
    &
    O_C
    &
    &
    &
    &
    &
    O_D
 \end{tikzcd}
\]
This looks very much like the signature for anafunctors, but we include an equality relation on $F_O$, and moreover the composition equations are different: we have $da \steq o\overline{c}$ and $ca \steq o\overline{d}$ imposing contravariance in the first factor.
The rest of the theory of this structure is just a two-sided version of \cref{eg:presheaves}.

Note that an adjunction between two categories is uniquely determined by a profunctor satisfying ``representability'' axioms on both sides.
Thus, by adding axioms to this example we obtain a theory for two categories together with an adjunction between them.
A different theory for the latter could be obtained by explicitly representing two categories with one anafunctor (\cref{eg:anafunctors}) in each direction between them, together with unit and counit natural transformations (\cref{eg:nat-trans}); we leave the details to the reader.
\end{example}

\begin{example}[Natural transformations]\label{eg:nat-trans}\index{natural transformation}
 The signature for two categories $C$ and $D$, two (ana)functors $F, G : D \to C$, and a natural transformation $\Lambda : F \to G$ can be given as follows.
 
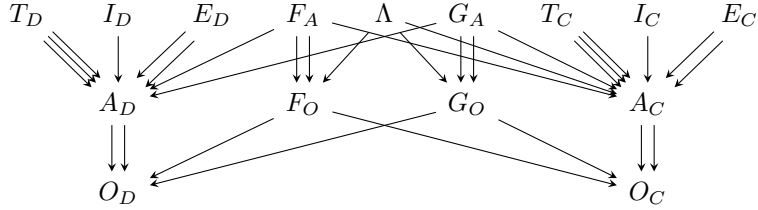
\[
  \begin{tikzcd}[column sep=small]
    T_D \ar[dr] \ar[dr, shift left] \ar[dr, shift right]
    &
    I_D \ar[d]
    &
    E_D \ar[dl, shift left] \ar[dl, shift right]
    &
    F_A \ar[d, shift right] \ar[d, shift left] \ar[dll] \ar[drrrr]
    &
    \Lambda \ar[dl] \ar[dr] \ar[drrr]
    &
    G_A \ar[dllll] \ar[drr] \ar[d, shift left] \ar[d, shift right]
    &
    T_C \ar[dr] \ar[dr, shift right] \ar[dr, shift left]
    &
    I_C \ar[d]
    &
    E_C \ar[dl, shift right] \ar[dl, shift left]
    \\
    &
    A_D \ar[d, shift right] \ar[d, shift left]
    &
    &
    F_O \ar[dll] \ar[drrrr]
    &
    &
    G_O \ar[dllll] \ar[drr]
    &
    &
    A_C \ar[d, shift left] \ar[d, shift right]
    \\
    &
    O_D
    &
    &
    &
    &
    &
    &
    O_C
  \end{tikzcd}
 \]
Here the composites $\Lambda \to F_O \to O_D$ and $\Lambda \to G_O \to O_D$ are equal, so that $\Lambda$ depends on one element of $O_D$ (but two elements of $O_C$).
We write an instance of $\Lambda$ as $\Lambda_{x,y,z}(w_1, w_2, \lambda)$, with $x : O_D$, $y, z : O_C$, $w_1 : F_O(x,y)$, $w_2 : G_O(x,z)$, and $\lambda : A_C(y,z)$; this signifies that $\lambda$ is the value of $\Lambda$ on the object $x : O_D$ (relative to the values $y$ and $z$ for $F(x)$ and $G(x)$).
For instance, if $\Lambda : F\Rightarrow G$ is an ordinary natural transformation between ordinary functors, and we make $F$ and $G$ into anafunctors in the standard way with $F_O(x,y) \eqdef (F x \cong y)$ and similarly for $G$, then we would define $\Lambda_{x,y,z}(w_1, w_2, \lambda)$ to mean that $\lambda : A_C(y,z)$ is equal to the composite $y \cong F x \xrightarrow{\Lambda_x} G x \cong z$.

In addition to the saturated anafunctor axioms for $F$ and $G$ separately, we assert that for any $x : O_D$, the fiber of $\Lambda$ over $x$ is inhabited.
Since $\sm{y:O_C} F_O(x,y)$ and $\sm{z:O_C} G_O(x,z)$ are contractible this equivalently says that for any $x$, $y$, $z$, $w_1$, $w_2$ there exists a $\lambda$ such that $\Lambda_{x,y,z}(w_1, w_2, \lambda)$.
In addition, we assert that such a $\lambda$ is unique (using equality $E_C$), and that these $\lambda$s are natural with respect to morphisms in $A_D$, in the sense that for any $h:A_D(x,x')$ with $w_1:F_O(x,y)$ and $w_1':F_O(x',y')$ while $w_2:G_O(x,z)$ and $w_2':G_O(x',z')$, if $\Lambda_{x,y,z}(w_1,w_2,\lambda)$ and $\Lambda_{x',y',z'}(w_1',w_2',\lambda')$ then $\lambda' \circ F_{w_1,w_1'}(h) = G_{w_2,w_2'}(h) \circ \lambda$ (expressed precisely using $T_C$, of course).
Finally, we assert that $E_C$ is a congruence for the relation $\Lambda$.

This last axiom ensures that the addition of $\Lambda$ doesn't change the notion of indiscernibility in $A_C$.
To show that it also doesn't change the notion of indiscernibility in $F_O$, we must show that given $w,w':F_O(x,y)$ such that $F_{w,w'}(1_x) = 1_y$, we necessarily have $\Lambda_{x,y,z}(w,v,\lambda) \to \Lambda_{x,y,z}(w',v,\lambda)$ for any $v:G_O(x,z)$ and $\lambda:A_C(y,z)$.
But if $\Lambda_{x,y,z}(w,v,\lambda)$ and $\Lambda_{x,y,z}(w',v,\lambda')$, then by naturality we have
\[\lambda' \circ 1_y = \lambda' \circ F_{w,w'}(1_x) = G_{v,v}(1_x) \circ \lambda = 1_z \circ \lambda,\]
i.e., $\lambda' = \lambda$.
A similar argument applies to $G_O$; hence univalence at these sorts once again says just that $F$ and $G$ are saturated anafunctors.

An indiscernibility $y\fiso y'$ in $O_C$ consists of an isomorphism $\phi : y\cong y'$ such that additionally $\Lambda_{x,y,z}(w_1,w_2,\lambda) \to \Lambda_{x,y',z}(\trans \phi {w_1},w_2,\lambda\circ \phi^{-1})$, where $\trans \phi {w_2}$ is the (necessarily unique) witness of $F_O(x,y')$ obtained from $w_2$ and $\phi$ by saturation, and similarly for other holes.
But since $F_{w_1,\trans\phi{w_1}}(1_x) = \phi$, this follows automatically from naturality.

Similarly, an indiscernibility $x\fiso x'$ in $O_D$ consists of an isomorphism $\phi : x\cong x'$ such that $\Lambda_{x,y,z}(w_1,w_2,\lambda) \to \Lambda_{x',y,z}(\trans \phi {w_1},\trans \phi {w_2},\lambda)$, and similarly for other holes.
Here $\trans \phi {w_1}$ is obtained as sketched in \cref{eg:anafunctors}, by choosing some witness $w_1':F_O(x',y')$ and transporting it back across the isomorphism $F_{w_1,w_1'}(\phi)$.
Analogously, $\trans \phi {w_2}$ is obtained by choosing some witness $w_2':G_O(x',z')$ and transporting back along $G_{w_2,w_2'}(\phi)$.
Since we have $F_{w_1,\trans \phi {w_1}}(\phi) = 1_y$ and $G_{w_2, \trans \phi {w_2}}(\phi) = 1_z$, the desired implication follows again by naturality.
Thus, univalence at $O_D$ and $O_C$ reduce respectively to ordinary univalence of the domain and codomain categories.
\end{example}

\begin{example}[Structures on categories]\label{eg:cat-struc}\index{category!with structure}
  Any structure on a category or family of categories that can be expressed in terms of functors and natural transformations can be represented by combining copies of \cref{eg:anafunctors,eg:nat-trans}, perhaps with domains and/or codomains identified.
  For instance, here is the signature for a category equipped with an endofunctor:
  
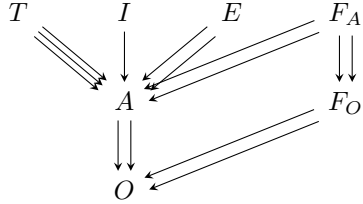
\[
   \begin{tikzcd}
    T \ar[dr] \ar[dr, shift left] \ar[dr, shift right]
    &
    I \ar[d]
    &
    E \ar[dl, shift left] \ar[dl, shift right]
    &
    F_A \ar[d, shift left] \ar[d, shift right] \ar[dll, shift left] \ar[dll, shift right]
    \\
    &
    A \ar[d, shift left] \ar[d, shift right]
    &
    &
    F_O \ar[dll, shift left] \ar[dll, shift right]
    \\
    &
    O
   \end{tikzcd}
 \]
 Here and in the remainder of this example, we leave it to the reader to formulate suitable axioms on top of the signatures we draw.\index{endofunctor}
 
  Here is the signature for a pointed endofunctor\index{endofunctor!pointed} (an endofunctor equipped with a natural transformation $\eta:\mathrm{Id}\to F$):
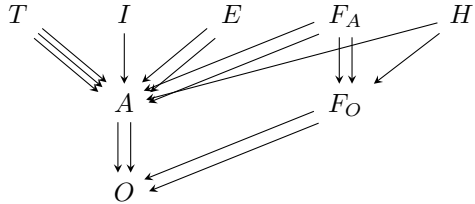
   \[
   \begin{tikzcd}
    T \ar[dr] \ar[dr, shift left] \ar[dr, shift right]
    &
    I \ar[d]
    &
    E \ar[dl, shift left] \ar[dl, shift right]
    &
    F_A \ar[d, shift left] \ar[d, shift right] \ar[dll, shift left] \ar[dll, shift right]
    &
    H \ar[dl] \ar[dlll]
    \\
    &
    A \ar[d, shift left] \ar[d, shift right]
    &
    &
    F_O \ar[dll, shift left] \ar[dll, shift right]
    \\
    &
    O
   \end{tikzcd}
  \]
  Since the domain of $\eta$ is the identity functor, we only need one arrow $H\to F_O$.
  An instance $H_{x,y}(w,\eta)$ for $w:F_O(x,y)$ and $\eta:A(x,y)$ says that $\eta$ is the component of the transformation at $x$ relative to the witness $w$ that $F(x) = y$.
  Similarly, here is the signature for a monad:
  
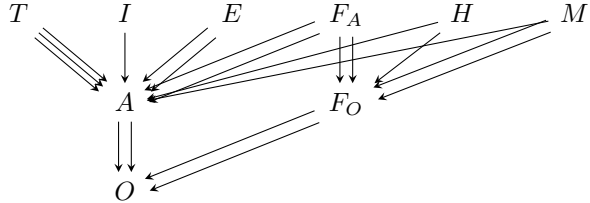
\[
   \begin{tikzcd}
    T \ar[dr] \ar[dr, shift left] \ar[dr, shift right]
    &
    I \ar[d]
    &
    E \ar[dl, shift left] \ar[dl, shift right]
    &
    F_A \ar[d, shift left] \ar[d, shift right] \ar[dll, shift left] \ar[dll, shift right]
    &
    H \ar[dl] \ar[dlll]
    &
    M \ar[dll, shift left] \ar[dll, shift right] \ar[dllll]
    \\
    &
    A \ar[d, shift left] \ar[d, shift right]
    &
    &
    F_O \ar[dll, shift left] \ar[dll, shift right]
    \\
    &
    O
   \end{tikzcd}
  \]
  An instance $M_{x,y,z}(w,w',\mu)$ says that $\mu:A(z,y)$ is the component of the multiplication $F^2\to F$ at $x$ relative to the witnesses $w:F_O(x,y)$ and $w':F_O(y,z)$.
  Functors involving product categories can also be represented by adding extra dependencies.
  For instance, a category equipped with a functor $C\times C \to C$ has the signature
  
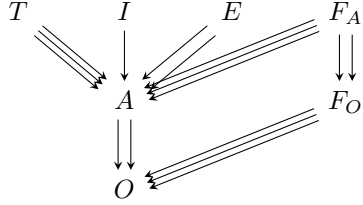
\[
   \begin{tikzcd}
     T \ar[dr] \ar[dr, shift right] \ar[dr, shift left]
     &
     I \ar[d] & E \ar[dl, shift right] \ar[dl, shift left]
     &
     F_A \ar[d, shift right] \ar[d, shift left] \ar[dll, shift right] \ar[dll, shift left] \ar[dll]
     \\
     &
     A \ar[d, shift right] \ar[d, shift left]
     &
     &
     F_O \ar[dll, shift right] \ar[dll, shift left] \ar[dll]
     \\
     &
     O
   \end{tikzcd}
  \]
  where $FO(x,y,z)$ means that ``$z = F(x,y)$''.
  A category equipped with an object $U$ (about which we assert nothing) has the signature
    
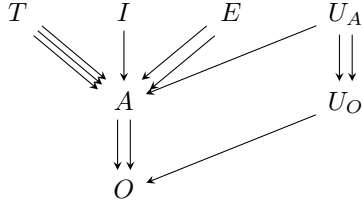
\[
   \begin{tikzcd}
     T \ar[dr] \ar[dr, shift right] \ar[dr, shift left]
     &
     I \ar[d] & E \ar[dl, shift right] \ar[dl, shift left]
     &
     U_A \ar[d, shift left] \ar[d, shift right] \ar[dll]
     \\
     &
     A \ar[d, shift right] \ar[d, shift left]
     &
     &
     U_O \ar[dll]
     \\
     &
     O
   \end{tikzcd}
  \]
  Here $U_O$ and $U_A$ represent an anafunctor whose domain is the terminal category, also known as an ``ana-object'' (with dependencies ignored, it is a functor out of a contractible groupoid).
  Univalence and saturation imply that $\sum_{x:O} U_O(x)$ is contractible, and that if we have some particular $u:O$ with $w:U_O(u)$ the type $U_O(x)$ for any other $x$ is equivalent to the type of isomorphisms $x\cong u$.
  The sort $U_A$ appears to be necessary, providing the ``anafunctorial action'' of the saturated ana-object: for $w_1:U_O(x_1)$ and $w_2:U_O(x_2)$ the proposition $U_A(w_1,w_2,g)$ says intuitively that $g$ is the composite isomorphism $x_1 \cong u \cong x_2$.

  Combining the latter two signatures with some natural transformations, here is the signature for a monoidal category:\index{category!monoidal}
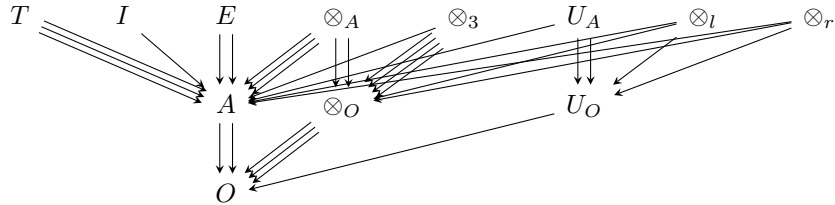
  \[
    \begin{tikzcd}
      T \ar[drr] \ar[drr, shift left] \ar[drr, shift right]
      &
      I \ar[dr]
      &
      E \ar[d, shift left] \ar[d, shift right]
      &
      \otimes_A \ar[d, shift left] \ar[d, shift right] \ar[dl, shift left] \ar[dl, shift right] \ar[dl]
      &
      \otimes_3 \ar[dl, shift left] \ar[dl, shift right] \ar[dl, shift left = 2] \ar[dl] \ar[dll]
      &
      U_A \ar[d, shift right] \ar[d, shift left] \ar[dlll]
      &
      \otimes_l \ar[dlll] \ar[dl] \ar[dllll]
      &
      \otimes_r \ar[dllll] \ar[dll] \ar[dlllll]
      \\
      &
      &
      A \ar[d, shift left] \ar[d, shift right]
      &
      \otimes_O \ar[dl, shift left] \ar[dl, shift right] \ar[dl]
      &
      &
      U_O \ar[dlll]
      \\
      &
      &
      O
    \end{tikzcd}
  \]
  Here $\otimes_O$ and $\otimes_A$ represent the tensor product (ana)functor $C\times C\to C$, while $U_O$ and $U_A$ represent the unit (ana)object, and $\otimes_3$, $\otimes_l$, and $\otimes_r$ represent the associativity and unit natural transformations.
  In fact, in this case we can omit the sort $U_A$; it is necessary in the above signature for an arbitrary ana-object to ensure that two different ``values'' of the object are canonically isomorphic, but when the ana-object is the unit object of a monoidal category this is automatic: if $U_1$ and $U_2$ are two different units then we have a composite of unit isomorphisms $U_1 \cong U_1 \otimes U_2 \cong U_2$.

  We can upgrade the previous theory to the theory of \emph{symmetric} monoidal categories by adding, to its signature, a \emph{braiding}, that is, a natural transformation $B$ with components $B_{x,y} : x\otimes y \to y \otimes x$ as follows (we only draw an excerpt of the signature):\index{category!symmetric monoidal}
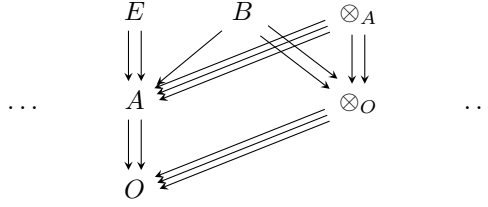
 \[
    \begin{tikzcd}
      &
      E \ar[d, shift left] \ar[d, shift right]
      &
      B \ar[dl] \ar[dr, shift left] \ar[dr, shift right]
      &
      \otimes_A \ar[d, shift left] \ar[d, shift right] \ar[dll, shift left] \ar[dll, shift right] \ar[dll]
      \\
      \ldots
      &
      A \ar[d, shift left] \ar[d, shift right]
      &
      &
      \otimes_O \ar[dll, shift left] \ar[dll, shift right] \ar[dll]
      &
      \ldots
      \\
      &
      O
    \end{tikzcd}
  \]
  Here, $B_{x,y,a, b}(w_1, w_2, \beta)$, with $w_1 : \otimes_O(x,y,a)$, $w_2 : \otimes_O(y,x,b)$, and $\beta : A(a, b)$ indicates that $\beta$ is the component of $B$ on $(x,y)$; we write $B(w_1,w_2) = \beta$.
  We assert that the braiding is symmetric, i.e., it satisfies $B(w_2, w_1) \circ B(w_1, w_2) = 1_{a}$ for any $w_1 : \otimes_O(x,y,a)$ and $w_2 : \otimes_O(y,x,b)$.
  (Of course, if we omitted this axiom we would obtain a theory for \emph{braided} monoidal categories.)\index{category!braided monoidal}

  In all these examples, we assert the same sorts of axioms, including existential saturation of the object-functor sorts (e.g., $F_O$, $\otimes_O$) and that the equalities on arrows are congruences for all rank-2 sorts.
  Univalence at the object-functor sorts then implies the uniqueness aspect of saturation, and together with functoriality and naturality it follows that univalence at each object sort $O$ reduces simply to ordinary univalence of the corresponding category.
\end{example}

\begin{example}[Cartesian monoidal categories]\label{eg:cart-monoidal-cats}\index{category!cartesian monoidal}
  A cartesian monoidal category can be characterized as a symmetric monoidal category (as specified in \cref{eg:cat-struc}) equipped with well-behaved diagonals and augmentations---specifically, with
  \begin{itemize}
  \item a natural transformation $\Delta : 1_C \to \otimes \circ \delta$ with $\delta : C \to C\times C$ the diagonal functor; and
  \item a natural transformation $H : 1_C \to U$;
  \end{itemize}
  satisfying certain conditions (\cite[Theorem 6.13]{heunen-vicary-lectures-quantum-mechanics}).
  We can add these natural transformations to the theory of symmetric monoidal categories by adding, to the underlying signature, the sorts $\Delta$ and $H$, and their dependencies, as follows:
 
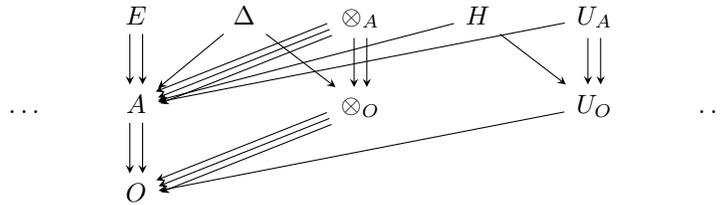
\[
    \begin{tikzcd}
      &
      E \ar[d, shift left] \ar[d, shift right]
      &
      \Delta \ar[dl] \ar[dr]
      &
      \otimes_A \ar[d, shift left] \ar[d, shift right] \ar[dll, shift left] \ar[dll, shift right] \ar[dll]
      &
      H \ar[dlll] \ar[dr]
      &
      U_A \ar[d, shift right] \ar[d, shift left] \ar[dllll]
      \\
      \ldots
      &
      A \ar[d, shift left] \ar[d, shift right]
      &
      &
      \otimes_O \ar[dll, shift left] \ar[dll, shift right] \ar[dll]
      &
      &
      U_O \ar[dllll]
      &
      \ldots
      \\
      &
      O
    \end{tikzcd}
  \]
  We assert the usual axioms for $\Delta$ and $H$ to represent natural transformations, along with the compatibility axioms mentioned above.
  As usual, univalence at the sorts $\Delta$ and $H$ entails that these are propositions pointwise; the addition of these natural transformations does not change indiscernibilities in the other sorts.
\end{example}

\begin{example}[Unbiased monoidal categories]\label{eg:unbiased-monoidal}\index{category!unbiased monoidal}
  In \cref{eg:cat-struc} we defined the theory of categories with a binary monoidal structure.
  Here, we define categories with \emph{unbiased} monoidal structure. By this, we mean a category equipped with a tensor product for any number $n$ of objects, instead of just two.

  Explicitly, such an unbiased monoidal category $C$ is equipped with
  \begin{enumerate}
   \item for any natural number $n$, an n-fold tensor product $\otimes^n : C^n \to C$;
   \item for any natural numbers $n, k_1, \ldots, k_n$, a natural isomorphism
          \[\gamma^{n;k_1,\ldots,k_n} : \otimes^n \circ \left(\otimes^{k_1} \times \ldots \times \otimes^{k_n}\right) \to \otimes^{\sum_{i=1}^n k_i}; \] and
   \item a natural isomorphism $\iota : 1_C \to \otimes^1$;
  \end{enumerate}
  subject to some axioms (for details, see, e.g., \cite[Def.~3.1.1]{higher-ops-higher-cats}).

  A suitable signature would look like this:
    
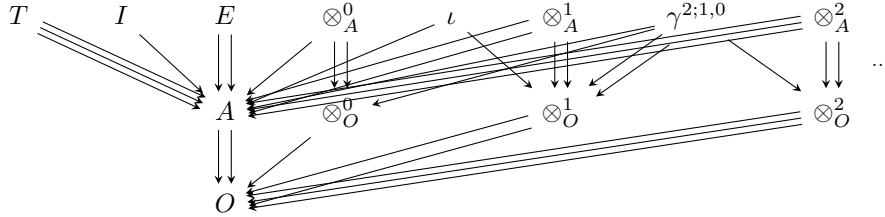
\[
    \begin{tikzcd}[]
      T \ar[drr] \ar[drr, shift left] \ar[drr, shift right]
      &
      I \ar[dr]
      &
      E \ar[d, shift left] \ar[d, shift right]
      &
      \otimes_A^0 \ar[d, shift left] \ar[d, shift right] \ar[dl]
      &
      \iota \ar[dll] \ar [dr]
      &
      \otimes_A^1 \ar[d, shift left] \ar[d, shift right] \ar[dlll, shift right] \ar[dlll, shift left]
      &
      \gamma^{2;1,0} \ar[dl, shift left] \ar[dl, shift right] \ar[dr] \ar[dlll] \ar[dllll]
      &
      \otimes_A^2 \ar[d, shift left, "\quad\ldots"] \ar[d, shift right] \ar[dlllll, shift right] \ar[dlllll, shift left] \ar[dlllll]
      \\
      &
      &
      A \ar[d, shift left] \ar[d, shift right]
      &
      \otimes_O^0  \ar[dl]
      &
      &
      \otimes_O^1  \ar[dlll, shift right] \ar[dlll, shift left]
      &
      &
      \otimes_O^2  \ar[dlllll, shift right] \ar[dlllll, shift left] \ar[dlllll]
      \\
      &
      &
      O
    \end{tikzcd}
  \]
  Here we have only only drawn, as exemplary for the family $\gamma$ of natural isomorphisms, the component $\gamma^{2;1,0} : \otimes^2 \circ (\otimes^1 \times \otimes^0) \to \otimes^1$.
  Note that this signature has (countably) infinitely many sorts of ranks 1 and 2.
  If we assume that $\exo{\Nat}$ is cofibrant (hence sharp), we can take $\L(1)$ and $\L(2)$ to both be $\exo{\Nat}$, in a similar spirit to our usual exofinite signatures.
  However, since the height of this signature is only 3, there are no nontrivial associativity relations, so it would also be possible to use $\Nat$ (see \cref{rmk:exofinite-sigs}).

  Indiscernibilities at $O$ are exactly isomorphisms; the category underlying an unbiased monoidal category is univalent if the structure is univalent.

  A morphism of models for this theory is a strong monoidal functor (called ``weak'' in \cite[Def.~3.1.3]{higher-ops-higher-cats}).
\end{example}

\begin{example}[Unbiased symmetric monoidal categories]\label{eg:unbiased-sym-monoidal}\index{category!unbiased symmetric monoidal}
 In \cref{eg:cat-struc} we obtained a theory of symmetric monoidal categories from that of monoidal categories by adding, to the signature of the latter, a braiding operation.

 Another theory of symmetric monoidal categories is obtained by modifying the signature of \cref{eg:unbiased-monoidal}; specifically by replacing finite \emph{ordered} sets there (given by natural numbers) by finite \emph{unordered} sets; this approach is analogous to that taken in \cref{eg:species} (and which will reappear in \cref{eg:petri,eg:fat-sym-multicats}).
 Concretely, this means that the family of tensor products is indexed by $\FinSet$ instead of $\Nat$, and the family $\gamma$ is parametrized by families of finite sets $[n]; [k_1],\ldots,[k_n]$; we have
 \[\gamma^{[n];[k_1],\ldots,[k_n]} : \otimes^{[n]} \circ \left(\otimes^{[k_1]} \times \ldots \times \otimes^{[k_n]}\right) \to \otimes^{\sm{i : [n]} k_i}. \]
 (Again, because the height is only 3, there are no nontrivial associativities, hence no problem with having arbitrary types of sorts and morphisms in our signature.)
 We can visualize the resulting signature as follows:
 
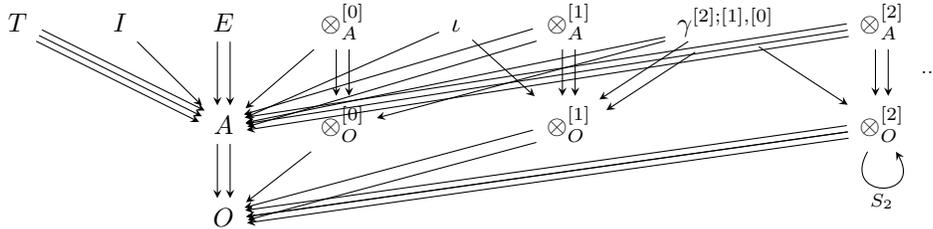
\[
    \begin{tikzcd}[]
      T \ar[drr] \ar[drr, shift left] \ar[drr, shift right]
      &
      I \ar[dr]
      &
      E \ar[d, shift left] \ar[d, shift right]
      &
      \otimes_A^{[0]} \ar[d, shift left] \ar[d, shift right] \ar[dl]
      &
      \iota \ar[dll] \ar [dr]
      &
      \otimes_A^{[1]} \ar[d, shift left] \ar[d, shift right] \ar[dlll, shift right] \ar[dlll, shift left]
      &
      \gamma^{[2];[1],[0]} \ar[dl, shift left] \ar[dl, shift right] \ar[dr] \ar[dlll] \ar[dllll]
      &
      \otimes_A^{[2]} \ar[d, shift left, "\quad\ldots"] \ar[d, shift right] \ar[dlllll, shift right] \ar[dlllll, shift left] \ar[dlllll]
      \\
      &
      &
      A \ar[d, shift left] \ar[d, shift right]
      &
      \otimes_O^{[0]}  \ar[dl]
      &
      &
      \otimes_O^{[1]}  \ar[dlll, shift right] \ar[dlll, shift left]
      &
      &
      \otimes_O^{[2]}  \ar[dlllll, shift right] \ar[dlllll, shift left] \ar[dlllll] \ar[dlllll] \ar[loop,out=-120,in=-60,looseness=5,"S_2"']
      \\
      &
      &
      O
    \end{tikzcd}
  \]
  The symmetric group $S_n$ acts on the sorts $\otimes_O^{[n]}$, as indicated by the loop in the signature.
  We also have more complicated actions on the sorts $\otimes_A^{[n]}$.
  Suitable axioms are asserted, see, e.g., \cite[Prop.~1.5]{Tannakian} or \cite{brandenburg_unbiased_sym_monoidal} for details.

  A morphism of such models is exactly an (unbiased) strong symmetric monoidal functor \cite[Def.~2.2]{brandenburg_unbiased_sym_monoidal}.
\end{example}

\begin{example}[Categories with a binary product functor]\label{eg:cat-w-binprod-functor}\index{category!with binary product functor}
 Here we present the theory of a category with a \emph{specified} cartesian monoidal structure.
 Its underlying signature is as follows:
 
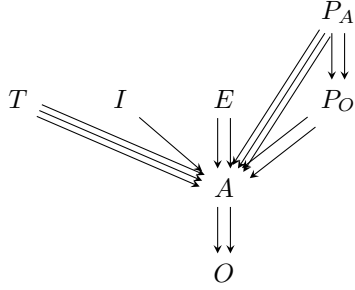
\[
     \begin{tikzcd}[ampersand replacement=\&]
     \&
     \&
     \&
     P_A \ar[d, shift left]  \ar[d, shift right]
             \ar[ddl] \ar[ddl, shift left]  \ar[ddl, shift right]
      \\
      T \ar[drr] \ar[drr, shift left] \ar[drr, shift right]
      \&
      I \ar[dr]
      \&
      E \ar[d, shift right] \ar[d, shift left]
      \& P_O \ar[dl, shift left]  \ar[dl, shift right]
      \\
      \&
      \& A \ar[d, shift left] \ar[d, shift right]
      \\
      \&
      \& O
    \end{tikzcd}
 \]
  Here, similar to \cref{eg:bin-product}, an element $p : (P_O)_{x,y,z}(f,g)$ denotes a product diagram $x \xleftarrow{f} y \xrightarrow{g} z$.
  Given another $p' : (P_O)_{x',y',z'}(f' ,g')$ and $h : A(x,x')$, $k : A(z, z')$, $\ell : A(y,y')$,
  an element $t : (P_A)_{x,y,z,x',y',z'}(p, p', h,k, \ell)$ signifies that $\ell = h \times k$.
  We assert as axioms that $P_O$ has a product for each input pair $(x,z)$.
  The sort $P_A$ is asserted to represent a function in $(p, p', k, h)$;
  univalence at $P_A$ ensures that $P_A$ is a proposition pointwise, and we write $(P_A)(p, p', h, k) = \ell$.
  We also assert functoriality axioms $(P_A)(p, p, 1_x, 1_z) = 1_y$
  and $(P_A)(p_2, p_3, k_2, h_2) \circ (P_A)(p_1, p_2, k_1, h_1) = (P_A)(p_1, p_3, k_2 \circ k_1, h_2 \circ h_1)$.
  Furthermore, we assert existential saturation: given $p : (P_O)_{a,y,b}(f,g)$ and $a \xleftarrow{f'} y' \xrightarrow{g'} b$ and $z : y \cong y'$ commuting with $f$, $f'$, $g$, and $g'$,
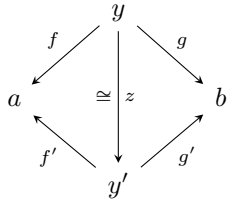
  \[
   \begin{tikzcd}[ampersand replacement = \&]
        \&   y  \ar[dl, "f"'] \ar[dr, "g"] \ar[dd, "\cong"', "z"]
      \\
      a    \&      \& b
      \\
            \&   y' \ar[ul, "f'"] \ar[ur, "g'"']
   \end{tikzcd}
  \]
  there exists $p' : (P_O)_{a, y', b}(f', g')$ such that $(P_A)(p, p', 1_a, 1_b) = z$.

  An indiscernibility $p \fiso p'$ in $P_O$ consists, in particular, of the assertion that $(F_A)(p, p', 1_a, 1_b) = 1_{ab}$; by functoriality, such an indiscernibility is nothing more than that.
  Univalence at $P_O$ then ensures the uniqueness part of saturation.

  If we have $E(f, f')$ and $w : (P_O)_{a, y, b}(f, g)$, then we obtain a unique $w' : (P_O)_{a, y, b}(f', g)$ by saturation; indeed, we have
  \[
   \begin{tikzcd}[ampersand replacement = \&]
        \&   y  \ar[dl, "f"'] \ar[dr, "g"] \ar[dd, "\cong"', "1"]
      \\
      a    \&      \& b
      \\
            \&   y' \ar[ul, "f'"] \ar[ur, "g"']
   \end{tikzcd}
  \]
  and the left-hand triangle commutes because $E$ is assumed to be a congruence for $T$, specifically,
  $T(1, f, f)$ and $E(f, f')$ imply $T(1, f', f)$.
  An indiscernibility $f \fiso f'$ is hence nothing more than an equality $E(f, f')$, and univalence at $A$ means exactly that $A$ is a set with equality $E$.

  Given objects $a, b : O$, an indiscernibility $a \fiso b$ consists of an isomorphism $a \cong b$ in the underlying category, together with transport functions for $P_O$ and $P_A$.
  But in $P_O$, these transport functions are uniquely specified by saturation, and in $P_A$ by the functorial axioms.
  Thus an indiscernibility $a \fiso b$ is just an isomorphism $a \cong b$, and the category underlying a univalent structure is univalent.

  A morphism of structures consists of a functor between the underlying categories and a natural transformation between the respective product functors (suitably composed with the functor between the categories); it is an equivalence if the functor is an equivalence and the natural transformation is an isomorphism.
\end{example}

\begin{example}[Multicategories/Colored operads]\label{eg:multicats}\index{multicategory}\index{operad}\index{operad!colored}
 A (non-symmetric) multicategory (or colored non-symmetric operad) (see, e.g., \cite[Section~I.2]{higher-ops-higher-cats}) has arrows of different arity, generalizing the notion of $n$-ary functions on sets.
 The data of a multicategory is specified via the signature below:
 \[
   \begin{tikzcd}[column sep=small]
     I\ar[rrrd]
     &
     T_{0;1} \ar[rrd] \ar[rd, shift left]\ar[rd, shift right]
     &
     T_{1;1} \ar[dr] \ar[dr, shift left] \ar[dr, shift right]
     &
     T_{0,0;2}\ar[dl] \ar[dl, shift left] \ar[dl, shift right] \ar[dr]
     &
     T_{1,1;2} \ar[ld, shift left]\ar[ld, , shift right] \ar[d, shift left] \ar[d, shift right]
     &
     T_{2;1}  \ar[dll] \ar[dl, shift left] \ar[dl, shift right]
     &
     \ldots
     &
     &
     E_i \ar[d, shift left] \ar[d, shift right]
     \\
     &
     &
     A_0 \ar[rd]
     &
     A_1 \ar[d, shift left] \ar[d, shift right]
     &
     A_2 \ar[dl] \ar[dl, shift left] \ar[dl, shift right]
     &
     \ldots
     &
     &
     \text{plus}
     &
     A_i
     \\
     &
     &
     &
     O
  \end{tikzcd}
 \]
 In this signature, $A_1$ is the sort of arrows known from categories, with one ``input'' object. The sort $A_0$ denotes arrows with no inputs, the sort $A_2$ arrows of 2 inputs, and so on.
 As in \cref{eg:unbiased-monoidal}, we can use either $\exo{\Nat}$ (if it is cofibrant) or $\Nat$ for the infinite families of objects at ranks 1 and 2.

 We then have composition operations for such morphisms.
 For instance, the sort $T_{1;1}$ denotes the composition of two arrows with one input each---the composition known from categories.
 Similarly, the sort $T_{1,1;2}$ denotes composition of two unary arrows with one arrow of two inputs, resulting in a composite arrow of two inputs.
 The general operation sort is $T_{k_1,\dots,k_n;n}$, denoting the composition of $n$ arrows having arities $k_1,\dots,k_n$ with one arrow having $n$ inputs.

 The signature for multicategories also includes an equality sort $E_i \rightrightarrows A_i$ for each $i$, but for readability we have omitted these from the main diagram.
 On a structure of this signature we can impose suitable axioms for the composition and identity in such a way that a model of the resulting theory is precisely a multicategory in the usual sense.

 An isomorphism $\phi : a \cong b$ in a multicategory is analogous to an isomorphism in a category: it consists of a morphism $f : A_1(a,b)$ together with $g : A_1(b,a)$ that is both pre- and post-inverse to $f$.
 Given a structure for the signature above, an indiscernibility $\phi : a \fiso b$ consists, in particular, of equivalences as in \cref{item:foldsiso1a,item:foldsiso1b,item:foldsiso1c,eq:Txya,eq:Txaz,eq:Tazw,eq:Txaa,eq:Taxa,eq:Taax,eq:Taaa,eq:Iaa,eq:Exa,eq:Eax,eq:Eaa} (where $A$ needs to be replaced by $A_1$);
 we have established in \cref{thm:iso-foldsiso} that this data determines uniquely an isomorphism in the multicategory.
 In addition, the indiscernibility $\phi$ consists of further analogous equivalences for the sorts of $n$-ary arrows $A_i$ and their compositions.
 For instance, it includes a family of equivalences
 \begin{equation*}
       \phi_{xy\bullet} : A_2(x,y;a) \cong A_2(x,y;b)
 \end{equation*}
 and a family of equivalences
 \begin{equation}\label{eq:nonsym-multicat-comp}
  (T_{2,1})_{w,x,y,a}(f, g, h) \leftrightarrow (T_{2,1})_{w,x,y,b}(f, \phi_{y\bullet}(g),\phi_{wx\bullet}(h)).
 \end{equation}
 This latter equivalence with $h \eqdef f$ and $g\eqdef 1_a$ shows that the family $\phi_{xy\bullet}$ is given by postcomposition with the isomorphism corresponding to $\phi$, and similarly for the other families of maps.
 Thus, indiscernibilities in a multicategory also coincide with isomorphisms, so a multicategory is univalent precisely when its underlying category is.

 A morphism of models accordingly corresponds to a functor between multicategories; it is an equivalence if the functor is an equivalence.
\end{example}

\begin{example}[(Fat) symmetric multicategories]\label{eg:fat-sym-multicats}\index{multicategory!symmetric}
 A symmetric multicategory is a multicategory with an action of the symmetric group $S_n$ on the arrows $A_n$; e.g., morphisms $A(x,y; z)$ correspond uniquely to morphisms $A(y,x; z)$.
 These actions are furthermore asserted to be compatible with composition.

 Here, we consider an equivalent formulation of symmetric multicategories, based on the notion of \emph{fat} symmetric multicategories (see, e.g., \cite[Appendix A.2]{higher-ops-higher-cats}).
 Its signature is similar to that of non-symmetric multicategories, but the arrows are instead indexed by unordered finite sets $[n]$ of cardinality $n$.
 That is, similarly to \cref{eg:species,eg:unbiased-sym-monoidal} (and, later, \cref{eg:petri}), the type $\L(1)$ of rank-1 sorts is $\FinSet$, a 1-type that is not a set.
 Similarly, the compositions are indexed by the type
 \[ \tsm{X:\FinSet} (X \to \FinSet) \]
 where $(X,Y)$ denotes the composition of one morphism whose inputs are indexed by $X$ with a family of $X$ morphisms of which the $x^{\mathrm{th}}$ has inputs indexed by $Y(x)$.
 As in \cref{eg:unbiased-sym-monoidal}, at height 3 there is no problem with using arbitrary types of sorts.
 This is even harder to draw non-misleadingly than our other examples with higher types of sorts, but we can give it a try:
 
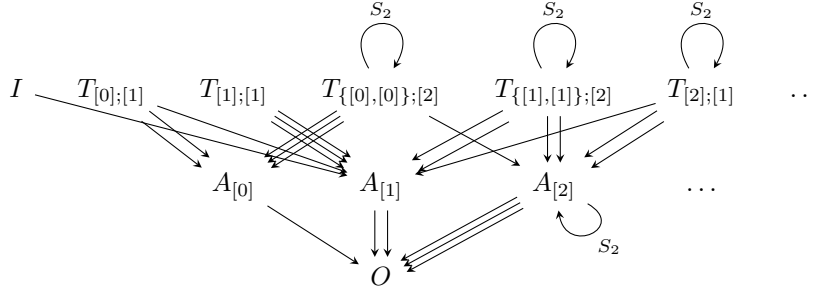
\[
   \begin{tikzcd}[column sep=small]
     I\ar[rrrd]
     &
     T_{[0];[1]} \ar[rrd] \ar[rd, shift left]\ar[rd, shift right]
     &
     T_{[1];[1]} \ar[dr] \ar[dr, shift left] \ar[dr, shift right]
     &
     T_{\{[0],[0]\};[2]}\ar[dl] \ar[dl, shift left] \ar[dl, shift right] \ar[dr] \ar[loop, out=120, in=60, looseness=7, "S_2"]
     &
     T_{\{[1],[1]\};[2]} \ar[ld, shift left]\ar[ld, , shift right] \ar[d, shift left] \ar[d, shift right] \ar[loop, out=120, in=60, looseness=7, "S_2"]
     &
     T_{[2];[1]}  \ar[dll] \ar[dl, shift left] \ar[dl, shift right] \ar[loop, out=120, in=60, looseness=7, "S_2"]
     &
     \ldots
     \\
     &
     &
     A_{[0]} \ar[rd]
     &
     A_{[1]} \ar[d, shift left] \ar[d, shift right]
     &
     A_{[2]} \ar[dl] \ar[dl, shift left] \ar[dl, shift right] \ar[loop, out=-30, in=-80, looseness=4, "S_2"]
     &
     \ldots
     \\
     &
     &
     &
     O
  \end{tikzcd}
 \]
 The sorts we have drawn, representing the elements of the 0-truncation of the types of sorts, are almost like the sorts for non-symmetric multicategories, but some get identified.
 For instance, for non-symmetric multicategories there are two different sorts $T_{0,1;2}$ and $T_{1,0;2}$ for composition of a binary operation with an ordered pair of a nullary and a unary operation, but for the symmetric variant, these two compositions collapse into one connected component that we have written $T_{\{[0],[1]\}; [2]}$, which has an $S_2$ symmetric action.
 In general, the isotropy group of $T_{\{[k_1],\dots,[k_n]\};[n]}$ is the semidirect product $(S_{k_1}\times\cdots\times S_{k_n})\rtimes S_n$.
 As in the non-symmetric case, we have omitted the equality sorts $E_{[n]}$ on $A_{[n]}$ for readability.
 We assert associativity and unitality axioms for composition as spelt out in \cite[Appendix A.2]{higher-ops-higher-cats}.

 Univalence at the top-level sorts entails that the equality, composition, and identity sorts are pointwise propositions;
 at $A_{[n]}$, it entails that $A_{[n]}$ are pointwise sets with equality given by their respective equality sorts.

 An indiscernibility $a \fiso b$ in $O$ consists of equivalences of sorts, e.g.,
 for $n\eqdef 1$ we have $\phi_{x\bullet} : A_{[1]}(\{x\};a) \cong A_{[1]}(\{x\}; b)$, and similar for the other hole and both holes in $A_{[1]}$.
 These equivalences are furthermore coherent with respect to the sorts $I$ and $T$, e.g., they satisfy the analogue of \cref{eq:nonsym-multicat-comp}.
 Given $a \fiso b$, we obtain in particular $\phi \eqdef \phi_{a\bullet}(1_a) : A_{[1]}(\{a\}; b)$. The morphism $\phi$ is an isomorphism; by the coherence with respect to $T$, the other equivalences for the sorts $A_{[n]}$ are given by suitable composition with $\phi$ or its inverse.
 Thus an indiscernibility $a \fiso b$ in $O$ is exactly an isomorphism $a \cong b$.

 A morphism of such models is precisely a functor of symmetric multicategories; it is an equivalence if the functor is an equivalence.

 If in this example we replace $\FinSet$ by $\Set$, we obtain a signature for a certain class of \emph{polynomial monads}, presented in a style with operations indexed by both their input and output sorts (see~\cite{capriotti:polynomials}).\index{monad!polynomial}
 Indeed, as shown in~\cite{ghk:analytic-monads}, symmetric $\infty$-multicategories can be identified with finitary polynomial $\infty$-monads on slice categories of $\infty$-groupoids, and hence symmetric 1-multicategories correspond to a subclass of the latter characterized by homotopy level restrictions (though more general than the classical class of finitary polynomial monads on slices of the category of sets: the polynomial data must be allowed to contain 1-types).
 Replacing $\FinSet$ by $\Set$ in our example removes the finiteness restriction, but retains the latter restriction.
 In \cref{sec:hcat-egs} we will define signatures and theories for higher categories, which could also be adapted to define fat symmetric $n$-multicategories for higher (finite) $n$.
\end{example}

\begin{example}[Semi-displayed categories; see also {\cite[p.~107]{MFOLDS}}]\label{eg:semi-displayed}\index{category!semi-displayed}
Displayed categories \cite{DBLP:journals/lmcs/AhrensL19} were developed, in particular, as a framework to define, in type theory, fibrations of categories without referring to equality of objects.\index{category!displayed}
A displayed category  $\D$  over a category $\C$
is given by, for any $c : \ob{\C}$, a type $\D(c)$ of ``objects over $c$'', and,
for any morphism $f : \C(a,b)$ and $x : \D(a)$ and $y : \D(b)$, a type $\D_f(x,y)$ of ``morphisms from $x$ to $y$ over $f$'', together with suitably typed composition and identity operations.
A na\"ive translation of this definition into a diagram signature might look like the following:
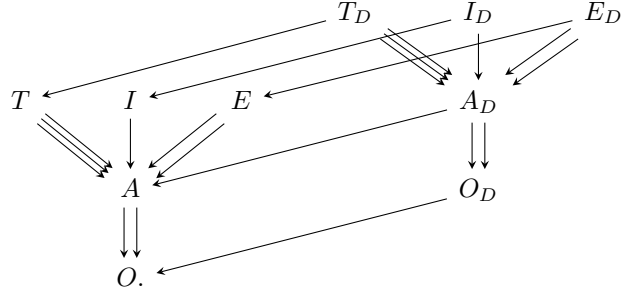
\[
  \begin{tikzcd}
    &
    &
    &
    T_D \ar[dlll] \ar[dr] \ar[dr, shift left] \ar[dr, shift right]
    &
    I_D \ar[dlll] \ar[d]
    &
    E_D \ar[dlll] \ar[dl, shift left] \ar[dl, shift right]
    \\
    T \ar[dr] \ar[dr, shift left] \ar[dr, shift right]
    &
    I \ar[d]
    &
    E \ar[dl, shift left] \ar[dl, shift right]
    &
    &
    A_D \ar[dlll] \ar[d, shift left] \ar[d, shift right]
    \\
    &
    A \ar[d, shift left] \ar[d, shift right]
    &
    &
    &
    O_D \ar[dlll]
    \\
    &
    O.
  \end{tikzcd}
\]
However, this is not well-behaved.
In particular, since $A$ has rank 1 in a height-4 signature, it might not be a set even in a univalent structure, and similarly $O$ might not be a 1-type.
The finger of blame can with some justification be pointed at the sort $E$, which cannot behave like an ordinary equality relation if it has a further sort $E_D$ depending on it.
Makkai makes essentially this point when discussing fibrations: it only makes sense to impose equality relations, in the usual sense, on sorts that are only one level below the top.

One way to solve this problem would be to allow the base category to be a bicategory (though the fibers are only 1-categories), as in \cref{eg:bicats} below.
In \cref{eg:disp-cat} we will see another way to solve it, using heterogeneous equality.
However, we can avoid the complexity of these approaches with the following signature due to Makkai, whose only dependency is for the objects:
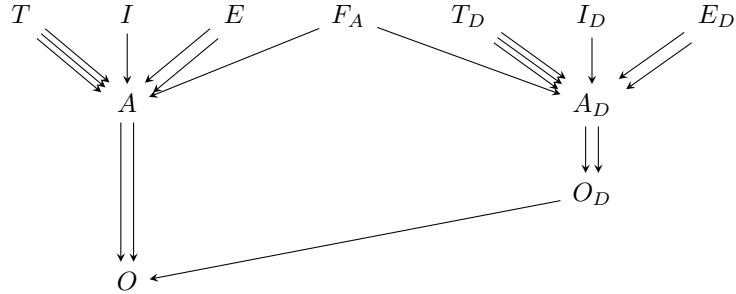
\[
  \begin{tikzcd}
    T \ar[dr] \ar[dr, shift left] \ar[dr, shift right]
    &
    I \ar[d]
    &
    E \ar[dl, shift left] \ar[dl, shift right]
    &
    F_A \ar[rrd] \ar[lld]
    &
    T_D  \ar[dr] \ar[dr, shift left] \ar[dr, shift right]
    &
    I_D  \ar[d]
    & E_D \ar[dl, shift left] \ar[dl, shift right]
    \\
    &
    A \ar[dd, shift left] \ar[dd, shift right]
    &
    &
    &
    &
    A_D \ar[d, shift left] \ar[d, shift right]
    \\
    &
    &
    &
    &
    &
    O_D \ar[dllll]
    \\
    &
    O
  \end{tikzcd}
\]
The dependency $A_D \to A$ is replaced by the relation $F_A$, asserted to be a functional relation, and the dependencies $T_D \to T$ and $I_D \to I$ are replaced by axioms, e.g., $ (I_D)_{c,x}(\overline{f}) \wedge (F_A)_{c,c,x,x}(f,\overline{f}) \to I_{c}(f)$.
A model of this theory might be called a ``semi-displayed category''; it consists of, for any $c : \ob{\C}$, a type $\D(c)$ of objects over $\C$, and for any $x : \D(a)$ and $y : \D(b)$ a type $\D(x,y)$ with a function $\D(x,y) \to \C(a,b)$.
While they may appear more \textit{ad hoc} than displayed categories, semi-displayed categories do suffice to define notions involving strict fibers of functors, such as fibrations of categories.\index{fibration!of categories}

As usual, we assert that $E$ and $E_D$ are congruences for all the relations, including $F_A$.
Thus, in a univalent model, all the top sorts are propositions, both $A$ and $A_D$ are sets with standard equality, and each fiber category over $c:O$ is a univalent category in the usual sense.
An indiscernibility between objects $c,d:O$ consists of an ordinary isomorphism $\phi:c\cong d$ in the underlying base category together with all possible liftings of it in both directions, e.g., for any $x:O_D(c)$ a choice of a $y:O_D(d)$ and an isomorphism $x\cong y$ over $\phi$, and dually.
Since (assuming univalence at $O_D$ and above) such liftings are unique when they exist, the type of such indiscernibilities is a subtype of that of ordinary isomorphisms.
Thus, in a univalent semi-displayed category, $O$ is a 1-type, even though \cref{thm:hlevel-folds} only implies that it is a 2-type.
Moreover, in a univalent model, the underlying ordinary category of the base category is univalent if and only if the semi-displayed category is an isofibration (which is a pure existence axiom; cf.\ also \cite[Problem~5.11]{DBLP:journals/lmcs/AhrensL19}).

A morphism of models consists of a functor between the underlying categories and a ``semi-displayed functor'' above it; it is an equivalence when both functors are equivalences.

Note that semi-displayed categories that are not isofibrations are not fully ``categorical'', i.e., not all of their structure is functorial or natural (here, the fibers do not vary functorially even with isomorphisms in the base).
As we will see in \cref{sec:unnat}, non-categorical structure on categories leads quite generally to indiscernibilities that are more restricted than ordinary isomorphisms, as we have seen in this case.
\end{example}

\begin{example}[Categorical structures for the interpretation of Martin-Löf Type Theory]\label{eg:cwf-cwa}
  Various structures on categories have been devised for the interpretation of Martin-Löf Type Theory.
  An analysis of some of these structures in univalent foundations is given in \cite{alv1-lmcs} (to which we also refer the reader for references to the original literature).
  We look here at split type categories (a.k.a.\ categories with attributes) and categories with families (a.k.a.\ natural models).\index{category!split type}
  
  A \defemph{split type category} consists of a category $\C$ equipped with
  \begin{itemize}
  \item a presheaf $T : \C^{\textrm{op}} \to \Set$;
  \item a ``comprehension structure'', associating to any $\Gamma : \ob{\C}$ and $B : T(\Gamma)$ an object $\Gamma.B$ and a morphism $\pi_B : \Gamma.B \to \Gamma$;
  \item for any $\Gamma$ and $B$ as above, and any $f:\Delta \to \Gamma$,
    \[
      \begin{tikzcd}
        \Delta.f^*B \ar[d, "\pi"] \ar[r, dotted, "q"]
        &
        \Gamma.B \ar[d, "\pi"]
        \\
        \Delta \ar[r, "f"]
        &
        \Gamma
      \end{tikzcd}
    \]
    where $f^*B$ is reindexing of $B$ along $f$ given by the action of $T$ on morphisms, an arrow $q = q(f,B)$ that completes the diagram to a pullback square; 
  \item such that $q$ is functorial.
  \end{itemize}
  Note that the comprehension structure, together with the operation $q$, forms a functor from the category of elements of $T$ to the underlying category $\C$.
  The family $\pi$ of morphisms then forms a natural transformation from that functor to the forgetful functor $\int T \to \C$.
  
A suitable signature for split type categories looks as follows:

    
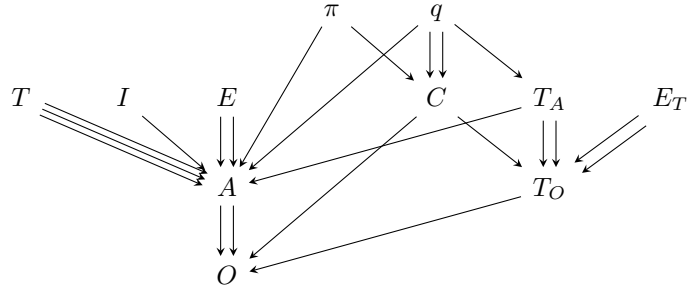
\[
    \begin{tikzcd}
      &
      &
      &
      \pi \ar[ddl] \ar[dr]
      &
      q \ar[d, shift left] \ar[d, shift right] \ar[ddll] \ar[dr]
      &
      &
      \\
      T \ar[drr] \ar[drr, shift left] \ar[drr, shift right]
      &
      I \ar[dr]
      &
      E \ar[d, shift left] \ar[d, shift right]
      &
      &
      C \ar[dr] \ar[ddll]
      &
      T_A\ar[d, shift left] \ar[d, shift right]\ar[dlll]
      &
      E_T \ar[dl, shift left] \ar[dl, shift right]
      \\
      &
      &
      A \ar[d, shift left] \ar[d, shift right]
      &
      &
      &
      T_O \ar[dlll]
      \\
      &
      &
      O
    \end{tikzcd}    
  \]

  Here, the presheaf $T$ is given by the sorts $(T_O,T_A)$ with the equality $E_T$, as in \cref{eg:presheaves}.
  The sort $C$ represents the comprehension: a witness $w : C_\Gamma(B,\Delta)$ signifies, intuitively, that $\Delta$ is the context $\Gamma$ extended by a variable of type $B$.
  The pair of sorts $(C,q)$ represent the comprehension functor.
  The sort $\pi$ represents the aforementioned natural transformation:
  a witness $t : \pi(w,f)$ with $w$ as above says that $f$ is a ``canonical projection'' from $\Delta$ to $\Gamma$.

  Since all the structure is categorical, an indiscernibility $a \fiso b$ in $O$ is just an isomorphism $a \cong b$ in the underlying category.
  This entails that in a univalent split type category, the underlying category is univalent.

  \defemph{Categories with families}, in the formulation of Fiore \cite{fiore_cwf} and Awodey \cite{awodey_natural} share some structure with split type categories, notably the category $\C$, the presheaf $T$ and the comprehension structure $\pi_B : \Gamma.B \to \Gamma$ for any $\Gamma$ and $B$.\index{category!with families}
  However, in this case the comprehension structure is not assumed \textit{a priori} to be functorial.
  Instead, a category with families has
  \begin{itemize}
  \item a presheaf $Tm$ on $\C$;
  \item a natural transformation $p : Tm \to T$;
  \item  for each object $\Gamma : \C$ and $B : T(\Gamma)$,
    an element $V_A:Tm(\Gamma.B)$, such that $p_{\Gamma.B}(V_A) = T(\pi_B)(B) : T(\Gamma.B)$
     and such that the induced commutative square of presheaves and natural transformations
      \[
        \begin{tikzcd}
          \yon(\Gamma.B)  \ar[r, "V_A"] \ar[d, "\yon(\pi)"]
          &
          Tm \ar[d, "p"]
          \\
          \yon(\Gamma) \ar[r, "B"]
          & T
      \end{tikzcd}
      \]
      is a pullback; here, $\yon$ denotes the Yoneda embedding.
  \end{itemize}

  Expressed as a signature, this yields
  
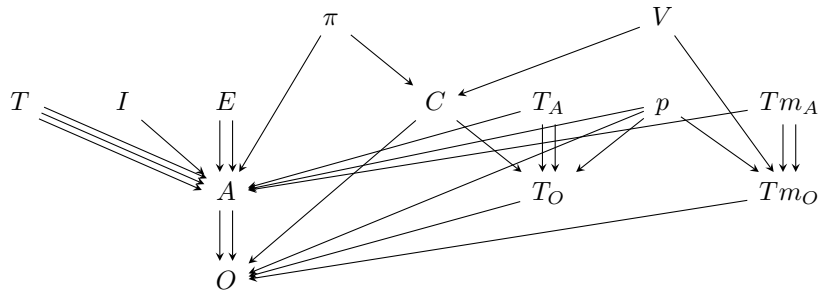
\[
    \begin{tikzcd}
      &
      &
      &
      \pi \ar[ddl] \ar[dr]
      &
      &
      &
      V \ar[dll] \ar[ddr]
      \\
      T \ar[drr] \ar[drr, shift left] \ar[drr, shift right]
      &
      I \ar[dr]
      &
      E \ar[d, shift left] \ar[d, shift right]
      &
      &
      C \ar[dr] \ar[ddll]
      &
      T_A\ar[d, shift left] \ar[d, shift right]\ar[dlll]
      &
      p \ar[ddllll] \ar[dllll] \ar[dl] \ar[dr]
      &
      Tm_A\ar[d, shift left] \ar[d, shift right] \ar[dlllll]
      \\
      &
      &
      A \ar[d, shift left] \ar[d, shift right]
      &
      &
      &
      T_O \ar[dlll]
      &
      &
      Tm_O \ar[dlllll]
      \\
      &
      &
      O
    \end{tikzcd}
  \]
  Here, $(T_O, T_A)$ and $(Tm_O, Tm_A)$ are assumed to be presheaves;
  for readability, we omit the equalities $E \rightrightarrows T_O$ and  $E \rightrightarrows Tm_O$ in the signature above.
  An element $x : V(\Gamma, \Delta, B, w, t)$ with $B : T_O(\Gamma)$, $t : Tm_O(\Delta)$, and $w : C(\Gamma, B, \Delta)$ states that $t$ is the generic variable obtained from the context extension $\Delta = \Gamma.B$.
  We assert, via suitable axioms, that the data thus given yields pullback squares.

  Prima facie, this structure is not categorical;
  in particular, an indiscernibility $a \fiso b$ in $O$ consists of an isomorphism $\phi : a \cong b$ together with transport functions in $C$, $\pi$, and $V$.
  But these sorts are exactly specifying the pullback data, and are hence closed under isomorphism in $O$. This means that transport in these sorts is for free; an indiscernibility $a \fiso b$ is exactly an isomorphism $\phi : a \cong b$.
  
  In a univalent model of the above theory, the data specified by $(C,\pi, V)$ exists uniquely \cite[Lemma~34]{alv1-lmcs}. The type of univalent models of this theory is hence equivalent to the type of univalent models of the theory of ``representable maps of presheaves'', where the representation of $p$ is merely assumed to exist:
  
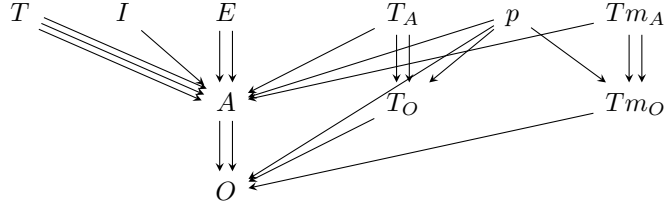
\[
    \begin{tikzcd}
      T \ar[drr] \ar[drr, shift left] \ar[drr, shift right]
      &
      I \ar[dr]
      &
      E \ar[d, shift left] \ar[d, shift right]
      &
      &
      T_A\ar[d, shift left] \ar[d, shift right]\ar[dll]
      &
      p \ar[ddlll] \ar[dlll] \ar[dl] \ar[dr]
      &
      Tm_A\ar[d, shift left] \ar[d, shift right] \ar[dllll]
      \\
      &
      &
      A \ar[d, shift left] \ar[d, shift right]
      &
      &
      T_O \ar[dll]
      &
      &
      Tm_O \ar[dllll]
      \\
      &
      &
      O
    \end{tikzcd}    
  \]
However, as in \cref{eg:pullbacks}, incorporating the comprehension structure in the signature ensures that it is preserved by arbitrary morphisms of models, which otherwise would not be the case.

In \cref{eg:ctx-cat} we discuss the theory of contextual categories, which includes ``non-categorical'' structure, that is, structure that does not transport along equivalence of categories.
\end{example}

\begin{example}[Semicategories]\label{eg:semicats}\index{semicategory}\index{category!semi-}
  In addition to adding more structure to a category, is also interesting to consider what happens if we \emph{remove} some of its structure.
  For instance, a \defemph{semicategory} is like a category, but has no identities, though its composition is still associative.
  Thus an appropriate signature for semicategories is:
  
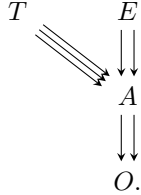
\[
    \begin{tikzcd}
      T \ar[dr] \ar[dr,shift left] \ar[dr,shift right] & E \ar[d,shift left] \ar[d, shift right]\\
      & A  \ar[d,shift left] \ar[d, shift right] \\
      & O.
    \end{tikzcd}
  \]
  As usual, univalence at $A$ makes it a set with equality $E$.
  An indiscernibility $a\foldsiso b$ in $O$ consists of:
  \begin{itemize}
  \item A natural isomorphism of representable presheaves $A(-,a) \cong A(-,b)$ (these notions make perfect sense for semicategories),
  \item A natural isomorphism of representable copresheaves $A(a,-) \cong A(b,-)$ (these notions make perfect sense for semicategories), and
  \item A semigroup isomorphism $A(a,a) \cong A(b,b)$,
  \item Which respect all the additional composition operations on these sets, i.e.,
    \begin{alignat*}{2}
      A(x,a) \times A(a,y) &\to A(x,y) &\qquad
      A(a,a) \times A(a,y) &\to A(a,y) \\
      A(x,a) \times A(a,a) &\to A(x,a) &\qquad
      A(a,x) \times A(x,a) &\to A(a,a)
    \end{alignat*}
    for all $x,y:O$.
  \end{itemize}

  In particular, if $f:A(a,b)$ is a morphism such that pre-composition and post-composition with $f$:
  \[ A(b,y) \to A(a,y) \qquad A(x,a) \to A(x,b) \]
  are isomorphisms for all $x,y:O$, then $f$ induces an indiscernibility $a \foldsiso b$.
  A morphism in a semicategory with this property is called an \emph{isomorphism} in~\cite{tringali:plots}, and \emph{neutral} in~\cite{CapriottiKraus:csst}.
  If $g:A(b,a)$ also has this property, then $f$ and $g$ induce the same indiscernibility if pre- and post-composition with $gf$ (or, equivalently, $fg$) are both the identity.
  However, not all indiscernibilities in a semicategory arise from morphisms: e.g., in a semicategory with \emph{no} morphisms, all objects are indiscernible!\footnote{In particular, the result of~\cite{CapriottiKraus:csst} that a univalent (there called ``complete'') semicategory is a category does not hold for our notion of univalence: the semicategory with one object and no morphisms is univalent.}

  A morphism of semicategories is a semifunctor, i.e., a graph morphism preserving composition.
  It is an equivalence if it is fully faithful and also split essentially surjective up to indiscernibility.
  We leave it to the reader to define a theory for structures such as two semicategories with a semifunctor between them, and so on.
\end{example}

In \cref{sec:graph-egs} we will study other structures that look like categories without composition and/or identity operations.
We will see that in contrast to the fairly well-behaved situation of semicategories, removing composition leads to rather strange notions of indiscernibility.

\begin{example}[A functor between two fixed categories]\label{eg:relative}
  In \cref{eg:anafunctors} we described a theory whose models consist of two categories together with an (ana)func\-tor between them.
  Alternatively, we might fix two categories $\C$ and $\D$ and ask for a signature for ``functors from $\C$ to $\D$''.
  Such a signature has a family of sorts $F_{O,x,y}$ indexed by $\ob{\C}\times \ob{\D}$, and another family of sorts $F_{A,f,g}$ indexed by pairs of arrows $f$ in $\C$ and $g$ in $\D$, with morphisms from $F_{A,f,g}$ to $F_{O,x_1,y_1}$ and $F_{O,x_2,y_2}$ where $f:\C(x_1,x_2)$ and $g:\D(y_1,y_2)$:
  \[
    \begin{tikzcd}[column sep=tiny]
      \dots && F_{A,f,g} \ar[dl] \ar[dr] && \dots \\
      \dots & F_{O,x_1,y_1} & & F_{O,x_2,y_2} & \dots
    \end{tikzcd}
  \]
  The axioms of an anafunctor can be simply restated about structures for this signature, with $F_{O,x,y}$ replacing $F_O(x,y)$ and so on.

  In fact this signature can be obtained from that of \cref{eg:anafunctors} in a straightforward way.
  First we modify the rank function of \cref{eg:anafunctors} so that the sorts $F_O$ and $F_A$ have rank 3 and 4 respectively.
  (We will show in \cref{cor:struc-norank} that this does not change the structures.)
  Now we take a third derivative of this signature; the input data for this consists precisely of the two categories $\C$ and $\D$, and the resulting derivative is precisely the signature described above.

  We can use the same method to obtain signatures for any kind of structure relative to some fixed ambient fragment of that structure.
  (To be precise, the fragment must be specified by a \emph{cosieve} in the diagram structure, \textit{qua} inverse exo-category, to enable a version of the above rank-reshuffling.)
  For instance, given a fixed category $\C$, there is a signature for ``presheaves on $\C$'' or ``monoidal structures on $\C$''.
\end{example}

\begin{remark}
  Note that, in contrast to ``algebraic'' approaches to categorical structure such as strict 2-monads, it is not possible to define a diagram signature for which the \emph{morphisms} of structures are \emph{strict}, \emph{lax}, or \emph{oplax} monoidal functors.  The only morphisms we can represent in this way are those that preserve structure up to coherent isomorphism.

  However, it is perfectly possible to define a diagram theory whose models are pairs of monoidal categories with a lax (or oplax) monoidal functor between them, similar to our example of anafunctors (\cref{eg:anafunctors}).
  See also \cref{rmk:lax-functors}.
\end{remark}

\chapter{Higher categories}
\label{sec:hcat-egs}

Just as theories of structured sets (\cref{sec:set-egs}) involve signatures of height 2, and theories of structured categories (\cref{sec:1cat-egs}) generally involve signatures of height 3 (with a few exceptions), theories of higher categories involve signatures of height $\ge 4$.
We begin with the theory of bicategories; strict 2-categories are somewhat subtler, and will be studied in \cref{eg:strict-2-cat}.

\begin{example}[(Ana)bicategories; {\cite[Section 7]{MFOLDS}}]\label{eg:bicats}\index{bicategory}\index{anabicategory}
  We can represent bicategories with the following signature from~\cite[p.~110]{MFOLDS} (with equality added):
  
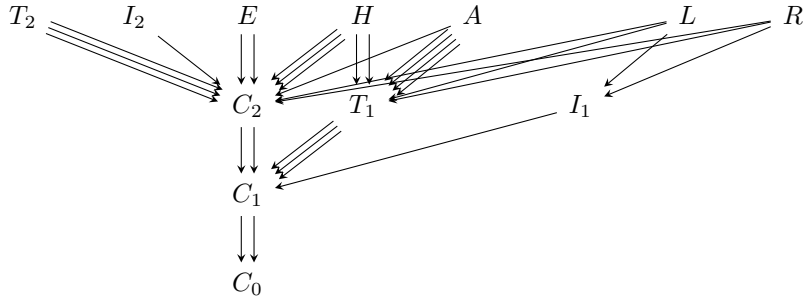
\[
    \begin{tikzcd}
      T_2 \ar[drr] \ar[drr, shift left] \ar[drr, shift right]
      &
      I_2 \ar[dr] & E \ar[d, shift left] \ar[d, shift right]
      &
      H \ar[d, shift left] \ar[d, shift right] \ar[dl, shift left] \ar[dl, shift right] \ar[dl]
      &
      A \ar[dl, shift left] \ar[dl, shift right] \ar[dl, shift left=2] \ar[dl] \ar[dll]
      &
      &
      L \ar[dlll] \ar[dl] \ar[dllll]
      &
      R \ar[dllll] \ar[dll] \ar[dlllll]
      \\
      &
      &
      C_2 \ar[d, shift left] \ar[d, shift right]
      &
      T_1 \ar[dl, shift left] \ar[dl, shift right] \ar[dl]
      &
      &
      I_1 \ar[dlll]
      \\
      &
      &
      C_1 \ar[d, shift left] \ar[d, shift right]
      \\
      &
      &
      C_0
    \end{tikzcd}
  \]
  Here $C_0,C_1,C_2$ are the sorts of objects, 1-cells, and 2-cells.
  The relations $T_2$, $I_2$, and $E$, with their axioms, make the 1-cells and 2-cells into hom-categories.
  The type $T_1$, which depends on a triangle of elements of $C_1$, represents composition of 1-cells: $t:T_1(f,g,h)$ is a ``reason why'' $g\circ f$ is equal to $h$.
  In general this is not a proposition, so composition is an anafunctor; thus (again following Makkai) we are actually representing ``anabicategories''.
  Similarly, $I_1(f)$ is the type of witnesses that $f$ is an identity 1-cell.
  The relation $A$ specifies the associativity isomorphisms: given $t_{12} : T_1(f_1,f_2,f_{12})$ and $t_{123} : T_1(f_{12},f_3,f_{123})$ and also $t_{23}:T_1(f_2,f_3,f_{23})$ and $t_{123}' : T_1(f_1,f_{23},f_{123}')$, the relation $A$ specifies a 2-cell in $C_2(f_{123},f_{123}')$ that plays the role of the associativity morphism (which we assert to be invertible with an axiom).
  Similarly, $L$ and $R$ specify the left and right unit isomorphisms.
  Finally, the relation $H$ specifies the ``horizontal'' composite of two 2-cells along an object, given witnesses for how to compose their domains and codomains.

  Note that if we drop the sorts $I_1,I_2,L,R$ relating to identities, the resulting inverse category is a truncation of the coface maps in Joyal's category $\Theta_2$~\cite{joyal:disks}.
  The identity sorts are a ``fattening'' of this to incorporate the degeneracies while still remaining an inverse category, as done for the simplex category in~\cite{kock:weakids}.
  We expect that under the simplicial set interpretation, a model of this theory can be rectified to obtain an actual simplicial presheaf on $\Theta_2$, analogously to the discussion for categories and bisimplicial sets in \cref{sec:categories-hott}.
  Under the interpretation of $\Theta_2$-spaces as models for $(\infty,2)$-categories~\cite{rezk:theta,ara:nqcats}, our anabicategories should be identified with the image of 2-categories inside $(\infty,2)$-categories.

  On the other hand, if we drop the bottom sort $C_0$, we obtain precisely the signature for monoidal categories from \cref{eg:cat-struc} (with the unnecessary sort $U_A$ removed and the others renamed) --- as we should expect, since a monoidal category is a one-object bicategory.
  Just as in that example, we assert as axioms that $E$ is a congrence for all the top-rank relations and that the functor $T_1$ is existentially saturated.
  This ensures that univalence at $C_2$ makes it consist of sets with equality $E$, that full saturation holds at $T_1$, and that univalence at $C_1$ reduces to ordinary univalence of each hom-category.
  Finally, a ``two-sided bicategorical Yoneda lemma'' implies that indiscernibilities in $C_0$ are equivalent to internal adjoint equivalences, so univalence at $C_0$ means that these are equivalent to identifications.
  The transport in $C_1$ and $C_2$ of an indiscernibility $a \fiso b$ is given by 1-composition and whiskering with the corresponding internal adjoint equivalence $\phi : a \simeq b$, respectively.

  A morphism of saturated anabicategories corresponds to a (pseudo) functor of bicategories.
  It is an equivalence of models if the functor is a (strong) biequivalence, i.e., such that the maps on hom-types of all dimensions are split essentially surjective.

  A signature for two bicategories and a pseudo-anafunctor between them can be designed in analogy to the signature of anafunctors of \cref{eg:anafunctors}.
  For pseudo-anafunctors, suitable saturation conditions need to be stated on the level of 0-cells (with respect to adjoint equivalences) and 1-cells (with respect to isomorphisms).
  The analysis of the univalence condition on a pseudo-anafunctor, and of morphisms of pseudo-anafunctors, is left as an exercise.
\end{example}

In \cite{DBLP:conf/rta/AhrensFMW19}, the authors define and study univalent bicategories in UF, and also a bicategorical version of the displayed categories of \cref{eg:disp-cat}.
Unlike in the models of our theory of bicategories above, their bicategories are algebraic, in that identities and composition are given as operations.
We expect that their univalent bicategories can be proven equivalent to our univalent anabicategories inside HoTT/UF.

\begin{example}[Opetopic bicategories]\index{bicategory!opetopic}
  \Cref{eg:bicats} is based on the classical definition of bicategory (adapted to use anafunctors).
  Another possibility is to use a ``non-algebraic'' definition of bicategory, such as a semisimplicial or fat-simplicial~\cite{kock:weakids} version of the Street--Duskin~\cite{street:orientals,duskin:bicatnerve} simplicial nerve, or the opetopic nerve~\cite{bd:hda3,Mult1}.

  The opetopic approach is particularly interesting because one of Makkai's original applications of FOLDS~\cite{Mak2004} was a definition of opetopic (or ``multitopic'') $\omega$-category.
  A suitable signature for opetopic bicategories is as follows:
  
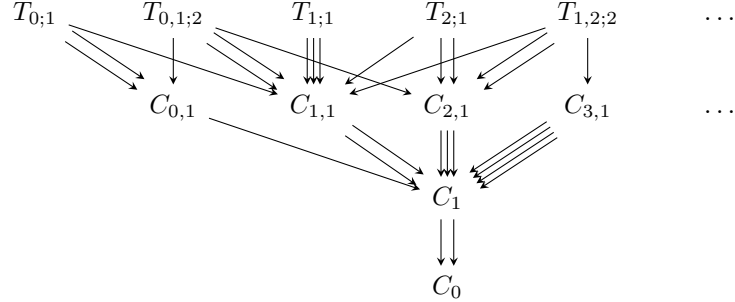
\[
    \begin{tikzcd}
      T_{0;1} \ar[dr,shift left] \ar[dr,shift right] \ar[drr] &
      T_{0,1;2} \ar[d] \ar[dr,shift left] \ar[dr,shift right] \ar[drr] &
      T_{1;1} \ar[d] \ar[d,shift left] \ar[d,shift right] &
      T_{2;1} \ar[dl] \ar[d,shift left] \ar[d,shift right] &
      T_{1,2;2} \ar[dll] \ar[dl,shift left] \ar[dl,shift right] \ar[d] &
      \dots
      \\
      &C_{0,1} \ar[drr] &
      C_{1,1} \ar[dr,shift left] \ar[dr,shift right] &
      C_{2,1} \ar[d] \ar[d,shift left] \ar[d,shift right] &
      C_{3,1} \ar[dl,shift left=.5] \ar[dl,shift right=.5] \ar[dl,shift left=1.5] \ar[dl,shift right=1.5] &
      \dots\\
      &&&C_1 \ar[d,shift left] \ar[d,shift right] \\
      &&&C_0
    \end{tikzcd}
  \]
  together with a binary equality predicate $E_{n,1}$ and a unary predicate $U_{n,1}$ on each $C_{n,1}$.
  We regard the elements of $C_{n,1}$ as 2-cells whose domain is a composable string of $n$ 1-cells and whose codomain is a single 1-cell; thus $C_{n,1}$ has $n+1$ arrows to $C_1$.
  Similarly, the relation $T_{k_1,\dots,k_n;n}$ represents a multicategorical-style composite of $n$ 2-cells whose domains are of length $k_1,\dots,k_n$ with one 2-cell whose domain is of length $n$; thus it has an arrow to each $C_{k_i}$ and to $C_n$, as well as an arrow to $C_{k_1+\cdots+k_n}$ specifying the composite.
  The predicate $U_{n,1}$ singles out certain 2-cells as \emph{universal}.

  Like \cref{eg:unbiased-monoidal,eg:multicats}, this signature has infinitely many sorts (in this case, at ranks 2 and 3); but since it has height 4, it also has nontrivial associativity relations.
  Thus, we can't use $\Nat$ to index these sorts in a na\"ive way and still get a strict exo-category.
  We can use $\exo{\Nat}$ if we assume that it is cofibrant (hence sharp), or we can encode composites and associativity using type dependency as in~\cite{shulman_univalence_ei}.

  The axioms (when suitably simplified from $\omega$-categories to $2$-categories) say that any composable diagram of 2-cells has a unique composite, that identity 2-cells exist, that 2-cells can be factored uniquely through universal ones, and that every composable string of 1-cells is the domain of a universal 2-cell.
  In addition, we assume an ``existential saturation'' property that universal cells are closed under composition with isomorphisms.
  It follows that the specified universal 2-cells (i.e., those satisfying $U_{n,1}$) are precisely those that satisfy the usual universality \emph{property}, and give identities and composition of 1-cells analogously to how a monoidal category can be characterized as a representable multicategory.

  A detailed comparison of opetopic bicategories with classical bicategories can be found in~\cite{cheng:opetopic-bicats}.
  Note, though, that all existing references we are aware of use only opetopic \emph{sets}, whereas our definition yields a notion of higher category based on opetopic \emph{spaces} (e.g., simplicial sets), analogous to~\cite{rezk:css,rezk:theta}.

  As usual, univalence at top-level ensures that each $T_{k_1,\dots,k_n;n}$ is a mere relation, and at the next level that each $C_{n,1}$ is a set with equality $E_{n,1}$.
  Since $C_{1,1}$ forms the morphisms of a category with objects $C_1$, an indiscernibility at $C_1$ includes the data of a 2-cell isomorphism, which can then be shown to uniquely determine the rest of an indiscernibility by the usual arguments; thus univalence at $C_1$ says that the category of 1-cells is univalent.

  From an indiscernibility $\phi:x\foldsiso y$ in $C_0$ we obtain morphisms $\phi:C_1(x,y)$ and $\phi':C_1(y,x)$ in the usual way, and from universal cells in $C_{2,1}(1_x,1_x;1_x)$ and $C_{2,1}(1_y,1_y;1_y)$ we obtain universal cells in $C_{2,1}(\phi,\phi';1_{x})$ and $C_{2,1}(\phi',\phi;1_{y})$ witnessing $f$ and $g$ as inverse adjoint equivalences.
  (These 2-cells are universal because the predicate $U_{2,1}$ is part of the signature, hence is preserved by indiscernibilities.)
  As usual, the entire indiscernibility can then be recovered uniquely from such an adjoint equivalence; thus univalence at $C_0$ is analogous to \cref{eg:bicats}.
  We expect other non-algebraic definitions to behave similarly.
\end{example}

In contrast to the fact that the vast majority of naturally-occurring 1-categories are univalent (although see \cref{sec:unnat} for some notable exceptions), there is a large class of naturally-occurring bicategories that are not univalent.
It is true that \emph{many} bicategories are univalent, such as the bicategory of univalent categories and functors, and many other related bicategories such as those of univalent monoidal categories and monoidal functors, univalent toposes and geometric morphisms, etc.\footnote{The bicategory of not-necessarily-univalent precategories and functors is not univalent, but this is a different sort of issue.}

However, the bicategory of rings and bimodules is not univalent: the identifications of objects are ring isomorphisms, while the equivalences are Morita equivalences.
For similar reasons, neither is the bicategory of categories (even univalent ones) and profunctors.\index{bicategory!non-univalent examples}
In~\cite{shulman:frbi} it was argued that bicategories of this second sort are more naturally viewed as \emph{double categories}.
Various advantages of this perspective were discussed therein, but a further one is that they do tend to be univalent as double categories.\index{category!double}

\begin{example}[Double (ana-bi)categories]\label{eg:double-cats}\index{category!double}\index{bicategory!double}\index{ana-bicategory}
 A \defemph{double category} is similar to a 2-category or bicategory, but has two families of 1-cells, called \emph{vertical} and \emph{horizontal}, respectively.
 The 2-cells take the shape of fillers for squares of 1-cells (two of each sort).
 For example, rings are the objects of a double category whose two families of 1-cells are ring homomorphisms and bimodules, while categories are the objects of a double category containing both functors and profunctors.

 Curiously, it is quite difficult to define a double category in which composition is weak in both directions.
 The closest approximation in the literature is the \emph{double bicategories} of Verity \cite[Definition~1.4.1]{verity-phd-tac}; in addition to squares, these have vertical and horizontal 2-cells of the usual ``globular'' shape, forming two separate bicategories with the same objects, together with operations by which the squares are acted on by the appropriate kind of globular 2-cells on all four sides.
 These are more general than the intuitive notion of ``doubly weak double category'' in that the globular 2-cells may not coincide with the squares having identity morphisms on two parallel sides, although we are free to assume as an additional axiom that the natural map between these two sets is a bijection.

 A suitable signature for double bicategories hence looks as follows:
   
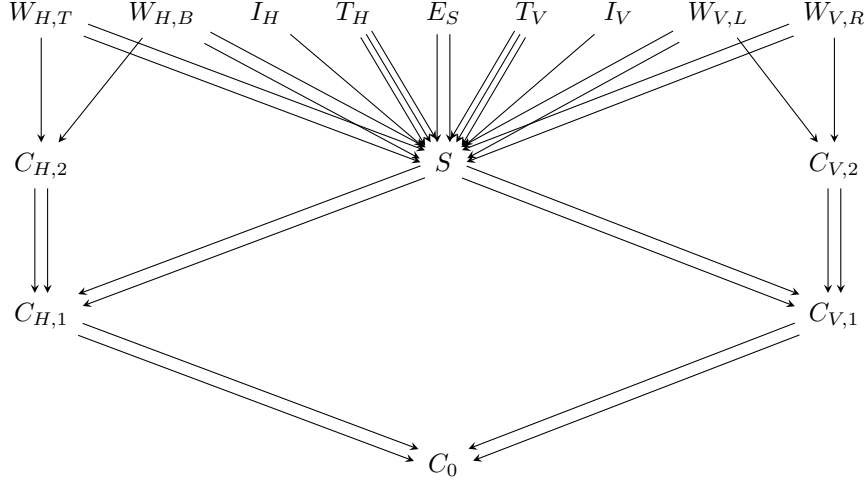
\[
    \begin{tikzcd}[ampersand replacement=\&, row sep = 40, column sep = small]
      W_{H,T} \ar[d] \ar[drrrr,shift right] \ar[drrrr,shift left] \&
      W_{H,B} \ar[dl] \ar[drrr,shift right] \ar[drrr,shift left] \&
      I_H \ar[drr] \&
      T_H \ar[dr] \ar[dr, shift left] \ar[dr, shift right]\&
      E_S  \ar[d,shift right] \ar[d,shift left] \&
      T_V \ar[dl] \ar[dl, shift left] \ar[dl, shift right] \&
      I_V \ar [dll] \&
      W_{V,L} \ar[dr] \ar[dlll,shift right] \ar[dlll,shift left] \&
      W_{V,R} \ar[d] \ar[dllll,shift right] \ar[dllll,shift left]
      \\
      C_{H,2} \ar[d,shift left] \ar[d,shift right] 
      \& \&      
      \& \&
      S \ar[dllll,shift left] \ar[dllll,shift right] \ar[drrrr,shift left] \ar[drrrr,shift right] 
      \& \&
      \& \&
      C_{V,2} \ar[d,shift left] \ar[d,shift right]
      \\
      C_{H,1} \ar[drrrr,shift left] \ar[drrrr,shift right] 
      \& \& \& \&
      \& \&
      \& \&
      C_{V,1} \ar[dllll,shift left] \ar[dllll,shift right]
      \\
      \& \& \& \& C_0
    \end{tikzcd}
  \]
  where we omit the bicategorical structure on $C_{X,2} \rightrightarrows C_{X,1} \rightrightarrows C_0$ for $X = H,V$ (see \cref{eg:bicats}) for readability.
  Intuitively, an element $\lambda : S_{w,x,y,z}(f, g, h, k)$ can be pictured as a filler 
  \[
   \begin{tikzcd}[ampersand replacement = \&]
           z \ar[d,"h"'] \ar[r, "k"] 
           \ar[dr, phantom, "\lambda"]
           \&  
           x \ar[d, "f"]
           \\
           y \ar[r, "g"'] \& w
   \end{tikzcd}
  \]
  and the vertical action $(W_{V,R})_{w,x,y,z,f,g,h,k,f'}(\alpha, \lambda, \lambda')$ attaches a vertical 2-cell $\alpha : f \Rightarrow f'$ on the right of $\lambda$ to yield a filler $\lambda'$ of a square of 1-cells $f',g,h,k$.
  Similarly, we have vertical action on the left ($W_{V,L}$) and horizontal action on the top and bottom.
  As usual, these relations are asserted to be functional.
  Squares can be composed vertically ($T_V$) and horizontally ($T_H$), and we have identities $I_V$ and $I_H$ for these compositions.
  We assert that the equalities (not pictured) on vertical and horizontal 2-cells, as well as the equality $E_S$ on squares $S$, are congruences with respect to these operations.
  
  Univalence at $S$ says that $S$ is pointwise a set with equality given by $E_S$.
  Given two vertical 1-cells $f, f' : C_{V,1}(x,w)$, an indiscernibility between them is given by an isomorphism $\phi : f \cong f'$ in the underlying vertical bicategory together with a transport function for $S$, e.g., $\transfun \phi : S_{w,x,y,z}(f,g,h,k) \to S_{w,x,y,z}(f',g,h,k)$. 
  But coherence with respect to $W_{V,R}$ says that this transport function is given by the action of the 1-isomorphism $\phi$, and similar for the other variables.
  
  Given $a, b : \C_0$, an indiscernibility $a \fiso b$ consists of a pair of a horizontal adjoint equivalence $\phi_H : a \simeq_H b$ and a vertical adjoint equivalence $\phi_V : a \simeq_V b$ together with transport functions for the sort $S$ that are coherent with respect to the top-level sorts.
  In particular, we have a transport function $S_{a,a,a,a}(1,1,1,1) \cong S_{b,a,a,a}(\phi_V,\phi_H,1,1)$; call $\Psi : S_{b,a,a,a}(\phi_V,\phi_H,1,1)$ the image of the identity filler under this isomorphism.
  Analogously, we have a transport function $S_{b,b,b,b}(1,1,1,1) \cong S_{b,b,b,a}(1,1,\phi_V,\phi_H)$; call $\Phi : S_{b,b,b,a}(1,1,\phi_V,\phi_H)$ the image of the identity filler under this isomorphism.
  The coherence laws for the top-sorts then entail that all the other transport functions are fully determined by the choice of $\Psi$ and $\Phi$, and that $\Psi$ and $\Phi$ compose with each other along $\phi_V$ and $\phi_H$ to identities.
  In summary, an indiscernibility in $C_0$ consists of a quadruple $(\phi_H,\phi_V,\Psi, \Phi)$ with these properties, a.k.a.\ a \defemph{companion pair} (see, e.g., \cite[\S1.2]{CTGDC_2004__45_3_193_0}) of adjoint equivalences, which Campbell~\cite{campbell:greg} has called a \defemph{gregarious equivalence}.\index{equivalence!of double categories!gregarious}

  The class of double categories studied in~\cite{shulman:frbi} as replacements for bicategories --- there called \emph{framed bicategories}, but elsewhere known as \emph{fibrant double categories} or \emph{proarrow equipments} --- in particular have the property that every vertical arrow has a horizontal companion.\index{bicategory!framed}\index{proarrow equipment}\index{double category!fibrant}
  Since the companion of an equivalence is always an equivalence, and companions are unique up to unique isomorphism, in a framed bicategory the indiscernibilities are simply the vertical equivalences (which might be simply vertical isomorphisms, if the vertical bicategory is locally discrete, i.e., equivalent to a 1-category).
  Thus, we can assemble rings, ring homomorphisms, and bimodules into a univalent double bicategory,\footnote{We do use here the extra generality of double bicategories over double categories: to ensure that the type of vertical equivalences is the set of ring isomorphisms, we must take the vertical globular 2-cells to be only the identities, rather than all the squares with two horizontal identities in their boundary.  The latter choice would yield instead the double category of one-object  $\mathsf{Ab}$-enriched categories, functors, and profunctors, and the type of equivalences between two one-object (enriched) categories is not equivalent to the set of isomorphisms between their hom-monoids (though one is inhabited if and only if the other is).} and likewise for categories, functors, and profunctors.\index{profunctor}\index{category!enriched}

  A morphism of double bicategories, regarded as models of our theory, is exactly a \emph{horizontal map} as defined in \cite[Definition~1.4.7]{verity-phd-tac} (a.k.a.\ a double pseudofunctor).
  It is an equivalence if it is fully faithful on squares, full on horizontal and vertical 1-cells up to globular isomorphism, and surjective on objects up to gregarious equivalence.\index{equivalence!gregarious}
  These are the weak equivalences in the model structure (on strict double categories) of~\cite{campbell:greg}, and when restricted to framed bicategories they reduce to the equivalences characterized in~\cite[\S7]{shulman:frbi}.
  Note that there is a multitude of model structures, and even a multitude of notions of weak equivalences, on the category of double categories;
  see, for instance, \cite{moser-sarazola-verdugo-2022,fiore-paoli-pronk-2008,moser-sarazola-verdugo-AGT}.
\end{example}

\begin{remark}\label{rmk:lax-functors}\index{functor!lax}
  In \cref{eg:bicats} we noted that morphisms of saturated anabicategories correspond to pseudofunctors, and that pseudo-anafunctors between saturated anabicategories can also be represented as models of a single diagram theory.
  However, the situation with \emph{lax} functors is much subtler, since lax functors between bicategories \emph{do not} preserve equivalences --- indeed, they need not even take \emph{identity} morphisms to equivalences --- nor are they invariant under equivalences of bicategories.
  For instance, a lax functor whose domain is the terminal bicategory sends the identity morphism in its domain to an arbitrary monad in its codomain (which need not be an equivalence), while a lax functor whose domain is \emph{equivalent} to the terminal bicategory instead selects a collection of monads with bimodules between them.

  Thus, it seems difficult, if not impossible, to represent a lax functor between univalent bicategories as a model of a diagram theory.
  However, many lax functors arising in practice actually involve \emph{non-univalent} bicategories.
  As a simple example, any monoidal category can be viewed as a one-object bicategory, which will be non-univalent whenever the monoidal category contains a $\otimes$-invertible object other than the unit (note that this has nothing to do with whether the monoidal category is univalent as a 1-category).
  Any lax monoidal functor then similarly induces a lax functor between these non-univalent one-object bicategories.\index{bicategory!one-object}

  Other examples of lax functors involve non-univalent bicategories that are best regarded as ``many-object monoidal categories'', sometimes called ``horizontally categorified'' rather than ``vertically categorified''.
  For instance, the bicategory of rings and bimodules mentioned above can be generalized to a bicategory $R\text{-Mod}$ of $R$-algebras and bimodules, for any commutative ring $R$.
  This is generally non-univalent, since equivalences therein between algebras are Morita equivalences rather than isomorphisms.
  A ring homomorphism $\phi : R\to S$ then induces a lax functor $S\text{-Mod} \to R\text{-Mod}$.

  However, as noted above and argued in~\cite{shulman:frbi}, such bicategories are often better viewed as double categories.
  Since a \emph{horizontally} lax \emph{double} functor does preserve \emph{vertical} equivalences, and is invariant under equivalence of double categories, there is no difficulty in writing down a signature for such functors, which includes many if not most naturally-occuring lax functors.
  We leave the details to the reader.
\end{remark}

\chapter{Strict categorical structures}
\label{sec:strict-egs}

All the higher-categorical structures considered in \cref{sec:hcat-egs} were ``maximally weak''.
One might guess that our framework can \emph{only} speak about maximally weak higher categories, but in fact this is not the case.
In this \lcnamecref{sec:strict-egs} we discuss a method for representing strict structures, starting with a notion that has no analogue in set-based category theory: the \emph{strict categories} mentioned in \cref{sec:categories-hott}.

\begin{example}[Strict categories]\label{eg:strict-1-cat}\index{category!strict}
  Recall from \cref{sec:categories-hott} that
  a strict category is a precategory whose underlying type of objects is a set.
  A natural way to force the type $O$ in an $\LcatE$-structure to be a set is to just add an equality predicate on it:
  \begin{equation}\label{sig:incorrect-sig-strict-cat}
    \begin{tikzcd}[ampersand replacement=\&]
      T \ar[dr] \ar[dr, shift left] \ar[dr, shift right] \& I \ar[d] \& E_A \ar[dl, shift right] \ar[dl, shift left]
      \\
      \& A \ar[d, shift left] \ar[d, shift right]
      \& E_O \ar[dl, shift right] \ar[dl, shift left]
      \\
      \& O
    \end{tikzcd}
  \end{equation}
  Since $E_O$ has no types that depend on it, univalence will make it a proposition, even though it is not at top rank.
  However, recall that in order for univalence to force equality at a sort (such as $O$) to coincide with a given equality predicate on that sort, the equality needs to be asserted to be a congruence for everything else dependent on that sort.
  In previous examples the equality predicate has been at top rank (like $E_A$), so that all the other sorts dependent on its sort (like $T, I$ dependent on $A$) are mere predicates, and the meaning of ``congruence'' is clear.
  But in the present situation it is less obvious what exactly it means for $E_O$ to be a ``congruence for $A$'' when $A$ is not a mere predicate on $O$.

  A solution is to make the equality of arrows \emph{heterogeneous}, meaning that we do not require that the two morphisms on which $E_A$ depends to have the same source and target, so that there are in total \emph{four} distinct arrows $E_A \to O$ in the signature.\index{equality!heterogeneous}
  We denote heterogeneous equalities by $\hetE$ rather than $E$; thus in this case we have types $(\hetE_A)_{x,y,z,w}(f,g)$ for $f:A(x,y)$ and $g:A(z,w)$.
  We still require $\hetE_A$ to be an equivalence relation, in the appropriate heterogeneous sense, and we assert that it is a congruence for $T$ and $I$ in appropriate ways that respect their dependencies.
  For instance, if $I_x(f)$ and $(\hetE_A)_{x,x,z,z}(f,g)$, then $I_z(g)$, and similarly for $T$.
  This suffices to ensure that the induced homogeneous equality $(E_A)_{x,y}(f,g) \eqdef (\hetE_A)_{x,y,x,y}(f,g)$, for $f,g:A(x,y)$, coincides with indiscernibilities in $A$.

  We will also write $\hetE_O$ for the equality on $O$, even though there is no distinction between homogeneous and heterogeneous equality on a rank-0 sort.
  We assert as axioms that $\hetE_O$ is an equivalence relation, and that if $(\hetE_A)_{x,y,z,w}(f,g)$ then $\hetE_O(x,z)$ and $\hetE_O(y,w)$.
  We can now also assert transport of $A$ along $\hetE_O$, in the sense that given $f:A(x,y)$ and $\hetE_O(x,z)$ and $\hetE_O(y,w)$, there exists a $g:A(z,w)$ such that $(\hetE_A)_{x,y,z,w}(f,g)$.
  This is sensible as a mere existence statement, since such a $g$ is unique up to homogeneous $\hetE_A$ (using the fact that $\hetE_A$ is a heterogeneous equivalence relation).

  Summarily, we write the resulting signature as follows:
    
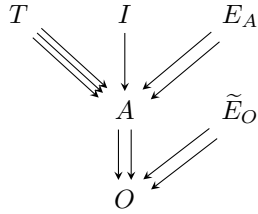
\[
    \begin{tikzcd}[ampersand replacement=\&]
      T \ar[dr] \ar[dr, shift left] \ar[dr, shift right] \& I \ar[d] \& \hetE_A \ar[dl, shift right] \ar[dl, shift left]
      \\
      \& A \ar[d, shift left] \ar[d, shift right]
      \& \hetE_O \ar[dl, shift right] \ar[dl, shift left]
      \\
      \& O
    \end{tikzcd}
  \]
  Since we are only drawing the generating arrows and not their compositions, this picture looks the same as \cref{sig:incorrect-sig-strict-cat}.
  However, it denotes a different category, because the definition of composition is different: as noted above, there are now four composite arrows from $E_A$ to $O$.
  The equality sorts are now denoted slightly differently to signal the different axioms that we impose on them.

  Now suppose univalence holds at all sorts above $O$, and that we have an indiscernibility $\phi:a\fiso b$ in $O$.
  This of course includes a map $\hetE_O(a,a) \to \hetE_O(a,b)$, so that we have $\hetE_O(a,b)$.
  It also includes maps like $\phi_{x\bullet}:A(x,a) \to A(x,b)$ as before, but now these must also respect $\hetE_A$.
  In particular, if $f:A(x,a)$, then since $(\hetE_A)_{x,a,x,a}(f,f)$, we have $(\hetE_A)_{x,a,x,b}(f,\phi_{x,\bullet}(f))$, so that $\phi_{x,\bullet}(f)$ is determined uniquely (up to homogeneous $\hetE_A$) as the $g$ asserted to exist by our transport axiom.

  Similar reasoning applies at other sorts, so any equality $\hetE_O(a,b)$ can be extended \emph{uniquely} to an indiscernibility $\phi:a\fiso b$, i.e., we have $(a\fiso b) \simeq \hetE_O(a,b)$.
  Since the latter is a mere relation, in a univalent structure $O$ must be a set, whose equality relation is $\hetE_O$.
  So the univalent models of this theory are precisely strict categories.
  
  A morphism of structures is precisely a functor; it is an equivalence if the functor is an \emph{isomorphism} of strict categories.
\end{example}

Note also that since strict categories are composed entirely of sets, they can also be considered as the (univalent) models of the essentially algebraic theory of categories, which can be regarded as a height-2 theory in our framework as in \cref{eg:fol}.\index{theory!essentially algebraic}
The value of presenting them as a height-3 theory instead is that it generalizes to ``partially-strict'' structures such as strict 2-categories and displayed categories, as we now explain.

\begin{example}[Strict 2-categories]\label{eg:strict-2-cat}\index{2-category!strict}
  In standard category-theoretic terminology, a strict 2-category \cite[XII.3]{CWM} is an ordinary category equipped with additional 2-cells; thus it is like a bicategory but composition of 1-cells is strictly associative and unital.  This suggests that in a univalent representation, the type of 1-cells should be a set, so that we can compare them for equality; in other words, the hom-categories should be strict categories.  But even in a strict 2-category, one generally does not compare \emph{objects} for equality, only for \emph{isomorphism}; thus a strict 2-category is not a fully strict set-level structure but should have a nontrivial univalence condition on the objects.  (Thus one might more precisely call it a ``locally strict 2-category.'')

  A suitable signature for strict 2-categories looks as follows:
  
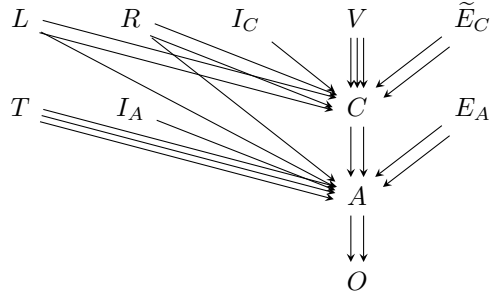
\[
    \begin{tikzcd}[ampersand replacement=\&]
      L \ar[ddrrr] \ar[drrr, shift left] \ar[drrr, shift right]
      \&
      R \ar[ddrr] \ar[drr, shift left] \ar[drr, shift right]
      \&
      I_C \ar[dr]
      \& 
      V \ar[d] \ar[d, shift left] \ar[d, shift right] 
      \& 
      \hetE_C \ar[dl, shift right] \ar[dl, shift left]
      \\
      T \ar[drrr] \ar[drrr, shift left] \ar[drrr, shift right]
      \& 
      I_A \ar[drr] 
      \&
      \& 
      C \ar[d, shift right] \ar[d, shift left]
      \& E_A \ar[dl, shift right] \ar[dl, shift left]
      \\
      \&
      \&
      \& A \ar[d, shift left] \ar[d, shift right]
      \& 
      \\
      \&
      \&
      \& O
    \end{tikzcd}
  \]
  Here, the sort $C$ denotes the sort of 2-cells, $V$ stands for vertical composition, $L$ and $R$ for left and right whiskering, respectively, and $I_C$ for identity 2-cells.
  For instance, $L_{a,b,c,g,h,k,\ell}(f, \alpha, \beta)$ signifies that $\beta : C_{a,c}(k,\ell)$ is the left whiskering of $f : A(a,b)$ and $\alpha : C_{b,c}(g,h)$.
  Intuitively, $L$ (and $R$) should not only depend on three 2-cells, but also on two triangles $T(f,g,k)$ and $T(f,h,\ell)$, to indicate that the boundary of the left and right whiskering is the composite of the boundaries of the input. However, this condition can be stated as an axiom instead: $L_{a,b,c,g,h,k,\ell}(f,\alpha,\beta)$ implies $T_{a,b,c}(f,g,k)$ and $T_{a,b,c}(f,h,\ell)$.
  The absence of these dependencies enables $T$ to be pointwise a proposition in a univalent structure.
  Similar to \cref{eg:strict-1-cat}, the sort $\hetE_C$ denotes a heterogeneous equality, but it is only ``partially heterogeneous'': e.g. it depends on four elements of $A$, but these four elements must be parallel in pairs, so that it depends on only four elements of $O$ rather than eight.
  
  Vertical composition and left and right whiskering are asserted to be functions.
  We also impose the usual axioms of a strict 2-category.
  The sort $\hetE_C$ is asserted to be a congruence for the other top-level sorts.
  By the same reasoning as in \cref{eg:strict-1-cat}, this entails that the homogeneous equality $(E_C)_{f,g}(\alpha,\alpha') \eqdef (\hetE_C)_{f,g,f,g}(\alpha, \alpha')$ induced by the heterogeneous equality coincides with indiscernibilities $\alpha \fiso \alpha'$.
  
  The equality $E_A$ is asserted to be a congruence with respect to the other sorts of the same rank and above. 
  In particular, transport of $C$ along $E_A$ is asserted as follows: 
  given $\alpha : C(f, g)$ and $E_A(f, f')$ and $E_A(g, g')$, there exists $\beta : C(f', g')$ such that $(\hetE_C)_{f, g, f', g'}(\alpha, \beta)$. Such $\beta$ is unique up to homogeneous equality, using that $\hetE_C$ is a heterogeneous equivalence relation.
  
  An indiscernibility $g \fiso g'$ at $A(x,y)$ comes, in particular, with an equality $(E_A)_{x,y}(g,g')$.
  It furthermore comes with a transport $\phi_{f\bullet} : C_{x,y}(f,g) \to C_{x,y}(f, g')$ respecting $\hetE_C$, in the sense that $(\hetE_C)_{x,y,f,g,f,g'}(\alpha, \phi_{f\bullet}(\alpha))$. Consequently, the map $\phi_{f\bullet}$ coincides, up to homogeneous equality, with the transport of $C$ along $E_A$ asserted as part of $E_A$ being a congruence.
  Conversely, any equality $(E_A)_{x,y}(g,g')$ gives rise, in a unique way, to the data of an indiscernibility $g \fiso g'$, and we obtain that $(E_A)_{x,y}(g, g') \simeq (g \fiso g')$.  In particular, in a univalent strict 2-category, each type $A(x,y)$ of arrows is a set.
  
  Given $a, b: O$, an indiscernibility $a \fiso b$ consists of an isomorphism $\phi : a \cong b$ in the underlying category together with transport functions
  $\phi_{x\bullet f f'} : C_{x,a}(f,f') \to C_{x,b}(\phi_{x\bullet}(f),\phi_{x\bullet}(f')) \converts C_{x,b}(\phi \circ f, \phi \circ f')$ and similar for the other dependency of $C$, and in both dependencies. These transport functions are furthermore compatible with sorts $L$, $R$, $I_C$, $V$, and $\hetE_C$.
  Compatibility of $\phi_{x\bullet f f'}$ with $R$ means that
  \[R_{x,a,a,f, f',f, f, f'}(\alpha,1, \alpha) \to R_{x,a, b, f, f', \phi_{x\bullet}(f), \phi_{x\bullet}(f')}(\alpha, \phi_{x\bullet}(1), \phi_{x\bullet f f'}(\alpha)),\]
  which means that $\phi_{x\bullet f g}(\alpha))$ is the right whiskering of $\alpha$ with $\phi : A(a,b)$.
  Analogously, we obtain that the transport function $\phi_{\bullet x f f'}$ is given exactly by left whiskering, and the transport in both dependencies by applying both left and right whiskering.
  This means that an indiscernibility $a \fiso b$ is simply an isomorphism $a \cong b$ in the category underlying the strict 2-category.

  A strict 2-category is univalent iff its underlying 1-category is univalent, $C$ is pointwise a set with equality given by $E_C$ (the homogeneous fragment of $\hetE_C$), and $L$, $R$, $I_C$, and $V$ are pointwise propositions.
  
  A morphism of structures is exactly a functor between strict 2-categories; it is an equivalence if the functor is an equivalence of 2-categories.
\end{example}

\begin{remark}
  It is worth repeating here the observation from \cref{sec:categories-hott} that nearly all naturally-occurring large categories, when defined in HoTT/UF, are univalent and not strict.
  Thus, the notion of strict category is of limited practical utility when working in HoTT/UF, although \emph{small} categories (such as the domains of diagrams) can often be defined in a strict way.
  For a similar reason, very few naturally-occurring bicategories are strict 2-categories in the sense of \cref{eg:strict-2-cat}: even the bicategory of univalent categories and functors is not a strict 2-category, because its hom-categories are not strict.
  There is, however, a strict 2-category of strict 1-categories.
  Moreover, the technique of heterogeneous equality is also useful for other examples that do occur more frequently in practice, such as the following.
\end{remark}

\begin{example}[Displayed categories]\label{eg:disp-cat}\index{category!displayed}
  Using heterogeneous equality, we can consider displayed categories instead of the semi-displayed categories of \cref{eg:semi-displayed}.
  \[
    \begin{tikzcd}[ampersand replacement = \&] 
        \& \& \& 
        T_D \ar[dr] \ar[dr, shift left] \ar[dr, shift right] 
        \&  
        I_D  \ar[d] 
        \& 
        \hetE_D  \ar[dl, shift right] \ar[dl, shift left] 
        \\
        T \ar[dr] \ar[dr, shift right] \ar[dr, shift left] 
        \& I \ar[d] 
        \& E \ar[dl, shift right] \ar[dl, shift left] 
        \& \& 
        A_D \ar[dlll] \ar[d, shift left] \ar[d, shift right] 
        \\
        \& A \ar[d, shift left] \ar[d, shift right] 
        \& 
        \& 
        \& 
        O_D \ar[dlll]
        \\
        \& 
        O
    \end{tikzcd}
  \]
  Here, the sort $\hetE_D$ is a heterogeneous equality as in \cref{eg:strict-1-cat,eg:strict-2-cat}.
  In the signature above, there are two dependencies of $\hetE_D$ on $A$.
  We assert transport of $A_D$ along $E$: given $g : (A_D)_{a,b}(x,y,f)$ and $E(f,f')$, there is $g' : (A_D)_{a,b}(x,y,f')$ and $(E_D)_{a,b,x,y,f,f'}(g,g')$.
  
  As before, indiscernibilities in $A_D$ correspond to homogeneous equalities $E_D$ induced by $\hetE_D$; hence, in a univalent structure, $A_D$ is a set with equality $E_D$.
  
  Given $f,f' : A(a,b)$, an indiscernibility $f \fiso f'$ comes with a transport function
  $\phi_{\bullet} : (A_D)_{a,b}(x,y,f) \to (A_D)_{a,b}(x,y,f')$ that is compatible with $\hetE_D$---in particular, we obtain $(\hetE_D)_{a,b,x,y,f,f'}(g,\phi_{\bullet}(g))$.
  This means that $\phi_{\bullet}(g)$ is, up to homogeneous equality $E_D$, the same as the postulated transport above.
  Altogether, an indiscernibility $f \fiso f'$ is exactly an equality $E_{a,b}(f, f')$; thus, in a univalent structure, $A$ is a set with equality $E$.
    
  Given $b: O$ and $y_1, y_2 : O_D(b)$, an indiscernibility $y_1 \fiso y_2$ is exactly a displayed isomorphism over the identity on $b$; this is shown analogously to the characterization of indiscernibilities of objects in a precategory (see \cref{sec:univalence-at-o}). In more detail, such an indiscernibility consists, in particular, of an equivalence
  \[\phi_{b,b,y_1,\bullet,1_b} : (A_D)_{b,b}(y_1,y_1, 1_b) \simeq (A_D)_{b,b}(y_1, y_2, 1_b),\] 
  and thus in particular of a displayed morphism \[\phi \eqdef \phi_{b,b,y_1,\bullet,1_b} (\overline{f}_{y_1}): (A_D)_{b,b}(y_1, y_2, 1_b).\]
  The morphism $\phi$ is an isomorphism, and transport in $A_D$ is given by composition with $\phi$ or its inverse, e.g.,
  \[(T_D)_{a,b,b,x,y_1,y_1,f,1,f}(\overline{f},\overline{1}_{y_1}, \overline{f}) \to (T_D)_{a,b,b,x,y_1,y_2,f,1,f}(\overline{f},\phi, \phi_{b,b,y_1,\bullet,1_b}(\overline{f})).\]
    
  Regarding indiscernibilities at $O$, the same reasoning as in \cref{eg:semi-displayed} applies. In particular, for a univalent structure, the underlying category is univalent if and only if the displayed category is an isofibration.
  
\end{example}

\begin{remark}\label{rmk:heteq}
  In~\cite[Appendix C]{MFOLDS}, Makkai describes a general method for starting with a diagram signature that has no equalities at all, and adding heterogeneous equalities (which he calls ``global equalities'') to all non-relational sorts, as well as a family of FOLDS-axioms for these equalities.
  Our signature for strict categories in \cref{eg:strict-1-cat} can be obtained by this method, if we start with the signature for categories with the equality on arrows removed (otherwise it would get duplicated), and our congruence and transport axioms are instances of Makkai's.
  Based on this example, it is natural to conjecture that a structure for any diagram signature with all heterogeneous equalities added that satisfies Makkai's equality axioms is univalent if and only if all its sorts are sets with standard equality.
  One might also hope to generalize Makkai's construction to add equalities only at some sorts (perhaps at a sieve of sorts) and thereby recover our \cref{eg:strict-2-cat,eg:disp-cat}.
\end{remark}

\begin{example}[Contextual categories, a.k.a.\ C-systems]\label{eg:ctx-cat}
  Contextual categories were introduced by Cartmell \cite{Cart86} as a mathematical structure in which to interpret generalized algebraic theories.
  A contextual category comes, in particular, with a ``father'' function $F : \C_0 \to \C_0$ on the objects of the underlying category, and with a length function $L : \C_0 \to \mathbb{N}$.
  Furthermore, it has a distinguished class of ``dependent projections'' $\pi_\Gamma : \C(\Gamma,F(\Gamma))$, with a functorial choice of pullbacks of dependent projections along any morphism.
  The pullback of a dependent projection is a dependent projection again;
  to state this condition, we need to postulate an equality of objects, which we hence postulate in our signature.
  The function $F : \C_0 \to \C_0$ decreases the length by 1, that is, $L(F(\Gamma) = L(\Gamma) -1$, whenever $L(\Gamma) > 0$.
  Furthermore, a contextual category has a distinguished terminal object, the only object whose length is 0.
  
  Thus, a suitable signature for the theory of contextual categories is given as follows:
  
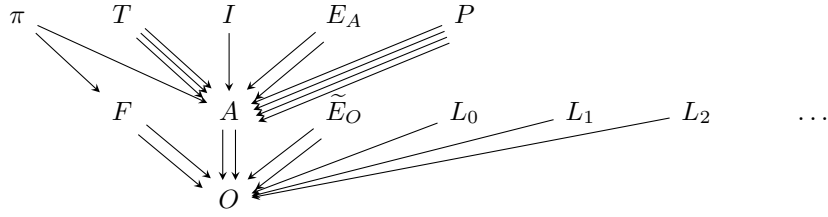
\[
    \begin{tikzcd}[ampersand replacement=\&]
      \pi \ar[drr] \ar[dr]
      \&
      T \ar[dr] \ar[dr, shift left] \ar[dr, shift right]
      \&
      I \ar[d]
      \& \hetE_A \ar[dl, shift right] \ar[dl, shift left]
      \& P \ar[dll, shift left]\ar[dll, shift left=2] \ar[dll, shift left = 3]\ar[dll]
      \&
      \\
      \&
      F \ar[dr, shift left] \ar[dr, shift right]
      \& 
      A \ar[d, shift left] \ar[d, shift right]
      \&
      \hetE_O \ar[dl, shift right] \ar[dl, shift left]
      \&
      L_0 \ar[dll]
      \&
      L_1 \ar[dlll]
      \&
      L_2 \ar[dllll]
      \&
      \ldots
      \\
      \&
      \&
      O
    \end{tikzcd}
  \]
  We recognize some components from previous examples:
  the core is given by the signature of a strict category as in \cref{eg:strict-1-cat}, and
  the sort $P$ specifies pullbacks as in \cref{eg:pullbacks}.
  The infinite family of rank-1 sorts can be represented using either $\exo{\Nat}$ (if cofibrant) or $\Nat$, as in \cref{eg:unbiased-monoidal,eg:multicats}.

  The sorts $F$, $\pi$, and $(L_i)_{i : \mathbb{N}}$ implement the contextual structure.
  The length function is implemented here as a sequence of predicates $(L_i)_{i : \mathbb{N}}$, where, in a model, an element $w : L_i(a)$ indicates that the length of $a : O$ is $i$. 
  An element $w : F(a,b)$ with $a,b : O$ indicates that $b$ is the father of $a$.
  An element $p : \pi_{a,b}(w,f)$ with $w : F(a,b)$ and $f : A(a,b)$ indicates that $f$ is the dependent projection associated to $a : O$.
  The equality $\hetE_O$ is assumed to be a congruence for the sorts $F$, $\pi$, and $\L_i$, and similar for $\hetE_A$. 

  Hence, $O$ is a set, and an indiscernibility $a \fiso b$ in a univalent model means that $a = b$.
  In particular, a univalent contextual category does not necessarily have a univalent underlying category.
  Compare this to the ``categorical'' structures for the interpretation of generalized algebraic theories studied in \cref{eg:cwf-cwa}.
  The non-categorical structure of a contextual category , with its effects pointed out here, motivated Voevodsky's renaming of contextual categories to ``C-systems'' \cite{Csystems}.

  A morphism of structures is a functor that, in particular, commutes with length and father \emph{strictly}; it is an equivalence just when it is an isomorphism of categories, similar to \cref{eg:strict-1-cat}.

\end{example}

\chapter{Graphs and Petri nets}
\label{sec:graph-egs}

In \cref{sec:set-egs} we considered signatures of height 2, which are (if univalent) necessarily built only out of sets; while in \cref{sec:1cat-egs,sec:hcat-egs,sec:strict-egs} we considered signatures of greater height for categorical structures.
Generally speaking, the presence of composition and identities in a categorical structure is what reduces the \textit{a priori} rather complicated notion of indiscernibility to a more familiar notion of isomorphism or equivalence; in \cref{eg:semicats} we saw a taste of what happens in the absence of identities.

In this \lcnamecref{sec:graph-egs} we look at a few signatures of height $>2$ for graphs and graph-like structures that entirely lack composition and identities.
The resulting notions of indiscernibility are a little strange, and naturally-occurring examples seem unlikely to be univalent.
However, with the technique of heterogeneous equality introduced in \cref{sec:strict-egs} we can eliminate the strange behavior and force all the types involved to be sets again.

\begin{example}[Directed multigraphs]
\label{eg:dir-multigraph}\index{multigraph}\index{multigraph!directed}
 The univalent models of the signature of \cref{eg:pre-po-sets}\ref{item:poset} (with the equality axioms but not the partial order axioms) are sets equipped with a binary relation.
 These, in turn, are special cases of directed graphs with at most one edge between any two nodes.
 A natural signature for directed \emph{multi}graphs, which may have several edges between nodes, is simply the signature for categories with both composition and identities removed:
   \[
     \begin{tikzcd}
            E \ar[d, shift left] \ar[d, shift right]
         \\
            A\ar[d, shift left] \ar[d, shift right]
         \\
            O
     \end{tikzcd}
   \]
  The sort $E$ is asserted to be an equivalence relation.
  In a univalent structure for this signature, $E$ is pointwise a proposition, and, for any $a,b : O$, the type $A(a,b)$ is a set with equality given by $E_{a,b}$.
  An indiscernibility $a \fiso b$ of objects $a, b: O$ consists of families of bijections $ A(x,a) \cong A(x,b)$ between the sets of edges into $a$ and $b$, respectively, and between the edges out of $a$ and $b$, and between the loops on $a$ and $b$.

  As noted above, if we add an equality relation on $O$ as well, and make the equality of $A$ heterogeneous:\index{equality!heterogeneous}
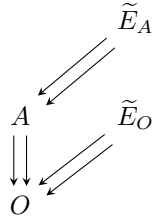
   \[
     \begin{tikzcd}
            & \hetE_A \ar[dl, shift left] \ar[dl, shift right]
         \\
         A\ar[d, shift left] \ar[d, shift right] &
         \hetE_O \ar[dl,shift left] \ar[dl,shift right]
         \\
            O
     \end{tikzcd}
   \]
   with suitable congruence axioms, then in a univalent structure both $O$ and $A$ will be sets with standard equality.
\end{example}

Recall that in \cref{eg:semicats} we considered semicategories, which are directed multigraphs with an associative composition (but no identities).
We could also consider \defemph{reflexive graphs}\index{graph!reflexive}, which have identities (i.e., specified loops at each vertex) but no composition.
We leave it to the reader to design a suitable diagram theory (with and without equality) for such graphs, and to characterize the indiscernibilities in each sort, as well as the morphisms and equivalences of models of the theory.

\begin{example}[Pre-nets, tensor schemes]
\label{eg:pre-net}\index{net!pre-}\index{tensor scheme}
A \emph{pre-net} \cite[Definition~3.1]{10.1006/inco.2001.3050} has a type of ``places'' and for each pair of natural numbers $m,n$, a type of ``transitions'' dependent on $m+n$ places. 
The same notion is known under the name of ``tensor scheme'' \cite[Definition~1.4]{joyal-street-tensor-calculus-I};
it is the natural underlying data from which to generate a free monoidal or symmetric-monoidal category. 
   
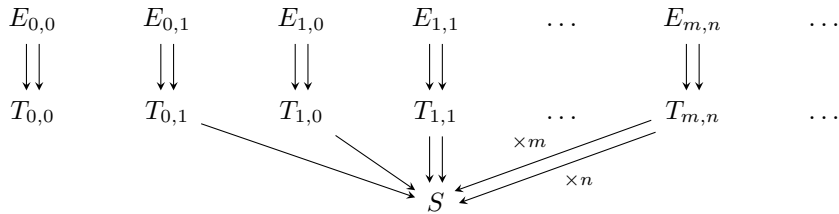
\[
     \begin{tikzcd}[ampersand replacement=\&]
            E_{0,0}\ar[d, shift left] \ar[d, shift right]
            \&
            E_{0,1}\ar[d, shift left] \ar[d, shift right]
            \&
            E_{1,0}\ar[d, shift left] \ar[d, shift right]
            \&
            E_{1,1}\ar[d, shift left] \ar[d, shift right]
            \&
            \ldots
            \&
            E_{m,n}\ar[d, shift left] \ar[d, shift right]
            \&
            \ldots
         \\
           T_{0,0}
           \&
           T_{0,1} \ar[drr]
           \& 
           T_{1,0} \ar[dr]
           \& 
           T_{1,1} \ar[d, shift left] \ar[d, shift right]
           \& 
           \ldots 
           \& T_{m,n}\ar[dll, "\times n", shift left] \ar[dll, "\times m"', shift right]
           \&
           \ldots
         \\
            \& \& \& S
     \end{tikzcd}
   \]
   Here, the sort $T_{m,n}$ of transitions has $m+n$ arrows to (a.k.a.\ dependencies on) the sort $S$ of places (a.k.a.\ ``species''), regarded as $m$ inputs and $n$ outputs.
   The infinite families of sorts can be represented using either $\exo{\Nat}$ (if cofibrant) or $\Nat$, as in \cref{eg:unbiased-monoidal,eg:multicats}.
   The discussion is largely the same as in \cref{eg:dir-multigraph};
   an indiscernibility $a \fiso b$ in $S$ consists of families of bijections between the sets of transitions with $a$ and $b$ appearing in one or more of their inputs or outputs, while by adding heterogeneous equalities we can force $S$ to also be a set with standard equality.
   A morphism of pre-nets \cite[Definition~3.2]{10.1006/inco.2001.3050} is a pair of functions on places and transitions that are compatible in a suitable sense; this is exactly a morphism of structures for the above signature.
\end{example}

\begin{example}[Other kinds of Petri nets]\label{eg:petri}\index{net!Petri}
  In a pre-net, the sets $T_{2,1}(x,y;z)$ and $T_{2,1}(y,x;z)$ are unrelated; but in a \emph{symmetric} monoidal category the hom-sets $\hom(x\otimes y,z)$ and $\hom(y\otimes x,z)$ are isomorphic.
  A \emph{Petri net} is a refinement of a pre-net that incorporates some kind of ``symmetry'' like this (though historically they are the earlier notion).
  There are many different inequivalent notions of ``Petri net'', not all of which are amenable to formalization in our framework.
  But from our present perspective, one natural approach to add symmetry to \cref{eg:pre-net} is to replace the indexing set $\Nat$ by the 1-type $\FinSet$, noting that $\Nat$ is both the 0-truncation of $\FinSet$ and the type of \emph{ordered} finite sets.

  That is, whereas in \cref{eg:pre-net} the exotypes of rank-1 and rank-2 sorts are both $\Nat\times\Nat$ (or $\exo{\Nat}\times\exo{\Nat}$), we now take these exotypes to be $\FinSet\times\FinSet$.
  Thus, as in \cref{eg:species,eg:unbiased-sym-monoidal,eg:fat-sym-multicats}, we consider here a diagram signature in which the types $\L(n)$ are not all sets, which we can attempt to draw as follows:
    
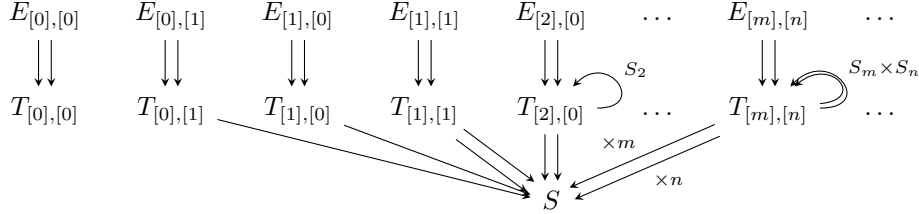
\[
     \begin{tikzcd}[column sep = small, ampersand replacement=\&]
            E_{[0],[0]}\ar[d, shift left] \ar[d, shift right]
            \&
            E_{[0],[1]}\ar[d, shift left] \ar[d, shift right]
            \&
            E_{[1],[0]}\ar[d, shift left] \ar[d, shift right]
            \&
            E_{[1],[1]}\ar[d, shift left] \ar[d, shift right]
            \&
            E_{[2],[0]}\ar[d, shift left] \ar[d, shift right]
            \&
            \ldots
            \&
            E_{[m],[n]}\ar[d, shift left] \ar[d, shift right]
            \&
            \ldots
         \\
           T_{[0],[0]}
           \&
           T_{[0],[1]} \ar[drrr]
           \& 
           T_{[1],[0]} \ar[drr]
           \& 
           T_{[1],[1]} \ar[dr, shift left] \ar[dr, shift right]
           \& 
           T_{[2],[0]} \ar[d, shift left] \ar[d, shift right] \ar[loop,out=0,in=45,looseness=4,"S_2"']
           \& 
           \ldots 
           \& T_{[m],[n]}\ar[dll, "\times n", shift left] \ar[dll, "\times m"', shift right]
           \ar[loop,out=5,in=40,looseness=4] \ar[loop,out=0,in=45,looseness=4,"S_m\times S_n"']
           \&
           \ldots
         \\
            \& \& \& \& S
     \end{tikzcd}
   \]
   In a univalent structure for this signature, the equality sorts $E_{[m],[n]}$ are pointwise propositions, and the sorts of edges $T_{[m],[n]}$ are pointwise sets with equality given by $E_{[m],[n]}$.
   The indiscernibilities behave just as in \cref{eg:pre-net}.
   
  The structures for this signature are closely related to the \emph{whole-grain Petri nets}\index{net!whole-grain Petri} of Kock \cite[Section~2.1]{kock_petri}, which are diagrams of sets
  \[ S \leftarrow I \rightarrow T \leftarrow O \rightarrow S \]
  in which the functions $I\to T$ and $O\to T$ have finite fibers.
  Thus these functions are jointly classified by a map $T \to \FinSet\times\FinSet$, which we can replace by a type family $T : \FinSet\times\FinSet \to \U$.
  If we also encode the remaining functions $I\to S$ and $O\to S$ by a further dependency of $T$ on some power of $S$, we obtain exactly a structure for the above signature.

  The structures arising from whole-grain Petri nets in this way can be characterized as those for which each $T_{[m],[n]}$ is a set with equality $E_{[m],[n]}$ (that is, the structure is univalent at $E_{[m],[n]}$ and $T_{[m],[n]}$), and in addition $S$ is a set and also
  \[T \converts \sm{X,Y:\FinSet}{s:X+Y \to S} T_{X,Y}(s) \]
  is a set.
  This latter requirement says equivalently that the action of $S_m\times S_n$ on $T_{[m],[n]}$ is free; if we drop it, we obtain a notion studied by~\cite{bgms:nets} under the name of ``$\Sigma$-nets'' (called a ``digraphical species'' in~\cite{kock_petri}).\index{net!$\Sigma$-}\index{species!digraphical}
  Thus, if we add heterogeneous equality and its axioms to this signature, its univalent models are precisely the $\Sigma$-nets.
\end{example}

\chapter{Enhanced categories and higher categories}
\label{sec:restr-indis-egs}

The phrase ``enhanced (higher) category'' was introduced, though not really defined, by~\cite{ls:limlax}.
Here we use it to mean a categorical structure that contains a ``underlying'' ordinary category or higher category, but in which the additional structure on that underlying category is not purely categorical, i.e., not expressed purely in terms of functors and natural transformations.
This frequently has the effect that, in contrast to the categorical structures studied in \cref{sec:1cat-egs,sec:hcat-egs}, the notion of indiscernibility often does \emph{not} coincide with that in the underlying category.
Yet, in most cases this different notion of indiscernibility turns out to have already been recognized in the literature as the ``correct'' notion of ``sameness''.

In this \lcnamecref{sec:restr-indis-egs} we describe some enhanced categorical structures where the extra structure can be expressed in terms of functors, but with additional strict conditions on equality of objects.
In \cref{sec:unnat} we will consider enhanced structures involving truly non-functorial or unnatural operations.

\begin{example}[$\dagger$-categories and $\dagger$-anafunctors]\label{ex:dagger}\index{category!$\dagger$-}
  A $\dagger$-category is a category with coherent isomorphisms $(\_)^\dagger : \hom(x,y) \to \hom(y,x)$.
  Historically this has been proposed as an especially interesting example to consider in structural approaches to category theory, since the correct notion of ``sameness'' for objects of a $\dagger$-category is \emph{not} ordinary isomorphism but rather \emph{unitary} isomorphism\index{isomorphism!unitary} (one satisfying $f^{-1} = f^\dagger$), and similarly ``$\dagger$-structure'' on a category does not transport naturally across equivalence of categories.
  
  In our framework we can deal with this by incorporating the $\dagger$-structure into the signature, represented of course by its graph.
  A signature for $\dagger$-categories is as follows:
  
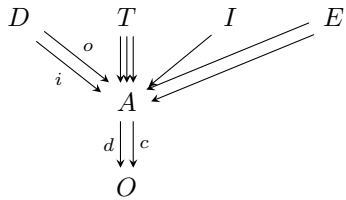
\[
    \begin{tikzcd}
      D \ar[dr, shift left, "o"] \ar[dr, shift right, "i"']
      &
      T \ar[d] \ar[d, shift left] \ar[d, shift right]
      &
      I \ar[dl]
      &
      E \ar[dll, shift left] \ar[dll, shift right]
      \\
      &
      A \ar[d, shift right, "d"'] \ar[d, shift left, "c"]
      \\
      &
      O
    \end{tikzcd}
  \]
  Here we have $co \steq di$ and $do \steq ci$, plus the exo-equalities of \cref{fig:signatures}.
  In addition to the axioms of a category, we require $E$ to be a congruence for $D$, and we require $D$ to be a functional relation that maps compositions to compositions and identities to identities.
  We also write $g = f^\dagger$ for $D(f,g)$.
  
  Given a model of this theory, univalence at $D$ means that $D$ is pointwise a proposition. 
  Since $E$ is a congruence for $D$, univalence at $A$ still entails that $A$ is a set with equality given by $E$.
  Given $a, b : O$, an indiscernibility $a \fiso b$ consists of an isomorphism $\phi : a \cong b$ such that (unfolding the definition of indiscernibility at $D$), for any morphism $f$, we have $\phi \circ f^\dagger = (f \circ \phi^{-1})^\dagger$, $(\phi\circ f)^\dagger = f^\dagger \circ \phi^{-1}$, and an equation about composition on both the left and the right.
  In particular, we have $D_{a,a}(1_a,1_a) \leftrightarrow D_{a,b}(\phi, \phi^{-1})$, where the left-hand side holds by one of the axioms imposed.
  We thus have $\phi^\dagger = \phi^{-1}$, and the other equations follow from this and the compabitibility of $(\_)^\dagger$ with composition.
  An isomorphism $\phi$ such that $\phi^\dagger = \phi^{-1}$ is called \emph{unitary}; thus an indiscernibility $a \fiso b$ is exactly a unitary isomorphism $a \cong b$.

  Consequently, an equivalence of $\dagger$-categories is a $\dagger$-functor that is fully faithful and unitarily-essentially split-surjective.
  This in turn corresponds exactly to an adjoint equivalence of $\dagger$-categories, involving $\dagger$-functors, such that the unit and counit are unitary natural isomorphisms;
  the usual construction (see, e.g., \cite[Lemma~6.6]{AKS13}) applies, using additionally that the constructions back and forth preserve unitarity of the input.

  The signature of a $\dagger$-anafunctor is analogous to that of an anafunctor between categories given in \cref{eg:anafunctors}.\index{anafunctor!$\dagger$-}
  In addition to the axioms given there, we require that the isomorphisms $(\_)^\dagger$ are preserved by the anafunctor, in the sense that for $w_1 : FO(a_1, b_1)$ and $w_2 : FO(a_2,b_2)$ and $f : DA(a_1, a_2)$, we have that $FA_{w_1,w_2}(f)^\dagger = FA_{w_2, w_1}(f^\dagger)$.
  This implies that $F$ preserves unitary isomorphisms.
  Since $1_x$ is a unitary isomorphism, it follows that the canonical isomorphisms $F_{w,w'}(1_x)$ relating different values of $F$ are unitary.
  Therefore, the existential saturation condition must be restricted to unitary isomorphisms: given $w : FO(x,y)$ and a unitary $g : y \cong y'$, there exists a $w' : FO(x,y')$ such that $F_{w, w'}(1_x) = g$.
  
  The rest of the theory is exactly parallel to \cref{eg:anafunctors} but with all isomorphisms being unitary.\index{anafunctor}
  In particular, an indiscernibility $y_1 \fiso y_2$ between objects in the codomain $\dagger$-category $C$ is exactly a unitary isomorphism $y_1 \cong y_2$, and similarly for the domain $D$.
  Furthermore, a morphism of structures is a square of $\dagger$-functors that commutes up to a unitary natural isomorphism.\index{natural isomorphism!unitary}
\end{example}

\begin{example}[$\M$-categories, e.g., homotopical categories]
\label{eg:m-cat}\index{category!$\M$-}\index{category!homotopical}
An $\M$-category (referred to as ``subset-category'' in \cite{power:premonoidal}) is a category enriched over the category whose objects are subset-inclusions and whose morphisms are commutative squares.
In detail, it consists of a type of objects and, for any two objects, \emph{two} sets of morphisms, which we call (following~\cite{ls:limlax}) \emph{tight} and \emph{loose}, and an inclusion of tight into loose morphisms.\index{morphism!tight}\index{morphism!loose}
One class of examples of $\M$-categories are \emph{homotopical categories}, in which the tight morphisms are called ``weak equivalences''; another class of examples is provided by the hereditary membership structures that model a ZF-like membership-based set theory (which can be constructed in UF as in~\cite[\S10.5]{HTT}), in which the loose morphisms are functions and the tight morphisms are actual subset inclusions (not just injections).

$\M$-categories can alternatively be expressed via a unary predicate ``being tight'' on one family of (loose) morphisms, such that the predicate is closed under identity and composition.
We first consider a signature for this alternative formulation:
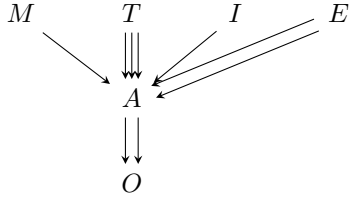
\[
  \begin{tikzcd}
    M \ar[dr]
    &
    T \ar[d, shift right] \ar[d, shift left] \ar[d]
    &
    I \ar[dl]
    &
    E \ar[dll, shift left] \ar[dll, shift right]
    \\
    &
    A \ar[d, shift left] \ar[d, shift right]
    \\
    &
    O
  \end{tikzcd}
\]
 We assert axioms asserting that $E$ is a congruence for $M$ and the other top-level sorts,
 and that $M$ is closed under identity and composition.
 Univalence at $M$ then means that $M$ is pointwise a proposition.
 Univalence at $A$ still means that $A$ is a set with equality given by $E$, since $E$ is required to be a congruence for $M$.
 Given $a,b : O$, an indiscernibility $a \fiso b$ consists of an isomorphism $\phi : a \cong b$ in the underlying category that is coherent with respect to $M$, i.e., such that $M_{a,y}(f) \leftrightarrow M_{b,y}(\phi_{\bullet y}(f))$ for any $y : O$ and $f : A(a,y)$ and similarly for holes on the right and in both variables.
 These coherence conditions simplify to the condition of $\phi$ being tight:
 on the one hand, setting $y\eqdef a$ and $f \eqdef 1_a$ yields the coherence condition $M_{a,a}(1_a) \leftrightarrow M_{a,b}(\phi)$, and since $M$ contains identities, this means in particular that $\phi$ is required to be tight.
 On the other hand, if $\phi$ is tight, the remaining coherence conditions then follow from $M$ being closed under composition.
 In summary, in a univalent $\M$-category, an indiscernibility $a \fiso b$ is exactly a tight isomorphism.\index{isomorphism!tight}

 A morphism of $\M$-categories is an $\M$-functor, i.e., a functor that preserves tightness.\index{functor!$\M$-}
 An equivalence of $\M$-categories is an $\M$-functor that is fully faithful, reflects tightness, and is split essentially surjective with respect to \emph{tight} isomorphism (i.e., every object of the codomain is tightly isomorphic to the image of some object in the domain).
 
 An alternative theory, closer to the first description, has the following signature:
 
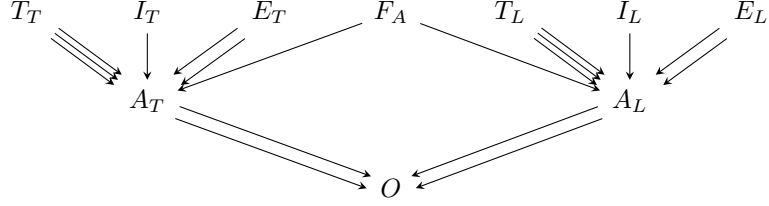
\[
   \begin{tikzcd}
     T_T \ar[dr] \ar[dr, shift left] \ar[dr, shift right]
     &
     I_T \ar[d]
     &
     E_T \ar[dl, shift left] \ar[dl, shift right]
     &
     F_A \ar[dll] \ar[drr]
     &
     T_L \ar[dr] \ar[dr, shift left] \ar[dr, shift right]
     &
     I_L \ar[d]
     &
     E_L \ar[dl, shift left] \ar[dl, shift right]
     \\
     &
     A_T \ar[drr, shift left] \ar[drr, shift right]
     &
     &
     &
     &
     A_L \ar[dll, shift left] \ar[dll, shift right]
     \\
     &
     &
     &
     O
    \end{tikzcd}
  \]
  Here $A_T$ represents the tight morphisms and $A_L$ the loose morphisms.
  This is essentially the signature of a functor (see \cref{eg:anafunctors}) whose map on objects is the identity.
  Here, $(F_A)_{x,y}(f,g)$ means that the tight morphism $f : A_T(x,y)$ is mapped to the loose morphism $g : A_L(x,y)$ by $F_A$.
  We also write this as $F_A(f) = g$.
  We impose axioms stating that $F_A$ is a function from $A_T$ to $A_L$, preserves identities and compositions, and is injective (so that the functor is faithful).
  
  Univalence at the top-level sorts, as usual, means exactly that these sorts are pointwise propositions.
  Univalence at $A_T$ and $A_L$ means that these sorts are sets with equality given by $E_T$ and $E_L$, respectively.
  
  Given $a, b : O$, an indiscernibility consists of an isomorphism $\phi_T$ in the tight fragment and an isomorphism $\phi_L$ in the loose fragment that are coherent with respect to $F_A$.
  This means for instance that $(F_A)_{y,a}(f, g) \leftrightarrow (F_A)_{y,b}(\phi_T \circ f, \phi_L \circ g)$, and similar for the other variable, and for both variables simultaneously.
  Since $F_A$ preserves identities, the previous condition in particular entails $F_A(\phi_T) = \phi_L$ (obtained for $y\eqdef a$ and $f,g \eqdef 1_a$), that is, $\phi_L$ is determined by $\phi_T$.
  All the coherences can then be deduced from this equation and the fact that $F_A$ preserves compositions.
  
  In summary, we again obtain that an indiscernibility $a \fiso b$ is exactly a tight isomorphism $a \cong b$.\index{isomorphism!tight}
  
  A morphism of such structures consists of a ``functor'' with an action on both tight and loose morphisms that preserves the inclusion of tight into loose morphisms.
  Such a functor is an equivalence when it is tight-essentially split-surjective and fully faithful on both tight and loose morphisms.

  Given a model of the first theory, we obtain a model of the second theory by defining $A_T$ to consist of those arrows that satisfy $M$.
  Using the univalence axiom, this construction can be shown to be an equivalence between the respective types of models of these theories.
  The first theory is of course simpler, but the second has the advantage that it can be generalized by removing the injectivity axiom on $F$; see for instance the example of Freyd-categories discussed after \cref{eg:premonoidal}.\index{category!Freyd-}
\end{example}

\begin{example}[$\F$-bicategories]\index{category!$\F$-}\index{bicategory!$\F$-}
  An $\F$-category is a 2-categorical version of an $\M$-category: it can be defined as a 2-category equipped with a subclass of its 1-morphisms called ``tight'', or as a 2-functor that is the identity on objects, injective on 1-cells, and locally fully faithful.
  $\F$-categories were introduced in~\cite{ls:limlax} to represent 2-categories of algebras for a 2-monad, where the tight morphisms are strict or pseudo algebra morphisms and the loose morphisms are lax or colax ones.

  The analogous weak notion of \emph{$\F$-bicategory} can be defined as a bicategory equipped with a subclass of its 1-morphisms called ``tight'' that is invariant under isomorphism, or as a pseudofunctor that is the identity on objects and locally fully faithful.
  The pseudoalgebras for a pseudomonad together with their pseudo and lax morphisms form an $\F$-bicategory, and likewise for the pseudo and colax morphisms.
  Another example of an $\F$-bicategory is a \emph{proarrow equipment}~\cite{wood:proarrows-i}.\index{proarrow equipment}

  If we represent an $\F$-bicategory analogously to the second variant of \cref{eg:m-cat}, we obtain the following signature:
   
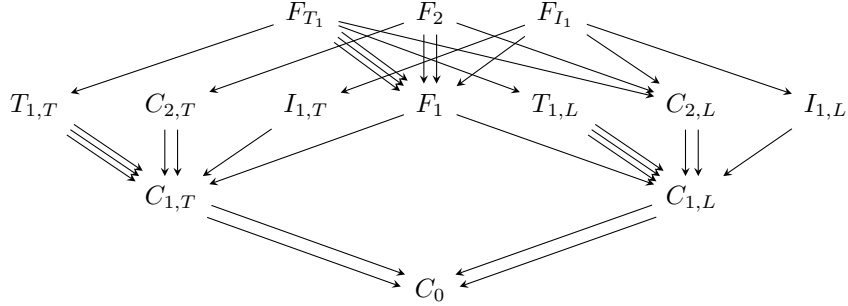
\[
    \begin{tikzcd}[ampersand replacement=\&]
      \& \& F_{T_1} \ar[drr] \ar[dll] \ar[drrr]
                    \ar[dr] \ar[dr,shift left] \ar[dr,shift right]
      \&
      F_{2} \ar[drr] \ar[dll] \ar[d,shift left] \ar[d,shift right]
      \&
      F_{I_1} \ar[drr] \ar[dll] \ar[dr] \ar[dl]
      \\
      T_{1,T} \ar[dr] \ar[dr,shift left] \ar[dr,shift right]  \& 
      C_{2,T} \ar[d,shift left] \ar[d,shift right] \& 
      I_{1,T} \ar[dl] \&
      F_{1} \ar[drr] \ar[dll] \&
      T_{1,L} \ar[dr] \ar[dr,shift left] \ar[dr,shift right] \&
      C_{2,L} \ar[d,shift left] \ar[d,shift right] \&
      I_{1,L} \ar[dl]
      \\
      \&  C_{1,T} \ar[drr,shift left] \ar[drr,shift right] 
         \& \& \& 
      \& C_{1,L} \ar[dll,shift left] \ar[dll,shift right] 
      \\
      \& \& \&  C_0 
    \end{tikzcd}
  \]
  where for readability we omit top-level sorts $A$, $H$, $E$, $L$, $R$, $T_2$, and $I_2$ as in \cref{eg:bicats} on both the tight (subscript $T$) and the loose (subscript $L$) fragment of the signature.
  We furthermore impose axioms asserting that $F_{2} \rightrightarrows F_{1}$ is a family of pointwise (i.e., for any two $x,y : C_0$) saturated anafunctors (cf.~\cref{eg:anafunctors}), and that these anafunctors are fully faithful.
  On $F_{T_1}$ and $F_{I_1}$ we impose the axioms of a family of natural transformations.

  As usual, univalence at $F_{T_1}$, $F_{C_2}$, and $F_{I_1}$ means that these sorts are pointwise propositions.
  Univalence at $F_1$ means that $F_1$ is pointwise a set.
  As per the discussion of \cref{eg:anafunctors} and \cref{eg:nat-trans}, 
  the indiscernibilities at $C_{1}$ (in $T$ and $L$) are exactly the isomorphisms; they are not changed by the presence of the functors $F_2 \rightrightarrows F_1$ and natural transformations $F_{T_1}$ and $F_{I_1}$.
  
  An indiscernibility $a \fiso b$ in $C_0$ consists of 
  \begin{enumerate*}
   \item a tight adjoint equivalence $\phi_T : a \simeq_T b$, i.e., an adjoint equivalence in the tight fragment of the signature;
   \item a loose adjoint equivalence $\phi_L : a \simeq_L b$ in the loose fragment of the signature; and
   \item transport functions for the sorts $F$ corresponding to the family of functors.
  \end{enumerate*}
  For instance, we have equivalences $(F_1)_{x,a}(f,g) \simeq (F_1)_{x,b}(\phi_{x\bullet}(f),\phi_{x\bullet}(g))$.
  Since $(F_1)_{a,a}(1_a, 1_a)$, an indiscernibility in particular yields $w : (F_{C_1})_{a,b}(\phi_T, \phi_L)$, meaning that, by saturation of $F$ in the fiber over $a,b$, the equivalence $\phi_L$ is determined by $\phi_T$.
  The other transport functions for $F_1$ are determined in turn by $w : (F_{C_1})_{a,b}(\phi_T, \phi_L)$, since $F_1$ is suitably compatible with 1-composition.
  Similarly, transport at $F_2$ is determined by compatibility of $F_2$ with action in source and target.
  Summarily, an indiscernibility in $C_0$ is exactly an adjoint equivalence internal to the bicategory spanned by the tight fragment (index $T$) of the signature.

  We can also write down an analogue of the first variant of \cref{eg:m-cat}: on top of the ordinary signature for bicategories we add one more sort $T$ dependent on a single 1-cell, where the intended interpretation of $T_{x,y}(f)$ is ``$f$ is tight''.
  Since there are no sorts dependent on this $T$, it is a family of propositions (even though it is not at top rank).
  We assert as an axiom that in addition to $T$ containing identities and being preserved by composition, it is also invariant under isomorphism: if $T_{x,y}(f)$ and $f\cong g$ then $T_{x,y}(g)$ --- this is an analogue of the saturation of the identity-on-objects anafunctor under the other approach.
  (Note that in contrast to the signature for ``a category with a specified object'', it does make sense for $T$ to be a mere predicate even though it is not at top rank, because ``being tight'' really is just an isomorphism-invariant property, whereas ``being the specified object'' is \emph{structure} that can be transported along an isomorphism, but only in a specified way.)
  This axiom, analogous to asserting that equality relations are congruences for top-level predicates, ensures that the notion of indiscernibility for 1-cells remains unchanged.
  And since identities are tight, indiscernibilities of objects once again reduce to tight adjoint equivalences.

  By combining this example with \cref{eg:strict-2-cat}, we can also obtain a signature for strict $\F$-categories in the original sense of~\cite{ls:limlax}.
\end{example}

\begin{example}[Bicategories with contravariance~\cite{shulman:contravariance}]\label{eg:bicat-contra}\index{bicategory!with contravariance}
  A \defemph{bicategory with contravariance} has a set of objects together with, for any two objects $x$ and $y$, two hom-categories $A^+(x,y)$ and $A^-(x,y)$, with four composition operations that multiply signs, and such that postcomposition with $A^-(x,y)$ is contravariant.
  The primordial example is $\mathsf{Cat}$, where the two hom-categories consist of covariant functors and contravariant functors.
  We can represent this with an adaptation of the signature of \cref{eg:bicats}, whose height-3 truncation is
  
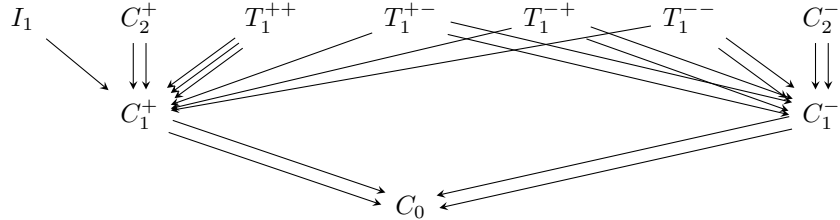
\[
    \begin{tikzcd}
      I_1 \ar[dr] & C_2^+ \ar[d,shift left] \ar[d,shift right] &
      T_1^{++} \ar[dl] \ar[dl,shift left] \ar[dl,shift right] &
      T_1^{+-} \ar[dll] \ar[drrr,shift left] \ar[drrr,shift right] &
      T_1^{-+} \ar[dlll] \ar[drr,shift left] \ar[drr,shift right] &
      T_1^{--} \ar[dllll] \ar[dr,shift left] \ar[dr,shift right] &
      C_2^- \ar[d,shift left] \ar[d,shift right]\\
      & C_1^+ \ar[drr,shift left] \ar[drr,shift right] &&&&& C_1^-\ar[dlll,shift left] \ar[dlll,shift right] \\
      &&& C_0
    \end{tikzcd}
  \]
  and whose rank-3 sorts implement equality, all sorts of composition of 2-cells, and the associativity and unit isomorphisms.

  Univalence at ranks $>0$ means that both hom-categories are univalent in the usual sense.
  An indiscernibility at $C_0$ is a \emph{covariant} adjoint equivalence.
  Thus, univalence at $C_0$ means that the underlying bicategory of covariant morphisms is univalent.

  By adding heterogeneous equality as in \cref{eg:strict-2-cat}, we can also represent strict 2-categories with contravariance.\index{equality!heterogeneous}
\end{example}

\chapter[Unnatural and nonfunctorial operations]{Unnatural transformations and nonfunctorial operations}
\label{sec:unnat}

Recall that in \cref{eg:anafunctors,eg:nat-trans,eg:cat-struc} we found that indiscernibilities between objects of a structured category reduced to ordinary indiscernibilities in the underlying category.
However, this conclusion depended crucially on the structure being composed of functors and natural transformations, and it can fail in the presence of \emph{non-functorial} operations on objects or \emph{unnatural} transformations.
Such structures may seem strange, but they do occur from time to time (as the examples below will show), and are also interesting for exploring the limits of our framework.

In general, an indiscernibility in such a structure turns out to be an isomorphism on which the non-functorial operations or unnatural transformations \emph{are} functorial or natural, respectively.
For structures appearing in the literature, this often reduces to a familiar notion in the relevant theory.

\begin{example}[Unnatural transformations]\label{eg:unnatural}\index{unnatural transformation}
  By an \emph{unnatural transformation} between two functors $F,G$ we mean an assignment of a morphism $\lambda_x:F x\to G x$ to each object $x$ of the domain, with no further conditions.\footnote{Thus ``unnatural'' means ``not necessarily natural'', just as a ``noncommutative ring'' means one that is not \emph{necessarily} commutative.}
  However, formulating an ``anafunctorial'' version of this requires a little thought.
  We start with the same signature from \cref{eg:nat-trans} for a natural transformation, with all the same axioms except naturality; but it turns out that we need to replace naturality by something weaker rather than omitting it entirely.

  Suppose that we have an unnatural transformation between ordinary functors $F$ and $G$ with components $\lambda_x:F x\to G x$, as above, and that we make $F$ and $G$ into anafunctors in the standard way with $F_O(x,y) \eqdef (Fx\cong y)$ and similarly for $G$.
  Then given $w_1 : F_O(x,y)$ and $w_2 : G_O(x,z)$, we can define $\Lambda_{x,y,z}(w_1, w_2, \lambda)$ to assert that $\lambda$ is the composite $y \cong F x \xrightarrow{\Lambda_x} G x \cong z$, as we did for a natural transformation.
  The resulting structure does not (of course) satisfy the naturality axiom, but it does satisfy \emph{naturality on identity morphisms}\index{naturality!on identity morphisms}: for any $w_i:FO(x,y_i)$ and $w'_j:GO(x,z_j)$, for $i,j\in \{1,2\}$, and $\lambda_i$ such that $\Lambda_{x,y_i,z_i}(w_i,w_i',\lambda_i)$, the following square commutes:
  \[
    \begin{tikzcd}[column sep=large]
      y_1 \ar[r,"F_{w_1,w_2}(1_x)"] \ar[d,"\lambda_1"'] & y_2 \ar[d,"\lambda_2"] \\
      z_1 \ar[r,"G_{w'_1,w'_2}(1_x)"'] & z_2
    \end{tikzcd}
  \]

  Intuitively, naturality on identity morphisms says that $\lambda_x$ is independent of which ``values'' we choose for $F x$ and $G x$.
  More precisely, suppose given a structure for this signature satisfying all the axioms of \cref{eg:nat-trans} except naturality.
  Then for any function assigning to each $x:O_D$ a pair of objects $y,z:O_C$ with $u:FO(x,y)$ and $v:GO(x,y)$ and a morphism $\lambda : A_C(y,z)$ such that $\Lambda_{x,y,z}(u,v,\lambda)$, we obtain ordinary functors $F'$ and $G'$ from the domain to the codomain and an unnatural transformation $\lambda'$ between them.
  Now if we have \emph{another} such function, we obtain two more ordinary functors $F''$ and $G''$ and an unnatural transformation $\lambda''$ between \emph{them}.
  The anafunctor structure yields natural isomorphisms $F' \cong F''$ and $G' \cong G''$, but naturality on identity morphisms is necessary in order to show that the two unnatural transformations $\lambda'$ and $\lambda''$ are related in the expected way, namely that the evident squares commute:
  \[
    \begin{tikzcd}
      F' x \ar[r,"\cong"] \ar[d,"\lambda'"'] & F'' x \ar[d,"\lambda''"] \\
      G' x \ar[r,"\cong"'] & G'' x.
    \end{tikzcd}
  \]
  Thus, we must assume naturality on identity morphisms to obtain an anafunctorial notion of ``unnatural transformation'' that corresponds to the ordinary such notion between ordinary functors.

  In \cref{eg:nat-trans} we used naturality in four places: to show that indiscernibilities in $FO$, $GO$, $O_C$, and $O_D$ reduce to ordinary ones for anafunctors and categories respectively.
  Naturality on identity morphisms suffices for the first three of these, but not the fourth.
  In the latter case, the extra condition $\Lambda_{x,y,z}(w_1,w_2,\lambda) \to \Lambda_{x',y,z}(\trans \phi {w_1},\trans \phi {w_2},\lambda)$ says precisely that the putative naturality square for $\phi:x\cong x'$ does commute:
  \[
    \begin{tikzcd}[column sep=huge]
      y \ar[r,"F_{w_1,\trans \phi {w_1}}(1_x)"] \ar[d,"\lambda_{w_1,w_2}"'] & y \ar[d,"\lambda_{\trans \phi {w_1},\trans \phi {w_2}}"] \\
      z \ar[r,"G_{w_2,\trans \phi {w_2}}(1_x)"'] & z
    \end{tikzcd}
  \]
  In other words, an indiscernibility $x\fiso x'$ in the domain category of an unnatural transformation $\lambda$ is an isomorphism on which $\lambda$ \emph{is} natural.
  In particular, the domain category is univalent in the ordinary sense if and only if $\lambda$ is natural on all isomorphisms.
\end{example}

Thus, new behavior in the indiscernibilities can only arise from unnatural transformations when the domain category is not univalent.
As we have said, most naturally-occurring categories in univalent foundations are univalent, but there are exceptions.

The most blatant example is that from any category $\D$, perhaps univalent, and any function $F:C \to \ob{\D}$, we can construct a new category $\D_F$ with $\ob{(\D_F)} = C$ and $\D_F(x,y) = \D(F(x),F(y))$.
The isomorphisms in $\D_F$ will be just those of $\D$, but the identifications will be those of $C$, which could be quite different from those of $\ob{\D}$; so $\D_F$ will often fail to be univalent.\footnote{In fact, this construction is universal, in the sense that \emph{every} not-necessarily-univalent category $\C$ can be obtained as $\D_F$ for some univalent category $\D$ and function $F:C\to \ob{\D}$: just let $\D$ be the univalent completion of $\C$.
  However, in practice it is much more common to start with a naturally-occurring univalent $\D$ and apply this construction to obtain a non-univalent $\D_F$.\label{fn:nonunivalent}}\index{category!universal non-univalent}

As a particular case of this example, $C$ could be the type of objects of $\D$ equipped with some structure.
For instance, if $\D$ is a symmetric monoidal category, then $C$ could be the type of monoid objects in $\D$; then $\D_F$ is weakly equivalent to the full subcategory of $\D$ consisting of those objects that can be equipped with a monoid structure, but it is not in general univalent: its isomorphisms are mere isomorphisms in $\D$, but its identifications are \emph{monoid} isomorphisms.\index{category!symmetric monoidal}
The abstract structure of this $\D_F$ is the following.

\begin{example}[Supply in monoidal categories]\index{supply!in monoidal categories}
  A symmetric monoidal category is said to \defemph{supply monoids}~\cite{FongSpivak:supply} if every object is equipped with a specified monoid structure, such that the specified monoid structure of $x\otimes y$ is that induced from those of $x$ and $y$.
  Note that these monoid structures consist of unnatural transformations $U\to x$ and $x\otimes x \to x$.
  Thus, we can obtain a signature for a symmetric monoidal category that supplies monoids by augmenting the signature of symmetric monoidal categories (\cref{eg:cat-struc}) with two predicates for these unnatural transformations:
  
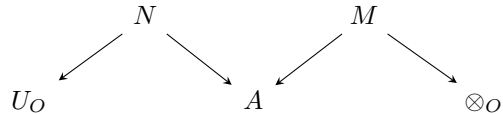
\[
    \begin{tikzcd}
      & N \ar[dl] \ar[dr] & & M \ar[dl] \ar[dr]\\
      U_O & & A & & \otimes_O
    \end{tikzcd}
  \]
  As in \cref{eg:unnatural}, the indiscernibilities of objects will be the isomorphisms on which these transformations \emph{are} natural: i.e., the monoid isomorphisms.
  Thus, the example $\D_F$ constructed above, though not univalent as a mere category, is univalent as a symmetric monoidal category that supplies monoids.

  More generally, there is a notion of when a symmetric monoidal category \defemph{supplies $\mathbb{P}$}, for any prop%
  \footnote{Recall that a \emph{prop}\index{prop} is a symmetric strict monoidal category $\mathbb{P}$ where every object is of the form $x^{\otimes n}$ for some generating object $x$, and a $\mathbb{P}$-structure on an object $a$ of some other symmetric monoidal category $\D$ is a symmetric monoidal functor $\mathbb{P}\to \D$ sending $x$ to $a$.  See, for instance, \cite[\S 24]{props}.}
  $\mathbb{P}$: namely, every object is equipped with a $\mathbb{P}$-structure, compatibly with the tensor product.
  (For example, when $\mathbb{P}$ is the prop for special commutative Frobenius algebras\index{Frobenius algebra}, a symmetric monoidal category that supplies $\mathbb{P}$ is called a \defemph{hypergraph category}~\cite{FongSpivak:hypergraph}.)
  Once $\mathbb{P}$ is fixed, we can write a signature for symmetric monoidal categories that supply $\mathbb{P}$, in which the indiscernibilities of objects will be the \defemph{supply isomorphisms}, i.e., the isomorphisms that commute with the $\mathbb{P}$-structures.
  In particular, given any symmetric monoidal category $\D$, if we let $C$ be the type of $\mathbb{P}$-algebras in $\D$ and $f:C\to \ob{\D}$ the forgetful map, then $\D_F$ will be univalent as a symmetric monoidal category that supplies $\mathbb{P}$.
\end{example}

Note that examples of the enhanced categorical structures considered in \cref{sec:restr-indis-egs} are also often obtained by this method.
For instance, the standard $\dagger$-category of Hilbert spaces is $\D_F$, where $\D$ is the univalent category of vector spaces, $C$ is the type of Hilbert spaces, and $F:C\to\ob{\D}$ is the forgetful function.

Another common example of a non-univalent category is a Kleisli category.\index{category!Kleisli}
In set-based category theory the Kleisli category $\C_T$ of a monad $T$ on a category $\C$ can be defined as either:
\begin{enumerate}
\item The category whose objects are those of $\C$ and with $\C_T(x,y)\eqdef \C(x,T y)$.
\item The full subcategory of the Eilenberg--Moore category $\C^T$ on the objects of the form $T x$ for some $x\in \C$.\index{category!Eilenberg-Moore}
\end{enumerate}
In Univalent Foundations, the former yields a precategory that is not univalent, while the latter yields a univalent category (at least if $\C$ is univalent) that is the univalent completion of the former.\footnote{As long as ``for some'' is interpreted with a propositional truncation.  Otherwise, it yields the same non-univalent result as the former definition.}
Indeed, the former definition is an instance of the construction $\D_F$ described above, where $F:\C\to\C^T$ is the free algebra functor.

This non-univalent Kleisli category (but not the univalent one!)\ can be equipped with several kinds of unnatural or nonfunctorial structure, motivated by the theory of programming languages in which the monad $T$ represents ``impure effects'' that can be added to a pure functional programming language.

\begin{example}[Thunk-force categories]\index{category!thunk-force}
 \label{eg:thunk-force}
  A \emph{thunk-force category} or \emph{abstract Kleisli category}~\cite{fuhrmann:abstract-kleisli} is a category $\D$ equipped with
  \begin{itemize}
  \item A functor $L:\D\to\D$,
  \item A natural transformation $\varepsilon : L \to 1_\D$, and
  \item An unnatural transformation $\vartheta : 1_\D \to L$,
  \end{itemize}
  such that $(L,\vartheta L, \varepsilon)$ is a comonad (so that in particular $\vartheta L : L \to L L$ \emph{is} a natural transformation) and each $\vartheta_x : x \to L x$ equips $x$ with the structure of an $L$-coalgebra.\index{comonad}\index{coalgebra}
  If $\D$ is the non-univalent Kleisli category $\C_T$ for a monad $T$, with corresponding adjunction $F : \C \rightleftarrows \D : U$ such that $F$ is the identity on objects, we can give it this structure where $(L,\vartheta L, \varepsilon) = (F U, F \eta U, \varepsilon)$ is the comonad induced by the adjunction and $\vartheta_x : \C_T(x,F U x) = \C(x,TTx)$ is the composite $\eta_{Tx} \circ \eta_x$.

  A signature for thunk-force categories is
  
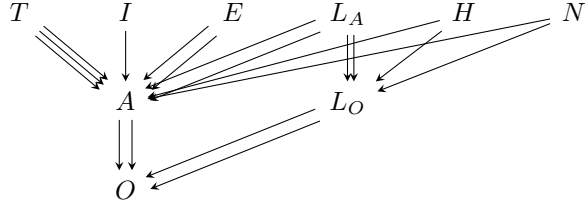
\[
    \begin{tikzcd}
      T \ar[dr] \ar[dr, shift left] \ar[dr, shift right]
      &
      I \ar[d] & E \ar[dl, shift left] \ar[dl, shift right]
      &
      L_A \ar[d, shift left] \ar[d] \ar[dll, shift left] \ar[dll, shift right]
      &
      H \ar[dl] \ar[dlll]
      &
      N \ar[dll] \ar[dllll]
      \\
      &
      A \ar[d, shift left] \ar[d, shift right]
      &
      &
      L_O \ar[dll, shift left] \ar[dll, shift right]
      \\
      &
      O
    \end{tikzcd}
  \]
  where $N$ and $H$ represent $\varepsilon$ and $\vartheta$ respectively.
  The axioms are straightforward to formulate, and as in \cref{eg:unnatural} we find that the indiscernibilities in $O$ are the isomorphisms on which $\vartheta$ is natural.
  In general, morphisms (not necessarily isomorphisms) on which $\vartheta$ is natural (that is, morphisms that are $L$-coalgebra maps) are called \emph{thunkable}.
  As shown in~\cite{fuhrmann:abstract-kleisli} they form a (non-full, but wide) subcategory that can be equipped with a monad whose Kleisli category is the given thunk-force category.
  (Indeed, they are the full subcategory of the Eilenberg-Moore category of the comonad $(L,\vartheta L, \varepsilon)$ on the objects $(x,\vartheta_x)$.)
  In this sense a thunk-force category is precisely ``what is left of a Kleisli category when we forget the underlying category''.

  In a non-univalent Kleisli category $\C_T$, the functor $F:\C\to\C_T$ lands inside the thunkable morphisms.
  Thus, if $\C$ is a univalent category, then $\C_T$ is univalent as a thunk-force category just when every thunkable isomorphism in $\C_T$ is the $F$-image of a unique isomorphism in $\C$.
  This is the case for any monad such that
  \[
    \begin{tikzcd}
      x \ar[r,"\eta"] & T x \ar[r,shift left,"T \eta"] \ar[r,shift right,"\eta T"'] & T T x
    \end{tikzcd}
  \]
  is an equalizer diagram, which happens frequently but not always.
  For instance, the trivial monad on $\Set$ defined by $T x = 1$ admits a thunkable isomorphism $0 \cong 1$ in $\Set_T$, but there is no isomorphism $0\cong 1$ in $\Set$.

  The opposite of a thunk-force structure is called a \emph{runnable monad} (thus a thunk-force structure could also be called a ``corunnable comonad''), and the duals of thunkable morphisms are called \emph{linear}.\index{morphism!linear}\index{monad!runnable}
\end{example}

Thunk-force categories are used to model call-by-value programming languages, while runnable monads are used for call-by-name languages.\index{programming language}\index{call-by-value}\index{call-by-name}
Since real-world programming languages allow functions to take more than one argument, these structures generally need to be enhanced with some kind of product; but in the presence of computational effects this is something weaker than a monoidal structure.

\begin{example}[Premonoidal categories]\label{eg:premonoidal}\index{category!premonoidal}
  A \emph{premonoidal category}~\cite{pr:premonoidal,power:premonoidal} is like a monoidal category, but the tensor product operation is only required to be functorial in each variable separately, rather than jointly.
  That is, for objects $x,y$ we have a tensor product object $x\otimes y$, and for any $f:x\to x'$ we have $f\otimes y : x\otimes y \to x'\otimes y$ and for $g:y\to y'$ we have $x\otimes g : x\otimes y \to x\otimes y'$, but there is no ``$f\otimes g : x\otimes y \to x'\otimes y'$'', and the square
  \[
    \begin{tikzcd}
      x\otimes y \ar[r,"x\otimes g"] \ar[d,"f\otimes y"'] & x\otimes y' \ar[d,"f\otimes y'"] \\
      x'\otimes y \ar[r,"x'\otimes g"'] & x'\otimes y'
    \end{tikzcd}
  \]
  need not commute.
  If for some $f$ this square does commute for all $g$, and a dual condition holds with $f$ on the right, we say that $f$ is \emph{central}.
  The associativity and unit isomorphisms in a premonoidal category are additionally asserted to be central; note that naturality of the associator has to be formulated as three different axioms relative to morphisms in the three possible places.

  One origin of premonoidal categories is bistrong monads on monoidal categories.
  Recall that a \emph{bistrong monad}\index{monad!bistrong} is a monad on a monoidal category equipped with strengths (see, for instance, \cite{monads-kock}) for both the tensor product and the reversed tensor product $a \otimes^{\mathrm{rev}} b \eqdef b \otimes a$.
  If $T$ is a bistrong monad\index{monad!bistrong} on a monoidal category $\C$, then its Kleisli category\index{category!Kleisli} $\C_T$ is premonoidal: its tensor product is that of $\C$, while for $f\otimes y$ is the composite $x\otimes y \xrightarrow{f \otimes y} T x' \otimes y \to T(x'\otimes y)$ with the strength, and dually for $x\otimes g$.

  We can obtain a signature for premonoidal categories by splitting the sort $\otimes_A$ of a monoidal category in two, for the two functors $x\otimes -$ and $-\otimes y$:
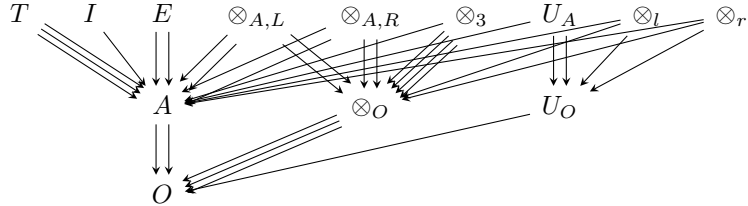
  \[
    \begin{tikzcd}[column sep=small]
      T \ar[drr] \ar[drr, shift left] \ar[drr, shift right]
      &
      I \ar[dr]
      &
      E \ar[d, shift left] \ar[d, shift right]
      &
      \otimes_{A,L} \ar[dr, shift left] \ar[dr, shift right] \ar[dl, shift left] \ar[dl, shift right]
      &
      \otimes_{A,R} \ar[d, shift left] \ar[d, shift right] \ar[dll, shift left] \ar[dll, shift right]
      &
      \otimes_3 \ar[dl, shift left] \ar[dl, shift right] \ar[dl, shift left = 2] \ar[dl] \ar[dlll]
      &
      U_A \ar[d, shift right] \ar[d, shift left] \ar[dllll]
      &
      \otimes_l \ar[dlll] \ar[dl] \ar[dlllll]
      &
      \otimes_r \ar[dllll] \ar[dll] \ar[dllllll]
      \\
      &
      &
      A \ar[d, shift left] \ar[d, shift right]
      &
      &
      \otimes_O \ar[dll, shift left] \ar[dll, shift right] \ar[dll]
      &
      &
      U_O \ar[dllll]
      \\
      &
      &
      O
    \end{tikzcd}
  \]
  That is, for $w:\otimes_O(x,y,z)$ and $w':\otimes_O(x',y,z')$ with $f:A(x,x')$ and $h:A(z,z')$, the relation $\otimes_{A,L}(w,w',f,h)$ says that $f\otimes y=h$ relative to $w$ and $w'$, and similarly for $\otimes_{A,R}$.
  We assert the usual axioms of a premonoidal category, including unique existence of an $h$ as in the previous sentence, which we denote $f\otimes^L_{w,w'} y$; similarly we have $x\otimes^R_{w,w'} g$ for $g:A(y,y')$.
  We define a morphism $f:A(x_1,x_2)$ to be \emph{central}\index{morphism!central} if for any $g:A(y_1,y_2)$ and \emph{any} $w_{ij}:\otimes_O(x_i,y_j,z_{ij})$ (for $i,j\in \{1,2\}$) the following square commutes:
  \[
    \begin{tikzcd}[column sep=huge]
      z_{11} \ar[r,"x_1\otimes^R_{w_{11},w_{12}} g"] \ar[d,"f\otimes^L_{w_{11},w_{21}} y_1"'] & z_{12} \ar[d,"f\otimes^L_{w_{12},w_{22}} y_2"] \\
      z_{21} \ar[r,"x_2\otimes^R_{w_{21},w_{22}} g"'] & z_{22},
    \end{tikzcd}
  \]
  as well as a dual property on the other side.
  However, the naturality of the isomorphisms between any two values of an anafunctor means that it suffices if this holds for \emph{some} $w_{ij}$.
  Recall also that the axioms of a premonoidal category include centrality of the associator and unit isomorphisms.

  We also assert that for any $w:\otimes_O(x,y,z)$ and $w':\otimes_O(x,y,z')$, we have $1_x \otimes^L_{w,w'} y = x \otimes^R _{w,w'} 1_y$.
  In other words, if we have two values of $x\otimes y$, the canonical isomorphisms between them obtained from the two anafunctors $x\otimes -$ and $-\otimes y$ coincide, giving a morphism that we denote $1_x \otimes_{w,w'}  1_y$.
  This is an ``anafunctorial'' version of the standard condition that the two functors are ``equal on objects'', and is necessary for similar reasons to the ``naturality on identity morphisms'' axiom from \cref{eg:unnatural}.
  In particular, it is necessary to prove that \emph{identity morphisms are central}: for any $x:O$ and $g:A(y_1,y_2)$ with $w_j:\otimes_O(x,y_j,z_j)$ the following squares are equal:
  \[
    \begin{tikzcd}[column sep=huge]
      z_{1} \ar[r,"x\otimes^R_{w_{1},w_{2}} g"] \ar[d,"1_x \otimes^L_{w_{1},w_{1}} y_1"'] & z_{2} \ar[d,"1_x \otimes^L_{w_{2},w_{2}} y_2"] \\
      z_{1} \ar[r,"x\otimes^R_{w_{1},w_{2}} g"'] & z_{2}
    \end{tikzcd}
    \qquad
    \begin{tikzcd}[column sep=huge]
      z_{1} \ar[r,"x\otimes^R_{w_{1},w_{2}} g"] \ar[d,"x \otimes^R_{w_{1},w_{1}} 1_{y_1}"'] & z_{2} \ar[d,"x \otimes^R_{w_{2},w_{2}} 1_{y_2}"] \\
      z_{1} \ar[r,"x\otimes^R_{w_{1},w_{2}} g"'] & z_{2}
    \end{tikzcd}
  \]
  and the right-hand square commutes by functoriality of $\otimes^R$.

  We can also show that if $g:A(y_1,y_2)$ is central, then so is any $x\otimes_{w_1,w_2}^R g$.
  For if we have $h:A(u_1,u_2)$ with appropriate witnesses of the tensor product, we can form the following diagram:
  \[
    \begin{tikzcd}[row sep=huge, column sep=huge]
      {} \ar[rrr,"(x\otimes y_1) \otimes^R h"] \ar[ddd,"(x\otimes^R g) \otimes^L u_1"'] \ar[dr] & & &
      {} \ar[ddd,"(x\otimes^R g) \otimes^L u_2"] \ar[dl] \\
      & {} \ar[r,"x\otimes ^R (y_1\otimes^R h)"] \ar[d,"x\otimes^R (g \otimes^L u_1)" description] &
      {} \ar[d,"x\otimes^R (g\otimes^L u_2)" description] \\
      & {} \ar[r,"x\otimes^R (y_2\otimes^R h)"'] & {} \\
      {} \ar[ur] \ar[rrr,"(x\otimes y_2) \otimes^R h"'] & & & {} \ar[ul]
    \end{tikzcd}
  \]
  Here the inner square commutes by centrality of $g$ and functoriality of $x\otimes^R -$, while the diagonal arrows are components of the associativity isomorphism and the trapezoids commute by naturality.
  Therefore, the outer square commutes.
  Together with a similar argument on the other side, this implies that $x\otimes_{w_1,w_2}^R g$ is central when $g$ is.
  Similarly, $f\otimes^L y$ is central as soon as $f$ is.

  In particular, it follows that the isomorphism $1_x \otimes_{w,w'} 1_y$ between any two values of $x\otimes y$ is central.
  Therefore, the existential saturation condition must be similarly restricted: it asserts that given $w:\otimes_O(x,y,z)$ and a \emph{central} $h:z\cong z'$, there exists a $w':\otimes_O(x,y,z')$ such that $1_x \otimes_{w,w'} 1_y = h$.

  As usual, $E$ is required to be a congruence for all rank-2 relations, so that univalence at $A$ means it is a set with $E$ as equality, and univalence at $U_O$ means it is a saturated ana-object.
  Now consider $w_1,w_2:\otimes_O(x,y,z)$; the indiscernibility type $w_1\fiso w_2$ is the proposition that $w_1$ and $w_2$ act the same on all arrows on both sides (the dependency of the natural transformations is automatically transportable by naturality, as in \cref{eg:nat-trans}).
  In other words, it says that $g \otimes^L_{w_1,w'} y = g\otimes^L_{w_2,w'} y$ for any $w':\otimes_O(x',y,z')$ and $g:A(x,x')$, and similarly on the other side.
  As for ordinary anafunctors, by functoriality this is equivalent to its special case $1_x \otimes_{w_1,w_2} 1_y = 1_z$.
  Thus, univalence at $\otimes_O$ means that if $1_x \otimes_{w_1,w_2} 1_y = 1_z$ then $w_1=w_2$, hence that the $w'$ asserted to exist in the existential saturation axiom is unique.

  Finally, an indiscernibility $x_1 \fiso x_2$ in $O$ consists of an isomorphism $\phi:x_1\cong x_2$ together with equivalences such as $\phi_{\bullet y} : \otimes_O(x_1,y,z) \simeq \otimes_O(x_2,y,z)$ and so on for the other holes, which respect all the rank-2 relations.
  Respect for $\otimes_{A,L}$ implies in particular that for $w:\otimes_O(x_1,y,z)$ we have $\phi \otimes^L_{w,\phi_{\bullet y}(w)} y = 1_{x_2} \otimes^L_{\phi_{\bullet y}(w),\phi_{\bullet y}(w)} 1_y = 1_{z}$.
  Now respect for $\otimes_{A,R}$ implies that for any $w_1:\otimes_O(x_1,y_1,z_1)$ and $w_2:\otimes_O(x_1,y_2,z_2)$ with $g:A(y_1,y_2)$ we have \[\otimes_{A,R}(w_1,w_2,g,h) \leftrightarrow \otimes_{A,R}(\phi_{\bullet y_1}(w_1),\phi_{\bullet y_2}(w_2),g,h),\] or equivalently $x_1 \otimes^L_{w_1,w_2} g = x_2 \otimes^R_{\phi_{\bullet y_1}(w_1),\phi_{\bullet y_2}(w_2)} g$.
  But since $\phi\otimes^L_{w_1,\phi_{\bullet y_1}(w_1)} y_1 = 1_{z_1}$ and $\phi\otimes^L_{w_2,\phi_{\bullet y_2}(w_2)} y_2 = 1_{z_2}$, this implies that the following square commutes, since its vertical arrows are identities and its horizontal arrows are equal:
  \[
    \begin{tikzcd}[row sep=large,column sep=huge]
      {} \ar[d,"\phi\otimes^L_{w_1,\phi_{\bullet y_1}(w_1)} y_1"'] \ar[r,"x_1 \otimes^L_{w_1,w_2} g"] &
      {} \ar[d,"\phi\otimes^L_{w_2,\phi_{\bullet y_2}(w_2)} y_2"] \\
      {} \ar[r,"x_2 \otimes^R_{\phi_{\bullet y_1}(w_1),\phi_{\bullet y_2}(w_2)} g"'] & {}.
    \end{tikzcd}
  \]
  Together with a similar argument on the other side, this implies that $\phi:x_1\cong x_2$ is necessarily central.

  From here the usual sort of arguments imply that the rest of the structure of an indiscernibility (such as the equivalences $\phi_{\bullet y}$ used above) is uniquely determined by saturation applied to $\phi$.
  This is perhaps least obvious in the case of the equivalences $\phi_{\bullet\bullet} : \otimes_O(x_1,x_1,z) \simeq \otimes_O(x_2,x_2,z)$, since $\otimes$ is not jointly functorial in its arguments.
  But once we have shown that $\phi_{x_1\bullet}$ and $\phi_{\bullet x_2}$ are uniquely determined, respect for $\otimes_{A,R}$ tells us that for any $w:\otimes_O(x_1,x_1,z)$ we have $\otimes_{A,R}(w,\phi_{x_1\bullet}(w),\phi,h) \leftrightarrow \otimes_{A,R}(\phi_{\bullet\bullet}(w),\phi_{\bullet x_2}(\phi_{x_1\bullet}(w)),1_{x_2},h)$, which uniquely determines $\phi_{\bullet\bullet}(w)$ by saturation.

  Thus $x_1 \fiso x_2$ is equivalent to the type of central isomorphisms $x_1\cong x_2$, and so in a univalent premonoidal category $x_1=x_2$ is also equivalent to this type.
  In particular, since as we noted above the values of the tensor product are also determined uniquely up to unique central isomorphism, the type of such values is contractible, so we obtain an actual function $\otimes : O\to O \to O$ as we would hope.

  Similarly to the situation for thunk-force categories, a Kleisli category $\C_T$ is univalent as a premonoidal category just when every central isomorphism in $\C_T$ is the image of a unique isomorphism in $\C$.
  This fails, for instance, when $T$ is a commutative monad, in which case every morphism is central but not every isomorphism of free algebras is in the image of the free functor (e.g., the nontrivial automorphism of the free abelian group on one generator).

  The situation for \emph{morphisms} of premonoidal categories is rather subtle, even classically.
  In~\cite{pr:premonoidal}, a \emph{premonoidal functor} is defined to be a functor that preserves centrality of morphisms and preserves the tensor product and unit object up to coherent central natural isomorphisms.
  However, there are also examples that one might like to call ``premonoidal functors'' but that do not preserve centrality of morphisms or even isomorphisms.
  For instance, any morphism of bistrong monads $T_1 \to T_2$ on a monoidal category $\C$ induces a functor $\C_{T_1} \to \C_{T_2}$ that preserves the tensor product \emph{strictly}, but need not preserve centrality of isomorphisms; a counterexample can be found in~\cite[Section 5.2]{sl:univprop-impure}.
  On the other hand, simply removing the preservation of centrality from the definition of premonoidal functor yields a notion that is not closed under composition.
  (We thank Paul Blain Levy for pointing out these subtleties.)

  Our morphisms of structures are, of course, always closed under composition.
  Between \emph{univalent} models of our theory of premonoidal categories, the structure morphisms are precisely the premonoidal functors of~\cite{pr:premonoidal}.
  But to understand the morphisms between non-univalent models, we have to pay more careful attention to how the operation $\otimes : O \to O\to O$ is made into the ``ana-function'' $\otimes_O$.

  The most obvious choice is to define a witness $w: \otimes_O(x,y,z)$ to be a \emph{central} isomorphism $x\otimes y \cong z$.
  This ensures that the existential saturation condition holds, and if the underlying $\LcatE$-structure is a $\Tcat$-precategory, the resulting structure is univalent at all sorts of rank $>0$; but it will not generally be univalent at $O$.
  The morphisms between structures of this kind are the functors of precategories that are premonoidal in the sense of~\cite{pr:premonoidal}.
  Thus, this approach includes the non-univalent Kleisli categories $\C_T$, but does not include all functors of the kind mentioned above.

  On the other hand, we could define a witness $w: \otimes_O(x,y,z)$ to be an \emph{identification} $x\otimes y = z$.
  If the underlying precategory is a strict category, then these are ``strict premonoidal functors'', but in general they need not be very strict; e.g., if $\C$ is univalent, then the identifications of objects in $\C_T$ are the isomorphisms in $\C$.
  Now there can be morphisms between non-univalent models of this kind that do not preserve centrality; e.g., every functor $\C_{T_1} \to \C_{T_2}$ induced by a morphism of bistrong monads does induce a morphism between models of this kind.
  However, these structures do not in general satisfy the existential saturation axiom.

  Note that this is a ``real-world'' example of the situation observed in \cref{eg:indis-notpres} that morphisms of structures need not preserve indiscernibility.
  In particular, if there is a ``univalent completion'' operation for premonoidal categories (with existential saturation omitted), then there will be morphisms between non-univalent structures that do not extend to their univalent completions.
\end{example}

Thus the Kleisli category of a bistrong monad is both a thunk-force category and a premonoidal category, and the two structures are not unrelated.
For instance, every thunkable morphism is also central; in the cartesian case this is~\cite[Proposition 2.20]{fuhrmann:abstract-kleisli}, while the general case can be found at~\cite{levy:thunk-cent}.
This makes the Kleisli category of a bistrong monad on a cartesian monoidal category into what is called a \emph{precartesian abstract Kleisli category}\index{category!precartesian abstract Kleisli} in~\cite{fuhrmann:abstract-kleisli}: a thunk-force category that is also premonoidal in which every thunkable morphism is central and the monoidal structure restricts to a cartesian monoidal structure on the thunkable morphisms.
We leave it to the reader to write down a signature for such things and check that its indiscernibilities of objects are the thunkable isomorphisms.

Note that although every morphism in the original category yields a thunkable (hence central) morphism in the Kleisli category, this operation may not be faithful.
If we remember the actual morphisms in the original category as extra data, we obtain a \emph{Freyd-category}~\cite{pt:freyd-cats}: a category $V$ with finite products, a symmetric premonoidal category $C$ with the same objects as $V$, and an identity-on-objects strict symmetric premonoidal functor $J:V\to C$ that lands in the center of $C$.\index{category!Freyd-}
We can write down a signature for Freyd-categories by combining \cref{eg:m-cat} (for the identity-on-objects functor, with faithfulness omitted) with \cref{eg:premonoidal} (for the premonoidal structure on $C$), and an enhancement of \cref{eg:cat-w-binprod-functor} (for the finite products on $V$).
Unsurprisingly, the indiscernibilities in a Freyd-category are just isomorphisms in $V$.

Levy \cite{DBLP:conf/popl/Levy17} develops a notion similar to our indiscernibilities, there called ``contextual isomorphism'', for the study of isomorphism of types in some simply-typed $\lambda$-calculi with effects.\index{isomorphism!contextual}
Given types $A,B$, a contextual isomorphism $A \cong B$ consists, very roughly, of a family of bijections
\[\theta_{\Gamma \vdash C} : Q((\Gamma \vdash C)[A]) \cong Q((\Gamma \vdash C)[B])\]
of (equivalence classes of) well-formed $\lambda$-terms, respectively, for each judgment $\Gamma \vdash C$ with a type-hole, filled with $A$ and $B$, respectively. 
Levy analyzes a particular $\lambda$-calculus called ``call-by-push-value'' with two kinds of types, value types and computation types, and consequently with two different kinds of judgments $\vdash^v$ (value judgment) and $\vdash^c$ (computation judgment), and with denotational (categorical) semantics in something akin to a Freyd-category (cf.\ \cite[\S5.1]{DBLP:conf/popl/Levy17}).
However, the contextual isomorphisms are defined, in \cite[\S5.2]{DBLP:conf/popl/Levy17}, via quantification over the  computation judgments only.  Thus, the end result is more akin to our \cref{eg:thunk-force}, and indeed
Levy finds that these ``partial'' contextual isomorphisms are precisely the thunkable isomorphisms \cite[\S7.2]{DBLP:conf/popl/Levy17}.

\begin{example}[Duploids]\index{duploid}
 \label{eg:duploid}
  A \emph{duploid}~\cite{munchmaccagnoni:thesis} is a structure that combines call-by-value structure (such as in a thunk-force category) and call-by-name structure (such as in its dual, a runnable monad) in one.
  It starts with a \emph{pre-duploid}\index{duploid!pre-}, which is almost like a category equipped with a map to the chaotic category on two objects $\{+,-\}$, except that the associativity law $(h \circ g) \circ f = h\circ (g\circ f)$ need only hold if either the codomain of $f$ (i.e., the domain of $g$) lies over $-$ (``is negative'') or the codomain of $g$ (i.e., the domain of $h$) lies over $+$ (``is positive'').
  A signature for pre-duploids is as follows:
  
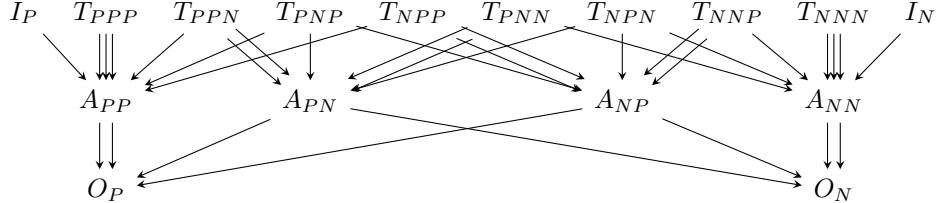
\[
    \begin{tikzcd}[column sep=.45em]
      I_P \ar[dr] &
      T_{PPP} \ar[d] \ar[d,shift left] \ar[d,shift right] &
      T_{PPN} \ar[dl] \ar[dr,shift left] \ar[dr,shift right] &
      T_{PNP} \ar[dll] \ar[d] \ar[drrr] &
      T_{NPP} \ar[dlll] \ar[drr,shift left] \ar[drr,shift right] &
      T_{PNN} \ar[dll,shift left] \ar[dll,shift right] \ar[drrr] &
      T_{NPN} \ar[dlll] \ar[d] \ar[drr] &
      T_{NNP} \ar[dl,shift left] \ar[dl,shift right] \ar[dr] &
      T_{NNN}  \ar[d] \ar[d,shift left] \ar[d,shift right] &
      I_N \ar[dl] \\
      & A_{PP} \ar[d,shift left] \ar[d,shift right] &&
      A_{PN} \ar[dll] \ar[drrrrr] &&&
      A_{NP} \ar[dlllll] \ar[drr] &&
      A_{NN} \ar[d,shift left] \ar[d,shift right] \\
      & O_P &&&&&&& O_N
    \end{tikzcd}
  \]
  plus equality congruences on all four sorts $A_{\bullet\bullet}$ that we have omitted to write, which thus coincide with the indiscernibilities on those sorts.

  Since the positive objects form a category in their own right, an indiscernibility $p_1\fiso p_2$ between $p_1,p_2:O_P$ consists in particular of an isomorphism $\phi:p_1\cong p_2$ in that category, together with equivalences $\phi_{\bullet n} : A_{PN}(p_1,n) \simeq A_{PN}(p_2,n)$ and $\phi_{n\bullet} : A_{NP}(n,p_1) \simeq A_{NP}(n,p_2)$ respecting composition of all sorts.
  The usual arguments imply that $\phi_{\bullet n}$ and $\phi_{n\bullet}$ are given by composition with $\phi$ or its inverse, so it remains to consider respect for composition.
  This includes, for instance, $T_{PNN}(g\circ \phi,h,k\circ \phi) \leftrightarrow T_{PNN}(g,h,k)$, which is to say that $h\circ (g\circ \phi) = (h\circ g) \circ \phi$ for all $g:A_{PN}(p_2,n_1)$ and $h:A_{NN}(n_1,n_2)$; and similarly for $T_{PNP}$.
  That is, the associativity law that isn't generally asserted in a pre-duploid does hold when $\phi$ is the first morphism.
  In the context of a pre-duploid, this is taken as the definition of when a morphism is \emph{thunkable}.\index{morphism!thunkable}
  The remaining conditions are automatic, so the indiscernibilities between positive objects are precisely the thunkable isomorphisms.
  Dually, the indiscernibilities between negative objects are precisely the \emph{linear} isomorphisms: those for which the missing associativities hold when they are the last morphism in the triple composite.

  A duploid is a pre-duploid together with ``parity shift'' functions $\Uparrow$ taking positive objects to negative ones and $\Downarrow$ taking negative objects to positive ones, together with unnatural families of linear isomorphisms $\mathsf{force}: \mathord{\Uparrow} p \cong p$, for positive $p$, and thunkable isomorphisms $\mathsf{wrap}:n \cong \mathord{\Downarrow} n$, for negative $n$.
  However, it turns out that $\Uparrow$ and $\Downarrow$ can in fact be made into functors, and $\mathsf{force}$ and $\mathsf{wrap}$ natural.
  Thus, this additional structure does not change the notions of indiscernibility or univalence.
\end{example}

\begin{example}[Factorization systems]\index{factorization system}
  A \emph{factorization system} on a category $\C$ consists of two classes of morphisms $L$ and $R$ satisfying certain axioms.
  For a \emph{weak} factorization system\index{factorization system!weak}, these axioms are that every morphism of $\C$ factors as an $L$-map followed by an $R$-map, that $L$ and $R$ are closed under retracts, and that any commutative square
  \begin{equation}
    \begin{tikzcd}
      x \ar[d,"\ell"'] \ar[r,"g"] & u \ar[d,"r"] \\
      y \ar[r,"h"'] \ar[ur,dotted,"s"] & v,
    \end{tikzcd}\label{eq:lift}
  \end{equation}
  with $\ell\in L$ and $r\in R$, has a diagonal filler $s$ as shown.
  For an \emph{orthogonal} or \emph{unique} factorization system, one requires that such diagonal fillers are unique, or equivalently that factorizations are unique up to unique isomorphism.\index{factorization system!orthogonal}\index{factorization system!unique}

  One natural signature for a factorization system (of either sort) is a slight generalization of \cref{eg:m-cat}, with two predicates $L$ and $R$ instead of just the one $M$.
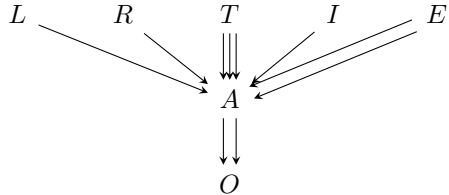
  \[
    \begin{tikzcd}
      L \ar[drr]
      &
      R \ar[dr]
      &
      T \ar[d] \ar[d, shift left] \ar[d, shift right]
      &
      I \ar[dl]
      &
      E \ar[dll, shift left] \ar[dll, shift right]
      \\
      &
      &
      A \ar[d, shift left] \ar[d, shift right]
      \\
      &
      &
      O
    \end{tikzcd}
  \]
  Building on \cref{eg:m-cat} (and the fact that in a weak factorization system, both $L$ and $R$ contain all isomorphisms), an indiscernibility between two objects $a$ and $b$ of a category with a weak factorization system will then be simply an isomorphism between $a$ and $b$.
  Similarly, a morphism between such structures is just a functor that preserves $L$ and $R$, and an equivalence of structures is an equivalence of the underlying categories that preserves and reflects $L$ and $R$.

  A weak factorization system is \emph{functorial} if there is a specified functor $\C^{\mathbf{2}} \to \C^{\mathbf{3}}$ factoring each morphism as an $L$-map followed by an $R$-map.\index{factorization system!functorial weak}
  (Here $\C^{\mathbf{2}}$ is the category whose objects are morphisms in $\C$ and whose morphisms are commutative squares, and similarly $\C^{\mathbf{3}}$ is the category whose objects are composable pairs of morphisms in $\C$.)
  An orthogonal factorization system is automatically and essentially-uniquely functorial, but a weak factorization system may not be.
  We can write down a signature for a functorial weak factorization system as follows:
  
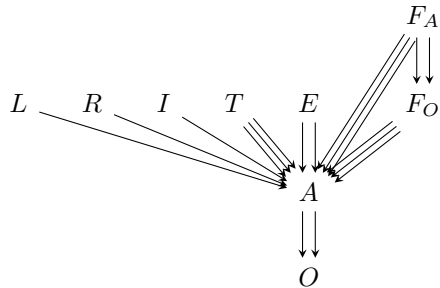
\[
    \begin{tikzcd}[column sep=small]
      &&&&&& F_A \ar[d,shift left] \ar[d,shift right] \ar[ddll] \ar[ddll,shift left] \ar[ddll,shift right] \\
      L \ar[drrrr] & R \ar[drrr] & I \ar[drr] & T \ar[dr,shift left] \ar[dr] \ar[dr,shift right] & E \ar[d,shift left] \ar[d,shift right] &
      & F_O \ar[dll,shift left] \ar[dll] \ar[dll,shift right] \\
      &&&& A \ar[d,shift left] \ar[d,shift right] \\
      &&&& O
    \end{tikzcd}
  \]
  Here for $f:A(x,y)$, the elements of $(F_O)_{x,y,z}(f,\ell,r)$ are witnesses that the functorial factorization of $f$ is $(\ell,r)$, where $\ell:A(x,z)$ and $r:A(z,y)$.
  We assert as an axiom that $E(f,r\circ \ell)$, where the composition equations in the signature ensure that this is well-typed.
  Similarly, for $w:F_O(f,\ell,r)$ and $w':F_O(f',\ell',r')$, $F_A(w,w',g,h,k)$ asserts that, assuming $E(f'\circ g,h\circ f)$, the image of this commutative square under the functorial factorization is $k$.
  Functoriality is straightforward to ensure with axioms; note we only assert that there is a $k$ satisfying $F_A(w,w',g,h,k)$ under the assumption that $E(f'\circ g,h\circ f)$.
  Since this $k$ is unique, we denote it by $F_{w,w'}(f,f',g,h)$.
  We also assert existential saturation: if $w:(F_O)_{x,y,z}(f,\ell,r)$ and $k:z\cong z'$, there exists $w':(F_O)_{x,y,z'}(f,k\circ \ell,r\circ k^{-1})$ such that $F_{w,w'}(f,f,1_x,1_y) = k$.
  
  As usual, univalence at $F_A$ makes it a proposition, and univalence at $F_O$ makes it a set such that the $w'$ in existential saturation is unique.
  We may worry about whether univalence at $A$ would disrupt its equality and even its h-level, since $A$ now has two ranks dependent on it.
  However, the functoriality of $F$ on identity squares implies that $E$ is also a ``congruence'' for $F_O$ in an appropriate sense: if $E(f,f')$ and we have $w:(F_O)_{x,y,z}(f,\ell,r)$ and $w':(F_O)_{x,y,z'}(f',\ell',r')$, then we have $F_{w,w'}(f,f',1_x,1_y):z\cong z'$, and by saturation and functoriality there is a uniquely determined $w':(F_O)_{x,y,z}(f',\ell,r)$.
  Thus, univalence at $A$ again simply ensures that it is a set with equality $E$, and similarly univalence at $O$ makes the underlying category univalent in the ordinary sense.
  A morphism of such structures is a functor that preserves both $L$ and $R$ as well as the functorial factorizations, up to coherent isomorphism.
  
  Things get more interesting if we consider weak factorization systems with a \emph{specified} but \emph{non-functorial} factorization.
  (Most weak factorization systems arising in practice are functorial, but some such as~\cite{isaksen:pross} are not; whereas recent work on constructive homotopy theory such as~\cite{henry:weak-modelcats} has found at least specified factorizations to be indispensable.)
  As we did for unnatural transformations, it is natural to assert at least \emph{functoriality on identities}, with a signature such as
  
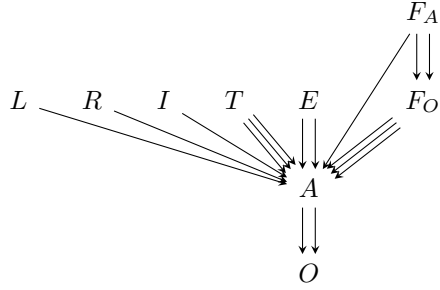
\[
    \begin{tikzcd}[column sep=small]
      &&&&&& F_A \ar[d,shift left] \ar[d,shift right] \ar[ddll] \\
      L \ar[drrrr] & R \ar[drrr] & I \ar[drr] & T \ar[dr,shift left] \ar[dr] \ar[dr,shift right] & E \ar[d,shift left] \ar[d,shift right] &
      & F_O \ar[dll,shift left] \ar[dll] \ar[dll,shift right] \\
      &&&& A \ar[d,shift left] \ar[d,shift right] \\
      &&&& O
    \end{tikzcd}
  \]
  in which $F_A(w,w',k)$, for $w:(F_O)_{x,y,z}(f,\ell,r)$ and $w':(F_O)_{x,y,z'}(f',\ell',r')$, asserts that assuming $E(f,f')$ then $k:z\cong z'$ is the image of that equality under the factorization.
  This suffices for the above analysis of saturation and univalence at all sorts except $O$.
  But at $O$, an indiscernibility $a\fiso b$ will now be an isomorphism $\phi:a\cong b$ together with ``all possible functorial actions of the factorization on $\phi$''.
  If we assumed that the factorization were functorial on all \emph{isomorphisms}, then this would reduce to simply an ordinary isomorphism in the underlying category; but in general this need not be the case.
  Hence, in particular, univalence of a ``category with weak factorization system'' in this sense is different from univalence of its underlying category.

  As in the case of thunk-force and premonoidal categories, to see nontrivial examples of this we need a construction that produces non-univalent categories.
  A naturally-occurring example in this case is the category of \emph{pro-objects} in a category $\C$.\index{pro-object}
  As with Kleisli categories, this has two natural definitions, one of which is naturally univalent and the other of which is not:
  \begin{enumerate}
  \item The objects of $\pro\C$ are functors $X:I\to \C$, where $I$ is a small cofiltered category.
    Its hom-sets are
    \[ \pro\C(X,Y) = \lim_{i,j} \C(X_i,Y_j). \]
  \item $\pro\C$ is the full subcategory of the presheaf category of $[\C,\Set]^{\mathsf{op}}$ on the objects that are cofiltered limits of representables.
  \end{enumerate}
  As before, the latter is always univalent, and if $\C$ is univalent then the latter is the Rezk completion\index{Rezk completion} of the former.\footnote{In view of \cref{fn:nonunivalent}, it is unsurprising that the former can be obtained as $\D_F$, where $\D = [\C,\Set]^{\mathsf{op}}$ and $F:(\sum_I (I\to \C))\to\ob{\D}$ takes a functor $X:I\to \C$ from a small cofiltered category, composes it with the Yoneda embedding, and then takes its limit.}
  The former is not univalent; an identification $X=Y$ therein is an isomorphism $I\cong J$ making a triangle that commutes up to isomorphism.
  This is called a \emph{level isomorphism}; more generally a \emph{level map} $X\to Y$ is an isomorphism $I\cong J$ with a natural transformation inhabiting the triangle.\index{isomorphism!of pro-objects!level}\index{morphism!of pro-objects!level}

  For example, in~\cite{isaksen:pross} it is shown that when $\C$ is the category of simplicial sets, $\pro\C$ supports a Quillen model structure, which includes two weak factorization systems.
  The construction of the factorizations of $f:X\to Y$ proceeds by first constructing a level map $f':X'\to Y'$ and non-level isomorphisms $X\cong X'$ and $Y\cong Y'$ such that the composite $X\cong X' \xrightarrow{f'} Y'\cong Y$ is $f$, and then factoring $f'$ levelwise.
  This can all be done in a specified way, but not (at least not obviously) in a way that respects non-level isomorphisms of pro-objects.
  Thus, this gives an example of a structure for the above signature in which the underlying category is not univalent.
  (We have not analyzed whether this structure is univalent at $O$, i.e., whether the level isomorphisms coincide with the indiscernibilities.)

  We can also regard the lifts in a weak factorization as structure.
  If we also relegate the factorization back to a property (for simplicity), this yields the following the signature.
  \[
    \begin{tikzcd}
      L \ar[drrr]
      &
      R \ar[drr]
      &
      S \ar[dr] \ar[dr, shift left] \ar[dr, shift left = 2] \ar[dr, shift right] \ar[dr, shift right = 2]
      &
      T \ar[d] \ar[d, shift left] \ar[d, shift right]
      &
      I \ar[dl]
      &
      E \ar[dll, shift left] \ar[dll, shift right]
      \\
      &
      &
      &
      A \ar[d, shift left] \ar[d, shift right]
      \\
      &
      &
      &
      O
 \end{tikzcd}
\]
  Here $S(\ell,g,h,r,s)$ says that $s$ is the chosen lift in the square~\eqref{eq:lift}.
  Now an indiscernibility between two objects $a,b:O$ consists of an isomorphism $\phi: a \cong b$ with the additional requirements that
  \begin{align*}
    S(\ell,y,x,r,s) &\leftrightarrow S(\ell \circ \phi,y,x \circ \phi,r,s)\\
    S(\ell,y,x,r,s) &\leftrightarrow S(\phi^{-1} \circ \ell ,y \circ \phi,x ,r,s \circ \phi)\\
    S(\ell,y,x,r,s) &\leftrightarrow S( \ell ,y , \phi \circ x ,r \circ \phi^{-1},  \phi \circ s )\\
    S(\ell,y,x,r,s) &\leftrightarrow S(\ell ,\phi \circ y ,x ,\phi \circ r,s)
  \end{align*}
  whenever these compositions exist.

  These requirements hold for any isomorphism if $S$ is compositional in the sense that
  \begin{align*}
    S(\ell_2, y, t, r, s) \times S(\ell_1,y\circ \ell_2, x, r, t) &\to S(\ell_2 \circ \ell_1 ,y, x, r, s)\qquad\text{and}\\
    S(\ell, y, r_1 \circ s, r_2, t) \times S(\ell, t, x, r_1, s) &\to S(\ell, t, x, r_2 \circ r_1, s).
  \end{align*}
  Orthogonal factorization systems are compositional (since lifts are unique), but for a weak factorization system it is rarely possible to choose a compositional lifting function.
  Instead, when regarding lifts as structure it is probably better to consider \emph{algebraic} weak factorization systems\index{factorization system!algebraic weak} (see, e.g.,~\cite{gt:nwfs,garner:soa,riehl:nwfs-model}), in which the predicates $L$ and $R$ are made into structure as well.
  We leave it to the reader to write down a signature for a category with an algebraic weak factorization system.
\end{example}

\part{Theory of functorial structures}\label{sec:ho}

In this \lcnamecref{sec:ho} we will give precise definitions and proofs of the general theorems we claimed in \cref{sec:general-theory}.
In addition, rather than working with diagram signatures, we will introduce, in \cref{sec:abstr-sign-transl}, a more general notion that we call a \emph{functorial signature},
equipped with a suitable notion of \emph{structure} for such signatures.
We call such structures ``functorial structures'' in this introduction, to distinguish them from structures for diagram signatures.

To link the two definitions of signature, in \cref{thm:translation} we construct a translation from the diagram signatures of \cref{sec:general-theory} to functorial signatures.
The diagram structures for a diagram signature are equivalent to the functorial structures for its induced functorial signature, so diagram signatures and structures are subsumed by their functorial counterparts.

Given this generalization, we then proceed to generalize the definitions and results from \cref{sec:general-theory}, leading us to a fully general statement and proof of our univalence principle.
We also discuss some examples of functorial theories.

Specifically, in \cref{sec:structures}, we introduce and study a notion of \emph{levelwise equivalence} for functorial structures, generalizing the levelwise equivalences for diagram structures from \cref{def:lvleqv-folds}.

In \cref{sec:FOLDS-iso-uni} we define indiscernibility and univalence for functorial structures,
thus generalizing the respective definitions of~\cref{def:indisc-in-diag-structure} and \cref{def:univalence-cond-on-diag-structures}.

In \cref{sec:hsip} we define equivalence of functorial structures, defined for diagram structures in \cref{def:eqv-folds}.
We then prove our main result in \cref{thm:hsip2}, which is a generalization to functorial signatures of the univalence principle stated for diagram signatures as \cref{thm:hsip-folds}.

Finally, since all the examples in \Cref{sec:examples} were diagram signatures, in \cref{sec:egs-higherorder} we discuss some examples of functorial signatures whose structures cannot, to our knowledge, be specified by a diagram signature.
Such examples include ``higher-order'' structures which involve quantification over subsets of a carrier set, such as topological spaces and suplattices.

\chapter{Functorial signatures}
\label{sec:abstr-sign-transl}

In this \lcnamecref{sec:abstr-sign-transl} we introduce another notion of signature for mathematical structures, called ``functorial signatures''.

Functorial signatures seem to be more general than diagram signatures;
some examples of functorial signatures whose structures cannot, to our knowledge, be defined by diagram signatures are presented in \cref{sec:egs-higherorder}.
However, our main reason for generalizing from diagram signatures to functorial ones is that the latter are easier to reason about in the abstract because of their inductive nature.
In particular, using functorial signatures simplifies the proof of our main result, \cref{thm:hsip2}.

Diagram signatures are nevertheless a very convenient way of specifying examples.
For this reason, we provide a translation from diagram signatures to functorial signatures, making our abstract results applicable to a wide range of examples, including all those presented in \cref{sec:examples}.

To motivate the notion of functorial signature, we start in \cref{sec:funct-deriv} by further analyzing the notion of derivative of a diagram signature.
In particular, we will prove our claim in \cref{sec:deriv-folds} that an exo-functor $\L\to\Ustrict$ is uniquely determined, up to isomorphism, by an $M:\L(0) \to \Ustrict$ and an exo-functor $\derivcat{\L}{M} \to \Ustrict$.
This requires investigating the functoriality of derivatives.

In \cref{sec:abstract-signatures} we define functorial signatures and their structures.
We do not only define the (exo)types of these things, but, at the same time, morphisms of such things and further categorical structure required, such as pullback of structures along morphisms of signatures.
From the analysis of diagram signatures in \cref{sec:funct-deriv}, we immediately obtain a translation from diagram signatures to functorial signatures, in \cref{thm:translation}.

The notions of axiom and theory for functorial signatures, defined in \cref{sec:ho-axioms}, are essentially copied over from the corresponding notions for diagram signatures given in \cref{sec:axioms}.

\section{Functoriality of derivatives}
\label{sec:funct-deriv}

Morphisms of diagram signatures will be exo-functors that preserve ranks strict\-ly and also ``preserve the dependency structure'', in the following sense.
Note that if $F:\L\to\M$ preserves ranks, then it induces a function $\fanoutfun{F}{K}{m} : \fanout{K}{m} \to  \fanout{FK}{m} $ for every $K: \L(n)$ and $m<n$. 
We sometimes write $F$ instead of $\fanoutfun{F}{K}{m}$ when no confusion can arise.

\begin{definition}
  Let $F:\L\to \M$ be a rank-preserving exo-functor between inverse exo-categories.
  It is a \defemph{discrete opfibration}\index{opfibration!discrete} when all the functions \[\fanoutfun{F}{K}{m} : \fanout{K}{m} \to  \fanout{FK}{m} \] are isomorphisms.
  Let $\hom_{\IC(p)}(\L,\M)$ denote the exotype of such discrete opfibrations.
\end{definition}

\begin{proposition}\label{prop:material-sigs-cat}
  The exotype $\IC(p)$ and hom-exotypes $\hom_{\IC(p)}(\L,\M)$ form an exo-category.
\end{proposition}
\begin{proof}
  Given $F:\hom_{\IC(p)}(\L,\M)$ and $G:\hom_{\IC(p)}(\M,\N)$, their composite is a discrete opfibration since for every $n: \exo{\Nat}_{< p}$ and $K: \L(n)$, \[G \circ F: \fanout{K}{n} \to  \fanout{GFK}{n}\] is the composition of the isomorphisms $F : \fanout{K}{n} \to  \fanout{FK}{n}$ and $G : \fanout{FK}{n} \to  \fanout{GFK}{n}$.
  Similarly, for any $\L \in \IC(p)$, the identity exo-functor $1_\L: \hom_{\IC(p)}(\L,\L)$ is a discrete opfibration.
  This composition is clearly associative and unital.
\end{proof}

\begin{proposition}\label{prop:dfib-matching}
  If $F:\L\to\M$ is a discrete opfibration, then for any exo-functor $M:\M \to\Ustrict$ and $K\in\L(m)$, we have an isomorphism
  \[ \match_K (M\circ F) \cong \match_{FK} M .\]
  In particular, if $M$ is Reedy fibrant, so is $M\circ F$.
\end{proposition}
\begin{proof}
  \cref{def:matching} is phrased mostly in terms of fanout exotypes, and the discrete opfibration condition also ensures that the morphism $g$ in $\L$ is also uniquely determined by its image $F(g)$ in $\M$.
  Thus, the definitions of both sides can be identified.
\end{proof}

In the following definition, for present purposes it would suffice to fix the signature $\L$ and let the structure $M$ vary; but for later use we allow the signature to vary as well.

\begin{definition}
  \label{def:derivation_functorial_action}\index{derivative!functoriality of}
  Let $\L$ and $\M$ be inverse exo-categories of height $p>0$, let $H: \L \to \M$ be a discrete opfibration, let $L: \L(0) \to \Ustrict$ and $M: \M(0) \to \Ustrict$, and let $h:\prd{K:\L(0)} L K \to M H K$.
  We define a functor $\derivcat{H}{h} : \derivcat{\L}{L} \to \derivcat{\M}{M}$  as follows.
  \begin{itemize}
  \item Consider an $n: \exo{\Nat}_{< p-1}$ and a $(K, \alpha): \derivcat{\L}{L}(n)$, so that in particular $\alpha : \prd{F:\fanout{0}{K}}L(\pi_1 F)$. We define $\derivcat{H}{h}(K, \alpha) : \derivcat{\M}{M}(n)$ to be $(H(K), \beta)$, where for $F: \fanout{HK}{0}$ we define
  \begin{align*}
     \beta(F) &\eqdef  h_{\pi_1 (H^{-1}F)} \big(\alpha(H^{-1}F)\big) : M{H} (\pi_1 H^{-1}F) \steq M(\pi_1 F)
  \end{align*}
  \item 
  Consider a morphism $(f,\phi):\hom_{\derivcat{\L}{L}}( (K_1,\alpha_1) , (K_2, \alpha_2))$. We can then define $\derivcat{H}{h}(f, \phi): \hom_{\derivcat{\M}{M}}( (HK_1,\beta_1) , (HK_2, \beta_2))$ to be $(Hf,\psi)$, where we define $\psi$ as follows: given $F : \fanout{HK_2}{0}$, we must check that
  \[ h_{\pi_1 (H^{-1}(F \circ Hf)}(\alpha_1(H^{-1}(F \circ Hf))) \steq h_{\pi_1 (H^{-1}F)} (\alpha_2(H^{-1}F)).\]
  But $H^{-1}( F \circ Hf) \steq H^{-1}( F ) \circ f $ (since applying the isomorphism $H$ produces $F \circ Hf$ on both sides) and $\alpha_1(H^{-1}( F ) \circ f) \steq \alpha_2(H^{-1}( F ))$ by $\phi$.
    \end{itemize}
\end{definition}

\begin{lemma}[Functoriality of derivatives]\label{lem:der_fun}
We have that 
\begin{enumerate}
\item \label{item:folds-deriv-id}
  for any inverse exo-category $\L$ of height $p>0$, and $L:\L(0) \to \Ustrict$,
\[\derivcat{(1_\L)}{\lambda x. 1_{Lx}} \steq 1_{\derivcat{\L}{L}}. \]
\item \label{item:folds-deriv-comp}
  for inverse exo-categories $\L,\M,\N$ of height $p>0$, discrete opfibrations $H: \L\to \M$ and $I:\M \to \N$, and families $L: {\L}(0) \to \Ustrict$, $M: {\M}(0) \to \Ustrict$, and $N: {\N}(0) \to \Ustrict$, plus $h:\prd{x:{\L}(0)} L x \to M {H} x$ and $i:\prd{x:{\M}(0)} M x \to N {I} x$,
\[\derivcat{I}{i} \circ \derivcat{H}{h} \steq \derivcat{(I \circ H)}{i \circ h}\] 
where 
$(i \circ h) (x,\ell) \define i({H}x, h(x,\ell))$.
\end{enumerate}
\end{lemma}
\begin{proof}
Note that the desired exo-equalities are obvious on the first components of objects and morphisms in ${\derivcat{\L}{L}}$. The desired exo-equalities on the second components of morphisms follow from UIP and function extensionality. Thus, we check the exo-equalities just on the second components of objects.

To check the exo-equality of \cref{item:folds-deriv-id} on objects, observe that
\begin{align*}
\pi_2 \derivcat{(1_\L)}{1_{L}}(K, \alpha) & \steq  \lambda F.  \lambda x. 1_{Lx} (\pi_1 (1_\L^{-1} F), \alpha(1_\L^{-1} F) \\
& \steq \alpha.
\end{align*}

To check the exo-equality of \cref{item:folds-deriv-comp} on objects, calculate that
\begin{align*}
\left( \pi_2 \derivcat{(I \circ H)}{i \circ h}(K, \alpha) \right) (F)
& \steq   ih(\pi_1 ((IH)^{-1}F), \alpha((IH)^{-1}F))
  \\ & \steq  i(\bottom{H} \pi_1 ((IH)^{-1}F) , h(\pi_1 ((IH)^{-1}F), \alpha((IH)^{-1}F)))
\\ 
& \steq i(\pi_1 (I^{-1}F), h(\pi_1 (H^{-1}I^{-1}F), \alpha(H^{-1}I^{-1}F)))
\\
& \steq \left( \pi_2 (\derivcat{I}{i} \circ \derivcat{H}{h})(K, \alpha) \right) (F).\qedhere
\end{align*}
\end{proof}

We also need to know that this functor lands in discrete opfibrations.

\begin{definition}
  For an inverse exo-category $\L$ of height $p>0$, let $\L_{>0}$ be its full subcategory on the objects of rank $>0$.
  This is an inverse exo-category of height $p-1$, where the rank of all objects is reduced by one from $\L$; thus $\L_{>0}(n) \define \L(n+1)$.
\end{definition}

\begin{lemma}\label{lem:Uisfib}
  For $\L$ an inverse exo-category of height $>0$ and $M:\L(0)\to\Ustrict$, the evident forgetful functor $U: \derivcat{\L}{M}\to \L_{>0}$ is a discrete opfibration.
\end{lemma}
\begin{proof}
  Consider a $(K,\alpha):\derivcat{\L}{M}$ where we have $K: \L(n+1)$ as well as a term $\alpha: \prd{F: \fanout{K}{0}} M(\pi_1 F)$, and a term $(L,f): \fanout{K}{m}$ composed of $L: \L(m)$ and $f: \hom(K,L)$.
  To show that the map on fanouts induced by $U$ is an isomorphism, we define a putative inverse $U^{-1}$ by letting $U^{-1} (L,f)$ be $((L,\beta),(f,\gamma))$, where we define
  $\beta: \prd{F: \fanout{L}{0}} M(\pi_1 F)$
  by $\beta(F) \define \alpha(F \circ f)$, and obtain
  $\gamma: \prd{F: \fanout{L}{0}} \alpha(F \circ f) \steq \beta (F)$
  by construction.

Clearly, $U U^{-1} \steq 1_{\fanout{K}{m}}$. To show $U^{-1} U \steq 1_{\fanout{(K,\alpha)}{m}}$, consider a $((L,\beta),(f,\gamma)): \fanout{(K,\alpha)}{m}$. 
We get that
\[ U^{-1} U((L,\beta),(f,\gamma)) \steq ((L,\beta'),(f,\gamma'))\] 
and
\[\gamma^{-1} * \gamma': \prd{F: \fanout{L}{0}}  \beta (F) \steq \beta'(F)  .
\]
By function extensionality, $\beta \steq \beta'$; and by UIP and function extensionality, $\gamma \steq \gamma'$.
\end{proof}

\begin{lemma}\label{lem:fibleftcan}
  For exo-functors $F:\L\to \M$ and $G:\M\to \N$ such that $G \circ F$ and $G$ are discrete opfibrations, also $F$ is a discrete opfibration.
\end{lemma}
\begin{proof}
  This is a standard lemma in category theory, but we reproduce the proof.
  Consider a
 $K:\L(n)$. The following strictly commutative diagram shows that $F$ is an isomorphism on fanouts, and thus a discrete opfibration.
 \[   
 \begin{tikzcd}[ampersand replacement = \&,baseline=(bottom.base)]
   \fanout{K}{m} \ar[r, "F"] \ar[dr, "G \circ F"', "\cong"]
   \&  
   \fanout{FK}{m} \ar[d, "G", "\cong"']
   \\
   \&
   |[alias=bottom]|
   \fanout{GFK}{m} 
  \end{tikzcd}\qedhere
 \]
\end{proof}

\begin{proposition}\label{prop:deriv-opfib}
The functor $\derivcat{H}{h}: \derivcat{\L}{L} \to \derivcat{\M}{M}$ from \cref{def:derivation_functorial_action} is a discrete opfibration.
\end{proposition}

\begin{proof}
The following square commutes.
\[
  \begin{tikzcd}
    \derivcat{\L}{L} \ar[r,"\derivcat{H}{h}"] \ar[d,"U"'] & \derivcat{\M}{M} \ar[d,"U"] \\
    \L_{>0} \ar[r,"{H_{>0}}"'] & \M_{>0}
  \end{tikzcd}
\]
Note that since $H$ is a discrete opfibration, so is ${H_{>0}}$.
Since both instances of $U$ are also discrete opfibrations (\cref{lem:Uisfib}), we find (using \cref{lem:fibleftcan}) that $\derivcat{H}{h}$ is a discrete opfibration.
\end{proof}

\begin{proposition}\label{prop:deriv-func}
  Let $\L$ be an inverse exo-category of height $>0$.
  Then we have an exo-functor $\derivcat{\L}{}$ from $\L(0) \to \Ustrict$ (with exo-category structure inherited from $\Ustrict$) to the exo-category of inverse exo-categories and discrete opfibrations over $\L_{>0}$.
\end{proposition}
\begin{proof}
  Take $H$ to be the identity in \cref{def:derivation_functorial_action}.
  \cref{lem:Uisfib,prop:deriv-opfib} show that this functor lands in discrete opfibrations over $\L_{>0}$.
\end{proof}

\begin{remark}\label{rem:two-functorialities}
  Note that \cref{def:derivation_functorial_action,prop:deriv-opfib} wrap up two kinds of functoriality in one statement:
  \begin{enumerate}
  \item For a fixed diagram signature $\L$, the derivative $\derivcat{\L}{}$ is functorial on maps of $\L(0)$-indexed families of exotypes.
    This is the functor $\derivcat{\L}{}$ discussed in \cref{prop:deriv-func}, and the one referred to when we wrote $\derivcat{\L}{\copair{1_{\bottom{M}}}{\hat{a}}}$ towards the beginning of \cref{sec:indisc-folds}.
  \item A morphism (i.e., discrete opfibration) $F : \L \to \M$ between diagram signatures induces a map from the derivative $\derivcat{\M}{M}$ of $\M$ at some family $M : \M(0) \to \U$ to the derivative $\derivcat{\L}{F^*M}$ of $\L$ at the induced family $F^*M : \L(0) \to \U$.
  \end{enumerate}
  The definition of functorial signature, given in \cref{def:abstract_signature}, distinguishes more explicitly between these two functorialities.
\end{remark}

Now recall from ordinary category theory that the category of discrete opfibrations over a given category $\C$ is equivalent to the functor category $[\C,\Set]$.
By the same argument in exo-category theory, the codomain of the functor $\derivcat{\L}{}$ from \cref{prop:deriv-func} is equivalent to the exo-functor exo-category $[\L_{>0}, \Ustrict]$.

\begin{theorem}\label{prop:deriv-struc}
  For an inverse exo-category $\L$ of height $p>0$, the exo-functor exo-category $\func{\L}{\Ustrict}$ is equivalent (as an exo-category) to the \emph{Artin gluing}\index{gluing!Artin} $\mathsf{Gl}(\derivcat{\L}{})$ of the functor $\derivcat{\L}{}$ from \cref{prop:deriv-func}.
  The latter is equivalently the following exo-category:
  \begin{itemize}
  \item Its objects are pairs $(\bottom{M},\derivdia{M})$, where $\bottom{M} : \L(0) \to \Ustrict$ and $\derivdia{M}$ is a functor $\derivcat{\L}{\bottom{M}} \to \Ustrict$.
  \item A morphism $(\bottom{M},\derivdia{M}) \to (\bottom{N},\derivdia{N})$ consists of a family of functions $\bottom\alpha : \prd{K:\L(0)} \bottom{M}(K) \to \bottom{N}(K)$ and a natural transformation
   \[ 
    \begin{tikzcd}[column sep = huge,row sep=small]
      \derivcat{\L}{\bottom{M}} \ar[rd, "\derivdia{M}"] \ar[dd, "\_ \circ \bottom{\alpha}"', ""'{name=F}]
      \\
      &
      |[alias=S]| \Ustrict
      \\
      \derivcat{\L}{\bottom{N}} \ar[ru, "\derivdia{N}"']
      \arrow[from = F, to = S, phantom, "\Downarrow \alpha'" description]
    \end{tikzcd}
   \]
  \end{itemize}  
\end{theorem}
\begin{proof}
  By definition, the Artin gluing is the comma exo-category of the identity functor of $[\L_{>0}, \Ustrict]$ over $\derivcat{\L}{}$.
  Thus, its objects are triples consisting of $\bottom{M} : \L(0) \to \Ustrict$, a functor $\derivdia{M} : \L_{>0} \to \Ustrict$, and a natural transformation from $\derivdia{M}$ to the $\L_{>0}$-diagram corresponding to $\derivcat{\L}{\bottom{M}}$.
  But the latter two data are equivalent to a functor $\derivcat{\L}{\bottom{M}} \to \Ustrict$.
  We leave it to the reader to similarly identify the morphisms.

  Now let $P$ be the profunctor from $\L(0)$ (considered as a discrete exo-category) to $\L_{>0}$ that is defined by the hom-functors of $\L$ with rank-0 codomain.
  Then $\L$ is the \emph{collage} of $P$, i.e., it is the disjoint union of $\L(0)$ and $\L_{>0}$ with extra morphisms from the latter to the former supplied by $P$.
  The collage of a profunctor has a universal property (see, e.g.,~\cite{street:cauchy-enr,wood:proarrows-ii}), which in this case says that $\func{\L}{\Ustrict}$ is equivalent to the category of triples consisting of a functor $\bottom{M} : \L(0) \to \Ustrict$, a functor $\derivdia{M} : \L_{>0} \to \Ustrict$, and a natural transformation from $P$ to the induced profunctor $\Ustrict(\derivdia{M},\bottom{M})$.
  The latter is equivalently a natural transformation from $\derivdia{M}$ to the $P$-weighted limit of $\bottom{M}$.
  (This is an instance of~\cite[Theorem 4.5]{shulman:reedy}.)
  But the latter weighted limit is precisely the $\L_{>0}$-diagram corresponding to $\derivcat{\L}{\bottom{M}}$.
\end{proof}

We can also make this functorial:

\begin{proposition}\label{prop:deriv-struc-fun}
  For $F:\hom_{\IC(p)}(\L,\M)$, the following square commutes up to natural exo-isomorphism:
  \[
    \begin{tikzcd}
      \func{\M}{\Ustrict} \ar[r,"\simeq"] \ar[d] \ar[dr,phantom,"\cong"] & \mathsf{Gl}(\derivcat{\M}{}) \ar[d] \\
      \func{\L}{\Ustrict} \ar[r,"\simeq"'] & \mathsf{Gl}(\derivcat{\L}{})
    \end{tikzcd}
  \]
  Here the horizontal arrows are the equivalences of exo-categories of \cref{prop:deriv-struc}, and the right-hand arrow sends $(\bottom{M},\derivdia{M})$ to $(\bottom{M}\circ F(0), \derivdia{M} \circ \derivcat{F}{})$.
  \qed
\end{proposition}

Next we compare Reedy fibrancy of diagrams, using the following fundamental lemma.
Here we start to assume that our inverse exo-categories are actually diagram signatures, although we will not need the full strength of \cref{def:material-sig}: for this result it would suffice to assume that the exotypes $\L(n)$ are cofibrant rather than sharp.

\begin{lemma}\label{lem:matching-fibers}
  Let $p>0$ and $\L: \IC(p)$, and let $M:\L\to \Ustrict$ be an exo-functor corresponding to $(\bottom{M}, \derivdia{M})$ under \cref{prop:deriv-struc}.
  Then for any $K:\L(n+1)$, the fibers of the map $MK \to \match_K M$ are precisely the fibers of all the maps $\derivdia{M}(K,\alpha) \to \match_{(K,\alpha)}\derivdia{M}$ as $\alpha$ varies.
\end{lemma}
\begin{proof}
  If we unwind the construction of \cref{prop:deriv-struc}, we see that the functor $\derivdia{M} : \derivcat{\L}{\bottom{M}} \to \Ustrict$ sends a sort $(K,\alpha) \in \derivcat{\L}{\bottom{M}}(n)$ to the strict fiber over $\alpha$ of the map
  \begin{equation}
    MK \to \tprd{(L,f): \fanout{K}{0}} ML\label{eq:deriv-fiber}
  \end{equation}
  that sends $x:MK$ to $\lambda (L,f). Mf(x)$.
  Now recall that $\match_K M$ is a sub-exotype (cf.~\cref{sec:identity}) of $\prd{m<n+1}{(L,f):\fanout{K}{m}} ML$.
  Thus,~\eqref{eq:deriv-fiber} factors through the matching map $MK \to \match_K M$:
  \[
    \begin{tikzcd}[column sep =small]
      MK \ar[rr] \ar[dr] && \match_K M \ar[dl] \\
      &  \tprd{(L,f): \fanout{K}{0}} ML
    \end{tikzcd}
  \]
  The fiber of $\match_K M$ over $\alpha$ is a sub-exotype of $\prd{m<n}{(L,f):\fanout{K}{m+1}} ML$, which as in \cref{prop:deriv-good} is isomorphic to $\prd{m<n}{(L,f):\fanout{K,\alpha}{m}} \derivdia{M}L$.
  If we unravel the conditions defining this subtype, we find that it is isomorphic to $\match_{(K,\alpha)} \derivdia{M}$.
\end{proof}

\begin{proposition}\label{thm:deriv-reedy}
  Let $p>0$ and $\L: \IC(p)$, and let $M:\L\to \Ustrict$ be an exo-functor corresponding to $(\bottom{M}, \derivdia{M})$ under \cref{prop:deriv-struc}.
  Then $M$ is Reedy fibrant if and only if
  \begin{enumerate}
  \item $\bottom{M}$ is pointwise fibrant, i.e., it is a function $\L(0) \to \U$; and\label{item:dr1}
  \item $\derivdia{M}$ is a Reedy fibrant diagram on $\derivcat{\L}{\bottom{M}}$.\label{item:dr2}
  \end{enumerate}
\end{proposition}
\begin{proof}
  Since the matching object at a rank-0 sort is always $\onetype$, condition~\ref{item:dr1} is equivalent to Reedy fibrancy of $M$ at rank-0 sorts.
  Thus, it suffices to show that for any $K:\L(n+1)$, the map $MK \to \match_K M$ is a fibration if and only if for all $\alpha:\prd{F: \fanout{K}{0}}\bottom{M}(\pi_1 F)$ the map $\derivdia{M}(K,\alpha) \to \match_{(K,\alpha)} \derivdia{M}$ is a fibration.
  But this follows from \cref{lem:matching-fibers}.
\end{proof}

By combining \cref{prop:deriv-struc,thm:deriv-reedy} we will be able to prove our claim that the exo-category of Reedy fibrant exo-diagrams on any diagram signature is equivalent, as an exo-category, to one whose exotype of objects is fibrant.
However, to describe and work with the latter type more easily, we introduce the notion of \emph{functorial signature}.

\section{Functorial signatures and structures}
\label{sec:abstract-signatures}

\cref{prop:deriv-struc,thm:deriv-reedy} show that the ``essential content'' of a diagram signature $\L$ consists of the type $\L(0)$ and the derived diagram signatures $\derivcat{\L}{M}$ for all $M:\L(0)\to\U$.
The notion of ``functorial signature'' simply takes an analogous decomposition as an \emph{inductive definition}.

\begin{definition}
\label{def:abstract_signature}\index{signature!functorial}
We define a family of exo-categories $\Sig(n)$ of \defemph{(functorial) signatures of height $n$} by induction.\nomenclature[Sig]{$\Sig(n)$}{exo-category of functorial signatures of height $n$}

Let $\Sig(0)$ be the trivial exo-category on $\onetype$.

An object $\L$ of $\Sig(n+1)$ consists of
\begin{enumerate}
\item a sharp exotype $\bottom{\L} : \Ustrict$;\nomenclature[L]{$\bottom{\L}$}{rank-0 part of functorial signature $\L$}
\item an exo-functor $\derivcat{\L}{} : \func{\bottom{\L}}{\U} \to \Sig(n)$, where $\func{\bottom{\L}}{\U}$ is the exo-functor exo-category from the discrete exo-category $\bottom{\L}$ to the canonical exo-category $\U$.\nomenclature[L]{$\derivcat{\L}{}$}{derivation exo-functor of functorial signature $\L$}
\end{enumerate}

\noindent
Arguments of $\derivcat{\L}{}$ will be written as subscripts, as in $\derivcat{\L}{M}$.\nomenclature[L]{$\derivcat{\L}{M}$}{value of derivation exofunctor $\derivcat{\L}{}$ of a functorial signature $\L$ at $M$}

For $\L, \M: \Sig(n+1)$, an element $\alpha$ of $\hom_{\Sig(n+1)}(\L, \M)$ consists of the following:
 \begin{enumerate}
 \item a function $\bottom{\alpha} : \bottom{\L} \to \bottom{\M}$
 \item an exo-natural transformation $\derivcat{\alpha}{}$ as in the diagram

   \[ 
    \begin{tikzcd}[ampersand replacement = \&, column sep = huge]
      \func{\bottom{\M}}{\U} \ar[rd, "\derivcat{\M}{}"] \ar[dd, "\_ \circ \bottom{\alpha}"', ""'{name=F}]
      \\
      \&
      |[alias=S]| \Sig(n)
      \\
      \func{\bottom{\L}}{\U} \ar[ru, "\derivcat{\L}{}"']
      \arrow[from = F, to = S, phantom, "\Uparrow \derivdia{\alpha}" description]
    \end{tikzcd}
   \]

 \end{enumerate}
  Arguments of $\derivdia{\alpha}$ will also be written as subscripts, as in $\derivcat{\alpha}{M}$.
 
 Composition and identities are given by function composition and identity at $\bottomlevel$, and inductively for the derivative.
 Similarly, the categorical laws are easily proved by induction. 
\end{definition}

The decomposition of a diagram signature into $\L(0)$ and its derivatives yields a functorial signature, and in fact this is an exo-functor.

\begin{theorem}\label{thm:translation}
For each $p: \exo{\Nat}$, define an exo-functor $E_p: \IC(p) \to \Sig(p)$ by induction on $p$ as follows.\nomenclature[E]{$E_p$}{functor from diagram signatures of height $p$ to functorial signatures of height $p$}

Since $\Sig(0)$ is the trivial category on $\onetype$, there is a unique exo-functor $\IC(0) \to \Sig(0)$ (which is actually an equivalence).

For $p > 0$, we assume given an exo-functor $E_{p-1} : \IC(p-1) \to \Sig(p-1)$. Let $\L: \IC(p)$.
We define  $E_p(\L)$ as follows:
\begin{enumerate}
\item The sharp exotype $\bottom{\L} \define \L(0) : \U$.
\item The functor $E_{p-1}  \derivcat{\L}{-}: (\bottom{\L} \to \U) \to \Sig(p-1)$ defined by composing the inductively given $E_{p-1} : \IC(p-1) \to \Sig(p-1)$ with the functor into $\IC(p-1)$ given on objects by \cref{app:defn:derivcat} and \cref{prop:deriv-good}
and on morphisms by 
 \cref{def:derivation_functorial_action} and \cref{prop:deriv-opfib}.
\end{enumerate}
For $\L,\M: \IC(p)$ and $F: \hom(\L,\M)$, let $E_p (F)$ consist of:
\begin{enumerate}
\item The function $\bottom{F} : \bottom{\L} \to \bottom{\M}$.
\item The natural transformation with underlying function 
\[E_{p-1}  \derivcat{F}{\lambda x. 1_{-  x}}: \prd{M : \bottom{\M} \to \U} \hom(\derivcat{\L}{M \circ \bottom{F}},\derivcat{\M}{M}) 
\]
defined in \cref{def:derivation_functorial_action} and \cref{prop:deriv-opfib}.
\end{enumerate}
\end{theorem}

\begin{proof}
We check that $E_p$ is functorial. 

For any $\L : \IC(p)$, we have the following.
\begin{align*} 
\pi_1 E_p(1_\mathcal L) &\steq 1_{\bottom{\L}} \\
& \steq \pi_1(1_{E_p \mathcal L}) \\
\pi_2 E_p(1_\mathcal L) &\steq E_{p-1} \circ \derivcat{1_\L}{\lambda x. 1_{-  x}}\\
& \steq 1_{E_{p-1} \derivcat{1_\L}{ \lambda x. 1_{-  x}}} \\
& \steq \pi_2(1_{E_p \mathcal L}) 
\end{align*}

For any $\M, \N, \mathcal P: \IC(p)$, $F: \hom(\M, \N)$, $G: \hom(\N,\mathcal P)$, we have the following.
\begin{align*} 
\pi_1 E_p(G \circ F) &\steq \bottom{(G \circ F)} \\
& \steq \bottom{G} \circ \bottom{F} \\
& \steq \pi_1 (E_p G \circ E_p F) \\
\left(  \pi_2 E_p(G \circ F) \right)(M) &\steq E_{p-1} \circ \derivcat{(G \circ F)}{\lambda x. 1_{M  x}}\\
&\steq E_{p-1}  (\derivcat{G}{\lambda x. 1_{M  x}} \circ \derivcat{F}{\lambda x. 1_{M  x}})\\
& \steq (E_{p-1}  \derivcat{G}{\lambda x. 1_{M  x}}) \circ (E_{p-1}  \derivcat{F}{\lambda x. 1_{M  x}}) \\
& \steq \pi_2 E_{p} G(M) \circ \pi_2 E_{p} F(M). \qedhere
\end{align*}
\end{proof}

Intuitively, this translation can be thought of as mapping into the exo-category \emph{coinductively}\index{coinduction} defined by a derivative functor, with the result landing inside the inductive part (our functorial signatures) because our diagram signatures have finite height.
(We have not investigated the possibility of signatures of infinite height.)
It can also be thought of as a sort of ``Taylor expansion'' of a diagram signature, consisting of all its iterated derivatives, with the functorial signatures playing the role of formal power series (although, again, since our signatures all have finite height, our ``power series'' are actually just polynomials).

The notion of functorial signature is perfectly adapted to define \emph{$\L$-structures} inductively.

\begin{definition}\label{def:functorial-structure}\index{structure!functorial}
Let $\L$ be a functorial signature; we define the type $\Struc{\L}$ of \defemph{$\L$-structures} by induction on $n$.\nomenclature[Str]{$\Struc{\L}$}{type of structures of functorial signature $\L$}

If $\L: \Sig(0)$, we define $\Struc{\L} \define \onetype$.

If $\L:\Sig(n+1)$, we define
\begin{equation}
  \Struc{\L} \define  \sm{\bottom{M} : {\bottom{\L}} \to \U } \Struc{\derivcat{\L}{\bottom{M}}}.\label{eq:func-struc}
\end{equation}
We write the two components of $M:\Struc{\L}$ as $(\bottom{M},\derivdia{M})$.
\end{definition}

Note that when $\L$ arises from a diagram signature as in \cref{thm:translation}, this definition reduces to \cref{def:folds-struc}.

We expect to have a whole exo-category of $\L$-structures.
This is true, but requires a bit more work.
The idea is that when $\L:\Sig(n+1)$, the exo-category $\Struc{\L}$ should be the Grothendieck construction of the composite functor
\[
  \begin{tikzcd}
    (\bottom{\L} \to \U)^{\mathsf{op}} \ar[r,"(\derivcat{\L}{})^{\mathsf{op}}"] &
    \Sig(n)^{\mathsf{op}} \ar[r,"\Struc{-}"] &
    \mathsf{Cat}.
  \end{tikzcd}
\]
Note that such a Grothendieck construction would indeed have \cref{eq:func-struc} as its type of objects.
Here $\bottom{\L}\to \U$ inherits its exo-category structure pointwise from $\U$; and $\Sig(n)$ and $\derivcat{\L}{}$ were defined to be a category and a functor, respectively, in \cref{def:abstract_signature}; but we have yet to make $\Struc{-}$ into a functor.
This means defining the ``pullback'' of an $\M$-structure along a morphism $\alpha : \L \to \M$ of signatures.

\begin{definition}\index{pullback!of functorial structures}
For any $\alpha: \hom_{\Sig(n)}(\L, \M)$, we define the \textbf{pullback} $\alpha^*: \Struc{\M} \to \Struc{\L}$ inductively as follows. 

If $n \eqdef 0$, then let $\alpha^*: \Struc{\M} \to \Struc{\L}$ be the identity. 

If $n > 0$, consider $M: \Struc{\M}$. We let $\bottom{ (\alpha^* M)}$ be $\bottom{M} \circ \bottom{\alpha}$. By induction, the morphism 
\[\derivcat{\alpha}{\bottom{M}}:  \hom_{\Sig(n-1)}(\derivcat{\L}{\bottom{M} \circ \bottom{\alpha}} , \derivcat{\M}{\bottom{M}})\]
produces a $(\derivcat{\alpha}{\bottom{M}})^*: \Struc{\derivcat{\M}{\bottom{M}}} \to \Struc{\derivcat{\L}{\bottom{M} \circ \bottom{\alpha}}}$, so we set $\derivdia{(\alpha^* M)} \define (\derivcat{\alpha}{\bottom{M}})^* \derivdia{M}$.
\end{definition}

Pullback is functorial: pullback along a composition of signature morphisms is the composition of pullbacks, and pullback along an identity morphism is the identity.
We can now define morphisms of structures, using the usual definition of morphisms in a Grothendieck construction.

\begin{definition}\index{morphism!of functorial structures}
Consider $\L: \Sig(n)$ and $M,N : \Struc{\L}$; we define the (fibrant) type $\hom_{\Struc{\L}}(M, N)$ of \defemph{morphisms of $\L$-structures} by induction on $n:\exo{\Nat}$.

When $n \eqdef 0$, we let $\hom_{\Struc{\L}}(M, N)\define \onetype$.

When $n > 0$, a morphism $f : \hom_{\Struc{\L}}(M,N)$ consists of
 \begin{enumerate}
 \item $\bottom{f}: \prd{K:\bottom{\L}} \bottom{M}(K) \to \bottom{N}(K)$
 \item $\derivdia{f}: \hom_{\Struc{\derivcat{\L}{\bottom{M}}}}(\derivdia{M},( \derivcat{\L}{\bottom{f}})^*\derivdia{N})$.
 \end{enumerate}
\end{definition}

Composition of structure morphisms requires pullback of structure morphisms along signature morphisms:

\begin{definition}\index{pullback!of morphisms of structures}
    Let $\L$ and $\M$ be signatures of height $n$, and $\alpha : \hom_{\Sig(n)} (\L,\M)$.
    Let $M, N : \Struc{\M}$ and $f : \hom_{\Struc{\M}}(M,N)$.
    We define the \defemph{pullback of $f$ along $\alpha$}, denoted $\alpha^*f : \hom_{\Struc{\L}}(\alpha^*M, \alpha^*N)$, by induction on $n$.

    If $n \eqdef 0$, then $\alpha^*f$ is the unique morphism in  $\hom_{\Struc{\L}}(\alpha^*M, \alpha^*N)$.

    If $n > 0$, then we define $\alpha^*f$ as follows:
    \begin{enumerate}
    \item $\bottom{(\alpha^*f)} : \prd{K : \bottom{\L}}\bottom{(\alpha^*M)}(K) \to \bottom{(\alpha^*N)}(K)$ is given by
      \[
        \bottom{(\alpha^*f)}(K) \eqdef \bottom{f}(\bottom{\alpha}(K)) : \bottom{M}(\bottom{\alpha}(K)) \to \bottom{N}(\bottom{\alpha}(K))
      \]
    \item
      We have $\derivdia{f} : \hom_{\Struc{\derivcat{\M}{\bottom{M}}}}(\derivdia{M}, (\derivcat{\M}{\bottom{f}})^* \derivdia{N})$.
      
      By the induction hypothesis, we can pull back $\derivdia{f}$ along the signature morphism $\derivcat{\alpha}{\bottom{M}} : \hom_{\Sig(n-1)}(\derivcat{\L}{M \circ \bottom{\alpha}}, \derivcat{\M}{M})$, yielding
      \[
        (\derivcat{\alpha}{\bottom{M}})^*(\derivdia{f}) : \hom_{\Struc{\derivcat{\L}{M \circ \bottom{\alpha}}}}((\derivcat{\alpha}{\bottom{M}})^* \derivdia{M} , (\derivcat{\alpha}{\bottom{M}})^* (\derivcat{\M}{\bottom{f}})^* \derivdia{N}).
      \]

      By naturality of $\derivdia{\alpha}$, the following square commutes up to exo-equality,

      \[
        \begin{tikzcd}
          \derivcat{\L}{\bottom{M}\circ \bottom{\alpha}} \ar[r, "\derivcat{\alpha}{\bottom{M}}"] \ar[d, "\derivcat{\L}{\bottom{f} \circ \bottom{\alpha}}"']
          &
          \derivcat{\M}{\bottom{M}} \ar[d, "\derivcat{\M}{\bottom{f}}"]
          \\
          \derivcat{\L}{\bottom{N}\circ \bottom{\alpha}} \ar[r, "\derivcat{\alpha}{\bottom{N}}"]
          &
          \derivcat{\M}{\bottom{N}}
        \end{tikzcd}
      \]
      and hence we define
      \[
        \derivdia{(\alpha^*f)} \eqdef (\derivcat{\alpha}{\bottom{M}})^*(\derivdia{f}) : \hom_{\Struc{\derivcat{\L}{\bottom{M}\circ \bottom{\alpha}}}}(\derivdia{(\alpha^*M)}, (\derivcat{\L}{\bottom{(\alpha^*f)}})^*\derivdia{(\alpha^*N)}).
      \] 
    \end{enumerate}
\end{definition}

\begin{definition}\index{identity!structure morphisms}
  Let $\L$ be a signature of height $n$, and $M$ be an $\L$-structure.
  We define the \defemph{identity structure morphism} of $M$ by induction on $n$.

  For $n \eqdef 0$, the unique morphism $M \to M$ is the identity on $M$.

  For $n > 1$, we define the identity $1_M$ to be given by
  \begin{enumerate}
  \item $\bottom{1_M}$ is given by $\bottom{1_M}(K) \eqdef 1_{\bottom{M}(K)}$
  \item $\derivdia{1_M} : \hom_{\Struc{\derivcat{\L}{\bottom{M}}}}(\derivdia{M}, (\derivcat{\L}{\bottom{(1_M)}})^*\derivdia{M}) \steq \hom_{\Struc{\derivcat{\L}{\bottom{M}}}}(\derivdia{M}, \derivdia{M})$ is given by the induction hypothesis, using that $(\derivcat{\L}{1})^*M \steq 1^*M \steq M$.
    
  \end{enumerate}
\end{definition}

\begin{definition}\index{composition!of structure morphisms}
  Let $\L$ be a signature of height $n$, and let $M,N,O : \Struc{\L}$, with $f : \hom_{\Struc{\L}}(M,N)$ and $g : \hom_{\Struc{\L}}(N,O)$.
  We define the \defemph{composite of structure morphisms} $g\circ f$ by induction on $n:\exo{\Nat}$.
  
  If $n \eqdef 0$, then $g \circ f : \hom_{\Struc{\L}}(M,O)$ is the unique morphism from $M$ to $O$.

  If $n > 0$, we define $g \circ f$ as follows:
  \begin{enumerate}
  \item $\bottom{(g \circ f)} (K) \eqdef \bottom{g}(K) \circ \bottom{f}(K)$
  \item We have \[\derivdia{f} : \hom_{\Struc{\derivcat{\L}{\bottom{M}}}} (\derivdia{M}, (\derivcat{\L}{\bottom{f}})^*(\derivdia{N}))\] and \[\derivdia{g} : \hom_{\Struc{\derivcat{\L}{\bottom{N}}}} (\derivdia{N}, (\derivcat{\L}{\bottom{g}})^*(\derivdia{O})).\]
    Then we have
    \[
      (\derivcat{\L}{\bottom{f}})^*\derivdia{g} : \hom_{\Struc{\derivcat{\L}{\bottom{M}}}} ((\derivcat{\L}{\bottom{f}})^*\derivdia{N}, (\derivcat{\L}{\bottom{f}})^*(\derivcat{\L}{\bottom{g}})^*(\derivdia{O}))
    \]
    and, using that $(\derivcat{\L}{\bottom{f}})^*(\derivcat{\L}{\bottom{g}})^*(\derivdia{O}) \steq (\derivcat{\L}{\bottom{(g \circ f)}})^*(\derivdia{O})$, we define (using the induction hypothesis)
      \[
        \derivdia{(g \circ f)} \eqdef (\derivcat{\L}{\bottom{f}})^*\derivdia{g} \circ \derivdia{f} .
      \]
    \end{enumerate}
\end{definition}

Similarly, we can prove by induction that pullback preserves composition and identities, and then that composition is associative and unital.
Thus, for any functorial signature $\L$, we have an exo-category $\Struc{\L}$ with a fibrant type of objects and fibrant hom-types, and for any signature morphism $\alpha:\L\to\M$ we have an exo-functor $\Struc{\M} \to \Struc{\L}$.
With a little more work, we can show that $\Struc{-}$ is a contravariant exo-functor from $\Sig(n)$ to the exo-category $\mathsf{Cat}$ of exo-categories, providing the inductive step in its definition as a Grothendieck construction.
We can now finally prove the theorem we have been leading up to.

\begin{theorem}\label{thm:reedy-struc}
  For any diagram signature $\L:\IC(p)$, the exo-category $\rfunc{\L}{\Ustrict}$ is equivalent, as an exo-category, to $\Struc{E_p(\L)}$.
  Moreover, for any discrete opfibration $F:\hom_{\IC(p)}(\L,\M)$, these equivalences commute, up to natural exo-isomorphism, with pullback and precomposition:
  \[
    \begin{tikzcd}
      \rfunc{\M}{\Ustrict} \ar[r,"\simeq"] \ar[d,"(-\circ F)"'] \ar[dr,phantom,"\cong"] &
      \Struc{E_p(\M)} \ar[d,"(E_p(F))^*"] \\
      \rfunc{\L}{\Ustrict} \ar[r,"\simeq"'] &
      \Struc{E_p(\L)}      
    \end{tikzcd}
    \]
\end{theorem}
\begin{proof}
  We prove the two statements by mutual induction on $p$.
  When $p\define 0$ they are trivial.

  When $p>0$, by \cref{prop:deriv-struc,thm:deriv-reedy}, $\rfunc{\L}{\Ustrict}$ is equivalent to the exo-category of pairs $(\bottom{M},\derivdia{M})$ where $\bottom{M}:\L(0)\to \Ustrict$ and $\derivdia{M} : \rfunc{\derivcat{\L}{\bottom{M}}}{\Ustrict}$, where a morphism $(\bottom{M},\derivdia{M})\to (\bottom{M},\derivdia{M})$ consists of a family of functions $\bottom{f} : \prd{K:\L(0)} \bottom{M}(K) \to \bottom{N}(K)$ and a natural transformation $\derivdia{M} \to \derivdia{N} \circ \derivcat{\L}{\bottom{f}}$.
  By the inductive hypotheses, this is equivalent to the exo-category of pairs $(\bottom{M},\derivdia{M})$ where instead $\derivdia{M} : \Struc{E_p ({\derivcat{\L}{\bottom{M}}})}$, and where the natural transformation is replaced by a morphism of structures $\derivdia{M} \to (E_p(\derivcat{\L}{\bottom{f}}))^*(\derivdia{N})$.
  But since $E_p$ commutes with derivation by definition, the latter is precisely the exo-category $\Struc{\L}$.

  The proof of the commutation statement is analogous, using \cref{prop:deriv-struc-fun} in place of \cref{prop:deriv-struc}.
\end{proof}

\section{Axioms and theories for functorial signatures}
\label{sec:ho-axioms}

Our definitions of axiom and theory from \cref{sec:axioms} can be copied essentially verbatim for functorial signatures.

\begin{definition}
 Let $\L$ be a functorial signature. An \defemph{$\L$-axiom}\index{axiom!for functorial structures} is a function
 $\Struc{\L} \to \PropU$.
  A \defemph{functorial theory}\index{theory!functorial} is a pair $(\L,T)$ of a functorial signature $\L$ and a family $T$ of $\L$-axioms indexed by a cofibrant exotype.  A \defemph{model of a theory $(\L,T)$}\index{model!of a functorial theory} then consists of a $\L$-structure $M$ together with a proof $t(M)$ for each axiom $t$ of $T$.
  A \defemph{morphism of models}\index{morphism!of models of a functorial theory} is a morphism of the underlying structures.
\end{definition}

Of course, any diagram theory gives rise to a functorial theory, including all the examples from \cref{sec:examples}.
In \cref{sec:egs-higherorder} we will discuss some examples of functorial theories not arising in this way.

\chapter{Levelwise equivalences of structures}\label{sec:structures}

In \cref{prop:uni1-folds} we asserted that all structures for a diagram signature satisfy a tautological ``levelwise'' form of univalence, saying that identifications of structures are equivalent to levelwise equivalences.
We now prove an analogous statement for all functorial signatures, in \cref{prop:uni1}.

For a functorial signature coming from a diagram signature $\L$, we also link the levelwise equivalences between $\L$-structures $M$ and $N$ to morphisms between $M$ and $N$ considered as Reedy-fibrant diagrams, in \cref{thm:lvle-folds}.

\begin{definition}
\label{def:iso-of-structures}\index{equivalence!of functorial structures!levelwise}
 Let $\L: \Sig(n)$ and $M,N : \Struc{\L}$; we define when a morphism $f: \hom_{\Struc{\L}}(M,N)$ is a \defemph{levelwise $\L$-equivalence} by induction on $n$.

If $n\eqdef 0$, every $f: \hom_{\Struc{\L}}(M,N)$ is a levelwise $\L$-equivalence. That is, we define
$\isIso_{\L}(f) \define \onetype$.

If $n > 0$, then $f: \hom_{\Struc{\L}}(M,N)$ is a levelwise $\L$-equivalence when
 \begin{enumerate}
 \item $\bottom{f}(K)$ is an equivalence of types for all $K: \bottom{\L}$, and
 \item $\derivdia{f}$ is a levelwise $\derivcat{\L}{\bottom{M}}$-equivalence.
 \end{enumerate}
 That is, we define
 \[ \isIso_\L(f) \define \left(\prd{(K: \bottom{\L})} \isEquiv(\bottom{f}(K))\right) \times \isIso_{\derivcat{\L}{\bottom{M}}} (\derivdia{f}). \]

 We denote the type of levelwise $\L$-equivalences between two $\L$-structures $M,N$ by $M \leqv[\L]  N$, or simply $M \leqv N$.
 That is,
 \[ (M \leqv N) \eqdef \sm{f:\hom_{\Struc{\L}}(M,N)} \isIso_\L(f). \]\nomenclature[\bequiv2]{$M \leqv N$}{type of levelwise equivalences between functorial structures $M$ and $N$}
\end{definition}

\begin{lemma}\label{thm:exoiso-lvleqv}
  If $f:M\to N$ is an isomorphism in the exo-category $\Struc{\L}$, then it is a levelwise $\L$-equivelence.
\end{lemma}
\begin{proof}
  If $f$ is an isomorphism, then $\bottom{f}$ is pointwise an isomorphism of exotypes (and hence an equivalence of types) and $\derivdia{f}$ is an isomorphism in $\Struc{\derivcat{\L}{\bottom{M}}}$.
  The result follows by induction.
\end{proof}

\begin{lemma}\label{thm:lvle-folds}
  If $\L$ is a diagram signature, then a morphism of $\L$-structures is a levelwise equivalence in the sense of \cref{def:iso-of-structures} if and only if its corresponding morphism of Reedy fibrant diagrams is a levelwise equivalence in the sense of \cref{def:lvleqv-folds}.
\end{lemma}
\begin{proof}
  By induction, it suffices to prove that when $\L$ has height $p>0$, a morphism $f:M\to N$ is a levelwise equivalence in the sense of \cref{def:lvleqv-folds} if and only if $f_K$ is an equivalence for all $K:\L(0)$ and $\derivdia{f}$ is also a levelwise equivalence in the sense of \cref{def:lvleqv-folds}.
  The former is exactly what \cref{def:lvleqv-folds} says for rank-0 sorts.
  For the latter, we note that by \cref{prop:dfib-matching}, the square
  \[
    \begin{tikzcd}
      \derivdia{M}(K,\alpha) \ar[d] \ar[r] & ((\derivcat{\L}{\bottom{f}})^*\derivdia{N})(K,\alpha) \ar[d] \\
      \match_{(K,\alpha)} \derivdia{M} \ar[r] & \match_{(K,\alpha)} ((\derivcat{\L}{\bottom{f}})^*\derivdia{N})
    \end{tikzcd}
  \]
  is isomorphic to
  \[
    \begin{tikzcd}
      \derivdia{M}{(K,\alpha)} \ar[d] \ar[r] & \derivdia{N} (K,f \alpha) \ar[d] \\
      \match_{(K,\alpha)} \derivdia{M} \ar[r] & \match_{(K,f\alpha)} \derivdia{N}.
    \end{tikzcd}
  \]
  Thus, by \cref{lem:matching-fibers}, the maps on fibers of all these squares, for all sorts $(K,\alpha)$ of $\derivcat{\L}{\bottom{\M}}$, are precisely the maps on fibers of the analogous squares for $f$ at all sorts $K$ of positive rank in $\L$.
\end{proof}

\begin{remark}
Just as equivalences of types can be characterized in several equivalent ways (see~\cite[Chapter 4]{HTT}), we could characterize levelwise equivalences in different ways in terms of the exo-category $\Struc{\L}$.

For example, given a morphism $f: \hom_{\Struc{\L}}(M,N)$, we could define a type $\isAdjEq_\L(f)$ of ways to make $f$ a half-adjoint equivalence:\index{equivalence!of functorial structures!half-adjoint} each term would consist of a morphism $g: \hom_{\Struc{\L}}(N,M)$, an identification $\eta: g \circ f = 1_N$, and an identification $\epsilon : f\circ g = 1_M$ such that $f \circ \eta = \epsilon \circ f$. Note that though $\Struc{\L}$ is an exo-category and thus does not have any intrinsic notion of half-adjoint equivalence, we are taking advantage of the fact that its $\hom$-types are actually fibrant types in order to define this notion.

We could similarly define a type $\isBiInv_\L(f)$ of ways to make $f$ a bi-invertible morphism.\index{equivalence!of functorial structures!bi-invertible} Then we could show that these types --- $\isIso_{\L}(f)$, $\isAdjEq_\L(f)$, and $\isBiInv_\L(f)$ --- are equivalent.

Though $\isAdjEq_\L(f)$ and $\isBiInv_\L(f)$ would be in some sense more explicit characterizations of levelwise equivalence, we stick with $\isIso_{\L}(f)$ for its simplicity.
\end{remark}

\begin{lemma}\label{lem:isisoprop}
For any morphism $f : M \to N$ between two $\L$-structures, the exotype $\isIso_\L(f)$ is a fibrant proposition. \qed
\end{lemma}

\begin{lemma}\label{lem:identity_is_lvle}
  For any signature $\L$ and $\L$-structure $M$, the identity morphism on $M$ is a levelwise equivalence.
\end{lemma}
\begin{proof}
  By induction on $n$: for $n \eqdef 0$, any morphism is a levelwise equivalence.
  For $n > 0$, we have that $\bottom{1_M}(K) = 1_{MK}$, which is an equivalence of types.
  The identity on $\derivdia{M}$ is a levelwise equivalence by induction hypothesis.
\end{proof}

\begin{definition}
  Let $\L$ be a signature, and $M$ be an $\L$-structure.
  We define, by \pathinduction, the function
  \begin{align*}
    \idtolvle : (M = N) &\to (M \leqv[\L] N)
    \\
                \refl_M &\mapsto 1_M
  \end{align*}\nomenclature[idtolvl]{$\idtolvle$}{function from identifications to levelwise equivalences of functorial structures}%
  Here we use that the identity morphism is a levelwise equivalence per \cref{lem:identity_is_lvle}.
\end{definition}

\begin{proposition}\label{prop:uni1}
 For structures $M,N$ of a signature $\L$, the canonical map 
 \[ \idtolvle_{M,N}: (M = N) \to (M \leqv N) \]
 is an equivalence of types.
\end{proposition}
\begin{proof}
  When $n \eqdef 0$, $ \idtolvle:\onetype \to \onetype$, hence is an equivalence.

Let $\ua : (\bottom{M} \simeq \bottom{N}) \to (\bottom{M}  = \bottom{N} )$ be given by the univalence axiom. 
First we show that $\trans{\ua(e)^{-1}}{\derivdia{N}} = (\derivcat{\L}{e})^*\derivdia{N}$ for any $e: \bottom{M} \simeq \bottom{N}$, where $\transfun{\ua(e)^{-1}}$ denotes transport along $\ua(e)^{-1}$.
Now the square in the diagram

\[
  \begin{tikzcd}[ampersand replacement = \&]
    (\bottom{M} \simeq \bottom{N}) \ar[r, "\ua"]
    \& 
    (\bottom{M} = \bottom{N}) \ar[d, "(-)^{-1}"] \ar[r, "\ua^{-1}"]
    \& 
    (\bottom{M} \simeq \bottom{N}) \ar[r, "\derivcat{\L}{-}"]
    \&
    \hom_{\Sig(n)}(\derivcat{\L}{\bottom{M}} , \derivcat{\L}{\bottom{N}}) \ar[d, "(-)^*"]
    \\
    \& 
    (\bottom{N} = \bottom{M}) \ar[rr, "(-)_*"] 
    \& 
    \& 
    \Struc{\derivcat{\L}{\bottom{N}}} \to \Struc{\derivcat{\L}{\bottom{M}}}
  \end{tikzcd}
\]
commutes (up to $=$) since both functions $(\bottom{M} = \bottom{N}) \to \Struc{\derivcat{\L}{\bottom{N}}} \to \Struc{\derivcat{\L}{\bottom{M}}} $ send $\refl_{\bottom{M}} $ to $ 1_{\Struc{\derivcat{\L}{M}}}$ (by exo-functoriality of the pullback). Precomposing these with $\ua$, we find that $(\derivcat{\L}{e})^*\derivdia{N} = \trans{\ua(e)^{-1}}{\derivdia{N}}$.
Now we have that
 \begin{align*}
 (M = N) & = \sm{p:\bottom{M} = \bottom{N}} \derivdia{M} = \trans{p^{-1}}{\derivdia{N}} \\
 &= \sm{e:\bottom{M} \simeq \bottom{N}} \derivdia{M} = \trans{\ua(e)^{-1}}{\derivdia{N}} \\
 &= \sm{e:\bottom{M} \simeq \bottom{N}} \derivdia{M} = (\derivcat{\L}{e})^*\derivdia{N} \\
 &= \sm{e:\bottom{M} \simeq \bottom{N}} \derivdia{M} \leqv (\derivcat{\L}{e})^*\derivdia{N} \\
 &\equiv ( M \leqv N )
 \end{align*}
 where the second identification is the univalence axiom and the fourth is our inductive hypothesis. This equivalence, from left to right, is $\idtolvle_{M,N}$.
 \end{proof}

To end this \lcnamecref{sec:structures}, we observe that we can now deduce that the type of structures for a diagram signature is independent of the rank function.

\begin{corollary}\label{cor:struc-norank}
  Let $\L:\IC(p)$ be an inverse exo-category, and let $\M:\IC(q)$ be the same exo-category but made into an inverse exo-category with a different rank function.
  Then the types $\Struc{E_p(\L)}$ and $\Struc{E_q(\M)}$ are equivalent.
\end{corollary}
\begin{proof}
  The notion of Reedy fibrancy is independent of the rank function, so the exo-categories $\rfunc{\L}{\Ustrict}$ and $\rfunc{\M}{\Ustrict}$ are equivalent.
  By \cref{thm:reedy-struc}, they are also equivalent to the exo-categories $\Struc{E_p(\L)}$ and $\Struc{E_q(\M)}$ respectively, so these two exo-categories are equivalent.
  However, by \cref{thm:exoiso-lvleqv,prop:uni1}, the natural exo-isomorphisms witnessing this equivalence of exo-categories yield identifications, so that the underlying types $\Struc{E_p(\L)}$ and $\Struc{E_q(\M)}$ are also equivalent.
\end{proof}

\chapter{Indiscernibility and univalence}\label{sec:FOLDS-iso-uni}

In this \lcnamecref{sec:FOLDS-iso-uni} we make the definitions of indiscernibility (\cref{def:indisc-in-diag-structure}) and univalence (\cref{def:univalence-cond-on-diag-structures}) from \cref{sec:indisc-folds} completely precise, in the generality of functorial signatures.
The corresponding definitions for functorial signatures are given in \cref{def:iso-within-a-structure} and \cref{def:functorial-univalence}, respectively.

We furthermore generalize the statements of \cref{thm:hlevel-folds,thm:hlevel1-folds} to functorial signatures, and prove them as \cref{thm:hlevel,thm:hlevel1}.
These results give an upper bound for the homotopy level of the types within a univalent $\L$-structure, and for the type of univalent $\L$-structures, respectively, in terms of the height of $\L$.

We begin with two auxiliary definitions which, for diagram signatures, were only sketched in \cref{sec:indisc-folds}.

\begin{definition}\label{def:ind}
  Let $L$ be a sharp exotype, $K:L$, $M: L \to \U$, and $a: M(K)$. We define the \defemph{indicator function of $K$}\index{function!indicator} to be
  \[ [K] \eqdef  \lambda x. \left(  x = K \right) : L \to \U\]
  and we define the
  function $\hat a: \prd{x: L} [K](x) \to M(x)$ by applying \cref{lem:sharp-id} to $a : M (K)$, so that $\hat a(K,\refl_K) = a$.
  When there is no risk of confusion, we write $\hat a$ as simply $a$.
  
Below we consider the pointwise disjoint union $M+[K]$ in $L \to \U$, the canonical injection $\iota_M: \prd{x: L} M(x) \to (M+[K])(x)$, and the induced function $\copair{1_M}{\hat{a}}:  \prd{x: L}  (M+[K])(x)\to M(x) $.
\end{definition}

\begin{definition}\label{def:partial-struc}
Consider $\L: \Sig(n+1)$, $K:\bottom{\L}$, $M: \Struc{\L}$, $a: \bottom{M}(K)$.
Define 
\[\partial_a M  \eqdef  (\derivcat{\L}{\copair{1_{\bottom{M}}}{\hat{a}}})^*\derivdia{M} : \Struc{\derivcat{\L}{\bottom{M} + [K]}}.\]
\end{definition}
 
This (along with the corresponding definition for diagram signatures in \cref{sec:indisc-folds}) is the first place we use our assumption that $\bottom{\L}$ is pointwise \emph{sharp}, rather than just \emph{cofibrant}.

Now we can define the type of indiscernibilities between objects within an $\L$-structure: 
 
\begin{definition}
\label{def:iso-within-a-structure} 
Consider $\L : \Sig(n+1)$, $K:\bottom{\L}$, $M: \Struc{\L}$, $a,b: \bottom{M}(  K)$. 
We define the type of \defemph{indiscernibilities from $a$ to $b$}\index{indiscernibility!in a functorial structure} to be
\[ 
(a \fiso b) \eqdef \sm{p \colon  \partial_a M \: = \: \partial_b M} \epsilon_a^{-1} \concat (\derivcat{\L}{\iota_{\bottom{M}}})^*p \concat \epsilon_b =_{\derivdia{M} = \derivdia{M}} \refl_{\derivdia{M}}, 
\]\nomenclature[\bbindisc2]{$a \fiso b$}{type of indiscernibilitites between $a$ and $b$ in a functorial structure}%
where $\epsilon_{x}$ is the concatenated identification
\begin{align*}
(\derivcat{\L}{\iota_{\bottom{M}}})^* \partial_x M  &\converts (\derivcat{\L}{\iota_{\bottom{M}}})^* (\derivcat{\L}{\langle 1_{\bottom{M}}, x \rangle})^*  \derivdia{M}\\ &= (\derivcat{\L}{  \langle 1_{\bottom{M}}, x \rangle \circ \iota_{\bottom{M}}})^* \derivdia{M}=  (\derivcat{\L}{  1_{\bottom{M}}})^* \derivdia{M} = \derivdia{M}.
\end{align*}
\end{definition}

\begin{remark}
  Using identification instead of levelwise equivalence of structures in \cref{def:iso-within-a-structure} is justified by \cref{prop:uni1}.
\end{remark}

\begin{lemma}\label{lem:iso-equivalent}
 The type of indiscernibilities $a \fiso b$ of \cref{def:iso-within-a-structure} is equivalent to the type
\[ 
\pushQED{\qed} 
 \sm{p \colon  \partial_a M \: = \: \partial_b M}  (\derivcat{\L}{\iota_{\bottom{M}}})^*p = \epsilon_a \concat \epsilon_b^{-1}.
 \qedhere
\popQED
\]
\end{lemma}

We now define \emph{univalence of $\L$-structures}. 
For this, we first need to define the canonical map from identifications to indiscernibilities.

\begin{definition}
\label{def:id-iso}
For $\L: \Sig(n+1)$, $K:\bottom{\L}$, $M: \Struc{\L}$, and $m : \bottom{M}(K)$, we define the \defemph{identity indiscernibility} $1 : m \fiso m$ as follows.\index{indiscernibility!identity}
Let $M: \Struc{\L}$.
For any $a: \bottom{M}(K)$, we have $\refl_{\partial_a M}: \partial_a M = \partial_a M$. Then
\begin{align*} \MoveEqLeft \epsilon_a^{-1} \concat (\derivcat{\L}{\iota_M})^*(\refl_{\partial_a M}) \concat \epsilon_a  
  \\
 &\steq \epsilon_a^{-1} \concat \refl_{(\derivcat{\L}{\iota_M})^*\partial_a M} \concat \epsilon_a  \\
 &=  \refl_{\derivdia{M}} \enspace ,
\end{align*}
where the second identification uses the groupoidal properties of types.
This gives the desired indiscernibility.
\end{definition}

\begin{definition}
Consider $\L: \Sig(n+1)$, $K:\bottom{\L}$, $M: \Struc{\L}$. 
For any $a,b: \bottom{M}(K)$, let $\idtoindisc_{a,b}: ( a = b) \to (a \fiso b) $ be the function which sends $\refl_a$ to the identity indiscernibility exhibited in \cref{def:id-iso}. 

We say that $M$ is \defemph{univalent at $K$}\index{structure!functorial!univalent} if for all $a,b : \bottom{M}(K)$, the map
\[\idtoindisc_{a,b} : (a = b) \to (a \fiso b)\]
is an equivalence.
\end{definition} 

\begin{definition}\label{def:functorial-univalence}
We define by induction what it means for a structure of a signature $\L:\Sig(n)$ to be \defemph{univalent}.

When $n \eqdef 0$, every structure $M: \Struc{\L}$ is univalent. 

When $n>0$, a structure $M: \Struc{\L}$ is univalent if $M$ is univalent at all $K:\bottom{\L}$ and $\derivdia{M}$ is univalent.

We denote by $\uStruc{\L}$ the type of univalent structures of $\L$.

Given a functorial theory $\T \steq (\L,T)$, a $\T$-model is \defemph{univalent}\index{model!of a functorial theory!univalent} if its underlying $\L$-structure is univalent.
\end{definition} 

\begin{lemma}
  Given a functorial signature $\L$,
  \begin{enumerate}
  \item for any $\L$-structure $M$, the exotype ``$M$ is univalent'' is a fibrant proposition,
  \item the exotype $\uStruc{\L}$ is fibrant, and
  \item  for $M,N:\uStruc{\L}$ we have
    \[
      \pushQED{\qed}
      (M =_{\uStruc{\L}} N) \simeq (M =_{\Struc{\L}} N) . \qedhere
      \popQED
    \]
  \end{enumerate}
\end{lemma}

Our first general observations about univalent structures give truncation bounds for their sorts and for the type of such structures.

\begin{theorem}\label{thm:hlevel}
Let $n>0$, $\L: \Sig(n)$, $M: \uStruc{\L}$, $K: \bottom{\L}$. Then $\bottom{M}(K)$ is an \nminustwo-type.
\end{theorem}

\begin{theorem}\label{thm:hlevel1}
Let $\L: \Sig(n)$. The type of univalent $\L$-struc\-tures is an \nminusone-type.
\end{theorem}

\begin{proof}[Proof of \cref{thm:hlevel,thm:hlevel1}]
Define the following exotypes.
\begin{align*}
P(n) &\define \prd{\L: \Sig(n+1)}{M: \uStruc{\L}}{K: \bottom{\L}} \istype{\nminusone}(\bottom{M}(K)) \\
Q(n) &\define \prd{\big(\substack{\M,\N:\Sig(n) \\ \alpha: \hom(\M  ,\N)}\big)}\prd{(N: \uStruc{\N})} \istype{(n-2)} (\alpha^* N = \alpha^* N)
\end{align*}
The exotype $P(n)$ is the statement of \cref{thm:hlevel}, and the exotype $Q(n)$ implies the statement of \cref{thm:hlevel1} by \cite[Thm.~7.2.7]{HTT}.
We prove $P(n)$ and $Q(n)$ simultaneously.

For $P(n)$, we need to show that $a = _{\bottom{M} K} b$ is an $(n-2)$-type for all $\L: \Sig(n+1), M: \uStruc{\L}, K: \bottom{\L}, a,b: {\bottom{M} K}$. But since $M$ is univalent, this type is equivalent to
\[ (a \fiso b) \equiv \sum_{e: \partial_a M = \partial_b M} \epsilon_a^{-1} \concat (\derivcat{\L}{\iota_M})^*p \concat \epsilon_b =_{\derivdia{M} = \derivdia{M}} \refl_{\derivdia{M}}.\] Thus, it will suffice to show that $\partial_a M = \partial_b M$ and $\epsilon_a^{-1} \concat (\derivcat{\L}{\iota_M})^*p \concat \epsilon_b =_{\derivdia{M} = \derivdia{M}} \refl_{\derivdia{M}}$ are $(n-2)$-types.

To show $P(0)$ and $Q(0)$ consider $\L: \Sig(1)$, $M: \uStruc{\L}$, $K: \bottom{\L}$, $a,b: {\bottom{M} K}$, $\M,\N:\Sig(0)$, $\alpha: \hom(\M  ,\N)$, $N: \uStruc{\N}$.
We have that $\derivdia{M}, \partial_a M,\partial_b M, \alpha^* N: \onetype$ so the types $\partial_a M = \partial_b M$, $\epsilon_a^{-1} \concat (\derivcat{\L}{\iota_M})^*p \concat \epsilon_b =_{\derivdia{M} = \derivdia{M}} \refl_{\derivdia{M}}$, and $ \alpha^* N = \alpha^* N$  are contractible. Thus, $P(0)$ and $Q(0)$ hold.

Suppose that $P(n)$ and $Q(n)$ hold. We first show $Q(n+1)$. Consider $\M,\N:\Sig(n+1), \alpha: \hom(\M  ,\N), N: \uStruc{\N}$. We have that 
\begin{align*} (\alpha^* N=\alpha^* N) 
&\simeq \Sigma_{e: \bottom{ (\alpha^* N)} = \bottom{ (\alpha^* N)}} (\alpha^*N) ' =  e_*(\alpha^* N)' \\
&\equiv \Sigma_{e: (\bottom{N} \circ \bottom{\alpha}) = (\bottom{N} \circ \bottom{\alpha})} (\derivcat{\alpha}{\bottom{N}})^* \derivdia{N} =  e_*(\derivcat{\alpha}{\bottom{N}})^* \derivdia{N}.
\end{align*}
Our inductive hypothesis $Q(n)$ ensures that $(\derivcat{\alpha}{\bottom{N}})^* \derivdia{N} = (\derivcat{\alpha}{\bottom{N}})^* \derivdia{N}$ is an $(n-2)$-type, and hence $(\derivcat{\alpha}{\bottom{N}})^* \derivdia{N} =  e_*(\derivcat{\alpha}{\bottom{N}})^* \derivdia{N}$  is an $(n-1)$-type by \cite[Thm.~7.2.7]{HTT}. 
It remains to show that $(\bottom{N} \circ \bottom{\alpha}) = (\bottom{N} \circ \bottom{\alpha})$ is an $(n-1)$-type. 
Note that $N$ is a univalent structure of an $(n+1)$-signature, and our inductive hypothesis $P(n)$ then implies that for all $K: \bottom{\N}$, the type $ \bottom{N}(K)$ is an $(n-1)$-type.
Then since $(\bottom{N} \circ \bottom{\alpha})$ is a function which takes values in $(n-1)$-types, we can conclude that $(\bottom{N} \circ \bottom{\alpha}) = (\bottom{N} \circ \bottom{\alpha})$ is an $(n-1)$-type \cite[Thm.~7.1.9]{HTT}. Thus, $Q(n+1)$ holds.

To show that $P(n+1)$ holds, consider $\L: \Sig(n+2), M: \uStruc{\L}, K: \bottom{\L}, a,b: {\bottom{M} K}$.
By \cite[Thm.~7.2.7]{HTT}, $Q(n+1)$ implies that $\partial_a M = \partial_b M$ and $\epsilon_a^{-1} \concat (\derivcat{\L}{\iota_M})^*p \concat \epsilon_b =_{\derivdia{M} = \derivdia{M}} \refl_{\derivdia{M}}$ are $(n-2)$-types. Therefore, $P(n+1)$ holds.
\end{proof}

\chapter{Equivalence of structures and the univalence principle}
\label{sec:hsip}

In this \lcnamecref{sec:hsip}, we define general notions of equivalence for structures, and prove our univalence principles.

Specifically, we start by defining ``split-surjective (weak) equivalences'' of $\L$-structures in \cref{def:vss}.
These are maps that are split-surjective up to the type of identifications.
Our first univalence principle, stated in \cref{thm:hsip}, characterizes the identification type $M = N$ of two $\L$-structures as the type of split-surjective equivalences from $M$ to $N$, whenever the source structure $M$ is univalent.

We also consider maps that are only \emph{essentially} split-surjective, that is, split-surjective up to indiscernibility, in \cref{def:eqv};
these are what we call simply ``equivalences''.
When the target structure $N$ is univalent, equivalences coincide with split-surjective equivalences; this is shown in \cref{thm:vss-to-eqv}.

Our ``main'' univalence principle, stated in \cref{thm:hsip2}, combines these two results, yielding a characterization of the type of identifications $M = N$, whenever both $M$ and $N$ are univalent, as the type of equivalences $M \simeq N$.

\begin{definition}\label{def:vss}
  Suppose $f:\hom_{\Struc{\L}}(M,N)$, where $M,N: \Struc{\L}$ and $\L: \Sig(n)$.
  We define what it means for $f$ to be a \defemph{split-surjective equivalence} by induction on $n$.
  
  If $n\eqdef 0$, then $f$ is always a split-surjective equivalence.\index{equivalence!of functorial structures!split-surjective}

  For $n>0$, $f$ is a split-surjective equivalence if
  \begin{enumerate}
  \item $\bottom{f}(K)$ is a split surjection for every $K:\bottom{\L}$, and
  \item $\derivdia{f}$ is a split-surjective equivalence.
  \end{enumerate}
  \defemph{Surjective weak equivalences} are defined similarly, but only requiring each $\bottom{f}(K)$ to be surjective.\index{equivalence!of functorial structures!surjective weak}
\end{definition}

As noted in \cref{sec:hsip-folds}, we are currently unable to prove our desired general result with surjective weak equivalences, so for the present we restrict to the split-surjective equivalences.
We write $\VSS(f)$ for the type ``$f$ is a split-surjective equivalence'', which in the inductive case is
\[ \VSS(f) \define \left(
\prd{K: \bottom{\L}}{y:\bottom{N}(K)} \sm{x:\bottom{M}(K)} (\bottom{f}(K)(x) = y) 
  \right) \times \VSS(\derivdia{f}),\]  
and $(M \strucequiv N) \eqdef \sm{f : \hom_{\Struc{\L}}(M,N)}\VSS(f)$ for the type of split-surjective equivalences.\nomenclature[\cequiv2]{$M \strucequiv N$}{type of split-surjective equivalences between functorial structures $M$ and $N$}

\begin{lemma}
  If $\L$ is a diagram signature, then a morphism of $\L$-structures is a split-surjective equivalence (resp.\ surjective weak equivalence) in the sense of \cref{def:vss} if and only if its corresponding morphism of Reedy fibrant diagrams is a split-surjective equivalence (resp.\ surjective weak equivalence) in the sense of \cref{def:vss-folds}.
\end{lemma}
\begin{proof}
  Just like \cref{thm:lvle-folds}, using (split-)surjective maps in place of equivalences.
\end{proof}

\begin{definition}\label{def:isotovss}
  Let $f : \hom_{\Struc{\L}}(M, N)$; we define
  \[U_f: \isIso(f) \to \VSS(f)\]
  by induction on $n$.
  If $n \eqdef 0$, $U_f$ is the identity function on $\onetype$.
  For $n>0$, we use that any equivalence of types is a split surjection, and the inductive hypothesis.
  Let $\lvletovss_{M,N} \eqdef (1,U) : (M \leqv N) \to (M \strucequiv N)$.
\end{definition}

\begin{definition}
  For $\L: \Sig(n)$ and $M,N: \Struc{\L}$ we define 
  \[\idtovss \eqdef \lvletovss \circ \idtolvle : (M = N) \to (M \strucequiv N).\]\nomenclature[idtosse]{$\idtovss$}{function from identifications to split-surjective equivalences of functorial structures}
\end{definition}

Our first univalence principle states that if $M$ is univalent, then $\idtovss_{M,N}$ is an equivalence.
It uses the following lemma.

 \begin{lemma} \label{lem:injwrtiso}
Let $\L: \Sig(n+1)$, $M, N : \Struc{\L}$, $\bottom{f}: \bottom{M} \to \bottom{N}$, and $e: \derivdia{M} = (\derivcat{\L}{\bottom{f}})^* \derivdia{N}$. Then for $x,y: \bottom{M}(K)$, an indiscernibility $\bottom{f}x \fiso \bottom{f} y$ produces an indiscernibility $x \fiso y$.
\end{lemma}
\begin{proof}
By \pathinduction on $e$, we may assume $\derivdia{M} \converts (\derivcat{\L}{\bottom{f}})^* \derivdia{N}$.

Consider the following diagram whose cells commute up to $\steq$ or $=$, as pictured.
\begin{equation}
\begin{tikzcd}[column sep=huge]
\bottom{M} \ar[dr,phantom,"\steq" description]  \ar[rr,phantom,"\steq", bend left=15]  \ar[d,"\bottom{f}"] \ar[r,"\iota_{\bottom{M}}"] \ar[rr, bend left=20, shift left=1,"1_{\bottom{M}}",near start] & \bottom{M} + [K] \ar[d,"\bottom{f}+1"] \ar[r,
"{\copair 1 x}"] & \bottom{M} \ar[d, "\bottom{f}"]\\
\bottom{N}  \ar[rr,phantom,"\steq", bend right=15] \ar[r,"{\iota_{\bottom{N}}}"] \ar[rr, bend right=20, shift right=1, "1_{\bottom{N}}"', near start]
 & \bottom{N} + [K] \ar[ur,phantom,"=" description]  \ar[r,"\copair{1}{\bottom{f}(K)x}"] & \bottom{N} 
\end{tikzcd} \label{diag:lemma}
\end{equation}
This diagram commutes 2-dimensionally, which is to say that the ``pasting'' of all four displayed identities is exo-equal to the strict
equality $\bottom{f}\circ 1_{\bottom{M}} \steq 1_{\bottom{N}} \circ \bottom{f}$.
Applying the composite exo-functor $\Struc{\derivcat{\L}{-}}$, we obtain:
\begin{equation}
\begin{tikzcd}[column sep=large]
\Struc{\derivcat{\L}{\bottom{M}}} \ar[dr,phantom,"\steq \scriptstyle(\alpha)" description]  \ar[rr,phantom,"\steq\scriptstyle(\epsilon_x)", bend left=15]  \ar[d,<-,"(\derivcat{\L}{\bottom{f}})^*"] \ar[r,<-,"(\derivcat{\L}{\iota_{\bottom{M}}})^*"] \ar[rr,<-, bend left=20, shift left=1,"1_{\Struc{\derivcat{\L}{\bottom{M}}}}",near start] & \Struc{\derivcat{\L}{\bottom{M} + [K]}} \ar[d,<-,"{(\derivcat{\L}{\bottom{f}+1})^*}"] \ar[r,<-,"(\derivcat{\L}{{\copair 1 x}})^*"] & \Struc{\derivcat{\L}{\bottom{M}}} \ar[d,<-, "(\derivcat{\L}{\bottom{f}})^*"]\\
\Struc{\derivcat{\L}{\bottom{N}}}  \ar[rr,phantom,"\steq\scriptstyle (\epsilon_{\bottom{f}x})", bend right=15] \ar[r,<-,"(\derivcat{\L}{\iota_{\bottom{N}}})^*"'] \ar[rr, <-,bend right=20, shift right=1, "1_{\Struc{\derivcat{\L}{\bottom{N}}}}"', near start]
 & \Struc{\derivcat{\L}{\bottom{N} + [K]}} \ar[ur,phantom,"\mathrlap{=\scriptstyle(\beta_x)}" description]  \ar[r,<-,"(\derivcat{\L}{\copair{1}{\bottom{f}x}})^*"'] & \Struc{\derivcat{\L}{\bottom{N}}}
\end{tikzcd} \label{diag:lemma2}
\end{equation}
which commutes in the same way.
Moreover, the upper and lower exo-equalities in this diagram are $\epsilon_x$ and $\epsilon_{\bottom{f}x}$ respectively; we call the others $\alpha$ and $\beta_x$. 

We have an analogous diagram for $y$, in which the left-hand square $\alpha$ is the same.

Then since it is the case that $ \partial_{\bottom{f}(K) x} N \converts (\derivcat{\L}{\copair{1}{\bottom{f}(K)x}})^* \derivdia{N}$,  $\derivdia{M} \converts (\derivcat{\L}{\bottom{f}})^* \derivdia{N}$, and also $\partial_{x} M \converts (\derivcat{\L}{\copair1x })^* \derivdia{M}$, we have an identification
\begin{equation*}
\beta_x N : (\derivcat{\L}{\bottom{f} +1 })^*\partial_{\bottom{f}(K) x} N = \partial_x M. 
\end{equation*}
The same can be shown for $y$.

Consider an indiscernibility $\bottom{f}x \fiso \bottom{f} y$ which consists, by \cref{lem:iso-equivalent}, of (1) an identification $i : \partial_{\bottom{f}x} N = \partial_{\bottom{f}y} N$ and (2) an identification 
$ j $ between $(\derivcat{\L}{\iota_{\bottom{N}}})^*i$ and the concatenation
\[
  \begin{tikzcd}
    (\derivcat{\L}{\iota_{\bottom{N}}})^*(\derivcat{\L}{\copair{1}{\bottom{f}x}})^*\derivdia{N}
    \ar[r,equals,"{\epsilon_{\bottom{f}x}}"] &
    \derivdia{N}
    \ar[r,equals,"{\epsilon_{\bottom{f}y}^{-1}}"] &
    (\derivcat{\L}{\iota_{\bottom{N}}})^*(\derivcat{\L}{\copair{1}{\bottom{f}y}})^*\derivdia{N}
  \end{tikzcd}
\]
(which is an exo-equality, though $i$ is not).

 We need to construct an indiscernibility $x \fiso y$ which consists of (1) an identification $k : \partial_{x} M = \partial_{y} M$ and (2) an identification $(\derivcat{\L}{\iota_{\bottom{M}}})^*k  = \epsilon_{ x}^{} \concat \epsilon_{ y}^{-1}$.

The first component, $k$, of our desired indiscernibility $x\fiso y$ is the following concatenation of labeled identifications:
\[
  \begin{tikzcd}[row sep=tiny,column sep=large]
    (\derivcat{\L}{\copair{1}{x}})^* (\derivcat{\L}{\bottom{f}})^* \derivdia{N}
    \ar[r,equals,"{\beta_x}"] &
    (\derivcat{\L}{\bottom{f}+1})^* (\derivcat{\L}{\copair{1}{\bottom{f}x}})^* \derivdia{N}\\
    \phantom{(\derivcat{\L}{\copair{1}{x}})^* (\derivcat{\L}{\bottom{f}})^* \derivdia{N}}
    \ar[r,equals,"{(\derivcat{\L}{\bottom{f}+1})^*i}"] &
    (\derivcat{\L}{\bottom{f}+1})^* (\derivcat{\L}{\copair{1}{\bottom{f}y}})^* \derivdia{N}\\
    \phantom{(\derivcat{\L}{\copair{1}{x}})^* (\derivcat{\L}{\bottom{f}})^* \derivdia{N}}
    \ar[r,equals,"{\beta_y^{-1}}"] &  (\derivcat{\L}{\copair{1}{y}})^* (\derivcat{\L}{\bottom{f}})^* \derivdia{N}
  \end{tikzcd}
\]

Now we need $ (\derivcat{\L}{\iota_{\bottom{M}}})^*k  =  \epsilon_{ x}^{} \concat \epsilon_{ y}^{-1}$. 
Consider the commutative diagram in \cref{diag:biglemma} (on page~\labelcpageref{diag:biglemma}) where straight lines denote exo-equalities, squiggly lines denote identifications, and double (squiggly) lines denote identifications between identifications. 
The 2-dimensional identification labeled $\nu$ arises from naturality, while those labeled $\sigma$ arise from the 2-di\-men\-sio\-nal commutativity of Diagram~\eqref{diag:lemma2}.
The concatenation of the three top horizontal identifications in \cref{diag:biglemma} is $ (\derivcat{\L}{\iota_{\bottom{M}}})^*k$.
Thus, \cref{diag:biglemma} exhibits an identification of this with $\epsilon_{x}^{} \concat \epsilon_{y}^{-1}$.
\end{proof}
 
\begin{figure}

  \centering
  \begin{sideways}
  
\begin{tikzcd}[column sep=large]
(\derivcat{\L}{\iota_{\bottom{M}}})^* \partial_x M  
\ar[ddrr,bend right,"\epsilon_x^{}"',no head,""{name=E}]
\ar[r,"(\derivcat{\L}{\iota_{\bottom{M}}})^*\beta_x ",squiggly, no head] 
\ar[r,phantom,""{above,name=A},shift left=4] 
\ar[r,phantom,""{below, name=B},shift right=4]
& (\derivcat{\L}{\iota_{\bottom{M}}})^* (\derivcat{\L}{\bottom{f} +1 })^* \partial_{\bottom{f}x} N
\ar[rr,"(\derivcat{\L}{\iota_{\bottom{M}}})^*(\derivcat{\L}{\bottom{f} +1 })^*i ",""{below,name=C},squiggly, no head]
\ar[d, no head,"\alpha "]
&& (\derivcat{\L}{\iota_{\bottom{M}}})^*(\derivcat{\L}{\bottom{f} +1 })^*\partial_{\bottom{f}y} N
\ar[r,"(\derivcat{\L}{\iota_{\bottom{M}}})^*(\beta_y )^{-1}",squiggly, no head]
\ar[r,phantom,""{above,name=A},shift left=4] 
\ar[r,phantom,""{below, name=B},shift right=4]
 \ar[d, no head,"\alpha "]
& (\derivcat{\L}{\iota_{\bottom{M}}})^* \partial_y M 
\\
&|[alias=X]|  (\derivcat{\L}{{\bottom{f}}})^*(\derivcat{\L}{\iota_{\bottom{N}}})^* \partial_{\bottom{f}x} N 
\ar[rr, "(\derivcat{\L}{{\bottom{f}}})^*(\derivcat{\L}{\iota_{\bottom{N}}})^*i"{name=D},""{below,name=DD}, squiggly, no head]
\ar[from=C, to=D,squiggly,shift left=0.5,no head,"\nu"]
\ar[from=C, to=D,squiggly,shift right=0.5, no head]
\ar[rd, "(\derivcat{\L}{{\bottom{f}}})^*( \epsilon_{\bottom{f}x}^{}  )"{description}, no head, bend right=10]
\ar[from=DD, to=Z,squiggly,shift right=0.5,no head]
\ar[from=DD, to=Z,squiggly,shift left=0.5, no head,"(\derivcat{\L}{{\bottom{f}}})^* j"]
&& |[alias=Y]|  (\derivcat{\L}{{\bottom{f}}})^*(\derivcat{\L}{\iota_{\bottom{N}}})^*\partial_{\bottom{f}y} N 
\\
&& |[alias=Z]| M \ar[uurr,bend right,"\epsilon_y^{-1}"',no head,""{name=F}]
\ar[ur, "(\derivcat{\L}{{\bottom{f}}})^*( \epsilon_{\bottom{f}y}^{-1}  )"{description}, no head, bend right=10]
\ar[from=E, to=X, equal,"\sigma"']
\ar[from=F, to=Y, equal,"\sigma"]
\end{tikzcd} 
  \end{sideways}
\caption{Diagram for proof of \cref{lem:injwrtiso}}

\label{diag:biglemma}

\end{figure}

\begin{theorem}[\defemph{Univalence principle, split-surjective case}]
\label{thm:hsip}\index{univalence principle!for functorial structures!split-surjective case}
Consider $\L: \Sig(n)$ and $M,N: \Struc{\L}$ such that $M$ is univalent.
The morphism $\idtovss: (M = N) \to (M \strucequiv N)$ is an equivalence.
\end{theorem}

\begin{proof}

It suffices to show that each $U_f$ of \cref{def:isotovss} is an equivalence. 
We proceed by induction on $n$.
When $n \eqdef 0$, each $U_f$ is a endofunction on $\onetype$, and so is an equivalence.

When $n > 0$,
we first construct a map $F_f: \VSS(f) \to  \isIso(f)$.
Consider an element of $\VSS(f)$: a right inverse $s(K)$ of $\bottom{f}(K)$ for each $K : \bottom{\L}$, and $s' : \VSS(\derivdia{f})$.
Since $\derivdia{M}$ is univalent, the inductive hypothesis for $s'$ implies $\derivdia{f}$ is a levelwise equivalence; thus it remains to show each $\bottom{f}(K)$ is an equivalence.

Since $s(K)$ is a right inverse of $\bottom{f}(K)$, it remains to show that we have $s(K)\bottom{f}(K) m = m$ for any $m: \bottom{M}(K)$. We have $\bottom{f}(K)s(K)\bottom{f}(K) m = \bottom{f}(K)m$ and thus $\bottom{f}(K)s(K)\bottom{f}(K) m \fiso \bottom{f}(K)m$.
We have already shown that $\derivdia{f}$ is a levelwise equivalence $\derivdia{M} \leqv ({\bottom{f}})^* \derivdia{N}$, so by \cref{prop:uni1}, we get $\derivdia{M} = (\derivcat{\L}{\bottom{f}})^* \derivdia{N}$.
Thus, by \cref{lem:injwrtiso}, we have $s(K)\bottom{f}(K) m \fiso m$; and since $M$ is univalent this yields $s(K)\bottom{f}(K) m = m$.

Thus, given our $(\lambda K.s(K), s'):\VSS(f)$, we have constructed an element of $\isIso(f)$; this defines $F_f: \VSS(f) \to  \isIso(f)$.
Since $\isIso(f)$ is a proposition (by \cref{lem:isisoprop}), $F_f U_f = 1$.
Moreover, we constructed $F_f$ and $U_f$ such that $U_f F_f = 1$.\footnote{Since we showed that $\bottom{f}(K)$ was an equivalence by making $s(K)$ a homotopy inverse of it, and $U_f$ remembers not just the inverse map but one of the homotopies, we technically have to use here the fact that a homotopy inverse of a function $g$ can be enhanced to an element of $\isEquiv(g)$ while changing at most one of the constituent homotopies.}
Hence, $U_f : \isIso(f) \to \VSS(f)$ is an equivalence.

Thus, the function $\lvletovss_{M,N} : (M \leqv[\L] N) \to (M \strucequiv N  )$ is also an equivalence. Using \cref{prop:uni1}, we find then that $\idtovss: (M = N) \to (M \strucequiv N)$ is an equivalence.
\end{proof}

We now move on to consider equivalences that are only \emph{essentially} surjective.
Here we have to be careful in the inductive step, because when considering $f:\hom_{\Struc{\L}}(M,N)$ we want all the indiscernibilities to lie in $N$ and its derivatives directly, \emph{not} in their pullbacks to derivatives at $M$.
This forces us to define a somewhat more general notion.

For $a,b:\bottom{M}(K)$, we write $a\fiso^M_K b$ instead of $a\fiso b$ if needed to eliminate ambiguity.

\begin{definition}\label{def:eqv}
  Consider $\L,\M:\Sig(n)$ as well as $\alpha:\hom_{\Sig(n)}(\L,\M)$, let $M: \Struc{\L}$ and $N:\Struc{\M}$, and let $f:\hom_{\Struc{\L}}(M,\alpha^* N)$.
  We define what it means for $f$ to be an \defemph{equivalence relative to $\alpha$} by induction on $n$.

  If $n\eqdef 0$, then $f$ is always an equivalence relative to $\alpha$.\index{equivalence!of functorial structures!relative}

  For $n>0$, $f$ is an equivalence relative to $\alpha$ if
  \begin{enumerate}
  \item For all $K:\bottom{\L}$ and $y:\bottom{N}(\bottom{\alpha}(K))$, we have a specified $x:\bottom{M}(K)$ and indiscernibility $\bottom{f}(x) \fiso^N_{\bottom{\alpha}(K)} y$.
  \item The morphism
    \begin{align*}
      \derivdia{f}: &\hom_{\Struc{\derivcat{\L}{\bottom{M}}}}(\derivdia{M},( \derivcat{\L}{\bottom{f}})^* \derivdia{(\alpha^* N)})
      \\
      \converts &\hom_{\Struc{\derivcat{\L}{\bottom{M}}}}(\derivdia{M},( \derivcat{\L}{\bottom{f}})^* (\derivcat{\alpha}{\bottom{N}})^* \derivdia{N})
    \end{align*}
   is an equivalence relative to the composite
    \[ \derivcat{\L}{\bottom{M}} \xrightarrow{\derivcat{\L}{\bottom{f}}} \derivcat{\L}{\bottom{N}\circ \bottom{\alpha}} \xrightarrow{\derivcat{\alpha}{\bottom{N}}} \derivcat{\M}{\bottom{N}}. \]
  \end{enumerate}
  \defemph{Relative weak equivalences}\index{equivalence!of functorial structures!relative weak} are defined similarly, but requiring only 
  \[ \Big\Vert \sm{x:\bottom{M}(K)} (\bottom{f}(x) \fiso^N_{\bottom{\alpha}(K)} y) \Big\Vert \] for each $K,y$.

  An unadorned \defemph{equivalence} means an equivalence relative to $\alpha\eqdef 1_{\L}$.
  We write $\streqv_\alpha(f)$ for the type ``$f$ is an equivalence relative to $\alpha$'', which in the inductive case means
  \begin{align*}
     & \streqv_\alpha(f) \eqdef \\
     & \left(\prd{K:\bottom{\L}}{y:\bottom{N}(\bottom{\alpha}(K))} \sm{x:\bottom{M}(K)} \left(\bottom{f}(K)(x)  \fiso^N_{\bottom{\alpha}(K)} y\right)\right) \times \streqv_{\derivcat{\alpha}{} \circ \bottom{f}} (\derivdia{f})
   \end{align*}
  and $(M\simeq N) \eqdef \sm{f:\hom_{\Struc{\L}}(M,N)} \streqv_1(f)$ for the type of equivalences.\nomenclature[\equiv2]{$M\simeq N$}{type of equivalences between functorial structures $M$ and $N$}
\end{definition}

\begin{remark}
  Importantly, $\bottom{f}(x) \fiso^N_{\bottom{\alpha}(K)} y$ is distinct from $\bottom{f}(x) \fiso^{\alpha^*N}_{K} y$, even though $\bottom{(\alpha^* N)}(K) \converts \bottom{N}(\bottom{\alpha}(K))$ by definition.

  For instance, consider the diagram signatures of \cref{eg:prop,eg:set} for propositions and sets, respectively.
  We have a morphism $\alpha$ of signatures from the former to the latter.
  \[
    \begin{tikzcd}
      1
      &
      &
      E \ar[d, shift left] \ar[d, shift right]
      \\
      0
      &
      P \ar[r, mapsto, "\alpha"]
      &
      X
    \end{tikzcd}
  \]
  Given a structure $N$ for sets, and $x,y : NX$, then $x \fiso^N_{\bottom{\alpha}(P)} y$ is the same as $E(x,y)$, whereas $x \fiso^{\alpha^*N}_{P} y$ is $\onetype$.
\end{remark}

\begin{lemma}
  If $\L$ is a diagram signature, then a morphism of $\L$-structures is an equivalence (resp.\ weak equivalence) in the sense of \cref{def:eqv} if and only if its corresponding morphism of Reedy fibrant diagrams is an equivalence (resp.\ weak equivalence) in the sense of \cref{def:eqv-folds}.
\end{lemma}
\begin{proof}
  This is mostly just like \cref{thm:lvle-folds}, but for the induction we need a relative version of \cref{def:eqv-folds}.
  We leave the details to the reader.
\end{proof}

\begin{lemma}\label{thm:vss-to-eqv}
  For $f:\hom_{\Struc{\L}}(M,\alpha^* N)$, we have a map
  \[\VSS(f) \to \streqv_\alpha(f),\]
  which is an equivalence if $N$ is univalent.
\end{lemma}
\begin{proof}
  By induction on $n$.
  When $n\eqdef 0$, both are $\onetype$.
  For $n>0$, the desired map consists of the inductively defined $\VSS(\derivdia{f}) \to \streqv_{\derivcat{\alpha}{\bottom{N}}\circ \derivcat{\L}{\bottom{f}}}(\derivdia{f})$ together with a morphism
  \begin{multline*}
  \left(\prd{K: \bottom{\L}}{y:\bottom{N}(\bottom{\alpha}(K))} \sm{x:\bottom{M}(K)} \left(\bottom{f}(K)(x) =_{\bottom{N}(\bottom{\alpha}(K))} y\right)\right)\\
    \to
    \left(\prd{K:\bottom{\L}}{y:\bottom{N}(\bottom{\alpha}(K))}\sm{x:\bottom{M}(K)} \left(\bottom{f}(K)(x) \fiso^N_{\bottom{\alpha}(K)} y\right)\right)
  \end{multline*}
  that is simply induced by $\idtoindisc_{\bottom{f}(K)(x),y}$.
  The latter is an equivalence when $N$ is univalent by definition, as is the inductively defined map since $\derivdia{N}$ is univalent.
  (This last step would fail if we worked only with absolute equivalences, since $\alpha^*N$ can fail to be univalent even if $N$ is so.)
\end{proof}

\begin{theorem}[\defemph{Univalence principle}]\label{thm:hsip2}\index{univalence principle!for functorial structures}
  Consider $\L: \Sig(n)$ and $M,N: \Struc{\L}$ such that $M$ and $N$ are both univalent.
  The canonical morphism
  \[\idtoeqv: (M = N) \to (M \simeq N)\]
  is an equivalence.\nomenclature[idtoeqv]{$\idtoeqv$}{function from identifications to equivalences of functorial structures}
\end{theorem}
\begin{proof}
  Combine \cref{thm:hsip,thm:vss-to-eqv}.
\end{proof}

As in the diagram case, this implies:

\begin{corollary}
  Any $\L$-axiom $t$ is invariant under equivalence of univalent $\L$-structures:
  given univalent $\L$-structures $M$, $N$ and an equivalence $M \simeq N$, then $t(M) \leftrightarrow t(N)$.
\end{corollary}

One might also hope for a univalence principle for \emph{weak} equivalences, i.e., an analogue of~\cite[Lemma 6.8]{AKS13}.\index{univalence principle!for weak equivalences}
A natural way to try to prove this would be by enhancing \cref{lem:injwrtiso} to say that some induced map ``$f : (x\fiso y)\to (f x \fiso f y)$'' is an equivalence, so that a weak equivalence between univalent structures would be an embedding and hence an equivalence.
Unfortunately, as we have seen in \cref{eg:indis-notpres,eg:premonoidal}, an arbitrary morphism between structures does not induce any such map on types of indiscernibilities, even when it is an identity on derived structures as in \cref{lem:injwrtiso}.

\chapter{Examples of functorial structures}
\label{sec:egs-higherorder}

Functorial signatures are significantly more general than diagram signatures.
As we saw in \cref{eg:fol}, exofinite height-2 diagram signatures are essentially the same as signatures for multi-sorted first-order logic.
However, height-2 functorial signatures can represent any signature in multi-sorted \emph{higher}-order logic.
Before describing such a representation in general, we give two classes of examples to illustrate the idea.

Throughout this chapter, we assume the \emph{propositional resizing axiom}\index{axiom!propositional resizing} mentioned in \cref{sec:logic}.
Since the type $\PropU$ is then independent of the universe $\U$, up to equivalence, we write it as simply $\Prop$ and assume that it lies in all universes.

\begin{example}[$T_0$-spaces]\label{eg:topological-spaces}\index{space!$T_0$}\index{space!topological}
  Since a topology is a structure on one underlying set, to describe a structure for topological spaces it suffices to consider height-2 signatures with $\bottom{\L}\eqdef \onetype$, with $\derivdia{\L}{} : \U\to\U$ remaining to be specified.
  A first guess might be $\derivcat{\L}{M} \eqdef (M \to \Prop)$, so that an $\L$-structure would be a type $M$ with a predicate on its ``type of subsets'' $M \to \Prop$ representing ``is open''.
  Unfortunately, this is not a covariant exo-functor.
  We can make it covariant via direct images (using propositional truncation), but this is not \emph{strictly} exo-functorial.\index{powerset}

  One way around this problem is to introduce a separate sort for open sets, with the following diagram signature:
  
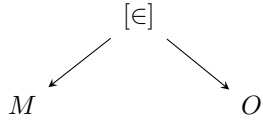
\[
    \begin{tikzcd}
      & {[\in]} \ar[dl] \ar[dr] \\
      M & & O
    \end{tikzcd}
  \]
  where $M$ represents the set of points, $O$ the set of opens, and $[\in]$ the membership relation; we write $[\in](x,w)$ infix as $x\in w$.
  (Note that there are \emph{no} ``equality'' relations.)
  We assert the usual axioms of a topology, e.g., for all $u,v:O$ there exists a $w:O$ such that for all $x:M$ we have $(x\in w) \leftrightarrow (x\in u)\land (x\in v)$.
  Univalence at $[\in]$ makes it a proposition.
  An indiscernibility $u\fiso v$ for $u,v:O$ then asserts that $(x\in u) \leftrightarrow (x\in v)$ for all $x:M$; thus in a univalent structure an element $u:O$ is uniquely determined by a subset of $M$.
  Similarly, an indiscernibility $x\fiso y$ for $x,y:M$ asserts that $(x\in u) \leftrightarrow (y\in u)$ for all $u:O$, which if $O$ is a topology amounts to saying that the topology is $T_0$, i.e., no two distinct points belong to the same sets.

  However, a morphism of topological spaces, regarded as structures for this signature, is a function on sets and a function on open sets that preserves the membership relation.
  In other words, we have $f:M\to N$ together with, for each open subset $u$ of $M$, an open subset $f_!(u)$ of $N$, such that if $x\in u$ then $f(x)\in f_!(u)$.
  This is quite different from the usual notion of continuous map, and does not even coincide with the standard notion of open map (that would require that $y\in f_!(u)$ if and only if $y = f(x)$ for some $x\in u$).

  A different way to obtain covariant exo-functoriality is to use the \emph{double}-powerset functor $M\mapsto ((M\to \Prop) \to \Prop)$.\index{powerset!double}
  The covariant functorial action of a function $f:M\to N$ takes a set of subsets $\mathfrak{x}$ to the set of all subsets $U$ of $N$ such that $f^{-1}(U)\in \mathfrak{x}$.

  In this case we need a definition of topological spaces that refers to sets of subsets instead of individual subsets.
  Perhaps the simplest approach is to take $\derivcat{\L}{M} \eqdef ((M\to \Prop) \to \Prop)$, so that a structure consists of a type $M$ together with a family of sets of subsets of $M$.
  We regard a topological space $M$ as such a structure by equipping it with the family of all supersets of the set of open subsets, i.e., a predicate that holds of $\mathfrak{x}$ just when $U\in \mathfrak{x}$ for every open subset $U$ of $M$.
  We can characterize the structures arising in this way by axioms asserting that the family of sets of subsets has a least element, and the elements of that least element satisfy the axioms of a topology.

  In this representation, a morphism of structures between two topological spaces is a function $f:M\to N$ such that if $\mathfrak{x}$ contains all opens in $M$, then its image under $f$ contains all opens in $N$, which is to say that $f^{-1}(U) \in \mathfrak{x}$ for all opens $U$ in $N$.
  This is equivalent to saying that $f^{-1}(U)$ is open in $M$ for all opens $U$ in $N$, i.e., that $f$ is continuous.

  Of course, univalence at rank 1 says that this predicate on sets of subsets is a proposition.
  For $x,y:M$, an indiscernibility $x\fiso y$ is the assertion that for a set $\mathfrak{x}$ of subsets of $M+\onetype$, its image under $\copair{1_M}{x}$ contains all opens if and only if its image under $\copair{1_M}{y}$ does.
  Such a $\mathfrak{x}$ is determined by two sets of subsets of $M$, say $\mathfrak{x} = (\mathfrak{x}_1 + \emptyset) \cup (\mathfrak{x}_2 + \{\ttt\})$, and its image under $\copair{1_M}{x}$ consists of those sets in $\mathfrak{x}_1$ that don't contain $x$ and those sets in $\mathfrak{x}_2$ that do contain $x$.
  Thus, $x\fiso y$ is equivalent to saying that any open set $U$ contains $x$ if and only if it contains $y$.
  Hence $M$ is univalent just when it is $T_0$, as before.
  In addition, a continuous map $f:M\to N$ between not-necessarily univalent structures is an equivalence if surjective up to indiscernibility --- i.e., for any $y\in N$ there is an $x\in M$ such that $f(x)$ and $y$ belong to the same open sets --- and moreover $M$ has the topology induced from $N$.

  Another way to present topological spaces using double-powersets is in terms of a \emph{convergence} relation between filters (which are sets of subsets) and points.\index{convergence}\index{filter}
  This suggests a different signature with
  \[\derivcat{\L}{M} \eqdef ((M\to \Prop) \to \Prop)\times M\]
  so that a structure is a set $M$ equipped with a relation between sets-of-subsets and points.
  When regarding a topological space as a structure for this signature, we could require that this relation holds of $(\mathfrak{x},x)$ only when $\mathfrak{x}$ is itself a filter converging to $x$, or when $\mathfrak{x}$ \emph{contains} some filter converging to $x$.
  With either choice, we can characterize the structures arising from topological spaces by extending the usual axioms for a topology in terms of convergence.

  Note that covariant functoriality of the double-powerset specializes to the direct image of filters.
  Thus, under either representation, the $\L$-structure morphisms between topological spaces will be functions that preserve convergence, a property which is equivalent to continuity.

  Univalence of such a structure means that convergence is a proposition, that $M$ is a set, and that two points $a, b$ are identified if exactly the same filters converge to them.
  For topological spaces, the latter is equivalent to saying that the principal filter at point $a$ converges to $b$ and vice versa, which is an equivalent way of saying the space is $T_0$.
  Finally, the equivalences are again the continuous maps that are surjective up to indiscernibility and give their domain the induced topology.
\end{example}

Other topological structures such as uniform spaces and proximity spaces, with the usual morphisms between them, can be represented in a similar way.

\begin{example}[Suplattices, DCPOs]\label{eg:suplattices}\index{suplattice}\index{DCPO}
 A \defemph{suplattice} is a partially ordered set that has joins of all subsets, or equivalently of all indexed families.
 One suitable signature for the theory of suplattices is given as follows.
 Consider the height-2 signature $\L$ with $\bottom{\L} \eqdef \onetype$, and with
 \[\derivcat{\L}{M} \eqdef (M \times M) + ((\tsm{A : \Set} (A \to M)) \times M);\]
 this assignment is covariantly exo-functorial.
 Here, the first summand $M \times M$ stands for the partial ordering---$(m,n)$ meaning $m \leq n$---whereas the second summand denotes suprema:
 $(X,s)$ holds if and only if $s$ is a supremum of the family $X$ of elements of $M$.
 We assert axioms saying that $\leq$ is a preorder, that $m$ is indeed the supremum of $X$, and that there exists some supremum of any family $X$.
 
 Given a structure $M$ for this signature, two elements $m_1, m_2 : M$ of the carrier type of $M$ are indiscernible if $m_1 \leq m_2$ and $m_2 \leq m_1$.
 The facts that $m_1$ and $m_2$ are suprema of exactly the same families $X$, and are interchangeable as elements of a family $X$ without altering its suprema, are then automatic.
 Univalence at $\bottomlevel$ hence means that $M$ is a set, and that the preorder $\leq$ on $M$ is antisymmetric, like in \cref{eg:pre-po-sets}.
 A morphism of structures is a sup-preserving morphism of preorders (in the sense that it takes any supremum to some other supremum); it is an equivalence if it is (split) surjective up to isomorphism and reflects the preorder (and hence also suprema of families).
 
 Directed-complete partial orders (DCPOs) can be formulated similarly, by restricting the families $(A,f) : \sm{A : \Set} A \to M$ of which we take suprema to directed ones, i.e., those such that for any $i, j : A$, there is $k : A$ such that $f(i)\leq f(k)$ and $f(j) \leq f(k)$.
 This restriction can't be made in the signature, but we can assert as an axiom that the directed families are exactly those that have suprema.
 It should also be possible to omit to include the partial ordering explicitly, since $m_1\le m_2$ holds precisely when $m_2$ is a supremum of the doubleton $\{m_1,m_2\}$.

 However, this signature does have the disadvantage that it is ``larger'' than its structures, in the sense of universe level.
 For it to correctly represent the suplattices in some universe $\U$, the type $\Set$ appearing in the definition of $\L$ must consist of all the sets in $\U$, with the consequence that $\L$ itself lives in the next higher universe $\U'$.
 In particular, this implies that we cannot use the same signature $\L$ to describe suplattices in all universes, but rather we need a different $\L$ for each ``size'' of suplattice.
 This creates no actual problems for our results in this \paperorbook, but it might become problematic when constructing univalent completions.

 The representations of \cref{eg:topological-spaces} avoid this problem due to our assumption of propositonal resizing, since they use only $\Prop$ rather than $\Set$.
 With this in mind, we can give a different presentation of suplattices using a similar double-powerset encoding, which also remains in the same universe under the assumption of propositional resizing.\index{axiom!propositional resizing}\index{powerset!double}
 Namely, we use the height-2 signature with $\bottom{\L} \eqdef \onetype$, and with
 \[\derivcat{\L}{M} \eqdef (M \times M) + (((M\to \Prop)\to \Prop)\times M).\]
 The first summand stands for the partial ordering, as before, and for the second summand we assert that $(X,s)$ holds if and only if $X$ is of the form $\{ B \subseteq M \mid A \subseteq B \}$ for some $A\subseteq M$, and $s$ is a supremum of $A$.
 We again assert suitable axioms.
 This representation is chosen to ensure the correct morphisms of structures: if $f:M\to N$ is a morphism of carriers, then the induced map on double powersets takes $\{ B \subseteq M \mid A \subseteq B \}$ to $\{ C \subseteq N \mid f_!(A) \subseteq C \}$, where $f_!(A)$ is the image of $A$ under $f$.
 Thus, a morphism of structures is again a sup-preserving map of preorders.
 And once again, univalence means $M$ is a set and $\leq$ is antisymmetric.
\end{example}

\Cref{eg:topological-spaces,eg:suplattices} illustrate both the potential and pitfalls of using functorial signatures to represent higher-order theories.
On one hand, unlike ordinary higher-order logic, our functorial signatures come with a canonical notion of \emph{non-invertible morphism} between structures.\index{morphism!non-invertible}
By taking care with the representation, we can often arrange that this notion coincides with some desired one, including notions of morphism that behave either covariantly or contravariantly on subsets.

On the other hand, this flexibility comes at a cost: we have to encode single powersets using double powersets, and in general there will be many different ways to encode a particular higher-order theory.
In particular, we do not expect that any one general method of translating higher-order theories into functorial theories would produce the desired result in all cases; some customization will usually be required.
However, to make the point about the extreme generality of functorial signatures, we will sketch a proof of the following.

\begin{theorem}
  Assume propositional resizing.\index{axiom!propositional resizing}
  Then for any exofinite multi-sorted relational higher-order signature $\mathcal{S}$, there is a height-2 functorial theory $(\L,T)$ such that the type of $\mathcal{S}$-structures is equivalent to the type of $(\L,T)$-models.\index{signature!of higher-order logic}
\end{theorem}

Note that although $\mathcal{S}$ is only a signature (with no axioms), we have to impose some axioms on the $\L$-structures to obtain an equivalence.
A theory over $\mathcal{S}$, of course, can then be transferred to a larger theory over $\L$.

\begin{proof}
  As we have seen in \cref{eg:topological-spaces,eg:suplattices}, when translating a higher-order signature to a functorial theory, the difficult question is how to make type constructions formed from iterated powersets into covariant exo-functors.
  The general idea is that when powersets are iterated an even number of times, they are already covariantly exo-functorial; while when they are iterated an odd number of times, we can apply an extra ``unnecessary'' powerset to make them so.
  This requires keeping track of parity, or equivalently variance, in the types of a higher-order signature.
  Thus, we will work with the following slightly idiosyncratic definition of higher-order signature, which is nevertheless equivalent to any of the usual (purely relational) formulations, at least for the purpose of defining structures.

  First, let the \emph{higher-order operations} of arity $n:\exo{\Nat}$ be the exo-functors
  \[ F: \U^n \times (\U^{\mathrm{op}})^n \to \U \]
  inductively generated as composites of projections $\pi_i:\U^n \to \U$, cartesian products $\U\times \U\to \U$, and powersets $P : \U^{\mathrm{op}}\to \U$, where $PA \eqdef (A \to \Prop)$.
  Such an $F$ can be thought of as a formal expression involving $n$ type variables built from cartesian products and powersets; the domain $\U^n \times (\U^{\mathrm{op}})^n$ records separately the covariant and contravariant occurrences of each type variable.
  For instance, in the formal expression $X\times P(X\times PY)$, the first occurrence of $X$ is covariant while the second is contravariant, and the only occurrence of $Y$ is covariant since it is nested within two powersets.
  Thus, we would represent this as the higher-order operation
  \[ ((X,Y),(X',Y')) \mapsto X\times P(X'\times PY). \]
  Note that since projections and products preserve both embeddings and surjections, while $P$ interchanges embeddings and surjections, any higher-order operation $F$ takes an input of $n$ embeddings and $n$ surjections to an embedding, and $n$ surjections and $n$ embeddings to a surjection.

  Composing the action of a higher-order operation $F$ on objects with the diagonal, we obtain a function (not a functor!)
  \[ \ob{F}\Delta : \U^n \to \U^n \times \U^n \to \U. \]
  The idea is that here we forget the variances and identify the covariant and contravariant occurrences of each type variable.

  We now define a (finite, relational) \emph{higher-order signature} to consist of:
  \begin{itemize}
  \item An exo-natural number $n:\exo{\Nat}$.
    We write the elements of $\exo{\Nat}_{<n}$ as $A_1$,\ldots, $A_n$ and call them the \emph{base sorts}.
  \item An exo-natural number $m:\exo{\Nat}$.
    We write the elements of $\exo{\Nat}_{<m}$ as $R_1$,\ldots, $R_m$ and call them the \emph{relation symbols}.
  \item For each relation symbol, a higher-order operation $F_R$ called its \emph{domain}.
  \end{itemize}
  A \emph{structure} for such a signature consists of
  \begin{itemize}
  \item For each base sort $A$, a set $MA : \SetU$.
    (Note that we consider only models in sets, not higher types.)
  \item For each relation symbol $R$ with domain $F_R$, a predicate
    \[MR : \ob{(F_R)}\Delta(M) \to \Prop.\]
  \end{itemize}
  Note that the definition of structures does not use the exo-functorial action of the higher-order operations $F_R$, only their action on objects composed with the diagonal.
  For this reason, signatures for higher-order logic do not usually track the variance of occurrences of type variables.
  However, a syntactic type expression uniquely determines a variance for each occurrence by counting the powersets it appears inside, so it is always possible to extract one of our higher-order signatures from a more ordinary one; we leave it to the reader to make this precise.

  Now, given such a higher-order signature, let $\L$ be the height-2 functorial signature with $\bottom{\L} \eqdef \exo{\Nat}_{<n}$ and $\derivcat{\L}{}$ the composite exo-functor
  \[ [\bottom{\L},\U] \cong \U^n \xrightarrow{\Delta} \U^n \times \U^n \xrightarrow{1 \times P} \U^n \times (\U^{\mathrm{op}})^n \xrightarrow{(F_j)_{0\le j<m}} \U^m \xrightarrow{\Sigma_m} \Usharp \cong \Sig(1).
  \]
  Here $\Sigma_m$ denotes the exo-$\Sigma$-type over $\exo{\Nat}_{<m}$, i.e., $\Sigma_m X \eqdef \sm{j:\exo{\Nat}_{<m}} X_j$, which is sharp since each $X_j$ is fibrant.
  (Recall that $\Sig(1)$ is just the exo-catgory of sharp exotypes.)
  A structure for this signature $\L$ then consists of a family of types $M:\U^n$ and a type family $\derivdia{M} : \derivcat{\L}{M} \to \U$.

  For instance, suppose $\mathcal{S}$ is a higher-order signature with two sorts and two relations symbols, with domain exo-functors sending $(X,Y,X',Y') : \U^2 \times \U^2$ to
  \[ X\times P(X'\times PY) \qquad\text{and}\qquad X\times P Y'.
  \]
  Then $\derivdia{\L} : \U^2 \to \Usharp$ is defined (up to exo-isomorphism) by
  \[ \derivdia{\L}(X,Y) \eqdef (X\times P(PX \times PY)) \exosum (X\times PPY).\]
  Thus, a structure for this functorial signature consists of two types $X$ and $Y$ and two families of types indexed by $X\times P(PX \times PY)$ and $X\times PPY$ respectively.

  Now for any $A:\U$, there is a map $A \to PA$ sending $a:A$ to the ``singleton'' $\lambda x. \Vert x=a\Vert$; and if $A$ is a set, then this map is an embedding.
  Thus, by contravariance, for any higher-order operation $F$ we have a map
  \[ \ob{F}(1\times P)\Delta(M) \to \ob{F}\Delta(M) \]
  which is a surjection if $M$ is a family of sets.
  Therefore, an $\mathcal{S}$-structure is equivalent to an $\L$-structure $M$ such that
  \begin{enumerate}
  \item $\derivdia{M}$ consists of propositions;\label{item:ho1}
  \item $M:\U^n$ consists of sets; and\label{item:ho2}
  \item $\derivdia{M} : \derivcat{\L}{M} \to \Prop$ factors through
    \[\derivcat{\L}{M} \converts
      \Sigma_{j:\exo{\Nat}_{<m}} \ob{(F_j)}(1\times P)\Delta(M)
      \twoheadrightarrow 
      \Sigma_{j:\exo{\Nat}_{<m}} \ob{(F_j)}\Delta(M)
      \]
      (necessarily uniquely, by surjectivity and since $\Prop$ is a set).
  \end{enumerate}
  This is a subtype of $\Struc{\L}$; hence by our very general notion of ``theory'', it is the type of models of a theory $T$ over $\L$.
  (Note also that~(\ref{item:ho1}) is equivalent to univalence of $\derivdia{M}$, while~(\ref{item:ho2}) is then equivalent to univalence of $M$ if $\mathcal{S}$ includes equality relations.)
\end{proof}

We end with an example suggesting that there may at least be some interest in non-diagram signatures of height greater than 2.

\begin{example}[Ultracategories~\cite{makkai:sdfol,lurie:ultracats,ct:mtm}]\label{eg:ultracats}\index{ultracategory}\index{ultrafilter}
  There are two notions of ``ultracategory'' in the literature.
  A Makkai--Lurie ultracategory~\cite{makkai:sdfol,lurie:ultracats} is a category $\C$ equipped with a functor $\int_S (-)d\mu : \C^S\to\C$ for any set $S$ and any ultrafilter on $S$.
  (In fact Makkai's and Lurie's definitions differ somewhat in the axioms imposed, but the basic structure is the same.)
  This can be represented by a diagram signature in the style of \cref{eg:cat-struc}, but with the type $\sum_{S:\Set} \mathsf{Ultra}(S)$ of ``sets equipped with an ultrafilter'' indexing a family of rank-1 sorts, and similarly a family of rank-2 sorts for the functoriality of these operations.
  As usual, univalence reduces to ordinary univalence of the underlying category.
  Note that $\sum_{S:\Set} \mathsf{Ultra}(S)$ is a 1-type, so this example exhibits behavior similar to \cref{eg:unbiased-monoidal}.
  It is also ``large and universe-sensitive'' in the same way as the ``suprema of families'' presentation of suplattices (\cref{eg:suplattices}).

  By contrast, a Clementino--Tholen ultracategory~\cite{ct:mtm} is more like a multicategory: it has a set $O$ of objects together with, for every ultrafilter $\mathfrak{x}$ \emph{on the set $O$} and every $y:O$, a hom-set $A(\mathfrak{x},y)$, with composition operations and axioms.
  We can represent this with a height-3 functorial signature with $\bottom{\L}\eqdef \onetype$, and for $MO:\U$
  \[\bottom{\derivdia{\L}{MO}} \eqdef \mathsf{Ultra}(MO) \times MO, \]
  with the top rank encoding the identity and composition operations as in \cref{eg:multicats,eg:fat-sym-multicats}.
  We mention this because it is our only example of a non-diagram signature of height greater than 2, but we have not investigated it in detail.
  In particular, since we can in general expect $MO$ to be a proper 1-type rather than a set, the correct notion of ``ultrafilter'' on it is perhaps not entirely clear.
\end{example}

\chapter{Conclusion}\label{sec:conclusion}

In \cref{sec:struc-folds}, we described the univalence principle as stating that ``equivalent mathematical structures are indistinguishable.''
In the course of this work, we have made this statement precise in \cref{thm:hsip2}.
Specifically, we have given a precise notion of ``mathematical structure'', through our notion of signature in \cref{def:abstract_signature} and (univalent) structure in \cref{def:functorial-structure}, and of ``equivalence'', through our notion of equivalence of structure in~\cref{def:vss}.
The notion of ``indistinguishability'' is given by the mathematical foundation we are working in, specifically, by the type of identifications.
Our main theorem identifies any two equivalent structures.

The key to this result is the notion of \emph{univalent} structure.
Using a relativized form of the \emph{identity of indiscernibles},
we defined a general notion of indiscernibility of objects in a categorical structure, yielding a notion of univalence for such structures.
These notions depend only on the shape of the structures as specified by the signature, not on any axioms they satisfy.
We then showed, in \cref{thm:hsip2}, a univalence principle for univalent structures that specializes to known results for first-order logic and univalent 1-categories, as well as many other important examples.

Regarding the setting we have chosen for our work, it seems impossible to define a fully coherent notion of
signature without 2LTT.  A sufficiently-coherent
``wild'' notion (in the sense of~\cite{CapriottiKraus:csst}) might suffice for our particular results,
but further development of the theory may require the fully coherent version. In addition, 2LTT seems to be
necessary for treating diagram signatures of arbitrary height (cf.\ \cref{sec:folds-signatures-two}).

In this paper we have focused on laying out the basic definitions, proving the fundamental univalence principle, and describing a large number of examples to show the wide applicability of the theory.
However, there are many important questions that we have left open, including the following.
\begin{itemize}
\item Can we remove the splitness condition from \cref{thm:hsip}, as discussed at the end of \cref{sec:hsip}?
\item Is there a completion operation for structures, i.e., a universal way to turn a structure into a univalent one, generalizing the Rezk completion for categories \cite[Section~8]{AKS13}?
\item As discussed in \cref{eg:ax-weq}, it should be the case that axioms expressed in Makkai's language FOLDS are invariant under our notion of weak equivalence.
\item Also as discussed in \cref{eg:ax-weq}, is there a weak-equivalence-invariant notion of ``axiom'' that also includes our examples involving non-diagram signatures?
\item As discussed in \cref{rmk:heteq}, can we prove a general theorem that univalent structures with fully heterogeneous equality consist of sets, and is there a general method to add heterogeneous equalities to only some sorts?
\item Can the theory of univalence be extended from our functorial signatures to a wider class of Generalized Algebraic Theories?
  In particular, can we deal directly with theories that include functions, perhaps by finding a uniform way to encode their graphs as relations?
  (We thank Steve Awodey for raising this question to us.)
\item We have so far considered only signatures of finite height, which permit arguments by induction.
  Can the theory be extended to signatures of infinite height, perhaps using coinduction?
\item The results presented here should be formalizable in a computer proof assistant implementing 2LTT.
\end{itemize}

\backmatter


\printnomenclature

\newcommand{\indexsee}[2]{\index{#1|see{#2}}}

\indexsee{type!exo-}{exotype}
\indexsee{semicategory}{category, semi-}
\indexsee{category, double}{double category}
\indexsee{bicategory, double}{double bicategory}
\indexsee{gregarious equivalence}{equivalence, gregarious}
\indexsee{combinatorial species}{species, combinatorial}
\indexsee{partial order}{poset}
\indexsee{Kleisli category}{category, Kleisli}
\indexsee{Eilenberg-Moore category}{category, Eilenberg-Moore}
\indexsee{exo-equality}{equality, exo-}
\indexsee{identification}{type, identity}
\indexsee{contractible}{type, contractible}
\indexsee{univalence axiom}{axiom, univalence}
\indexsee{propositional truncation}{truncation, propositional}
\indexsee{finite type}{type, finite}
\indexsee{exofinite exotype}{exotype, exofinite}
\indexsee{cofibrant exotype}{exotype, cofibrant}
\indexsee{sharp exotype}{exotype, sharp}
\indexsee{equality!strict}{equality, exo-}
\indexsee{exo-isomorphism}{isomorphism, exo-}
\indexsee{discrete opfibration}{opfibration, discrete}
\indexsee{Artin gluing}{gluing, Artin}
\indexsee{functorial signature}{signature, functorial}
\indexsee{functorial structure}{structure, functorial}
\indexsee{functorial theory}{theory, functorial}
\indexsee{levelwise equivalence}{equivalence, levelwise}
\indexsee{indicator function}{function, indicator}
\indexsee{univalent functorial structure}{structure, functorial, univalent}
\indexsee{$T_0$-space}{space, $T_0$}
\indexsee{topological space}{space, topological}
\indexsee{propositional resizing}{axiom, propositional resizing}
\indexsee{resizing!propositional}{axiom, propositional resizing}
\indexsee{pointed set}{set, pointed}
\indexsee{precategory}{category, pre-}
\indexsee{E-category}{category, E-}
\indexsee{exo-category}{category, exo-}
\indexsee{exo-functor}{functor, exo-}
\indexsee{exo-natural transformation}{natural transformation, exo-}
\indexsee{inverse exo-category}{category, exo-, inverse}
\indexsee{diagram signature}{signature, diagram}
\indexsee{diagram theory}{theory, diagram}
\indexsee{axiom!excluded middle}{excluded middle}
\indexsee{functorality!of derivatives}{derivative, functoriality of}
\indexsee{exo-diagram}{diagram, exo-}

\printindex

\bibliographystyle{amsalpha}  
\bibliography{foldssatrefs}

\end{document}